\pdfoutput=1
\RequirePackage{ifpdf}
\ifpdf 
\documentclass[pdftex]{sigma}
\else
\documentclass{sigma}
\fi

\usepackage{amstext}
\usepackage{amscd}
\usepackage{enumerate}
\usepackage{mathrsfs}
\usepackage{mathtools}

\usepackage{tikz}
\usepackage{tikz-cd}
\usetikzlibrary{arrows}
\usetikzlibrary{decorations.markings}
\usetikzlibrary{patterns,decorations.pathreplacing}

\numberwithin{equation}{section}
\numberwithin{figure}{section}
\numberwithin{table}{section}

\newtheorem{Theorem}{Theorem}[section]
\newtheorem*{Theorem*}{Theorem}
\newtheorem{Corollary}[Theorem]{Corollary}
\newtheorem{Lemma}[Theorem]{Lemma}
\newtheorem{Proposition}[Theorem]{Proposition}
 { \theoremstyle{definition}
\newtheorem{Definition}[Theorem]{Definition}

\newtheorem{Example}[Theorem]{Example}
\newtheorem{Remark}[Theorem]{Remark}
\newtheorem{observation}[Theorem]{Observation}
\newtheorem{operation}[Theorem]{Operation}
\newtheorem{setting}[Theorem]{Setting}}

\newcommand{\ZZ}{\mathbb{Z}}
\newcommand{\QQ}{\mathbb{Q}}
\newcommand{\RR}{\mathbb{R}}
\newcommand{\CC}{\mathbb{C}}
\newcommand{\PP}{\mathbb{P}}
\newcommand{\TT}{\mathbb{T}}

\newcommand{\rmH}{\mathrm{H}}
\newcommand{\rmI}{\mathrm{I}}
\newcommand{\calP}{\mathcal{P}}
\newcommand{\calQ}{\mathcal{Q}}
\newcommand{\calR}{\mathcal{R}}
\newcommand{\calV}{\mathcal{V}}
\newcommand{\calX}{\mathcal{X}}
\newcommand{\calY}{\mathcal{Y}}
\newcommand{\calZ}{\mathcal{Z}}
\newcommand{\fkX}{\mathfrak{X}}

\newcommand{\pt}{\mathsf{pt}}
\newcommand{\sfh}{\mathsf{h}}

\newcommand{\sfP}{\mathsf{P}}
\newcommand{\sfQ}{\mathsf{Q}}

\newcommand{\bfp}{\mathbf{p}}
\newcommand{\bfq}{\mathbf{q}}

\newcommand{\Hom}{\operatorname{Hom}}
\newcommand{\GL}{\operatorname{GL}}

\newcommand{\Zig}{\mathsf{Zig}}
\newcommand{\Zag}{\mathsf{Zag}}
\newcommand{\zig}{\mathsf{zig}}
\newcommand{\zag}{\mathsf{zag}}
\newcommand{\PM}{\mathsf{PM}}

\newcommand{\Newt}{\mathsf{Newt}}
\newcommand{\mutation}{\mathsf{mut}}
\newcommand{\mut}{\mathsf{mut}}
\newcommand{\width}{\mathsf{width}}

\begin{document}
\allowdisplaybreaks

\newcommand{\arXivNumber}{1903.01636}

\renewcommand{\PaperNumber}{030}

\FirstPageHeading

\ShortArticleName{Deformations of Dimer Models}

\ArticleName{Deformations of Dimer Models}

\Author{Akihiro HIGASHITANI~$^{\rm a}$ and Yusuke NAKAJIMA~$^{\rm b}$}
\AuthorNameForHeading{A.~Higashitani and Y.~Nakajima}

\Address{$^{\rm a)}$~Department of Pure and Applied Mathematics, Graduate School of Information Science\\
\hphantom{$^{\rm a)}$}~and Technology, Osaka University, Osaka 565-0871, Japan} \EmailD{\href{mailto:higashitani@ist.osaka-u.ac.jp}{higashitani@ist.osaka-u.ac.jp}}

\Address{$^{\rm b)}$~Department of Mathematics, Kyoto Sangyo University,\\
\hphantom{$^{\rm b)}$}~Motoyama, Kamigamo, Kita-Ku, Kyoto, 603-8555, Japan}
\EmailD{\href{mailto:ynakaji@cc.kyoto-su.ac.jp}{ynakaji@cc.kyoto-su.ac.jp}}

\ArticleDates{Received August 06, 2021, in final form April 10, 2022; Published online April 16, 2022}

\Abstract{The combinatorial mutation of polygons, which transforms a given lattice po\-ly\-gon into another one, is an important operation to understand mirror partners for two-dimensional Fano manifolds, and the mutation-equivalent polygons give ${\mathbb Q}$-Gorenstein deformation-equivalent toric varieties. On the other hand, for a dimer model, which is a~bipartite graph described on the real two-torus, one can assign a lattice polygon called the perfect matching polygon.
It is known that for each lattice polygon $P$ there exists a dimer model having $P$ as the perfect matching polygon and satisfying certain consistency conditions. Moreover, a dimer model has rich information regarding toric geometry associated with the perfect matching polygon. In this paper, we introduce a set of operations which we call deformations of consistent dimer models, and show that the deformations of consistent dimer models realize the combinatorial mutations of the associated perfect matching polygons.}


\Classification{52B20; 14M25; 14J33}

\tableofcontents

\section{Introduction}\label{sec_intro}

\subsection{Background and motivations}
Fano manifolds are one of the well studied classes in geometry, and the classification of Fano manifolds, which has been done in low dimensions, is a fundamental problem. Here, a \emph{Fano manifold} $X$ is a complex projective manifold such that the anticanonical line bundle $-K_X$ is ample.
Recently, a new approach that uses mirror symmetry has been proposed for classifying Fano manifolds as follows.
First, a Fano manifold is expected to correspond to a certain Laurent polynomial via mirror symmetry (see~\cite{CCGGK}).
That is, a Laurent polynomial $f\in\CC\big[x_1^\pm,\dots,x_n^\pm\big]$ is said to be a \emph{mirror partner} for an $n$-dimensional Fano manifold $X$
if the Taylor expansion of the \emph{classical period} $\pi_f$ of $f$ coincides with a \emph{generating function for Gromov--Witten invariants} of $X$
(see the references quoted above for the details of these terminologies).
Furthermore, if a Fano manifold $X$ is a mirror partner of~$f$, it is expected that~$X$ admits a toric degeneration~$X_P$.
Here, $P\coloneqq \Newt(f)$ is the Newton polytope of $f$, which is defined as the convex hull of exponents of monomials of $f$, and $X_P$ is the toric variety defined by the \emph{spanning fan} of~$P$ (i.e., the fan whose cones are spanned by the faces of $P$).
Thus, Laurent polynomials having the same classical period are considered as mirror partners for the same Fano manifold $X$, and in general there are many Laurent polynomials that are mirror partners for~$X$.
In order to understand the relationship between such Laurent polynomials, an operation called the \emph{mutation} of $f$, which is a birational transformation analogue to a cluster transformation, was introduced in~\cite{GU}.
In particular, it was shown that if $f, g\in\CC\big[x_1^\pm,\dots,x_n^\pm\big]$ are transformed into each other by mutations,
then their classical periods are the same, that is, $\pi_f=\pi_g$~\cite[Lemma~1]{ACGK}.
Moreover, this mutation of Laurent polynomials induces an operation, called the \emph{combinatorial mutation}, which connects
the associated Newton polytopes $P=\Newt(f)$ and $Q=\Newt(g)$ as defined in~\cite{ACGK} (see also Section~\ref{sec_pre_mutation_polygon}).
Also, it was shown in \cite{Ilt} that if $P$ and $Q$ are Fano polytopes and are transformed into each other by combinatorial mutations,
then the associated toric varieties~$X_P$ and~$X_Q$ are related by a \emph{$\QQ$-Gorenstein $(=$~qG$)$ deformation}; that is, there exists a flat family $\fkX\rightarrow\PP^1$ such that the relative canonical divisor is $\QQ$-Cartier and $\fkX_0\cong X_P$, $\fkX_\infty\cong X_Q$, where~$\fkX_p$ is the fiber of $p\in\PP^1$.
Thus, it has been conjectured that there is a bijection between qG-deformation equivalence classes of ``class TG'' Fano manifolds and mutation equivalence classes of Fano polytopes. There have previously been several affirmative results (e.g., \cite{Akhtar_etc,KNP}).

\subsection{Results}
As mentioned above, the combinatorial mutation of polytopes is quite important in mirror symmetry of Fano manifolds.
In this paper, we focus on the two-dimensional case, and the polygons which we are interested in are not necessarily Fano.
First, it is known that any lattice polygon in~$\RR^2$ can be realized as the \emph{perfect matching polygon} $\Delta_\Gamma$ of a dimer model $\Gamma$ satisfying the consistency condition (see Sections~\ref{sec_pre_dimer} and~\ref{sec_zigzag}).
A \emph{dimer model} is a bipartite graph on the real two-torus (see Section~\ref{sec_pre_dimer} for more details), which was first introduced in the field of statistical mechanics.
Since the 2000s, string theorists have been using dimer models to study quiver gauge theories (e.g., \cite{FHK,HV,Kenn,Keny} and references therein).
As a result, relationships between dimer models and many branches of mathematics have been discovered, e.g., see \cite{Boc_abc} and references therein.
(Note that a bipartite graph described on other surfaces, which is also called a dimer model, has also been studied actively.
For example, dimer models on the disk are related to Grassmannians and their mirror symmetry, e.g., see \cite{HN, RW}.)
Based on this background, we expect that there is a certain operation on a consistent dimer model that realizes
the combinatorial mutation of the associated perfect matching polygon.
In this paper, we introduce a concept called the \emph{deformations of consistent dimer models}.
Note that there is another operation on dimer models, called \emph{mutation of dimer models} (see Appendix~\ref{app_mutationdimer}).
Although a mutation changes the shape of a dimer model, it does not change the associated perfect matching polygon.
Thus, we need a~new operation different from the mutation to realize our expectation.\looseness=-1

To explain our main theorem, we briefly recall the combinatorial mutations of polygons (see Section~\ref{sec_pre_mutation_polygon} for a more precise definition).
Let $N\cong\ZZ^2$ be a rank two lattice and $M\coloneqq\Hom_\ZZ(N,\ZZ)\cong\ZZ^2$.
First, we consider a lattice polygon $P$ in $N_\RR\coloneqq N\otimes_\ZZ \RR$ and choose an edge~$E$ of~$P$.
We then take a primitive inner normal vector $w\in M$ for $E$, and consider the linear map $\langle w,-\rangle\colon N_\RR\rightarrow \RR$.
Using these notions, we determine the \emph{height} $\langle w,u\rangle$ of each point $u\in P$.
In particular, a primitive lattice element $u_E\in N$ satisfying $\langle w,u_E\rangle=0$ plays an important role in defining the combinatorial mutation.
Such an element $u_E$ is determined uniquely up to sign; thus we fix one of them.
Then, we define the line segment $F\coloneqq\operatorname{conv}\{{\bf 0},u_E\}$, which is called a~\emph{factor} of~$P$ with respect to~$w$.
Using these data, we have the lattice polygon $\mut_w(P,F)$, which is called the \emph{combinatorial mutation} of $P$ given by the vector $w$ and the factor $F$, as defined in Definition~\ref{def_mutation_polygon}.
In addition, we also define another combinatorial mutation $\mut_w(P,-F)$ in a similar way.
We note that although $\mut_w(P,F)$ looks different from $\mut_w(P,-F)$, they are transformed into each other by a $\GL(2,\ZZ)$-transformation.
On the other hand, for the given lattice polygon $P$ there exists a consistent dimer model $\Gamma$ such that $P=\Delta_\Gamma$.

Based on this background, the deformation of a consistent dimer model is compatible with the above combinatorial mutations in the following sense.
Let $\Gamma$ be the consistent dimer model.
The deformations of $\Gamma$ are defined for a certain set of ``\emph{zigzag paths}" $\{z_1,\dots,z_r\}$ on $\Gamma$ corresponding to the vector $-w$ (see Section~\ref{sec_zigzag} concerning zigzag paths),
and there are two kinds of deformations which we call the \emph{zig-deformation} and the \emph{zag-deformation},
which are respectively denoted by $\nu_\bfp^\zig(\Gamma,\{z_1,\dots,z_r\})$ and $\nu_\bfp^\zag(\Gamma,\{z_1,\dots,z_r\})$; see Definitions~\ref{def_deformation_zig} and~\ref{def_deformation_zag}.
For some cases, these deformations are compatible with the combinatorial mutations of the perfect matchings as follows.\looseness=-1

\begin{Theorem}[{see Proposition~\ref{prop_preserve_consistent1} and Theorem~\ref{mutation=deformation1} for more details}]\label{thm_main_intro}
In addition to the above settings, we also assume that either one of the following conditions is satisfied:
\begin{itemize}\itemsep=0pt
\item[$(i)$] $r=1$ or
\item[$(ii)$] $\Gamma$ is a hexagonal or rectangular dimer model $($see Definition~{\rm \ref{def_hexagonal_square})}.
\end{itemize}
Then the deformations $\nu_\bfp^\zig(\Gamma,\{z_1,\dots,z_r\})$ and $\nu_\bfp^\zag(\Gamma,\{z_1,\dots,z_r\})$ are consistent dimer models,
and we have
\begin{gather*}\mutation_{w}(\Delta_\Gamma,F)=\Delta_{\nu^\zig_\bfp(\Gamma,\{z_1,\dots,z_r\})},\qquad
\mutation_{w}(\Delta_\Gamma,-F)=\Delta_{\nu^\zag_\bfp(\Gamma,\{z_1,\dots,z_r\})},
\end{gather*}
where $\Delta_{\nu^\zig_\bfp(\Gamma,\{z_1,\dots,z_r\})}$ and $\Delta_{\nu^\zag_\bfp(\Gamma,\{z_1,\dots,z_r\})}$ are the perfect matching polygons of
$\nu_\bfp^\zig(\Gamma,\{z_1,\dots,\allowbreak z_r\})$ and $\nu_\bfp^\zag(\Gamma,\{z_1,\dots,z_r\})$, respectively.
\end{Theorem}

Note that the combinatorial mutations $\mutation_{w}(\Delta_\Gamma,\pm F)$ are defined for the perfect matching polygon $\Delta_\Gamma$,
whereas the deformations $\nu^{\zig/\zag}_\bfp(\Gamma,\{z_1,\dots,z_r\})$ are defined independently of $\Delta_\Gamma$.
When we drop the both conditions~$(i)$ and $(ii)$ in Theorem~\ref{thm_main_intro}, the deformed dimer model does not necessarily realize the combinatorial mutation.
In order to obtain the same results for a~general situation, we have to add some algorithmic operations to the definitions of deformations.
We formulate such operations
as the \emph{extended zig-deformation} $\nu_\calX^\zig(\Gamma,\{z_1,\dots,z_r\})$ and \emph{extended zag-deformation} $\nu_\calY^\zag(\Gamma,\{z_1,\dots,z_r\})$ as in Definitions~\ref{def_exdeformation_zig} and \ref{def_exdeformation_zag}.
Then we can prove the statement in Theorem~\ref{thm_main_intro} for the extended deformations (see Proposition~\ref{prop_make_consistent} and Theorem~\ref{mutation=deformation}).
By these results, the perfect matching polygons of the (extended) deformation of dimer models satisfy
the properties which are exactly the same as the combinatorial mutation of a polygon (see Section~\ref{mutation_vs_deformation}).

Since dimer models are related with many branches of mathematics and physics, we considered that it would be of interest to compare certain objects (e.g., perfect matchings, the associated toric variety) related with
a consistent dimer model $\Gamma$ to those of $\nu_\calX^\zig(\Gamma,\{z_1,\dots,z_r\})$ (or $\nu_\calY^\zag(\Gamma,\{z_1,\dots,z_r\})$).

The structure of this paper is as follows.
In Section~\ref{sec_pre_dimer}, we introduce dimer models and related concepts. In particular, the notion of the perfect matching polygon introduced in this section is one of main concepts in this paper.
In Section~\ref{sec_zigzag}, we introduce the notion of zigzag paths, which are special paths on a dimer model.
We then define the consistency condition using zigzag paths, and discuss the relationships between perfect matchings and zigzag paths on a~consistent dimer model.
Next, we focus specifically on type I zigzag paths, which are zigzag paths having typical properties.
Using these type~I zigzag paths, we introduce the deformations of consistent dimer models in Section~\ref{sec_def_deform}.
In Section~\ref{sec_mutation_polygon1}, we compare the perfect matching polygon of the deformed dimer model with
the combinatorial mutation of the perfect matching polygon of the original dimer model.
In particular, we see that these polygons coincide in some cases as in Theorem~\ref{thm_main_intro} (see Theorem~\ref{mutation=deformation1}).
In Section~\ref{sec_def_exdeform}, we consider some algorithmic operations and define the extended deformations of consistent dimer models.
In Section~\ref{sec_proof}, we show the fundamental properties of the extended deformation of consistent dimer models.
In particular, we prove the compatibility of the extended deformations and the combinatorial mutations for general situations as in Theorem~\ref{mutation=deformation}, and this induces Theorem~\ref{mutation=deformation1}.
As we will mention in Remark~\ref{rem_def_deform_not_unique}, the extended deformation of a consistent dimer model
depends on the choice of some of the data used in the processes (zig-5), (zag-5), as explained in
Definitions~\ref{def_exdeformation_zig} and~\ref{def_exdeformation_zag} (see also Operation~\ref{operation_BIU}),
whereas the perfect matching polygon of the extended deformation of a consistent dimer model is determined uniquely.
This ambiguity is caused by the fact that there are several consistent dimer models giving the same perfect matching polygon.
However, it has been conjectured that such consistent dimer models are transformed into each other by \emph{mutations of dimer models}, and hence ``conjecturally" our deformation of a consistent dimer model is determined uniquely up to the mutation.
We include a survey of this mutation of dimer models in Appendix~\ref{app_mutationdimer}.
In Appendix~\ref{app_large_example}, we provide an additional example of the extended deformation, the explanation of which is too lengthy to include in the main body of this paper.

\section{Dimer models and perfect matching polygons}\label{sec_pre_dimer}

\subsection{What is a dimer model?}
We first introduce dimer models. Some ideas and concepts contained in this subsection are originally derived from theoretical physics (e.g., \cite{FHK,HV}).

A \emph{dimer model} (or \emph{brane tiling}) $\Gamma$ is a finite bipartite graph on the real two-torus~$\TT\coloneqq\RR^2/\ZZ^2$; that is, the set~$\Gamma_0$ of nodes is divided into two parts~$\Gamma_0^+$,~$\Gamma_0^-$, and the set~$\Gamma_1$ of edges consists of the edges connecting nodes in~$\Gamma_0^+$ with those in~$\Gamma_0^-$.
In order to make the situation clear, we color the nodes in~$\Gamma_0^+$ white, and those in~$\Gamma_0^-$ black.
A connected component of $\TT{\setminus}\Gamma_1$ is called a \emph{face} of $\Gamma$, and we denote by~$\Gamma_2$ the set of faces.
We also obtain the bipartite graph $\widetilde{\Gamma}$ on $\RR^2$ induced via the universal cover $\RR^2\rightarrow\TT$.
We call $\widetilde{\Gamma}$ the universal cover of the dimer model~$\Gamma$.
For example, the bipartite graph shown on the left side of Figure~\ref{ex_dimer_4b} is a dimer model
where the outer frame is the fundamental domain of~$\TT$.

As the dual of a dimer model $\Gamma$, we define the quiver $Q_\Gamma$ associated with $\Gamma$.
Namely, we assign a vertex dual to each face in $\Gamma_2$, and an arrow dual to each edge in $\Gamma_1$.
The orientation of arrows is determined so that the white node is on the right of the arrow.
For example, the right side of Figure~\ref{ex_dimer_4b} is the quiver associated with the dimer model on the left.
Sometimes we simply denote the quiver $Q_\Gamma$ by $Q$.

\newcommand{\basicdimerB}{
\coordinate (W1) at (5,1); \coordinate (W2) at (1,3); \coordinate (W3) at (3,5);
\coordinate (B1) at (3,1); \coordinate (B2) at (1,5); \coordinate (B3) at (5,3);
\draw[line width=\edgewidth] (0,0) rectangle (6,6);
\draw[line width=\edgewidth] (W2)--(B1)--(W1)--(B3)--(W3)--(B2)--(W2); \draw[line width=\edgewidth] (B1)--(W3);
\draw[line width=\edgewidth] (0,3)--(W2); \draw[line width=\edgewidth] (3,0)--(B1); \draw[line width=\edgewidth] (6,0)--(W1);
\draw[line width=\edgewidth] (6,3)--(B3); \draw[line width=\edgewidth] (3,6)--(W3); \draw[line width=\edgewidth] (0,6)--(B2);
\draw [line width=\nodewidth, fill=black] (B1) circle [radius=\noderad] ;
\draw [line width=\nodewidth, fill=black] (B2) circle [radius=\noderad] ;
\draw [line width=\nodewidth, fill=black] (B3) circle [radius=\noderad] ;
\draw [line width=\nodewidth, fill=white] (W1) circle [radius=\noderad] ;
\draw [line width=\nodewidth, fill=white] (W2) circle [radius=\noderad] ;
\draw [line width=\nodewidth, fill=white] (W3) circle [radius=\noderad] ;
}

\begin{figure}[h!]\centering
\newcommand{\edgewidth}{0.07cm}
\newcommand{\nodewidth}{0.07cm}
\newcommand{\noderad}{0.24} 
\newcommand{\arrowwidth}{0.08cm}
\begin{tikzpicture}
\node (DM) at (0,0)
{\scalebox{0.45}{
\begin{tikzpicture}
\basicdimerB
\end{tikzpicture}
} };

\node (QV) at (5,0)
{\scalebox{0.45}{
\begin{tikzpicture}[sarrow/.style={black, -latex, very thick}]
\node (Q1) at (2,3.5){{\LARGE$1$}}; \node (Q0) at (4,2.5){{\LARGE$0$}};
\node (Q3) at (5,5){{\LARGE$3$}}; \node (Q2) at (1,1){{\LARGE$2$}};
\draw[line width=\edgewidth] (0,0) rectangle (6,6);

\draw[lightgray, line width=\edgewidth] (W2)--(B1)--(W1)--(B3)--(W3)--(B2)--(W2); \draw[lightgray, line width=\edgewidth] (B1)--(W3);
\draw[lightgray, line width=\edgewidth] (0,3)--(W2); \draw[lightgray, line width=\edgewidth] (3,0)--(B1);
\draw[lightgray, line width=\edgewidth] (6,0)--(W1); \draw[lightgray, line width=\edgewidth] (6,3)--(B3);
\draw[lightgray, line width=\edgewidth] (3,6)--(W3); \draw[lightgray, line width=\edgewidth] (0,6)--(B2);

\draw [line width=\nodewidth, draw=lightgray, fill=lightgray] (B1) circle [radius=\noderad] ;
\draw [line width=\nodewidth, draw=lightgray, fill=lightgray] (B2) circle [radius=\noderad] ;
\draw [line width=\nodewidth, draw=lightgray, fill=lightgray] (B3) circle [radius=\noderad] ;
\draw [line width=\nodewidth, draw=lightgray, fill=white] (W1) circle [radius=\noderad] ;
\draw [line width=\nodewidth, draw=lightgray, fill=white] (W2) circle [radius=\noderad] ;
\draw [line width=\nodewidth, draw=lightgray, fill=white] (W3) circle [radius=\noderad] ;

\draw[sarrow, line width=\arrowwidth] (Q0)--(Q1); \draw[sarrow, line width=\arrowwidth] (Q3)--(Q0); \draw[sarrow, line width=\arrowwidth] (Q1)--(Q2);
\draw[sarrow, line width=\arrowwidth] (Q2)--(3,0); \draw[sarrow, line width=\arrowwidth] (3,6)--(Q3); \draw[sarrow, line width=\arrowwidth] (Q3)--(4.7,6);
\draw[sarrow, line width=\arrowwidth] (4.7,0)--(Q0); \draw[sarrow, line width=\arrowwidth] (Q0)--(6,1.5); \draw[sarrow, line width=\arrowwidth] (0,1.5)--(Q2);
\draw[sarrow, line width=\arrowwidth] (Q2)--(0,3); \draw[sarrow, line width=\arrowwidth] (6,3)--(Q3); \draw[sarrow, line width=\arrowwidth] (Q3)--(6,4.5);
\draw[sarrow, line width=\arrowwidth] (0,4.5)--(Q1);
\draw[sarrow, line width=\arrowwidth] (1.3,0)--(Q2); \draw[sarrow, line width=\arrowwidth] (Q1)--(1.3,6); \draw[sarrow, line width=\arrowwidth] (Q2)--(0,0);
\draw[sarrow, line width=\arrowwidth] (6,6)--(Q3);
\end{tikzpicture}
} };

\end{tikzpicture}
\caption{A dimer model and the associated quiver.}\label{ex_dimer_4b}
\end{figure}

The \emph{valency} of a node is the number of edges incident to that node.
We say that a~node on a~dimer model is \emph{$n$-valent} if its valency is $n$.
We then define several operations on a dimer model.
The \emph{join move} is the operation removing a $2$-valent node and joining the two distinct nodes connected to it as shown in Figure~\ref{split_join}.
Thus, using join moves we obtain a dimer model having no $2$-valent nodes.
We say that a dimer model is \emph{reduced} if it has no $2$-valent nodes.
We can see that the quiver associated with a reduced dimer model contains no $2$-cycles.
On the other hand, there is the operation called the \emph{split move}, which inserts a $2$-valent node (see Figure~\ref{split_join}).

We say that two reduced dimer models $\Gamma$, $\Gamma^\prime$ are \emph{isomorphic}, which is denoted by $\Gamma\cong\Gamma^\prime$, if their underlying cell decompositions of $\TT$ are homotopy-equivalent.

\begin{figure}[h!]\centering
{\scalebox{1}{
\begin{tikzpicture}

\node at (3.8,0.5) {join move} ; \node at (3.8,-0.5) {split move} ;
\draw[->, line width=0.03cm] (2.6,0.15)--(5,0.15);
\draw[<-, line width=0.03cm] (2.6,-0.15)--(5,-0.15);

\node (Dimer_a) at (0,0)
{\scalebox{0.7}{
\begin{tikzpicture}
\coordinate (b1) at (-1,0); \coordinate (b2) at (1,0);
\coordinate (w1) at (-2,1); \coordinate (w2) at (-2.5,0); \coordinate (w3) at (-2,-1);
\coordinate (w4) at (0,0); \coordinate (w5) at (2,1); \coordinate (w6) at (2,-1);

\draw[line width=0.05cm] (w1)--(b1); \draw[line width=0.05cm] (w2)--(b1);
\draw[line width=0.05cm] (w3)--(b1); \draw[line width=0.05cm] (b1)--(w4); \draw[line width=0.05cm] (w4)--(b2);
\draw[line width=0.05cm] (w5)--(b2); \draw[line width=0.05cm] (w6)--(b2);
\filldraw [line width=0.05cm, fill=black] (b1) circle [radius=0.16] ; \filldraw [line width=0.05cm, fill=black] (b2) circle [radius=0.16] ;
\draw [line width=0.05cm, fill=white] (w4) circle [radius=0.16] ;
\end{tikzpicture} }} ;

\node (Dimer_b) at (7,0)
{\scalebox{0.7}{
\begin{tikzpicture}
\coordinate (b1) at (0,0);
\coordinate (w1) at (-1,1); \coordinate (w2) at (-1.5,0); \coordinate (w3) at (-1,-1);
\coordinate (w5) at (1,1); \coordinate (w6) at (1,-1);
\draw[line width=0.05cm] (w1)--(b1); \draw[line width=0.05cm] (w2)--(b1);
\draw[line width=0.05cm] (w3)--(b1); \draw[line width=0.05cm] (w5)--(b1); \draw[line width=0.05cm] (w6)--(b1);
\filldraw [line width=0.05cm, fill=black] (b1) circle [radius=0.16] ;
\end{tikzpicture} }} ;

\end{tikzpicture}
}}

\caption{An example of the join and split move.}\label{split_join}
\end{figure}

We sometimes focus on the following type of dimer models.

\begin{Definition}\label{def_hexagonal_square}
We say that a dimer model $\Gamma$ is \emph{hexagonal} (resp.\ \emph{rectangular})
if any face of~$\Gamma$ is hexagon (resp.\ rectangle) and any node of $\Gamma$ is $3$-valent (resp.\ 4-valent).
In particular, a~hexagonal (resp.\ rectangular) dimer model is homotopy-equivalent to a dimer model whose faces are all regular hexagons (resp.\ squares).
\end{Definition}

\subsection{Perfect matchings and the perfect matching polygon}\label{subsection_PM}

Next, we assign a lattice polygon to each dimer model.
For this purpose, we will introduce the notion of perfect matchings, and we construct a polygon called the \emph{perfect matching polygon}.

\begin{Definition}A \emph{perfect matching} (or \emph{dimer configuration}) on a dimer model $\Gamma$ is a subset $\sfP$ of $\Gamma_1$ such that each node is
the end point of precisely one edge in $\sfP$.
\end{Definition}

In general, not every dimer model necessarily has a perfect matching.
In this paper, we will mainly discuss consistent dimer models, and such dimer models have perfect matchings.
Moreover, we can extend a perfect matching $\sfP$ to one on $\widetilde{\Gamma}$ via the universal cover $\RR^2\rightarrow\TT$.
We call this a \emph{perfect matching} on $\widetilde{\Gamma}$, and use the same notation $\sfP$.
For example, some perfect matchings on the dimer model given in Figure~\ref{ex_dimer_4b} are shown in Figure~\ref{pm_4b}.
(This dimer model has eight perfect matchings in total.)

\begin{figure}[h!]\centering
\newcommand{\edgewidth}{0.07cm}
\newcommand{\nodewidth}{0.07cm}
\newcommand{\noderad}{0.24} 

\newcommand{\pmwidth}{0.5cm}
\newcommand{\pmcolor}{red}

\begin{tikzpicture}
\node at (0,-1.6) {$\sfP_0$};
\node at (3.2,-1.6) {$\sfP_1$}; \node at (6.4,-1.6) {$\sfP_2$}; \node at (9.6,-1.6) {$\sfP_3$}; \node at (12.8,-1.6) {$\sfP_4$};

\node (PM0) at (0,0)
{\scalebox{0.35}{
\begin{tikzpicture}
\draw[line width=\pmwidth, color=\pmcolor] (W1)--(B3); \draw[line width=\pmwidth, color=\pmcolor] (W3)--(B1);
\draw[line width=\pmwidth, color=\pmcolor] (W2)--(B2);
\basicdimerB
\end{tikzpicture} }};

\node (PM1) at (3.2,0)
{\scalebox{0.35}{
\begin{tikzpicture}
\draw[line width=\pmwidth, color=\pmcolor] (W3)--(B1);
\draw[line width=\pmwidth, color=\pmcolor] (B3)--(6,3); \draw[line width=\pmwidth, color=\pmcolor] (W2)--(0,3);
\draw[line width=\pmwidth, color=\pmcolor] (W1)--(6,0); \draw[line width=\pmwidth, color=\pmcolor] (B2)--(0,6);
\basicdimerB
\end{tikzpicture} }};

\node (PM2) at (6.4,0)
{\scalebox{0.35}{
\begin{tikzpicture}
\draw[line width=\pmwidth, color=\pmcolor] (W2)--(B1); \draw[line width=\pmwidth, color=\pmcolor] (W3)--(B3);
\draw[line width=\pmwidth, color=\pmcolor] (W1)--(6,0); \draw[line width=\pmwidth, color=\pmcolor] (B2)--(0,6);
\basicdimerB
\end{tikzpicture} }};

\node (PM3) at (9.6,0)
{\scalebox{0.35}{
\begin{tikzpicture}
\draw[line width=\pmwidth, color=\pmcolor] (W2)--(B2); \draw[line width=\pmwidth, color=\pmcolor] (W1)--(B3);
\draw[line width=\pmwidth, color=\pmcolor] (W3)--(3,6); \draw[line width=\pmwidth, color=\pmcolor] (B1)--(3,0);
\basicdimerB
\end{tikzpicture} }};

\node (PM4) at (12.8,0)
{\scalebox{0.35}{
\begin{tikzpicture}
\draw[line width=\pmwidth, color=\pmcolor] (W1)--(B1); \draw[line width=\pmwidth, color=\pmcolor] (W3)--(B2);
\draw[line width=\pmwidth, color=\pmcolor] (B3)--(6,3); \draw[line width=\pmwidth, color=\pmcolor] (W2)--(0,3);
\basicdimerB
\end{tikzpicture} }};

\end{tikzpicture}
\caption{Some perfect matchings on the dimer model given in Figure~\ref{ex_dimer_4b}.}\label{pm_4b}
\end{figure}

We say that a dimer model is \emph{non-degenerate} if every edge is contained in some perfect matchings.
It is known that this non-degeneracy condition is equivalent to the \emph{strong marriage condition}; that is, the dimer model has equal numbers of black and white nodes and every proper subset of the black nodes of size $n$ is connected to at least $n+1$ white nodes (e.g., \cite[Remark~2.12]{Bro}).

Following \cite[Section~5]{IU2}, we next define the perfect matching polygon.
We first fix a perfect matching $\sfP_0$, and call this the \emph{reference perfect matching}.
For any perfect matching $\sfP$, we consider the connected components of the universal cover $\RR^2$ divided by $\sfP\cup\sfP_0$.
Then, we consider the \emph{height function} $\sfh_{\sfP,\sfP_0}$, which is a locally constant function on $\RR^2{\setminus}(\sfP\cup\sfP_0)$, defined as follows.
First, we choose a connected component of $\RR^2{\setminus}(\sfP\cup\sfP_0)$, and define the value of $\sfh_{\sfP,\sfP_0}$ as $0$.
Then, this function increases by $1$ when we cross
\begin{itemize}\itemsep=0pt
\item[--] an edge $e\in\sfP$ with the black node on the right, or
\item[--] an edge $e\in\sfP_0$ with the white node on the right,
\end{itemize}
and decreases by $1$
when we cross
\begin{itemize}\itemsep=0pt
\item[--] an edge $e\in\sfP$ with the white node on the right, or
\item[--] an edge $e\in\sfP_0$ with the black node on the right.
\end{itemize}
This function is determined up to a choice of a connected component of value $0$.
For example, Figure~\ref{fig_heightchnge} shows the height function $\sfh_{\sfP_2,\sfP_3}$ on the dimer model given in Figure~\ref{ex_dimer_4b},
where the red square stands for a fundamental domain of $\TT$, the edges in $\sfP_2$ (resp.~$\sfP_3$) are colored blue (resp.\ green),
and the number filled in each component is the value of $\sfh_{\sfP_2,\sfP_3}$.

\begin{figure}[h!]\centering
\newcommand{\edgewidth}{0.07cm}
\newcommand{\nodewidth}{0.07cm}
\newcommand{\noderad}{0.24} 
\newcommand{\arrowwidth}{0.08cm}
\newcommand{\pmwidth}{0.5cm}

\scalebox{0.33}{
\begin{tikzpicture}
\coordinate (W1a) at (5,1); \coordinate (W2a) at (1,3); \coordinate (W3a) at (3,5);
\coordinate (B1a) at (3,1); \coordinate (B2a) at (1,5); \coordinate (B3a) at (5,3);
\coordinate (W1b) at (5,7); \coordinate (W2b) at (1,9); \coordinate (W3b) at (3,11);
\coordinate (B1b) at (3,7); \coordinate (B2b) at (1,11); \coordinate (B3b) at (5,9);
\coordinate (W1c) at (11,1); \coordinate (W2c) at (7,3); \coordinate (W3c) at (9,5);
\coordinate (B1c) at (9,1); \coordinate (B2c) at (7,5); \coordinate (B3c) at (11,3);
\coordinate (W1d) at (11,7); \coordinate (W2d) at (7,9); \coordinate (W3d) at (9,11);
\coordinate (B1d) at (9,7); \coordinate (B2d) at (7,11); \coordinate (B3d) at (11,9);

\draw[line width=\pmwidth, color=blue] (W2a)--(B1a); \draw[line width=\pmwidth, color=blue] (W3a)--(B3a);
\draw[line width=\pmwidth, color=blue] (W1a)--(6,0); \draw[line width=\pmwidth, color=blue] (B2a)--(0,6);
\draw[line width=\pmwidth, color=blue] (W2b)--(B1b); \draw[line width=\pmwidth, color=blue] (W3b)--(B3b);
\draw[line width=\pmwidth, color=blue] (W1b)--(B2c); \draw[line width=\pmwidth, color=blue] (B2b)--(0,12);
\draw[line width=\pmwidth, color=blue] (W2c)--(B1c); \draw[line width=\pmwidth, color=blue] (W3c)--(B3c);
\draw[line width=\pmwidth, color=blue] (W1c)--(12,0); \draw[line width=\pmwidth, color=blue] (B2c)--(6,6);
\draw[line width=\pmwidth, color=blue] (W2d)--(B1d); \draw[line width=\pmwidth, color=blue] (W3d)--(B3d);
\draw[line width=\pmwidth, color=blue] (W1d)--(12,6); \draw[line width=\pmwidth, color=blue] (B2d)--(6,12);

\draw[line width=\pmwidth, color=green] (W2a)--(B2a); \draw[line width=\pmwidth, color=green] (W1a)--(B3a);
\draw[line width=\pmwidth, color=green] (W3a)--(B1b); \draw[line width=\pmwidth, color=green] (B1a)--(3,0);
\draw[line width=\pmwidth, color=green] (W2b)--(B2b); \draw[line width=\pmwidth, color=green] (W1b)--(B3b);
\draw[line width=\pmwidth, color=green] (W3b)--(3,12);
\draw[line width=\pmwidth, color=green] (W2c)--(B2c); \draw[line width=\pmwidth, color=green] (W1c)--(B3c);
\draw[line width=\pmwidth, color=green] (W3c)--(B1d); \draw[line width=\pmwidth, color=green] (B1c)--(9,0);
\draw[line width=\pmwidth, color=green] (W2d)--(B2d); \draw[line width=\pmwidth, color=green] (W1d)--(B3d);
\draw[line width=\pmwidth, color=green] (W3d)--(9,12);

\draw[line width=\edgewidth] (W2a)--(B1a)--(W1a)--(B3a)--(W3a)--(B2a)--(W2a); \draw[line width=\edgewidth] (B1a)--(W3a);
\draw[line width=\edgewidth] (0,3)--(W2a); \draw[line width=\edgewidth] (3,0)--(B1a); \draw[line width=\edgewidth] (6,0)--(W1a);
\draw[line width=\edgewidth] (0,6)--(B2a);
\draw[line width=\edgewidth] (W2b)--(B1b)--(W1b)--(B3b)--(W3b)--(B2b)--(W2b); \draw[line width=\edgewidth] (B1b)--(W3b);
\draw[line width=\edgewidth] (0,9)--(W2b); \draw[line width=\edgewidth] (0,12)--(B2b); \draw[line width=\edgewidth] (3,12)--(W3b);
\draw[line width=\edgewidth] (W2c)--(B1c)--(W1c)--(B3c)--(W3c)--(B2c)--(W2c); \draw[line width=\edgewidth] (B1c)--(W3c);
\draw[line width=\edgewidth] (9,0)--(B1c); \draw[line width=\edgewidth] (12,0)--(W1c); \draw[line width=\edgewidth] (12,3)--(B3c);
\draw[line width=\edgewidth] (W2d)--(B1d)--(W1d)--(B3d)--(W3d)--(B2d)--(W2d); \draw[line width=\edgewidth] (B1d)--(W3d);
\draw[line width=\edgewidth] (12,6)--(W1d);\draw[line width=\edgewidth] (6,12)--(B2d);
\draw[line width=\edgewidth] (12,9)--(B3d); \draw[line width=\edgewidth] (9,12)--(W3d);
\draw[line width=\edgewidth] (W3a)--(B1b); \draw[line width=\edgewidth] (B3a)--(W2c);
\draw[line width=\edgewidth] (B3b)--(W2d); \draw[line width=\edgewidth] (B2c)--(W1b);
\draw[line width=\edgewidth] (W3c)--(B1d);

\draw [line width=\nodewidth, fill=black] (B1a) circle [radius=\noderad] ;
\draw [line width=\nodewidth, fill=black] (B2a) circle [radius=\noderad] ;
\draw [line width=\nodewidth, fill=black] (B3a) circle [radius=\noderad] ;
\draw [line width=\nodewidth, fill=black] (B1b) circle [radius=\noderad] ;
\draw [line width=\nodewidth, fill=black] (B2b) circle [radius=\noderad] ;
\draw [line width=\nodewidth, fill=black] (B3b) circle [radius=\noderad] ;
\draw [line width=\nodewidth, fill=black] (B1c) circle [radius=\noderad] ;
\draw [line width=\nodewidth, fill=black] (B2c) circle [radius=\noderad] ;
\draw [line width=\nodewidth, fill=black] (B3c) circle [radius=\noderad] ;
\draw [line width=\nodewidth, fill=black] (B1d) circle [radius=\noderad] ;
\draw [line width=\nodewidth, fill=black] (B2d) circle [radius=\noderad] ;
\draw [line width=\nodewidth, fill=black] (B3d) circle [radius=\noderad] ;
\draw [line width=\nodewidth, fill=white] (W1a) circle [radius=\noderad] ;
\draw [line width=\nodewidth, fill=white] (W2a) circle [radius=\noderad] ;
\draw [line width=\nodewidth, fill=white] (W3a) circle [radius=\noderad] ;
\draw [line width=\nodewidth, fill=white] (W1b) circle [radius=\noderad] ;
\draw [line width=\nodewidth, fill=white] (W2b) circle [radius=\noderad] ;
\draw [line width=\nodewidth, fill=white] (W3b) circle [radius=\noderad] ;
\draw [line width=\nodewidth, fill=white] (W1c) circle [radius=\noderad] ;
\draw [line width=\nodewidth, fill=white] (W2c) circle [radius=\noderad] ;
\draw [line width=\nodewidth, fill=white] (W3c) circle [radius=\noderad] ;
\draw [line width=\nodewidth, fill=white] (W1d) circle [radius=\noderad] ;
\draw [line width=\nodewidth, fill=white] (W2d) circle [radius=\noderad] ;
\draw [line width=\nodewidth, fill=white] (W3d) circle [radius=\noderad] ;

\node at (1,1) {\Huge$\mathchar`-1$}; \node at (0.35,4) {\Huge$\mathchar`-1$};
\node at (2,3.5) {\Huge$0$}; \node at (4,2.5) {\Huge$0$}; \node at (4,0.5) {\Huge$0$}; \node at (1,7) {\Huge$0$}; \node at (0.35,10) {\Huge$0$};
\node at (5,5) {\Huge$1$}; \node at (7,1) {\Huge$1$}; \node at (4,8.5) {\Huge$1$}; \node at (2,9.5) {\Huge$1$}; \node at (2,11.65) {\Huge$1$};
\node at (10,0.5) {\Huge$2$}; \node at (8,3.5) {\Huge$2$}; \node at (10,2.5) {\Huge$2$}; \node at (7,7) {\Huge$2$}; \node at (5,11) {\Huge$2$};
\node at (11.65,2) {\Huge$3$}; \node at (11,5) {\Huge$3$}; \node at (8,9.5) {\Huge$3$}; \node at (10,8.5) {\Huge$3$}; \node at (8,11.65) {\Huge$3$};
\node at (11,11) {\Huge$4$}; \node at (11.65,8) {\Huge$4$};

\draw[line width=0.12cm, red] (0,0) rectangle (6,6);
\end{tikzpicture}
}
\caption{The height function $\sfh_{\sfP_2,\sfP_3}$.}\label{fig_heightchnge}
\end{figure}

We then take a point $\pt\in\RR^2{\setminus}(\sfP\cup\sfP_0)$, and define the \emph{height change}
\begin{displaymath}
h(\sfP,\sfP_0)=(h_x(\sfP,\sfP_0),h_y(\sfP,\sfP_0))\in\ZZ^2
\end{displaymath}
of $\sfP$ with respect to $\sfP_0$ as the differences of the height function:
\begin{gather*}
h_x(\sfP,\sfP_0)=\sfh_{\sfP,\sfP_0}(\pt+(1,0))-\sfh_{\sfP,\sfP_0}(\pt),
\\
h_y(\sfP,\sfP_0)=\sfh_{\sfP,\sfP_0}(\pt+(0,1))-\sfh_{\sfP,\sfP_0}(\pt).
\end{gather*}
We remark that this does not depend on the choice of $\pt\in\RR^2{\setminus}(\sfP\cup\sfP_0)$.
We then consider the height change $h(\sfP,\sfP^\prime)$ for any pair of perfect matchings $\sfP$, $\sfP^\prime$,
but since we have
\begin{equation*}
h(\sfP,\sfP^\prime)=h(\sfP,\sfP_0)-h(\sfP^\prime,\sfP_0),
\end{equation*}
it is enough to consider height changes with respect to the reference perfect matching $\sfP_0$.
Then, the \emph{perfect matching} (= \emph{PM}) \emph{polygon} (or \emph{characteristic polygon}) $\Delta_\Gamma\subset\RR^2$ of a dimer model $\Gamma$ is defined
as the convex hull of $\big\{h(\sfP,\sfP_0)\in\ZZ^2\,|\, \sfP\in\PM(\Gamma)\big\}$ where $\PM(\Gamma)$ is the set of perfect matchings on~$\Gamma$.

\begin{Remark}\label{rem_polygon}
The description of height changes depends on the choice of the coordinate system fixed in $\TT$
(i.e., the choice of a fundamental domain).
A change of a coordinate system induces a $\GL(2,\ZZ)$-action on the PM polygon, and this action does not affect our problem.
In the following, we say that two polygons $P$ and $Q$ are \emph{$\GL(2,\ZZ)$-equivalent} if they are transformed into each other by $\GL(2,\ZZ)$-transformations, in which case we denote $P\cong Q$.
Thus, we may fix some fundamental domain of~$\TT$.
Also, we remark that the description of the polygon $\Delta_\Gamma$ depends on the choice of a reference perfect matching, but it is determined up to translations.
\end{Remark}

\begin{Definition}\label{def_corner}
Fix a perfect matching $\sfP_0$ and let $\Delta_\Gamma$ be the perfect matching polygon.
We say that a perfect matching $\sfP$ is
\begin{itemize}\itemsep=0pt
\item a \emph{corner} (or \emph{extremal}) \emph{perfect matching} if $h(\sfP,\sfP_0)$ is a vertex of $\Delta_\Gamma$,
\item a \emph{boundary} (or \emph{external}) \emph{perfect matching} if $h(\sfP,\sfP_0)$ is a lattice point on an edge of $\Delta_\Gamma$
(in particular, corner perfect matchings are boundary perfect matchings),
\item an \emph{internal} \emph{perfect matching} if $h(\sfP,\sfP_0)$ is an interior lattice point of $\Delta_\Gamma$.
\end{itemize}
\end{Definition}

In the next subsection, we will introduce consistent dimer models (see Definition~\ref{def_consistent}), which have several nice properties.
If a dimer model is consistent, then there exists a unique corner perfect matching corresponding to each vertex of $\Delta_\Gamma$
(e.g., \cite[Corollary~4.27]{Bro}, \cite[Proposition~9.2]{IU2}).
Thus, we can give a cyclic order to corner perfect matchings along the corresponding vertices of $\Delta_\Gamma$ in the anti-clockwise direction.
We say that two corner perfect matchings are \emph{adjacent} if they are adjacent with respect to the given cyclic order.

\begin{Example}
We consider the dimer model given in Figure~\ref{ex_dimer_4b}, which is consistent as we will see in the next subsection.
Fix the perfect matching $\sfP_0$ shown in Figure~\ref{pm_4b} as the reference one.
Then, we see that the perfect matchings $\sfP_1,\dots,\sfP_4$ correspond to lattice points $(1,0)$, $(1,1)$, $(-1,0)$, $(0,-1)$, respectively.
Since $\sfP_0$ is the reference perfect matching, it corresponds to $(0,0)$.
In addition, this dimer model has three perfect matchings that are not listed in Figure~\ref{pm_4b},
and such perfect matchings also correspond to $(0,0)$.
Thus, the PM polygon takes the form shown in Figure~\ref{polygon_4b}, and hence $\sfP_1,\dots,\sfP_4$ are corner perfect matchings.

\begin{figure}[h!]\centering
\scalebox{0.75}{
\begin{tikzpicture}
\coordinate (00) at (0,0); \coordinate (10) at (1,0); \coordinate (01) at (0,1); \coordinate (-10) at (-1,0); \coordinate (11) at (1,1);
\coordinate (0-1) at (0,-1);

\draw [step=1, gray] (-2.3,-2.3) grid (2.3,2.3);
\draw [line width=0.06cm] (10)--(11) ; \draw [line width=0.06cm] (11)--(-10) ;
\draw [line width=0.06cm] (-10)--(0-1) ; \draw [line width=0.06cm] (0-1)--(10) ;

\draw [line width=0.05cm, fill=black] (00) circle [radius=0.1] ; \draw [line width=0.05cm, fill=black] (10) circle [radius=0.1] ;
\draw [line width=0.05cm, fill=black] (11) circle [radius=0.1] ;
\draw [line width=0.05cm, fill=black] (-10) circle [radius=0.1] ; \draw [line width=0.05cm, fill=black] (0-1) circle [radius=0.1] ;

\node [red] at (1.5,0) {$\sfP_1$} ;
\node [red] at (1.5,1) {$\sfP_2$} ; \node [red] at (-1.5,0) {$\sfP_3$} ;
\node [red] at (0,-1.5) {$\sfP_4$} ;

\end{tikzpicture}
}
\caption{The PM polygon of the dimer model given in Figure~\ref{ex_dimer_4b}.}\label{polygon_4b}
\end{figure}
\end{Example}

In this way, we can obtain the PM polygon from a dimer model.
On the other hand, it is known that any lattice polygon can be obtained as the PM polygon of a certain dimer model.

\begin{Theorem}[\cite{Gul,IU2}]\label{existence_dimer}
For any lattice polygon $\Delta$ in $\RR^2$, there exists a dimer model $\Gamma$ giving $\Delta$ as the PM polygon $\Delta_\Gamma$.
Furthermore, we can take this $\Gamma$ as it satisfies the consistency condition $($see Definition~{\rm \ref{def_consistent})}.
\end{Theorem}

Thus, for a given lattice polygon $\Delta$, we say that $\Gamma$ is a \emph{dimer model associated with $\Delta$} if the PM polygon of~$\Gamma$ coincides with~$\Delta$.
We remark that for a given polygon $\Delta$, the associated consistent dimer model is not unique in general.

\section{Zigzag paths and their properties}\label{sec_zigzag}

\subsection{Consistency conditions}\label{subsec_consistency}

In this subsection, we introduce the consistency condition.
In order to define this condition, we first introduce the notion of zigzag paths.
These paths are also the main ingredients for introducing deformations of dimer models.

\begin{Definition}We say that a path on a dimer model is a \emph{zigzag path} if it makes a maximum turn to the right on a white node and a maximum turn to the left on a black node.
Also, we say that a zigzag path is \emph{reduced} if it does not pass through $2$-valent nodes.
(We remark that we can make a zigzag path reduced using the join moves. In particular, any zigzag path on a~reduced dimer model is reduced.)
\end{Definition}

Since a dimer model has only finitely many edges, we see that all zigzag paths are periodic.
For a zigzag path $z$ on $\Gamma$, we define the length of~$z$, which is denoted by~$\ell(z)$, as the number of edges of~$\Gamma$ constituting~$z$.
In particular, we see that $\ell(z)$ is an even integer.
Thus, edges on a~zigzag path are indexed by elements in $\ZZ/(2n)\ZZ$ for some integer $\ell(z)/2=n\ge 1$.
Fix a black node on a zigzag path $z$ as the starting point of~$z$, and denote~$z$ as a sequence of edges starting from the fixed black node:
$z=z[1]z[2]\cdots z[2n-1]z[2n]$.

\begin{center}
{\scalebox{0.7}{
\begin{tikzpicture}
\newcommand{\edgewidth}{0.05cm} 
\newcommand{\nodewidth}{0.05cm} 
\newcommand{\noderad}{0.16} 

\coordinate (B1) at (0,0); \coordinate (B2) at (3,0); \coordinate (B3) at (6,0);
\coordinate (W1) at (1.5,1.5); \coordinate (W2) at (4.5,1.5); \coordinate (W3) at (7.5,1.5);

\path (B1) ++(135:1.5cm) coordinate (B1w); \path (B1) ++(225:0.7cm) coordinate (B1s); \path (B1) ++(315:0.7cm) coordinate (B1e);
\path (W1) ++(45:0.7cm) coordinate (W1w); \path (W1) ++(135:0.7cm) coordinate (W1s);
\path (B2) ++(225:0.7cm) coordinate (B2s); \path (B2) ++(315:0.7cm) coordinate (B2e);
\path (W2) ++(45:0.7cm) coordinate (W2w); \path (W2) ++(135:0.7cm) coordinate (W2s);
\path (B3) ++(225:0.7cm) coordinate (B3s); \path (B3) ++(315:0.7cm) coordinate (B3e);
\path (W3) ++(315:1.5cm) coordinate (W3n); \path (W3) ++(45:0.7cm) coordinate (W3w); \path (W3) ++(135:0.7cm) coordinate (W3s);

\path (B1) ++(245:0.5cm) coordinate (B1ss); \path (B1) ++(295:0.5cm) coordinate (B1ee);
\path (W1) ++(65:0.5cm) coordinate (W1ww); \path (W1) ++(115:0.5cm) coordinate (W1ss);
\path (B2) ++(245:0.5cm) coordinate (B2ss); \path (B2) ++(295:0.5cm) coordinate (B2ee);
\path (W2) ++(65:0.5cm) coordinate (W2ww); \path (W2) ++(115:0.5cm) coordinate (W2ss);
\path (B3) ++(245:0.5cm) coordinate (B3ss); \path (B3) ++(295:0.5cm) coordinate (B3ee);
\path (W3) ++(65:0.5cm) coordinate (W3ww); \path (W3) ++(115:0.5cm) coordinate (W3ss);
\draw [line width=\edgewidth] (B1w)--(B1)--(W1)--(B2)--(W2)--(B3)--(W3)--(W3n);
\draw [line width=\edgewidth] (B1s)--(B1)--(B1e); \draw[line width=\edgewidth] (B2s)--(B2)--(B2e); \draw [line width=\edgewidth] (B3s)--(B3)--(B3e);
\draw [line width=\edgewidth] (W1w)--(W1)--(W1s); \draw[line width=\edgewidth] (W2w)--(W2)--(W2s); \draw [line width=\edgewidth] (W3w)--(W3)--(W3s);

\draw [line width=\edgewidth,line cap=round, dash pattern=on 0pt off 2.5\pgflinewidth] (B1ss)--(B1ee);
\draw [line width=\edgewidth,line cap=round, dash pattern=on 0pt off 2.5\pgflinewidth] (W1ww)--(W1ss);
\draw [line width=\edgewidth,line cap=round, dash pattern=on 0pt off 2.5\pgflinewidth] (B2ss)--(B2ee);
\draw [line width=\edgewidth,line cap=round, dash pattern=on 0pt off 2.5\pgflinewidth] (W2ww)--(W2ss);
\draw [line width=\edgewidth,line cap=round, dash pattern=on 0pt off 2.5\pgflinewidth] (B3ss)--(B3ee);
\draw [line width=\edgewidth,line cap=round, dash pattern=on 0pt off 2.5\pgflinewidth] (W3ww)--(W3ss);
\draw [line width=\nodewidth, fill=black] (B1) circle [radius=\noderad] ;
\draw [line width=\nodewidth, fill=black] (B2) circle [radius=\noderad] ;
\draw [line width=\nodewidth, fill=black] (B3) circle [radius=\noderad] ;
\draw [line width=\nodewidth, fill=white] (W1) circle [radius=\noderad] ;
\draw [line width=\nodewidth, fill=white] (W2) circle [radius=\noderad] ;
\draw [line width=\nodewidth, fill=white] (W3) circle [radius=\noderad] ;
\path (B1w) ++(90:0.2cm) coordinate (B1w+);
\path (B1) ++(90:0.2cm) coordinate (B1+); \path (W1) ++(90:0.2cm) coordinate (W1+);
\path (B2) ++(90:0.2cm) coordinate (B2+); \path (W2) ++(90:0.2cm) coordinate (W2+);
\path (B3) ++(90:0.2cm) coordinate (B3+); \path (W3) ++(90:0.2cm) coordinate (W3+);
\path (W3n) ++(90:0.2cm) coordinate (W3n+);
\draw [->, rounded corners, line width=0.08cm, red] (B1w+)--(B1+)--(W1+)--(B2+)--(W2+)--(B3+)--(W3+)--(W3n+) ;
\draw [line width=0.05cm, red, line cap=round, dash pattern=on 0pt off 2.5\pgflinewidth] (-1.8,0.8)--(-2.4,0.8) ;
\draw [line width=0.05cm, red, line cap=round, dash pattern=on 0pt off 2.5\pgflinewidth] (9.3,0.8)--(9.9,0.8) ;
\node[red] at (0.5,1.3) {\large$z[1]$} ; \node[red] at (2,0.5) {\large$z[2]$} ;
\node[red] at (3.5,1.3) {\large$z[3]$} ; \node[red] at (5,0.5) {\large$z[4]$} ;
\node[red] at (6.5,1.3) {\large$z[5]$} ; \node[red] at (8,0.5) {\large$z[6]$} ;
\node[red] at (-1.1,0.5) {\large$z[2n]$} ;
\end{tikzpicture}
} }
\end{center}

An edge in a zigzag path $z$ is called a \emph{zig} (resp.\ \emph{zag}) of $z$ if it is indexed by an odd (resp.~even) integer.
We denote by $\Zig(z)$ (resp.~$\Zag(z)$) the set of zigs (resp.~zags) appearing in a~zigzag path~$z$, which is a~finite set.
Two zigzag paths are said to \emph{intersect} if they share an edge (not a node).
We note that if $z$ does not have a self-intersection, $\Zig(z)$ and $\Zag(z)$ are disjoint sets.
For any edge~$e$ of a dimer model, we can consider the zigzag path containing $e$ as a zig and the zigzag path containing $e$ as a zag.
Thus, any edge $e$ is contained in at most two zigzag paths.
If such zigzag paths do not have a self-intersection, $e$ is contained in exactly two zigzag paths.
For example, zigzag paths on the dimer model given in the left of Figure~\ref{ex_dimer_4b} are shown in Figure~\ref{zigzag_4b}.

\begin{figure}[h!]
\centering

\newcommand{\edgewidth}{0.07cm}
\newcommand{\nodewidth}{0.07cm}
\newcommand{\noderad}{0.24} 

{\scalebox{1}{
\begin{tikzpicture}

\node at (0,-1.6) {$z_1$}; \node at (3.5,-1.6) {$z_2$}; \node at (7,-1.6) {$z_3$}; \node at (10.5,-1.6) {$z_4$};

\node (ZZ1) at (0,0)
{\scalebox{0.4}{
\begin{tikzpicture}
\basicdimerB
\draw[->, line width=0.2cm, rounded corners, color=red] (0,3)--(W2)--(B1)--(W3)--(B3)--(6,3);
\end{tikzpicture} }};

\node (ZZ2) at (3.5,0)
{\scalebox{0.4}{
\begin{tikzpicture}
\basicdimerB
\draw[->, line width=0.2cm, rounded corners, color=red] (3,0)--(B1)--(W2)--(B2)--(0,6);
\draw[->, line width=0.2cm, rounded corners, color=red] (6,0)--(W1)--(B3)--(W3)--(3,6);
\end{tikzpicture} }} ;

\node (ZZ3) at (7,0)
{\scalebox{0.4}{
\begin{tikzpicture}
\basicdimerB
\draw[->, line width=0.2cm, rounded corners, color=red] (3,6)--(W3)--(B2)--(W2)--(0,3);
\draw[->, line width=0.2cm, rounded corners, color=red] (6,3)--(B3)--(W1)--(B1)--(3,0);
\end{tikzpicture} }};

\node (ZZ4) at (10.5,0)
{\scalebox{0.4}{
\begin{tikzpicture}
\basicdimerB
\draw[->, line width=0.2cm, rounded corners, color=red] (0,6)--(B2)--(W3)--(B1)--(W1)--(6,0);
\end{tikzpicture} }} ;

\end{tikzpicture}
}}
\caption{All zigzag paths on the dimer model given in Figure~\ref{ex_dimer_4b}.}\label{zigzag_4b}
\end{figure}

For a zigzag path $z$ on a dimer model $\Gamma$, we also consider the lift of $z$ to the universal cover~$\widetilde{\Gamma}$.
Let $\widetilde{z}{(\alpha)}$ denote a zigzag path on $\widetilde{\Gamma}$ whose projection on $\Gamma$ is $z$ where $\alpha\in\ZZ$.
When we do not need to specify these, we simply denote each of them by~$\widetilde{z}$.
Then, we see that a zigzag path on~$\widetilde{\Gamma}$ is either periodic or infinite in both directions.
Using these notions, we introduce the consistency condition.

\begin{Definition}[{see \cite[Definition~3.5]{IU1}}]
\label{def_consistent}
We say that a dimer model is (\emph{zigzag}) \emph{consistent} if it satisfies the following conditions:
\begin{itemize}\itemsep=0pt
\item [(1)] there is no homologically trivial zigzag path,
\item [(2)] no zigzag path on the universal cover has a self-intersection,
\item [(3)] no pair of zigzag paths on the universal cover intersect each other in the same direction more than once.
That is, if a pair of zigzag paths $(\widetilde{z},\widetilde{w})$ on the universal cover has two intersected edges $a_1$, $a_2$
and $\widetilde{z}$ points from $a_1$ to $a_2$, then $\widetilde{w}$ points from $a_2$ to $a_1$.
\end{itemize}
\end{Definition}

In the literature, there are several conditions that are equivalent to Definition~\ref{def_consistent} (for more details, see \cite{Boc1,IU1}),
and it is known that a consistent dimer model is non-degenerate (e.g., \cite[Proposition~8.1]{IU2}).
For example, we see that the dimer model given in Figure~\ref{ex_dimer_4b} is consistent by checking all zigzag paths, which are shown in Figure~\ref{zigzag_4b}.
We also remark that this dimer model satisfies the stronger condition called \emph{isoradial} (see Definition~\ref{def_isoradial}).

In this paper, we also use another condition called \emph{properly ordered}.
To explain the proper ordering, we prepare several notation.
First, considering a zigzag path $z$ as a $1$-cycle on $\TT$, we have the homology class $[z]\in\rmH_1(\TT)\cong\ZZ^2$. We call this element $[z]\in\ZZ^2$ the \emph{slope} of $z$.
We remark that even if we apply the join and split moves to nodes contained in a zigzag path, such operations do not change the slope.
If a zigzag path does not have any self-intersection, the slope of each zigzag path is primitive.
Now, we consider slopes $(a,b)\in\ZZ^2$ of zigzag paths that are not homologically trivial.
The set of such slopes has a natural cyclic order by considering~$(a,b)$ as an element of the unit circle:
\begin{displaymath}
\frac{(a,b)}{\sqrt{a^2+b^2}}\in S^1.
\end{displaymath}
We say that two zigzag paths are \emph{adjacent} if their slopes are adjacent with respect to the above cyclic order.
Using this cyclic order, we define a properly ordered dimer model below.
In particular, it is known that a dimer model is consistent in the sense of Definition~\ref{def_consistent} if and only if it is properly ordered (see \cite[Proposition~4.4]{IU1}).

\begin{Definition}[{see \cite[Section~3.1]{Gul}}]\label{def_properly}
A dimer model is said to be \emph{properly ordered} if
\begin{itemize}\itemsep=0pt
\item [(1)] there is no homologically trivial zigzag path,
\item [(2)] no zigzag path on the universal cover has a self-intersection,
\item [(3)] no pair of zigzag paths with the same slope have a common node,
\item [(4)] for any node on the dimer model, the natural cyclic order on the set of zigzag paths incident to that node
coincides with the cyclic order determined by their slopes.
\end{itemize}
\end{Definition}

We also introduce isoradial dimer models which are stronger than consistent ones.
The dimer model given in Figure~\ref{ex_dimer_4b} is isoradial in particular.

\begin{Definition}[{\cite[Theorem~5.1]{KS}; see also \cite{Duf,Mer}}]
\label{def_isoradial}
We say that a dimer model $\Gamma$ is \emph{isoradial} (or \emph{geometrically consistent}) if
\begin{itemize}\itemsep=0pt
\item [(1)] every zigzag path is a simple closed curve,
\item [(2)] any pair of zigzag paths on the universal cover share at most one edge.
\end{itemize}
\end{Definition}

\subsection{Relationships between perfect matchings and zigzag paths}\label{subsec_relation_zigzag_pm}

We can now discuss the relationship between perfect matchings and zigzag paths.
The following proposition is essential throughout this paper.

\begin{Proposition}[{see \cite[Theorem~3.3 and Corollary~3.8]{Gul}, \cite[Proposition~9.2 and Corollary~9.3]{IU2}}]\label{zigzag_sidepolygon}
There exists a one-to-one correspondence between the set of slopes of zigzag paths on a consistent dimer model $\Gamma$ and
the set of primitive side segments of the PM polygon $\Delta_\Gamma$. More precisely, each slope of a zigzag path is the primitive outer normal vector for a primitive side segment of~$\Delta_\Gamma$.

Moreover, zigzag paths having the same slope arise as the difference of two adjacent corner perfect matchings $\sfP$, $\sfP^\prime$ $($i.e., the edges in $\sfP\cup\sfP^\prime{\setminus}\sfP\cap\sfP^\prime$ form zigzag paths$)$. Thus, any corner perfect matching intersects with half of the edges constituting a certain zigzag path.
\end{Proposition}

For example, the zigzag path $z_1$ shown in Figure~\ref{zigzag_4b} is obtained from the pair of adjacent corner perfect matchings $(\sfP_1,\sfP_2)$ given in Figure~\ref{pm_4b}. Also, the zigzag paths $z_2$, $z_3$, $z_4$ are obtained by pairs $(\sfP_2,\sfP_3)$, $(\sfP_3,\sfP_4)$, $(\sfP_4,\sfP_1)$, respectively.

By Proposition~\ref{zigzag_sidepolygon}, we can assign each edge of the PM polygon to a zigzag path~$z$;
thus we will call this the edge corresponding to~$z$.
In particular, the edges corresponding to zigzag paths having the same slope are all the same.

Let $\sfP$, $\sfP^\prime$ be adjacent corner perfect matchings on a consistent dimer model, and $z_1,\dots, z_r$ be the zigzag paths arising from $\sfP$ and $\sfP^\prime$ as in Proposition~\ref{zigzag_sidepolygon}.
In particular, these zigzag paths have the same slope.
We see that $\sfP\cap z_i=\Zig(z_i)$ and $\sfP^\prime\cap z_i=\Zag(z_i)$ (or $\sfP\cap z_i=\Zag(z_i)$ and $\sfP^\prime\cap z_i=\Zig(z_i)$) for any $i=1,\dots,r$. Here, $\sfP\cap z_i$ denotes the subset of edges in $\sfP$ contained in~$z_i$.
Then, we have the description of boundary perfect matchings using the corner ones.

\begin{Proposition}[{e.g., \cite[Proposition~4.35]{Bro}, \cite[Corollary~3.8]{Gul}}]\label{char_bound}
Let $\sfP$, $\sfP^\prime$ and $z_1,\dots, z_r$ be as above.
Let $E$ be the edge of the PM polygon of $\Gamma$ corresponding to $z_1,\dots, z_r$.
We assume that $\sfP\cap z_i=\Zig(z_i)$ and $\sfP^\prime\cap z_i=\Zag(z_i)$.
Then, any boundary perfect matching corresponding to a~lattice point on $E$ can be described as
\begin{displaymath}
\bigg(\sfP{\setminus}\bigcup_{i\in I}\Zig(z_i)\bigg)\cup\bigcup_{i\in I}\Zag(z_i) \qquad or \qquad \bigg(\sfP^\prime{\setminus}\bigcup_{i\in I}\Zag(z_i)\bigg)\cup\bigcup_{i\in I}\Zig(z_i),
\end{displaymath}
where $I$ is a subset of $\{1, \dots, r\}$.
In particular, the number of perfect matchings corresponding to a~lattice point~$\mathsf{q}$ on~$E$ is~$\dbinom{r}{m}$,
where~$m$ is the number of primitive side segments of~$E$ between~$\mathsf{q}$ and one of the endpoints of~$E$.
\end{Proposition}

We then observe the relationships between zigzag paths and height changes of perfect matchings.
Some of these relationships are well-known to experts, but we will consider them in detail because these statements are quite important
when we define the deformation of consistent dimer models in Section~\ref{sec_def_deform}, and also for the self-containedness.

\begin{observation}[{cf.~\cite[Section~5.3]{IU2}}]\label{height_zigzag}
Let $\Gamma$ be a consistent dimer model.
For any zigzag path~$z$, the slope $[z]$ is an element in $\rmH_1(\TT)$.
On the other hand, we can consider height changes as elements in the cohomology group $\rmH^1(\TT)\cong\ZZ^2$.
We consider a pairing \mbox{$\langle{-},{-}\rangle\colon \! \rmH^1(\TT)\!\times\!\rmH_1(\TT)\!\rightarrow\!\ZZ$}.
By Propositions~\ref{zigzag_sidepolygon} and \ref{char_bound}, there is a perfect matching $\sfP^\prime$ that intersects half of the edges constituting~$z$.
Then, for any perfect matching $\sfP$, we have $\langle h(\sfP,\sfP^\prime),[z]\rangle\le 0$.
In fact, we first replace~$z$ by the path~$p_z$ on the quiver $Q_\Gamma$ going along the left side of~$z$ (see the figure below).

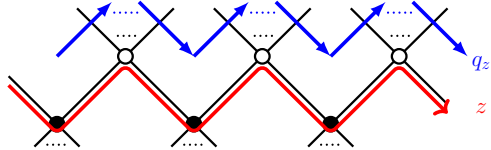
\begin{center}
{\scalebox{0.6}{
\begin{tikzpicture}[sarrow/.style={black, -latex, very thick},tarrow/.style={black, latex-, very thick}]
\newcommand{\edgewidth}{0.05cm} 
\newcommand{\nodewidth}{0.05cm} 
\newcommand{\noderad}{0.16} 

\coordinate (B1) at (0,0); \coordinate (B2) at (3,0); \coordinate (B3) at (6,0);
\coordinate (W1) at (1.5,1.5); \coordinate (W2) at (4.5,1.5); \coordinate (W3) at (7.5,1.5);

\path (B1) ++(135:1.5cm) coordinate (B1w); \path (B1) ++(225:0.7cm) coordinate (B1s); \path (B1) ++(315:0.7cm) coordinate (B1e);
\path (W1) ++(45:1.5cm) coordinate (W1w); \path (W1) ++(135:1.5cm) coordinate (W1s);
\path (B2) ++(225:0.7cm) coordinate (B2s); \path (B2) ++(315:0.7cm) coordinate (B2e);
\path (W2) ++(45:1.5cm) coordinate (W2w); \path (W2) ++(135:1.5cm) coordinate (W2s);
\path (B3) ++(225:0.7cm) coordinate (B3s); \path (B3) ++(315:0.7cm) coordinate (B3e);
\path (W3) ++(315:1.5cm) coordinate (W3n); \path (W3) ++(45:1.5cm) coordinate (W3w); \path (W3) ++(135:1.5cm) coordinate (W3s);

\path (B1) ++(245:0.5cm) coordinate (B1ss); \path (B1) ++(295:0.5cm) coordinate (B1ee);
\path (W1) ++(65:0.5cm) coordinate (W1ww); \path (W1) ++(115:0.5cm) coordinate (W1ss);
\path (B2) ++(245:0.5cm) coordinate (B2ss); \path (B2) ++(295:0.5cm) coordinate (B2ee);
\path (W2) ++(65:0.5cm) coordinate (W2ww); \path (W2) ++(115:0.5cm) coordinate (W2ss);
\path (B3) ++(245:0.5cm) coordinate (B3ss); \path (B3) ++(295:0.5cm) coordinate (B3ee);
\path (W3) ++(65:0.5cm) coordinate (W3ww); \path (W3) ++(115:0.5cm) coordinate (W3ss);
\draw [line width=\edgewidth] (B1w)--(B1)--(W1)--(B2)--(W2)--(B3)--(W3)--(W3n);
\draw [line width=\edgewidth] (B1s)--(B1)--(B1e); \draw[line width=\edgewidth] (B2s)--(B2)--(B2e); \draw [line width=\edgewidth] (B3s)--(B3)--(B3e);
\draw [line width=\edgewidth] (W1w)--(W1)--(W1s); \draw[line width=\edgewidth] (W2w)--(W2)--(W2s); \draw [line width=\edgewidth] (W3w)--(W3)--(W3s);

\draw [line width=\edgewidth,line cap=round, dash pattern=on 0pt off 2.5\pgflinewidth] (B1ss)--(B1ee);
\draw [line width=\edgewidth,line cap=round, dash pattern=on 0pt off 2.5\pgflinewidth] (W1ww)--(W1ss);
\draw [line width=\edgewidth,line cap=round, dash pattern=on 0pt off 2.5\pgflinewidth] (B2ss)--(B2ee);
\draw [line width=\edgewidth,line cap=round, dash pattern=on 0pt off 2.5\pgflinewidth] (W2ww)--(W2ss);
\draw [line width=\edgewidth,line cap=round, dash pattern=on 0pt off 2.5\pgflinewidth] (B3ss)--(B3ee);
\draw [line width=\edgewidth,line cap=round, dash pattern=on 0pt off 2.5\pgflinewidth] (W3ww)--(W3ss);
\draw [line width=\nodewidth, fill=black] (B1) circle [radius=\noderad] ;
\draw [line width=\nodewidth, fill=black] (B2) circle [radius=\noderad] ;
\draw [line width=\nodewidth, fill=black] (B3) circle [radius=\noderad] ;
\draw [line width=\nodewidth, fill=white] (W1) circle [radius=\noderad] ;
\draw [line width=\nodewidth, fill=white] (W2) circle [radius=\noderad] ;
\draw [line width=\nodewidth, fill=white] (W3) circle [radius=\noderad] ;
\path (B1w) ++(-90:0.2cm) coordinate (B1w+);
\path (B1) ++(-90:0.2cm) coordinate (B1+); \path (W1) ++(-90:0.2cm) coordinate (W1+);
\path (B2) ++(-90:0.2cm) coordinate (B2+); \path (W2) ++(-90:0.2cm) coordinate (W2+);
\path (B3) ++(-90:0.2cm) coordinate (B3+); \path (W3) ++(-90:0.2cm) coordinate (W3+);
\path (W3n) ++(-90:0.2cm) coordinate (W3n+);
\draw [->, rounded corners, line width=0.08cm, red] (B1w+)--(B1+)--(W1+)--(B2+)--(W2+)--(B3+)--(W3+)--(W3n+) ;

\path (B1)++(90:1.5cm) coordinate (V1);
\path (B2)++(90:1.5cm) coordinate (V2);
\path (B3)++(90:1.5cm) coordinate (V3);
\path (9,0)++(90:1.5cm) coordinate (V4);
\draw [sarrow, line width=0.08cm, blue] (V1)-- ++(45:1.7cm);
\draw [sarrow, line width=0.08cm, blue] (V2)-- ++(45:1.7cm);
\draw [tarrow, line width=0.08cm, blue] (V2)-- ++(135:1.7cm);
\draw [sarrow, line width=0.08cm, blue] (V3)-- ++(45:1.7cm);
\draw [tarrow, line width=0.08cm, blue] (V3)-- ++(135:1.7cm);
\draw [tarrow, line width=0.08cm, blue] (V4)-- ++(135:1.7cm);

\path (W1) ++(75:1cm) coordinate (W1www); \path (W1) ++(105:1cm) coordinate (W1sss);
\path (W2) ++(75:1cm) coordinate (W2www); \path (W2) ++(105:1cm) coordinate (W2sss);
\path (W3) ++(75:1cm) coordinate (W3www); \path (W3) ++(105:1cm) coordinate (W3sss);
\draw [line width=\edgewidth,line cap=round, dash pattern=on 0pt off 2.5\pgflinewidth,blue] (W1www)--(W1sss);
\draw [line width=\edgewidth,line cap=round, dash pattern=on 0pt off 2.5\pgflinewidth,blue] (W2www)--(W2sss);
\draw [line width=\edgewidth,line cap=round, dash pattern=on 0pt off 2.5\pgflinewidth,blue] (W3www)--(W3sss);
\node[red] at (9.3,0.35) {\LARGE $z$} ; \node[blue] at (9.3,1.35) {\LARGE $p_z$} ;
\end{tikzpicture}
} }
\end{center}

Then, considering this path $p_z$ as the element $[p_z]\in\rmH_1(\TT)$, we have $[z]=[p_z]$.
By a choice of $\sfP^\prime$, this $p_z$ does not cross any edge in $\sfP^\prime$, and if $p_z$ crosses an edge in $\sfP$, we can see the white node on the right by the definition of $Q_\Gamma$. Thus, we have the desired inequality.
\end{observation}

For a perfect matching $\sfP$ and a zigzag path $z$ on a dimer model $\Gamma$, we denote by $|\sfP\cap z|$ the number of edges in $\sfP\cap z$.
Since the number of perfect matchings is finite, the maximum (resp.\ minimum) number $\omega_{\max}(z)$ (resp.\ $\omega_{\min}(z)$) of $|\sfP\cap z|$ exists for each zigzag path $z$.
For a~consistent dimer model, $z$ can be obtained as the difference of adjacent perfect matchings (see Proposition~\ref{char_bound});
thus we clearly have $\ell(z)/2=\omega_{\max}(z)$.
We set
\begin{gather*}
\PM_{\max}(z)=\{\sfP\in\PM(\Gamma) \,|\, |\sfP\cap z|=\omega_{\max}(z)\},
\\
\PM_{\min}(z)=\{\sfP\in\PM(\Gamma) \,|\, |\sfP\cap z|=\omega_{\min}(z)\}.
\end{gather*}
In particular, if $\sfP$, $\sfP^\prime$ are adjacent corner perfect matchings on a consistent dimer model $\Gamma$, and $z$ is one of the zigzag paths obtained by $\sfP$, $\sfP^\prime$, then we have $\sfP\cap z=\Zig(z)$ and $\sfP^\prime\cap z=\Zag(z)$ (or $\sfP\cap z=\Zag(z)$ and $\sfP^\prime\cap z=\Zig(z)$),
and hence the next lemma easily follows from Propositions~\ref{zigzag_sidepolygon} and~\ref{char_bound}.

\begin{Lemma}\label{number_zigzag_corner}
Let $z$ be a zigzag path on a consistent dimer model $\Gamma$, and $E$ be the edge of the PM polygon of~$\Gamma$ corresponding to~$z$.
If $\sfP_1,\dots,\sfP_s$ are boundary perfect matchings corresponding to lattice points on~$E$, then we have
$\{\sfP_1,\dots,\sfP_s\}=\PM_{\max}(z)$, and $\omega_{\max}(z)=|\sfP_i\cap z|=\ell(z)/2$ for any $i=1,\dots,s$.
\end{Lemma}

Next, we prepare several lemmas, which play crucial roles to define deformations of consistent dimer models.

\begin{Lemma}\label{zigzag_lem1}
Let the notation be the same as Lemma~{\rm \ref{number_zigzag_corner}}.
For any perfect matching~$\sfP$, we have
\begin{displaymath}
|\sfP\cap z|=\ell(z)/2-\langle h(\sfP,\sfP_i),-[z]\rangle.
\end{displaymath}
In particular, we have
\begin{displaymath}
\langle h(\sfP,\sfP_i),-[z]\rangle\le \omega_{\max}(z)-\omega_{\min}(z),
\end{displaymath}
and the equality holds for $\sfP\in \PM_{\min}(z)$.
\end{Lemma}

\begin{proof}
First, the maximum number of $|\sfP\cap z|$ is $\ell(z)/2$, in which case $\sfP=\sfP_i$ by Lemma~\ref{number_zigzag_corner}.
If the path $p_z$ as in Observation~\ref{height_zigzag} crosses an edge $e$ in $\sfP$, it means that $e$ is not an edge constituting $z$, and thus any edge sharing the same white node as $e$ is not contained in $\sfP$.
By Observation~\ref{height_zigzag}, we see that for any perfect matching $\sfP$ the number of edges in $\sfP$ intersecting with $p_z$ coincides with $-\langle h(\sfP,\sfP_i),[z]\rangle$, and thus we have the first equation.

The second assertion follows from the first equation and Lemma~\ref{number_zigzag_corner}.
\end{proof}

By this lemma, we see that $\sfP\in\PM_{\min}(z)$ if and only if $\langle h(\sfP,\sfP_i),[z]\rangle{\le} \langle h(\sfP^\prime,\sfP_i),[z]\rangle$ for any $\sfP^\prime\in\PM(\Gamma)$.
Thus, we see that $\sfP\in\PM_{\min}(z)$ lies either on a vertex of the PM polygon $\Delta_\Gamma$ or an edge of $\Delta_\Gamma$.
Also, even if two zigzag paths $z_j$ and $z_k$ have the same slope, $\ell(z_j)\neq\ell(z_k)$ and $|\sfP\cap z_j|\neq|\sfP\cap z_k|$ in general,
but the difference of these values is the same in the following sense:

\begin{Lemma}\label{zigzag_lem2}
Let $\Gamma$ be a consistent dimer model, and $z$, $z^\prime$ be zigzag paths on $\Gamma$ having the same slope.
Then, for any perfect matching $\sfP$, we have
\begin{displaymath}
\ell(z)/2-|\sfP\cap z|=\ell(z^\prime)/2-|\sfP\cap z^\prime|.
\end{displaymath}
In particular, we have
\begin{displaymath}
\ell(z)/2-\omega_{\min}(z)=\ell(z^\prime)/2-\omega_{\min}(z^\prime).
\end{displaymath}
\end{Lemma}

\begin{proof}Since $[z]=[z^\prime]$, the first equation follows from Lemma~\ref{zigzag_lem1}.
Considering a perfect matching $\sfP$ such that the value of $\langle h(\sfP,\sfP_i),-[z]\rangle=\langle h(\sfP,\sfP_i),-[z^\prime]\rangle$ is maximal, we have the second equation.
\end{proof}

We then divide zigzag paths on a consistent dimer model into the following two types. Note that type I zigzag paths are used to define the deformation of consistent dimer models.

\begin{Definition}\label{def_zigzag_type}
Let $\Gamma$ be a dimer model, and $z$ be a zigzag path on $\Gamma$.
\begin{itemize}\itemsep=0pt
\item[(1)] We say that $z$ is \emph{type I} if $z$ is reduced and $\widetilde{z}$ intersects with any other zigzag path on the universal cover $\widetilde{\Gamma}$ at most once.
\item[(2)] We say that $z$ is \emph{type II} if $z$ is reduced and there exists a zigzag path $\widetilde{w}$ on the universal cover $\widetilde{\Gamma}$ such that $\widetilde{w}$ intersects with $\widetilde{z}$ in the opposite direction more than once.
\end{itemize}
We note that any zigzag path on a reduced consistent dimer model is either type~I or type~II.
In particular, if $\Gamma$ is isoradial, then all zigzag paths are type~I (see Definition~\ref{def_isoradial}).
\end{Definition}

As the following lemmas show, the properties of type I zigzag paths are particularly nice.

\begin{Lemma}\label{lem_existence_pm}
Let $z$ be a type I zigzag path on a consistent dimer model $\Gamma$.
Then, there exists a perfect matching $\sfP$ on $\Gamma$ satisfying $|\sfP\cap z|=0$, which means that $\sfP$ is in $\PM_{\min}(z)$.
\end{Lemma}

\begin{proof}
In order to find a perfect matching $\sfP$, we will use the method discussed in \cite[Section~3]{Gul}, \cite[Section~4]{Bro}.
For this, we first prepare some notation.

We consider the sequence $[z_1],\dots,[z_n]$ of slopes of zigzag paths on~$\Gamma$.
Since $\Gamma$ is consistent, it is properly ordered, and thus we can assume that the slopes are ordered cyclically with this order.
We note that some of the slopes may coincide.
Then, we define the normal fan in $\rmH_1(\TT)\otimes_\ZZ\RR$ whose rays are slopes $[z_1],\dots,[z_n]$.
In particular, each two-dimensional cone $\sigma$ is generated by different adjacent slopes.
We denote by $\theta_i$ the angle formed by $[z_i]$. Here, we suppose that $z=z_k$.
Let $\calR$ be a ray whose angle is $\theta_k+\pi+\epsilon$ where $\epsilon>0$ is a sufficiently small angle satisfying the condition that $\theta_k+\pi+\epsilon$ does not coincide with any $\theta_i$.

Then, for each node $v\in\Gamma_0$, we define the fan $\xi(v)$ generated by the slopes of zigzag paths factoring through $v$.
In this fan $\xi(v)$, we can find the zigzag path $z^\prime_v$ whose slope makes the smallest clockwise angle with $\calR$, and the zigzag path $z^{\prime\prime}_v$ whose slope makes the smallest anti-clockwise angle with~$\calR$.
Since $\Gamma$ is properly ordered, these zigzag paths are consecutive around~$v$.
The intersection of~$z^\prime_v$ and $z^{\prime\prime}_v$ is the edge~$e(v)$ which has~$v$ as an endpoint.

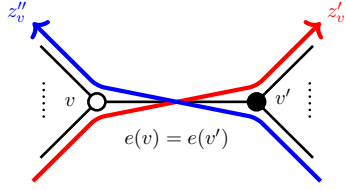
\begin{center}
{\scalebox{0.7}{
\begin{tikzpicture}
\newcommand{\edgewidth}{0.05cm} 
\newcommand{\nodewidth}{0.05cm} 
\newcommand{\noderad}{0.16} 

\coordinate (W1) at (0,0); \coordinate (B1) at (3,0);
\path (W1) ++(135:1.5cm) coordinate (W1a); \path (W1) ++(225:1.5cm) coordinate (W1b);
\path (B1) ++(45:1.5cm) coordinate (B1a); \path (B1) ++(315:1.5cm) coordinate (B1b);

\draw [line width=\edgewidth] (W1)--(B1); \draw [line width=\edgewidth] (W1)--(W1a); \draw [line width=\edgewidth] (W1)--(W1b);
\draw [line width=\edgewidth] (B1)--(B1a); \draw [line width=\edgewidth] (B1)--(B1b);
\draw [line width=\edgewidth,line cap=round, dash pattern=on 0pt off 2.5\pgflinewidth] (-1,0.35)--(-1,-0.35);
\draw [line width=\edgewidth,line cap=round, dash pattern=on 0pt off 2.5\pgflinewidth] (4,0.35)--(4,-0.35);
\draw [line width=\nodewidth, fill=white] (W1) circle [radius=\noderad] ;
\draw [line width=\nodewidth, fill=black] (B1) circle [radius=\noderad] ;

\path (W1) ++(270:0.3cm) coordinate (W1-); \path (W1) ++(90:0.3cm) coordinate (W1+);
\path (B1) ++(270:0.3cm) coordinate (B1-); \path (B1) ++(90:0.3cm) coordinate (B1+);
\path (W1-) ++(225:1.7cm) coordinate (W1--); \path (W1+) ++(135:1.7cm) coordinate (W1++);
\path (B1-) ++(315:1.7cm) coordinate (B1--); \path (B1+) ++(45:1.7cm) coordinate (B1++);
\draw [->, rounded corners, line width=0.08cm, red] (W1--)--(W1-)--(B1+)--(B1++) ;
\draw [->, rounded corners, line width=0.08cm, blue] (B1--)--(B1-)--(W1+)--(W1++) ;

\node at (-0.5,0) {\large$v$} ; \node at (3.5,0.1) {\large$v^\prime$} ;
\node at (1.5,-0.7) {\large$e(v)=e(v^\prime)$} ;
\node[red] at (4.5,1.7) {\large$z_v^\prime$}; \node[blue] at (-1.5,1.7) {\large$z_v^{\prime\prime}$};

\end{tikzpicture}
} }
\end{center}

We then apply the same argument to the node $v^\prime$ which is the other endpoint of $e(v)$.
The proper ordering on $\Gamma$ leads to the conclusion that $e(v)=e(v^\prime)$ (see the above figure).
We repeat these arguments for any node, but it is enough to consider $e(v)$'s for any $v\in\Gamma^+_0$ (or $v\in\Gamma^-_0$).
By \cite[Section~3.2]{Gul} or \cite[Lemma~4.19]{Bro}, we see that the subset of edges $e(v)$ for all $v\in\Gamma^+_0$ forms a perfect matching.
Furthermore, since $z=z_k$ is type I, there exists a zigzag path whose slope is located at an angle less than $\pi$ in an anti-clockwise (resp.\ clockwise) direction from~$[z]$ in~$\xi(v)$ by \cite[Lemma~4.11 and its proof]{Bro}, in which case such a slope is located between $[z]$ and $[z^\prime_v]$ (resp.~$[z^{\prime\prime}_v]$) or coincides with $[z^\prime_v]$ (resp.~$[z^{\prime\prime}_v]$). Thus, we have $z\neq z^\prime_v$ and $z\neq z^{\prime\prime}_v$ by the definition of the ray $\calR$.
Therefore, in this case $e(v)$ is not contained in~$z$ by the above construction.
Also, we see that if a node~$v$ does not lie on~$z$, $e(v)$ is not contained in~$z$.
Thus, the perfect matching constructed in the above fashion satisfies the desired condition.
\end{proof}

\begin{Lemma}\label{lem_same_length}
Let $z$ be a type I zigzag path on a consistent dimer model.
Then, we have $\omega_{\min}(z)=0$, and hence $\ell(z)$ is the same for all type I zigzag paths having the same slope.
\end{Lemma}

\begin{proof}
This follows from Lemmas~\ref{zigzag_lem2} and \ref{lem_existence_pm}.
\end{proof}

For zigzag paths $z$, $w$ on a dimer model $\Gamma$, we denote by $z\cap w$ the subset of edges that are intersections of $z$ and $w$ on $\Gamma$.
We remark that if $z$ is type I then the number of intersections of~$\widetilde{z}$ and~$\widetilde{w}$ on $\widetilde{\Gamma}$ is less than or equal to one, but there are possibly more intersections of~$z$ and~$w$ if we consider them on~$\Gamma$.

\begin{Lemma}\label{intersect_zigorzag}
Let $z$ be a type I zigzag path on a reduced consistent dimer model $\Gamma$.
We suppose that a zigzag path $w$ has intersections with~$z$ on~$\Gamma$.
Then, we see that $z\cap w\subset \Zig(z)$ or $z\cap w\subset \Zag(z)$.
\end{Lemma}

\begin{proof}
Let $e_1$, $e_2$ be edges of $\Gamma$, and we assume that $w$ intersects with $z$ at $e_1$ and~$e_2$.
If~$e_i$ is a~zig of~$z$, then it is a~zag of~$w$, and vice versa. We assume that $e_1$ is a zig of~$z$ and~$e_2$ is a~zag of~$z$.

We then lift these edges on the universal cover $\widetilde{\Gamma}$. Let $\widetilde{e_1}$, $\widetilde{e_2}$ be edges of $\widetilde{\Gamma}$ whose restrictions on $\Gamma$ are $e_1$, $e_2$, respectively. In particular, $\widetilde{e_1}$ is a~zig of~$\widetilde{z}$ and~$\widetilde{e_2}$ is a~zag of~$\widetilde{z}$.
Then, there exist zigzag paths $\widetilde{w}_1$, $\widetilde{w}_2$ on~$\widetilde{\Gamma}$ whose restriction on~$\Gamma$ is just~$w$,
and $\widetilde{w}_1$ (resp.~$\widetilde{w}_2$) intersects with~$\widetilde{z}$ at~$\widetilde{e_1}$ (resp.~$\widetilde{e_2}$).
The zigzag path $\widetilde{z}$ splits $\RR^2$ into two pieces, and~$\widetilde{w}_1$ (resp.~$\widetilde{w}_2$) intersects with~$\widetilde{z}$ from left to right (resp.\ right to left) by the definition of zigzag paths (see the figure below).

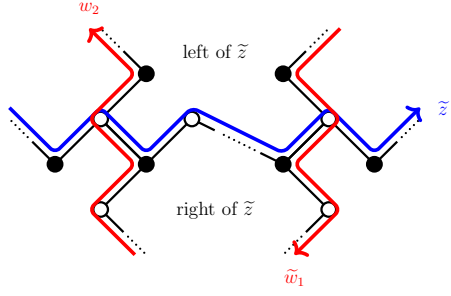
\begin{center}
\newcommand{\edgewidth}{0.05cm} 
\newcommand{\nodewidth}{0.05cm} 
\newcommand{\noderad}{0.16} 
\scalebox{0.6}{
\begin{tikzpicture}
\coordinate (B1) at (1,2); \coordinate (B2) at (3,2); \coordinate (B3) at (3,4);
\coordinate (B4) at (6,2); \coordinate (B5) at (6,4); \coordinate (B6) at (8,2);

\coordinate (W1) at (2,1); \coordinate (W2) at (2,3); \coordinate (W3) at (4,3);
\coordinate (W4) at (7,1); \coordinate (W5) at (7,3);
\draw [line width=\edgewidth] (2.5,0.5)--(W1)--(B2)--(W2)--(B3)--(2.5,4.5) ;
\draw [line width=\edgewidth,line cap=round, dash pattern=on 0pt off 2.5\pgflinewidth] (2.6,0.4)--(3,0);
\draw [line width=\edgewidth,line cap=round, dash pattern=on 0pt off 2.5\pgflinewidth] (2.4,4.6)--(2,5);
\draw [line width=\edgewidth] (0.5,2.5)--(B1)--(W2)--(B2)--(W3)--(4.5,2.75) ;
\draw [line width=\edgewidth,line cap=round, dash pattern=on 0pt off 2.5\pgflinewidth] (0.4,2.6)--(0,3);
\draw [line width=\edgewidth,line cap=round, dash pattern=on 0pt off 2.5\pgflinewidth] (4.7,2.65)--(5.4,2.3);
\draw [line width=\edgewidth] (5.5,2.25)--(B4)--(W5)--(B6)--(8.5,2.5) ;
\draw [line width=\edgewidth,line cap=round, dash pattern=on 0pt off 2.5\pgflinewidth] (8.6,2.6)--(9,3);
\draw [line width=\edgewidth] (6.5,0.5)--(W4)--(B4)--(W5)--(B5)--(6.5,4.5) ;
\draw [line width=\edgewidth,line cap=round, dash pattern=on 0pt off 2.5\pgflinewidth] (6.4,0.4)--(6,0);
\draw [line width=\edgewidth,line cap=round, dash pattern=on 0pt off 2.5\pgflinewidth] (6.6,4.6)--(7,5);

\draw [line width=\nodewidth, fill=black] (B1) circle [radius=\noderad] ; \draw [line width=\nodewidth, fill=black] (B2) circle [radius=\noderad] ;
\draw [line width=\nodewidth, fill=black] (B3) circle [radius=\noderad] ;
\draw [line width=\nodewidth, fill=black] (B4) circle [radius=\noderad] ; \draw [line width=\nodewidth, fill=black] (B5) circle [radius=\noderad] ;
\draw [line width=\nodewidth, fill=black] (B6) circle [radius=\noderad] ;
\draw [line width=\nodewidth, fill=white] (W1) circle [radius=\noderad] ; \draw [line width=\nodewidth, fill=white] (W2) circle [radius=\noderad] ;
\draw [line width=\nodewidth, fill=white] (W3) circle [radius=\noderad] ;
\draw [line width=\nodewidth, fill=white] (W4) circle [radius=\noderad] ; \draw [line width=\nodewidth, fill=white] (W5) circle [radius=\noderad] ;

\draw [->, rounded corners, line width=0.08cm, blue] (0,3.25)--(1,2.25)--(2,3.25)--(3,2.25)--(4,3.25)--(6,2.25)--(7,3.25)--(8,2.25)--(9,3.25);
\node[blue] at (9.5,3.25) {\Large$\widetilde{z}$} ;
\draw [->, rounded corners, line width=0.08cm, red] (2.75,0)--(1.75,1)--(2.75,2)--(1.75,3)--(2.75,4)--(1.75,5);
\node[red] at (1.75,5.5) {\Large$\widetilde{w}_2$} ;
\draw [<-, rounded corners, line width=0.08cm, red] (6.25,0)--(7.25,1)--(6.25,2)--(7.25,3)--(6.25,4)--(7.25,5);
\node[red] at (6.25,-0.5) {\Large$\widetilde{w}_1$} ;

\node at (4.5,1) {\Large right of $\widetilde{z}$} ; \node at (4.5,4.5) {\Large left of $\widetilde{z}$} ;
\end{tikzpicture}
}
\end{center}

Since $z$ is type I, there is no intersections of $\widetilde{z}$ and $\widetilde{w}_i$ except $\widetilde{e_i}$ where $i=1,2$.
Thus, we can not superimpose $\widetilde{w}_1$ and $\widetilde{w}_2$ using translations.
This contradicts a choice of $\widetilde{w}_1$, $\widetilde{w}_2$.
\end{proof}

The following lemma follows from the argument in the proof of \cite[Proposition~3.12]{Bro}.

\begin{Lemma}\label{slope_linearly_independent}
Let $z$, $w$ be zigzag paths on a consistent dimer model~$\Gamma$.
We assume that $\widetilde{z}$ intersects with $\widetilde{w}$ on the universal cover $\widetilde{\Gamma}$ at most once.
Then, the slopes $[z]$, $[w]$ are linearly independent if and only if~$\widetilde{z}$ and~$\widetilde{w}$ intersect on precisely one edge.
\end{Lemma}

\begin{Lemma}\label{intersect_zigorzag2}
Let $z_1,\dots,z_r$ be type~I zigzag paths on a reduced consistent dimer model $\Gamma$ having the same slope.
We suppose that a zigzag path~$w$ has intersections with $z_j$ for some $j$.
Then, $w$~intersects with every $z_i$ $(i=1,\dots,r)$, and the intersections $w\cap z_1,\dots, w\cap z_r$ are all in $\Zig(w)$ or all in $\Zag(w)$.
Moreover, we have $|w\cap z_1|=\cdots=|w\cap z_r|$.
\end{Lemma}

\begin{proof}
Since $z_1,\dots,z_r$ are type I, $\widetilde{w}$ intersects with each $\widetilde{z_i}$ at most once.
Thus, each pair of zigzag paths $(\widetilde{w},\widetilde{z_i})$ for $i=1,\dots,r$ satisfies the assumption in Lemma~\ref{slope_linearly_independent}.
Since~$w$ has intersections with~$z_j$, $\widetilde{w}$ intersects with $\widetilde{z_j}$ precisely once on the universal cover.
By Lemma~\ref{slope_linearly_independent}, we see that $[w]$ and $[z_j]$ are linearly independent, and hence $[w]$ and $[z_i]$ are linearly independent for any~$i$.
Thus, $\widetilde{w}$ intersects with $\widetilde{z_i}$ precisely once for any $i=1,\dots,r$.

The latter assertion follows from a similar argument as in the proof of Lemma~\ref{intersect_zigorzag}.
More precisely, since $\widetilde{w}$ intersects with $\widetilde{z_i}$ precisely once for any $i=1,\dots,r$,
if $\widetilde{w}$ intersects with $\widetilde{z_i}$ from right to left (resp.\ left to right), then so do zigzag paths having the same slope.
This means that intersections are all in $\Zig(w)$ (resp.\ all in $\Zag(w))$.

Also, let $\widetilde{z_1}$ and $\widetilde{z_1}^\prime$ be zigzag paths on $\widetilde{\Gamma}$ that are projected onto the zigzag path $z_1$ on $\Gamma$.
We assume that there is no zigzag path projected onto $z_1$ between $\widetilde{z_1}$ and $\widetilde{z_1}^\prime$.
Also, we assume that~$\widetilde{w}$ first intersects with~$\widetilde{z_1}$, then intersects with $\widetilde{z_1}^\prime$.
Then, for all $i=2,\dots,r$ we can find a~unique zigzag path~$\widetilde{z_i}$ on $\widetilde{\Gamma}$ such that it is projected onto $z_i$ and is located between~$\widetilde{z_1}$ and~$\widetilde{z_1}^\prime$.
Thus, after $\widetilde{w}$ intersects with $\widetilde{z_1}$, it intersects with $\widetilde{z_2}, \dots, \widetilde{z_r}$ precisely once and then arrives at~$\widetilde{z_1}^\prime$.
We can apply the same arguments for any pair $(\widetilde{z_1},\widetilde{z_1}^\prime)$ of zigzag paths on~$\widetilde{\Gamma}$ satisfying the above properties. Thus, projecting onto~$\Gamma$ we have $|w\cap z_1|=\cdots=|w\cap z_r|$.
\end{proof}

\section{Deformations of dimer models}\label{sec_def_deform}

In this section, we will introduce the concept of a deformation of consistent dimer models.
This operation is defined for type I zigzag paths on a consistent dimer model, and there are two kinds of deformations,
which we call the zig-deformation (see Definition~\ref{def_deformation_zig}) and the zag-deformation (see Definition~\ref{def_deformation_zag}).
For some special cases, these deformations preserve the consistency condition (see Proposition~\ref{prop_preserve_consistent1}),
but they change the associated PM polygon, whereas we will see that the PM polygon of the deformed dimer model is exactly the combinatorial mutation of a~polygon (see Theorem~\ref{mutation=deformation1}).

\subsection{Definition of deformations of dimer models}\label{subsec_def_deform}

Let $\Gamma$ be a reduced consistent dimer model.
In particular, any slope of a zigzag path on $\Gamma$ is primitive.
Let $\calZ_v(\Gamma)$ be the subset of zigzag paths on $\Gamma$ whose slopes are the primitive vector~$v\in\ZZ^2$,
and $\calZ_v^\rmI(\Gamma)$ be the subset of $\calZ_v(\Gamma)$ consisting of type I zigzag paths.
We first prepare the \emph{deformation data}.

\begin{Definition}[deformation data]\label{def_deformation_data}
Let $\Gamma$ be a reduced consistent dimer model. In order to define the deformation of $\Gamma$, we fix the following data.
\begin{itemize}\itemsep=0pt
\item[(1)] We choose a type I zigzag path $z$, and let $2n\coloneqq\ell(z)$ and $v\coloneqq [z]$.
\item[(2)] We then fix a positive integer $r$ such that $r\le |\calZ_v^\rmI(\Gamma)|$, and let $h=n-r$.
\item[(3)] We take a subset $\{z_1,\dots,z_r\}\subset\calZ_v^\rmI(\Gamma)$ of type I zigzag paths, in which case we have $2n=\ell(z_1)=\cdots=\ell(z_r)$ by Lemma~\ref{lem_same_length}.
Therefore, each $z_i$ can be described as
\begin{displaymath} z_i=z_i[1]z_i[2]\cdots z_i[2n-1]z_i[2n]. \end{displaymath}
\item[(4)]
We consider non-negative integers $p_1,\dots,p_r\in\ZZ_{\ge0}$ such that $p_1+\cdots+p_r=n-r=h$.
We call $\bfp\coloneqq(p_1,\dots,p_r)$ a \emph{deformation weight} of $\{z_1, \dots, z_r\}$.
\end{itemize}
\end{Definition}

\begin{Remark}\label{rem_typeI_to_typeII}
To define the deformation data, we need a type I zigzag path.
If a dimer model is isoradial, then any zigzag path is type I (see Definition~\ref{def_isoradial}), and hence $\big|\calZ_v^\rmI(\Gamma)\big|=|\calZ_v(\Gamma)|$.
Also, even if $\Gamma$ contains no type I zigzag paths, we sometimes change a type II zigzag path into type I by using the \emph{mutations of dimer models}
(see Appendix~\ref{app_mutationdimer}, especially Example~\ref{mutation_typeII_to_I}).
\end{Remark}

\begin{Definition}[zig-deformation]\label{def_deformation_zig}
Let the notation be the same as in Definition~\ref{def_deformation_data}.
For a~deformation weight $\bfp=(p_1,\dots,p_r)$ of $\{z_1, \dots, z_r\}$, we consider the following procedures:
\begin{enumerate}[(\text{zig}-1)]
\setlength{\parskip}{0pt}\itemsep=0pt
\item Using split moves, we insert $p_i$ white nodes and $p_i$ black nodes in each zig of $z_i$.

\noindent{\bf [Notation]}
\begin{itemize}\itemsep=0pt
\item For any zig $z_i[2m-1]$ of $z_i$ where $m=1,\dots,n$ and $i=1,\dots,r$, we denote by $b_i[2m-1]$ (resp.~$w_i[2m-1]$) the black (resp.\ white) node that is the endpoint of $z_i[2m-1]$.
\item We denote the white nodes added in the zig $z_i[2m-1]$ by $w_{i,1}[2m-1],\dots,w_{i,p_i}[2m-1]$, and the black ones by $b_{i,1}[2m-1],\dots,b_{i,p_i}[2m-1]$.
Here, the subscripts increase in the direction from $b_i[2m-1]$ to $w_i[2m-1]$.
\end{itemize}
\item We remove all zags of $z_i$ for every $i=1,\dots,r$.
\item If $p_i\neq 0$, then we connect the white node $w_{i,j}[2m-1]$ to the black node $b_{i,j}[2m+1]$ where $j=1,\dots,p_i$ and $m=1,\dots,n$.
(Note that $w_{i,j}[2n-1]$ is connected to $b_{i,j}[2n+1]\coloneqq b_{i,j}[1]$.)
We denote by $z_{i,j}$ the new $1$-cycle, which will be a zigzag path on the deformed dimer model, obtained by connecting
\begin{displaymath}
w_{i,j}[2n-1], b_{i,j}[2n-1], w_{i,j}[2n-3], b_{i,j}[2n_i-3],\dots,w_{i,j}[1], b_{i,j}[1]
\end{displaymath}
cyclically ($i=1,\dots,r$ and $j=1,\dots,p_i$).
\item[(join)] If there exist $2$-valent nodes, then we apply the join moves to the dimer model obtained by the above procedures and make it reduced.
\end{enumerate}
We denote the resulting dimer model by $\nu^\zig_\bfp(\Gamma, \{z_1,\dots,z_r\})$,
and call it the \emph{zig-deformation of~$\Gamma$ at $\{z_1,\dots,z_r\}$ with respect to the weight $\bfp$}.
If a situation is clear, we simply denote this by~$\nu^\zig_\bfp(\Gamma)$.
\end{Definition}

\begin{figure}[h!]\centering
\scalebox{0.55}{
\begin{tikzpicture}
\newcommand{\edgewidth}{0.05cm} 
\newcommand{\nodewidth}{0.05cm} 
\newcommand{\noderad}{0.16} 
\coordinate (W1) at (0,1.2); \coordinate (W2) at (0,3.6);
\path (W1) ++(-20:4cm) coordinate (B1); \path (W2) ++(-20:4cm) coordinate (B2);
\path (B1) ++(200:2cm) coordinate (B0); \path (W2) ++(20:2cm) coordinate (W3);

\path (W1) ++(-20:0.8cm) coordinate (B1-1); \path (W1) ++(-20:1.6cm) coordinate (W1-1);
\path (W1) ++(-20:2.4cm) coordinate (B1-2); \path (W1) ++(-20:3.2cm) coordinate (W1-2);
\path (W2) ++(-20:0.8cm) coordinate (B2-1); \path (W2) ++(-20:1.6cm) coordinate (W2-1);
\path (W2) ++(-20:2.4cm) coordinate (B2-2); \path (W2) ++(-20:3.2cm) coordinate (W2-2);

\path (B1) ++(30:1cm) coordinate (B1e); \path (B1) ++(330:1cm) coordinate (B1s);
\path (W1) ++(150:1cm) coordinate (W1w); \path (W1) ++(210:1cm) coordinate (W1s);
\path (B2) ++(30:1cm) coordinate (B2e); \path (B2) ++(330:1cm) coordinate (B2s);
\path (W2) ++(150:1cm) coordinate (W2w); \path (W2) ++(210:1cm) coordinate (W2s);

\path (W1) ++(160:0.7cm) coordinate (W1+); \path (W1) ++(200:0.7cm) coordinate (W1-);
\path (W2) ++(160:0.7cm) coordinate (W2+); \path (W2) ++(200:0.7cm) coordinate (W2-);
\path (B1) ++(20:0.7cm) coordinate (B1+); \path (B1) ++(340:0.7cm) coordinate (B1-);
\path (B2) ++(20:0.7cm) coordinate (B2+); \path (B2) ++(340:0.7cm) coordinate (B2-);

\node (original) at (0,0) {
\begin{tikzpicture}
\draw [line width=\edgewidth] (B1)--(W1); \draw [line width=\edgewidth] (B2)--(W1) ; \draw [line width=\edgewidth] (B2)--(W2) ;
\draw [line width=\edgewidth] (B1)--(B0) ; \draw [line width=\edgewidth] (W2)--(W3) ;
\draw [line width=\edgewidth] (B1)--(B1e); \draw [line width=\edgewidth] (B1)--(B1s);
\draw [line width=\edgewidth] (W1)--(W1w); \draw [line width=\edgewidth] (W1)--(W1s);
\draw [line width=\edgewidth] (B2)--(B2e); \draw [line width=\edgewidth] (B2)--(B2s);
\draw [line width=\edgewidth] (W2)--(W2w); \draw [line width=\edgewidth] (W2)--(W2s);

\draw [line width=\edgewidth,line cap=round, dash pattern=on 0pt off 2.5\pgflinewidth] (W1+)--(W1-);
\draw [line width=\edgewidth,line cap=round, dash pattern=on 0pt off 2.5\pgflinewidth] (W2+)--(W2-);
\draw [line width=\edgewidth,line cap=round, dash pattern=on 0pt off 2.5\pgflinewidth] (B1+)--(B1-);
\draw [line width=\edgewidth,line cap=round, dash pattern=on 0pt off 2.5\pgflinewidth] (B2+)--(B2-);
\draw [line width=\nodewidth, fill=black] (B1) circle [radius=\noderad] ; \draw [line width=\nodewidth, fill=black] (B2) circle [radius=\noderad] ;
\draw [line width=\nodewidth, fill=white] (W1) circle [radius=\noderad] ; \draw [line width=\nodewidth, fill=white] (W2) circle [radius=\noderad] ;

\coordinate (W1z) at (0,1.2); \coordinate (W2z) at (0,3.6);
\path (W1) ++(90:0.2cm) coordinate (W1z); \path (W2) ++(90:0.2cm) coordinate (W2z);
\path (W1z) ++(-20:4cm) coordinate (B1z); \path (W2z) ++(-20:4cm) coordinate (B2z);
\path (B1z) ++(200:2cm) coordinate (B0z); \path (W2z) ++(20:2cm) coordinate (W3z);
\path (B1z) ++(200:2cm) coordinate (B0z); \path (W2z) ++(20:2.2cm) coordinate (W3z);

\path (W1) ++(-90:0.2cm) coordinate (W1zz); \path (W1zz) ++(150:1.5cm) coordinate (W1wzz);
\draw [->, rounded corners, line width=0.08cm, red] (B0z)--(B1z)--(W1z)--(B2z)--(W2z)--(W3z);

\node at (2.5,3.5) {\color{red}\large$z_i[2m+1]$};
\node at (1.2,2.2) {\color{red}\large$z_i[2m]$};
\node at (2.5,1.1) {\color{red}\large$z_i[2m-1]$};
\end{tikzpicture}};

\node (original_insert) at (10,0) {
\begin{tikzpicture}
\draw [line width=\edgewidth] (B1)--(W1); \draw [line width=\edgewidth] (B2)--(W1) ; \draw [line width=\edgewidth] (B2)--(W2) ;
\draw [line width=\edgewidth] (B1)--(B0) ; \draw [line width=\edgewidth] (W2)--(W3) ;
\draw [line width=\edgewidth] (B1)--(B1e); \draw [line width=\edgewidth] (B1)--(B1s);
\draw [line width=\edgewidth] (W1)--(W1w); \draw [line width=\edgewidth] (W1)--(W1s);
\draw [line width=\edgewidth] (B2)--(B2e); \draw [line width=\edgewidth] (B2)--(B2s);
\draw [line width=\edgewidth] (W2)--(W2w); \draw [line width=\edgewidth] (W2)--(W2s);

\draw [line width=\edgewidth,line cap=round, dash pattern=on 0pt off 2.5\pgflinewidth] (W1+)--(W1-);
\draw [line width=\edgewidth,line cap=round, dash pattern=on 0pt off 2.5\pgflinewidth] (W2+)--(W2-);
\draw [line width=\edgewidth,line cap=round, dash pattern=on 0pt off 2.5\pgflinewidth] (B1+)--(B1-);
\draw [line width=\edgewidth,line cap=round, dash pattern=on 0pt off 2.5\pgflinewidth] (B2+)--(B2-);
\draw [line width=\nodewidth, fill=black] (B1) circle [radius=\noderad] ; \draw [line width=\nodewidth, fill=black] (B2) circle [radius=\noderad] ;
\draw [line width=\nodewidth, fill=black] (B1-1) circle [radius=\noderad] ; \draw [line width=\nodewidth, fill=black] (B1-2) circle [radius=\noderad] ;
\draw [line width=\nodewidth, fill=black] (B2-1) circle [radius=\noderad] ; \draw [line width=\nodewidth, fill=black] (B2-2) circle [radius=\noderad] ;
\draw [line width=\nodewidth, fill=white] (W1) circle [radius=\noderad] ; \draw [line width=\nodewidth, fill=white] (W2) circle [radius=\noderad] ;
\draw [line width=\nodewidth, fill=white] (W1-1) circle [radius=\noderad] ; \draw [line width=\nodewidth, fill=white] (W1-2) circle [radius=\noderad] ;
\draw [line width=\nodewidth, fill=white] (W2-1) circle [radius=\noderad] ; \draw [line width=\nodewidth, fill=white] (W2-2) circle [radius=\noderad] ;
\end{tikzpicture}};

\node (original_remove) at (10,-6.5) {
\begin{tikzpicture}
\draw [line width=\edgewidth] (B1)--(W1); \draw [line width=\edgewidth] (B2)--(W2);
\draw [line width=\edgewidth] (B1)--(B1e); \draw [line width=\edgewidth] (B1)--(B1s);
\draw [line width=\edgewidth] (W1)--(W1w); \draw [line width=\edgewidth] (W1)--(W1s);
\draw [line width=\edgewidth] (B2)--(B2e); \draw [line width=\edgewidth] (B2)--(B2s);
\draw [line width=\edgewidth] (W2)--(W2w); \draw [line width=\edgewidth] (W2)--(W2s);

\draw [line width=\edgewidth,line cap=round, dash pattern=on 0pt off 2.5\pgflinewidth] (W1+)--(W1-);
\draw [line width=\edgewidth,line cap=round, dash pattern=on 0pt off 2.5\pgflinewidth] (W2+)--(W2-);
\draw [line width=\edgewidth,line cap=round, dash pattern=on 0pt off 2.5\pgflinewidth] (B1+)--(B1-);
\draw [line width=\edgewidth,line cap=round, dash pattern=on 0pt off 2.5\pgflinewidth] (B2+)--(B2-);
\draw [line width=\nodewidth, fill=black] (B1) circle [radius=\noderad] ; \draw [line width=\nodewidth, fill=black] (B2) circle [radius=\noderad] ;
\draw [line width=\nodewidth, fill=black] (B1-1) circle [radius=\noderad] ; \draw [line width=\nodewidth, fill=black] (B1-2) circle [radius=\noderad] ;
\draw [line width=\nodewidth, fill=black] (B2-1) circle [radius=\noderad] ; \draw [line width=\nodewidth, fill=black] (B2-2) circle [radius=\noderad] ;
\draw [line width=\nodewidth, fill=white] (W1) circle [radius=\noderad] ; \draw [line width=\nodewidth, fill=white] (W2) circle [radius=\noderad] ;
\draw [line width=\nodewidth, fill=white] (W1-1) circle [radius=\noderad] ; \draw [line width=\nodewidth, fill=white] (W1-2) circle [radius=\noderad] ;
\draw [line width=\nodewidth, fill=white] (W2-1) circle [radius=\noderad] ; \draw [line width=\nodewidth, fill=white] (W2-2) circle [radius=\noderad] ;
\end{tikzpicture}};

\node (before) at (20,-6.5) {
\begin{tikzpicture}
\draw [line width=\edgewidth] (B1)--(W1); \draw [line width=\edgewidth] (B2)--(W2) ;
\draw [line width=\edgewidth] (B1)--(B1e); \draw [line width=\edgewidth] (B1)--(B1s);
\draw [line width=\edgewidth] (W1)--(W1w); \draw [line width=\edgewidth] (W1)--(W1s);
\draw [line width=\edgewidth] (B2)--(B2e); \draw [line width=\edgewidth] (B2)--(B2s);
\draw [line width=\edgewidth] (W2)--(W2w); \draw [line width=\edgewidth] (W2)--(W2s);

\draw [line width=\edgewidth,line cap=round, dash pattern=on 0pt off 2.5\pgflinewidth] (W1+)--(W1-);
\draw [line width=\edgewidth,line cap=round, dash pattern=on 0pt off 2.5\pgflinewidth] (W2+)--(W2-);
\draw [line width=\edgewidth,line cap=round, dash pattern=on 0pt off 2.5\pgflinewidth] (B1+)--(B1-);
\draw [line width=\edgewidth,line cap=round, dash pattern=on 0pt off 2.5\pgflinewidth] (B2+)--(B2-);

\draw [line width=\edgewidth] (B2-1)--(W1-1); \draw [line width=\edgewidth] (B2-2)--(W1-2);
\path (B1-1) ++(-70:1.5cm) coordinate (B1-1-); \draw [line width=\edgewidth] (B1-1)--(B1-1-);
\path (B1-2) ++(-70:1.5cm) coordinate (B1-2-); \draw [line width=\edgewidth] (B1-2)--(B1-2-);
\path (W2-1) ++(110:1.5cm) coordinate (W2-1-); \draw [line width=\edgewidth] (W2-1)--(W2-1-);
\path (W2-2) ++(110:1.5cm) coordinate (W2-2-); \draw [line width=\edgewidth] (W2-2)--(W2-2-);
\draw [line width=\nodewidth, fill=black] (B1) circle [radius=\noderad] ; \draw [line width=\nodewidth, fill=black] (B2) circle [radius=\noderad] ;
\draw [line width=\nodewidth, fill=black] (B1-1) circle [radius=\noderad] ; \draw [line width=\nodewidth, fill=black] (B1-2) circle [radius=\noderad] ;
\draw [line width=\nodewidth, fill=black] (B2-1) circle [radius=\noderad] ; \draw [line width=\nodewidth, fill=black] (B2-2) circle [radius=\noderad] ;
\draw [line width=\nodewidth, fill=white] (W1) circle [radius=\noderad] ; \draw [line width=\nodewidth, fill=white] (W2) circle [radius=\noderad] ;
\draw [line width=\nodewidth, fill=white] (W1-1) circle [radius=\noderad] ; \draw [line width=\nodewidth, fill=white] (W1-2) circle [radius=\noderad] ;
\draw [line width=\nodewidth, fill=white] (W2-1) circle [radius=\noderad] ; \draw [line width=\nodewidth, fill=white] (W2-2) circle [radius=\noderad] ;
\end{tikzpicture}};

\path (original) ++(0:3.5cm) coordinate (original+); \path (original_insert) ++(180:3.5cm) coordinate (original_insert+);
\draw [->, line width=\edgewidth] (original+)--(original_insert+) node[midway,xshift=0cm,yshift=0.5cm] {\LARGE(zig-1)} ;

\path (original_insert) ++(0:3.5cm) coordinate (original_insert-); \path (original_remove) ++(180:3.5cm) coordinate (original_remove+);
\path (original_remove) ++(0:3.5cm) coordinate (original_remove++);
\path (before) ++(180:3.5cm) coordinate (before-); \path (0,-6.5) ++(0:3.5cm) coordinate (before--);

\draw [->, line width=\edgewidth] (before--)--(original_remove+) node[midway,xshift=0cm,yshift=0.5cm] {\LARGE(zig-2)} ;
\draw [->, line width=\edgewidth] (original_remove++)--(before-) node[midway,xshift=0cm,yshift=0.5cm] {\LARGE(zig-3)} ;

\end{tikzpicture}
}
\caption{The zig-deformation at $z_i$ with $p_i=2$.}\label{fig_zigdeform_p2}
\end{figure}

Similarly, we can define the ``zag version'' of this deformation as follows.

\begin{Definition}[zag-deformation]\label{def_deformation_zag}
Let the notation be the same as Definition~\ref{def_deformation_data}.
For a deformation weight $\bfp=(p_1,\dots,p_r)$ of $\{z_1, \dots, z_r\}$, we consider the following procedures:
\begin{enumerate}[(\text{zag}-1)]
\setlength{\parskip}{0pt}
\itemsep=0pt
\item Using split moves, we insert $p_i$ white nodes and $p_i$ black nodes in each zag of $z_i$.

\noindent{\bf [Notation]}
\begin{itemize}
\setlength{\parskip}{0pt}
\setlength{\itemsep}{3pt}
\item For any zag $z_i[2m]$ of $z_i$ where $m=1,\dots,n$ and $i=1,\dots,r$, we denote by $w_i[2m]$ (resp.~$b_i[2m]$) the white (resp.\ black) node that is the endpoint of $z_i[2m]$.
\item We denote the white nodes added in the zag $z_i[2m]$ by $w_{i,1}[2m],\dots,w_{i,p_i}[2m]$, and the black ones by $b_{i,1}[2m],\dots,b_{i,p_i}[2m]$.
Here, the subscripts increase in the direction from $w_i[2m]$ to $b_i[2m]$.
\end{itemize}
\item We remove all zigs of $z_i$ for every $i=1,\dots,r$.
\item If $p_i\neq 0$, then we connect the black node $b_{i,j}[2m]$ to the white node $w_{i,j}[2m+2]$ where $j=1,\dots,p_i$ and $m=1,\dots,n$.
(Note that $b_{i,j}[2n]$ is connected to $w_{i,j}[2n+2]\coloneqq w_{i,j}[2]$.)
We denote by $z_{i,j}$ the new $1$-cycle, which will be a zigzag path on the deformed dimer model, obtained by connecting
\begin{displaymath}
b_{i,j}[2n], w_{i,j}[2n], b_{i,j}[2n-2], w_{i,j}[2n-2],\dots,b_{i,j}[2], w_{i,j}[2]
\end{displaymath}
cyclically ($i=1,\dots,r$ and $j=1,\dots,p_i$).
\item[(join)] If there exist $2$-valent nodes, then we apply the join moves to the dimer model obtained by the above procedures and make it reduced.
\end{enumerate}
We denote the resulting dimer model by $\nu^\zag_\bfp(\Gamma, \{z_1,\dots,z_r\})$,
and call it the \emph{zag-deformation of~$\Gamma$ at $\{z_1,\dots,z_r\}$ with respect to the weight~$\bfp$}.
If a situation is clear, we simply denote this by~$\nu^\zag_\bfp(\Gamma)$.
\end{Definition}

\begin{figure}[h!]\centering
\scalebox{0.55}{
\begin{tikzpicture}
\newcommand{\edgewidth}{0.05cm} 
\newcommand{\nodewidth}{0.05cm} 
\newcommand{\noderad}{0.16} 
\coordinate (W1) at (0,0); \coordinate (W2) at (0,2.4);
\path (W1) ++(20:4cm) coordinate (B1); \path (W2) ++(20:4cm) coordinate (B2);
\path (W1) ++(340:2cm) coordinate (W0); \path (B2) ++(160:2cm) coordinate (B3);

\path (W1) ++(20:0.8cm) coordinate (B1-1); \path (W1) ++(20:1.6cm) coordinate (W1-1);
\path (W1) ++(20:2.4cm) coordinate (B1-2); \path (W1) ++(20:3.2cm) coordinate (W1-2);
\path (W2) ++(20:0.8cm) coordinate (B2-1); \path (W2) ++(20:1.6cm) coordinate (W2-1);
\path (W2) ++(20:2.4cm) coordinate (B2-2); \path (W2) ++(20:3.2cm) coordinate (W2-2);

\path (B1) ++(30:1cm) coordinate (B1e); \path (B1) ++(330:1cm) coordinate (B1s);
\path (W1) ++(150:1cm) coordinate (W1w); \path (W1) ++(210:1cm) coordinate (W1s);
\path (B2) ++(30:1cm) coordinate (B2e); \path (B2) ++(330:1cm) coordinate (B2s);
\path (W2) ++(150:1cm) coordinate (W2w); \path (W2) ++(210:1cm) coordinate (W2s);

\path (W1) ++(160:0.7cm) coordinate (W1+); \path (W1) ++(200:0.7cm) coordinate (W1-);
\path (W2) ++(160:0.7cm) coordinate (W2+); \path (W2) ++(200:0.7cm) coordinate (W2-);
\path (B1) ++(20:0.7cm) coordinate (B1+); \path (B1) ++(340:0.7cm) coordinate (B1-);
\path (B2) ++(20:0.7cm) coordinate (B2+); \path (B2) ++(340:0.7cm) coordinate (B2-);

\node (original) at (0,0) {
\begin{tikzpicture}
\draw [line width=\edgewidth] (W1)--(B1); \draw [line width=\edgewidth] (B1)--(W2) ; \draw [line width=\edgewidth] (W2)--(B2) ;
\draw [line width=\edgewidth] (W0)--(W1) ; \draw [line width=\edgewidth] (B2)--(B3) ;
\draw [line width=\edgewidth] (B1)--(B1e); \draw [line width=\edgewidth] (B1)--(B1s);
\draw [line width=\edgewidth] (W1)--(W1w); \draw [line width=\edgewidth] (W1)--(W1s);
\draw [line width=\edgewidth] (B2)--(B2e); \draw [line width=\edgewidth] (B2)--(B2s);
\draw [line width=\edgewidth] (W2)--(W2w); \draw [line width=\edgewidth] (W2)--(W2s);

\draw [line width=\edgewidth,line cap=round, dash pattern=on 0pt off 2.5\pgflinewidth] (W1+)--(W1-);
\draw [line width=\edgewidth,line cap=round, dash pattern=on 0pt off 2.5\pgflinewidth] (W2+)--(W2-);
\draw [line width=\edgewidth,line cap=round, dash pattern=on 0pt off 2.5\pgflinewidth] (B1+)--(B1-);
\draw [line width=\edgewidth,line cap=round, dash pattern=on 0pt off 2.5\pgflinewidth] (B2+)--(B2-);
\draw [line width=\nodewidth, fill=black] (B1) circle [radius=\noderad] ; \draw [line width=\nodewidth, fill=black] (B2) circle [radius=\noderad] ;
\draw [line width=\nodewidth, fill=white] (W1) circle [radius=\noderad] ; \draw [line width=\nodewidth, fill=white] (W2) circle [radius=\noderad] ;

\coordinate (W1z) at (0,1.2); \coordinate (W2z) at (0,3.6);
\path (W1) ++(90:0.2cm) coordinate (W1z); \path (W1z) ++(-20:2cm) coordinate (W0z);
\path (B1) ++(90:0.2cm) coordinate (B1z); \path (W2) ++(90:0.2cm) coordinate (W2z);
\path (B2) ++(90:0.2cm) coordinate (B2z); \path (B2z) ++(160:2.2cm) coordinate (B3z);

\path (W2) ++(-90:0.2cm) coordinate (W2zz); \path (W2zz) ++(210:1cm) coordinate (W2wzz);

\draw [->, rounded corners, line width=0.08cm, red] (W0z)--(W1z)--(B1z)--(W2z)--(B2z)--(B3z);

\node at (1.4,3.7) {\color{red}\large$z_i[2m+2]$};
\node at (2.4,2.4) {\color{red}\large$z_i[2m+1]$};
\node at (1.4,1.2) {\color{red}\large$z_i[2m]$};
\end{tikzpicture}};

\node (original_insert) at (10,0) {
\begin{tikzpicture}
\draw [line width=\edgewidth] (W1)--(B1); \draw [line width=\edgewidth] (B1)--(W2) ; \draw [line width=\edgewidth] (W2)--(B2) ;
\draw [line width=\edgewidth] (W0)--(W1) ; \draw [line width=\edgewidth] (B2)--(B3) ;
\draw [line width=\edgewidth] (B1)--(B1e); \draw [line width=\edgewidth] (B1)--(B1s);
\draw [line width=\edgewidth] (W1)--(W1w); \draw [line width=\edgewidth] (W1)--(W1s);
\draw [line width=\edgewidth] (B2)--(B2e); \draw [line width=\edgewidth] (B2)--(B2s);
\draw [line width=\edgewidth] (W2)--(W2w); \draw [line width=\edgewidth] (W2)--(W2s);

\draw [line width=\edgewidth,line cap=round, dash pattern=on 0pt off 2.5\pgflinewidth] (W1+)--(W1-);
\draw [line width=\edgewidth,line cap=round, dash pattern=on 0pt off 2.5\pgflinewidth] (W2+)--(W2-);
\draw [line width=\edgewidth,line cap=round, dash pattern=on 0pt off 2.5\pgflinewidth] (B1+)--(B1-);
\draw [line width=\edgewidth,line cap=round, dash pattern=on 0pt off 2.5\pgflinewidth] (B2+)--(B2-);
\draw [line width=\nodewidth, fill=black] (B1) circle [radius=\noderad] ; \draw [line width=\nodewidth, fill=black] (B2) circle [radius=\noderad] ;
\draw [line width=\nodewidth, fill=black] (B1-1) circle [radius=\noderad] ; \draw [line width=\nodewidth, fill=black] (B1-2) circle [radius=\noderad] ;
\draw [line width=\nodewidth, fill=black] (B2-1) circle [radius=\noderad] ; \draw [line width=\nodewidth, fill=black] (B2-2) circle [radius=\noderad] ;
\draw [line width=\nodewidth, fill=white] (W1) circle [radius=\noderad] ; \draw [line width=\nodewidth, fill=white] (W2) circle [radius=\noderad] ;
\draw [line width=\nodewidth, fill=white] (W1-1) circle [radius=\noderad] ; \draw [line width=\nodewidth, fill=white] (W1-2) circle [radius=\noderad] ;
\draw [line width=\nodewidth, fill=white] (W2-1) circle [radius=\noderad] ; \draw [line width=\nodewidth, fill=white] (W2-2) circle [radius=\noderad] ;
\end{tikzpicture}};

\node (original_remove) at (10,-6.5) {
\begin{tikzpicture}
\draw [line width=\edgewidth] (W1)--(B1); 
\draw [line width=\edgewidth] (W2)--(B2) ;
\draw [line width=\edgewidth] (B1)--(B1e); \draw [line width=\edgewidth] (B1)--(B1s);
\draw [line width=\edgewidth] (W1)--(W1w); \draw [line width=\edgewidth] (W1)--(W1s);
\draw [line width=\edgewidth] (B2)--(B2e); \draw [line width=\edgewidth] (B2)--(B2s);
\draw [line width=\edgewidth] (W2)--(W2w); \draw [line width=\edgewidth] (W2)--(W2s);

\draw [line width=\edgewidth,line cap=round, dash pattern=on 0pt off 2.5\pgflinewidth] (W1+)--(W1-);
\draw [line width=\edgewidth,line cap=round, dash pattern=on 0pt off 2.5\pgflinewidth] (W2+)--(W2-);
\draw [line width=\edgewidth,line cap=round, dash pattern=on 0pt off 2.5\pgflinewidth] (B1+)--(B1-);
\draw [line width=\edgewidth,line cap=round, dash pattern=on 0pt off 2.5\pgflinewidth] (B2+)--(B2-);
\draw [line width=\nodewidth, fill=black] (B1) circle [radius=\noderad] ; \draw [line width=\nodewidth, fill=black] (B2) circle [radius=\noderad] ;
\draw [line width=\nodewidth, fill=black] (B1-1) circle [radius=\noderad] ; \draw [line width=\nodewidth, fill=black] (B1-2) circle [radius=\noderad] ;
\draw [line width=\nodewidth, fill=black] (B2-1) circle [radius=\noderad] ; \draw [line width=\nodewidth, fill=black] (B2-2) circle [radius=\noderad] ;
\draw [line width=\nodewidth, fill=white] (W1) circle [radius=\noderad] ; \draw [line width=\nodewidth, fill=white] (W2) circle [radius=\noderad] ;
\draw [line width=\nodewidth, fill=white] (W1-1) circle [radius=\noderad] ; \draw [line width=\nodewidth, fill=white] (W1-2) circle [radius=\noderad] ;
\draw [line width=\nodewidth, fill=white] (W2-1) circle [radius=\noderad] ; \draw [line width=\nodewidth, fill=white] (W2-2) circle [radius=\noderad] ;
\end{tikzpicture}};

\node (before) at (20,-6.5) {
\begin{tikzpicture}
\draw [line width=\edgewidth] (W1)--(B1); \draw [line width=\edgewidth] (W2)--(B2) ;
\draw [line width=\edgewidth] (B1)--(B1e); \draw [line width=\edgewidth] (B1)--(B1s);
\draw [line width=\edgewidth] (W1)--(W1w); \draw [line width=\edgewidth] (W1)--(W1s);
\draw [line width=\edgewidth] (B2)--(B2e); \draw [line width=\edgewidth] (B2)--(B2s);
\draw [line width=\edgewidth] (W2)--(W2w); \draw [line width=\edgewidth] (W2)--(W2s);

\draw [line width=\edgewidth] (B1-1)--(W2-1); \draw [line width=\edgewidth] (B1-2)--(W2-2);

\path (B2-1) ++(70:1.5cm) coordinate (B2-1-); \draw [line width=\edgewidth] (B2-1)--(B2-1-);
\path (B2-2) ++(70:1.5cm) coordinate (B2-2-); \draw [line width=\edgewidth] (B2-2)--(B2-2-);
\path (W1-1) ++(-110:1.5cm) coordinate (W1-1-); \draw [line width=\edgewidth] (W1-1)--(W1-1-);
\path (W1-2) ++(-110:1.5cm) coordinate (W1-2-); \draw [line width=\edgewidth] (W1-2)--(W1-2-);

\draw [line width=\edgewidth,line cap=round, dash pattern=on 0pt off 2.5\pgflinewidth] (W1+)--(W1-);
\draw [line width=\edgewidth,line cap=round, dash pattern=on 0pt off 2.5\pgflinewidth] (W2+)--(W2-);
\draw [line width=\edgewidth,line cap=round, dash pattern=on 0pt off 2.5\pgflinewidth] (B1+)--(B1-);
\draw [line width=\edgewidth,line cap=round, dash pattern=on 0pt off 2.5\pgflinewidth] (B2+)--(B2-);
\draw [line width=\nodewidth, fill=black] (B1) circle [radius=\noderad] ; \draw [line width=\nodewidth, fill=black] (B2) circle [radius=\noderad] ;
\draw [line width=\nodewidth, fill=black] (B1-1) circle [radius=\noderad] ; \draw [line width=\nodewidth, fill=black] (B1-2) circle [radius=\noderad] ;
\draw [line width=\nodewidth, fill=black] (B2-1) circle [radius=\noderad] ; \draw [line width=\nodewidth, fill=black] (B2-2) circle [radius=\noderad] ;
\draw [line width=\nodewidth, fill=white] (W1) circle [radius=\noderad] ; \draw [line width=\nodewidth, fill=white] (W2) circle [radius=\noderad] ;
\draw [line width=\nodewidth, fill=white] (W1-1) circle [radius=\noderad] ; \draw [line width=\nodewidth, fill=white] (W1-2) circle [radius=\noderad] ;
\draw [line width=\nodewidth, fill=white] (W2-1) circle [radius=\noderad] ; \draw [line width=\nodewidth, fill=white] (W2-2) circle [radius=\noderad] ;
\end{tikzpicture}};

\path (original) ++(0:3.5cm) coordinate (original+); \path (original_insert) ++(180:3.5cm) coordinate (original_insert+);
\draw [->, line width=\edgewidth] (original+)--(original_insert+) node[midway,xshift=0cm,yshift=0.5cm] {\LARGE(zag-1)} ;

\path (original_insert) ++(0:3.5cm) coordinate (original_insert-); \path (original_remove) ++(180:3.5cm) coordinate (original_remove+);
\path (original_remove) ++(0:3.5cm) coordinate (original_remove++);
\path (before) ++(180:3.5cm) coordinate (before-); \path (0,-6.5) ++(0:3.5cm) coordinate (before--);

\draw [->, line width=\edgewidth] (before--)--(original_remove+) node[midway,xshift=0cm,yshift=0.5cm] {\LARGE(zag-2)} ;
\draw [->, line width=\edgewidth] (original_remove++)--(before-) node[midway,xshift=0cm,yshift=0.5cm] {\LARGE(zag-3)} ;

\end{tikzpicture}
}
\caption{The zag-deformation at $z_i$ with $p_i=2$.}\label{fig_zagdeform_p2}
\end{figure}

For some special cases, these deformations preserve the consistency condition as in Proposition~\ref{prop_preserve_consistent1} below. Note that in Section~\ref{sec_def_exdeform} we will introduce extended deformations of dimer models which always preserve the consistency condition (see Proposition~\ref{prop_make_consistent}).

\begin{Proposition}\label{prop_preserve_consistent1}
Let the notation be the same as Definitions~{\rm \ref{def_deformation_data}}, {\rm \ref{def_deformation_zig}} and {\rm \ref{def_deformation_zag}}.
We assume that either one of the following conditions is satisfied:
\begin{itemize}\itemsep=0pt
\item[$(i)$] $r=1$ or
\item[$(ii)$] $\Gamma$ is a hexagonal or rectangular dimer model $($see Definition~{\rm \ref{def_hexagonal_square})}.
\end{itemize}
Then, the dimer models $\nu^\zig_\bfp(\Gamma, \{z_1,\dots,z_r\})$ and $\nu^\zag_\bfp(\Gamma, \{z_1,\dots,z_r\})$ are consistent.
\end{Proposition}

\begin{proof}This follows from Proposition~\ref{prop_make_consistent} (see also Remark~\ref{r=1case} and Proposition~\ref{skip_hexagonal_square}).
\end{proof}

We note that for some cases the zig-deformation and zag-deformation are mutually inverse operations
on the level of dimer models as in Proposition~\ref{prop_mut_inversecase} below.
To observe this, let $\nu^\zig_\bfp(\Gamma, \{z_1,\dots,z_r\})$ be the zig-deformed dimer model as in Definition~\ref{def_deformation_zig}.
We assume that $p_i=1$ for some $i$ and consider the zigzag path $z_{i,p_i}=z_{i,1}$ on~$\nu^\zig_\bfp(\Gamma)$ created in place of~$z_i$ on~$\Gamma$ via the zig-deformation. Note that~$z_{i,1}$ is type~I (see Proposition~\ref{deform_vector}).
Then, we take a new deformation data~$\bfq$ so that~$z_{i,1}$ is chosen in the data with the weight~$1$,
and consider the zag-deformation of $\nu^\zig_\bfp(\Gamma)$ at some set of zigzag paths with respected to~$\bfq$.
Through this operation, $z_{i,1}$~changes into the original zigzag path~$z_i$ as in Figure~\ref{fig_zigzag_recover}.

\begin{figure}[h!]\centering
\scalebox{0.55}{
\begin{tikzpicture}
\newcommand{\edgewidth}{0.05cm} 
\newcommand{\nodewidth}{0.05cm} 
\newcommand{\noderad}{0.16} 
\coordinate (W1) at (0,1.2); \coordinate (W2) at (0,3.6);
\path (W1) ++(-20:4cm) coordinate (B1); \path (W2) ++(-20:4cm) coordinate (B2);
\path (B1) ++(200:2cm) coordinate (B0); \path (W2) ++(20:2cm) coordinate (W3);

\path (W1) ++(-20:0.8cm) coordinate (B1-1); \path (W1) ++(-20:1.6cm) coordinate (W1-1);
\path (W1) ++(-20:2.4cm) coordinate (B1-2); \path (W1) ++(-20:3.2cm) coordinate (W1-2);
\path (W2) ++(-20:0.8cm) coordinate (B2-1); \path (W2) ++(-20:1.6cm) coordinate (W2-1);
\path (W2) ++(-20:2.4cm) coordinate (B2-2); \path (W2) ++(-20:3.2cm) coordinate (W2-2);

\path (B1) ++(30:1cm) coordinate (B1e); \path (B1) ++(330:1cm) coordinate (B1s);
\path (W1) ++(150:1cm) coordinate (W1w); \path (W1) ++(210:1cm) coordinate (W1s);
\path (B2) ++(30:1cm) coordinate (B2e); \path (B2) ++(330:1cm) coordinate (B2s);
\path (W2) ++(150:1cm) coordinate (W2w); \path (W2) ++(210:1cm) coordinate (W2s);

\path (W1) ++(160:0.7cm) coordinate (W1+); \path (W1) ++(200:0.7cm) coordinate (W1-);
\path (W2) ++(160:0.7cm) coordinate (W2+); \path (W2) ++(200:0.7cm) coordinate (W2-);
\path (B1) ++(20:0.7cm) coordinate (B1+); \path (B1) ++(340:0.7cm) coordinate (B1-);
\path (B2) ++(20:0.7cm) coordinate (B2+); \path (B2) ++(340:0.7cm) coordinate (B2-);

\node (original) at (0,0) {
\begin{tikzpicture}
\draw [line width=\edgewidth] (B1)--(W1); \draw [line width=\edgewidth] (B2)--(W1) ; \draw [line width=\edgewidth] (B2)--(W2) ;
\draw [line width=\edgewidth] (B1)--(B0) ; \draw [line width=\edgewidth] (W2)--(W3) ;
\draw [line width=\edgewidth] (B1)--(B1e); \draw [line width=\edgewidth] (B1)--(B1s);
\draw [line width=\edgewidth] (W1)--(W1w); \draw [line width=\edgewidth] (W1)--(W1s);
\draw [line width=\edgewidth] (B2)--(B2e); \draw [line width=\edgewidth] (B2)--(B2s);
\draw [line width=\edgewidth] (W2)--(W2w); \draw [line width=\edgewidth] (W2)--(W2s);

\draw [line width=\edgewidth,line cap=round, dash pattern=on 0pt off 2.5\pgflinewidth] (W1+)--(W1-);
\draw [line width=\edgewidth,line cap=round, dash pattern=on 0pt off 2.5\pgflinewidth] (W2+)--(W2-);
\draw [line width=\edgewidth,line cap=round, dash pattern=on 0pt off 2.5\pgflinewidth] (B1+)--(B1-);
\draw [line width=\edgewidth,line cap=round, dash pattern=on 0pt off 2.5\pgflinewidth] (B2+)--(B2-);
\draw [line width=\nodewidth, fill=black] (B1) circle [radius=\noderad] ; \draw [line width=\nodewidth, fill=black] (B2) circle [radius=\noderad] ;
\draw [line width=\nodewidth, fill=white] (W1) circle [radius=\noderad] ; \draw [line width=\nodewidth, fill=white] (W2) circle [radius=\noderad] ;

\path (W1) ++(90:0.2cm) coordinate (W1z); \path (W2) ++(90:0.2cm) coordinate (W2z);
\path (W1z) ++(-20:4cm) coordinate (B1z); \path (W2z) ++(-20:4cm) coordinate (B2z);
\path (B1z) ++(200:2cm) coordinate (B0z); \path (W2z) ++(20:2cm) coordinate (W3z);
\path (B1z) ++(200:2cm) coordinate (B0z); \path (W2z) ++(20:2.2cm) coordinate (W3z);
\draw [->, rounded corners, line width=0.08cm, red] (B0z)--(B1z)--(W1z)--(B2z)--(W2z)--(W3z);
\node at (1.8,2.3) {\color{red}\LARGE$z_i$};
\end{tikzpicture}};

\node at (9,0) {
\begin{tikzpicture}
\path (W1) ++(-20:1.2cm) coordinate (B1-3); \path (W1) ++(-20:2.8cm) coordinate (W1-3);
\path (W2) ++(-20:1.2cm) coordinate (B2-3); \path (W2) ++(-20:2.8cm) coordinate (W2-3);
\draw [line width=\edgewidth] (B1)--(W1); \draw [line width=\edgewidth] (B2)--(W2) ;
\draw [line width=\edgewidth] (B1)--(B1e); \draw [line width=\edgewidth] (B1)--(B1s);
\draw [line width=\edgewidth] (W1)--(W1w); \draw [line width=\edgewidth] (W1)--(W1s);
\draw [line width=\edgewidth] (B2)--(B2e); \draw [line width=\edgewidth] (B2)--(B2s);
\draw [line width=\edgewidth] (W2)--(W2w); \draw [line width=\edgewidth] (W2)--(W2s);

\draw [line width=\edgewidth,line cap=round, dash pattern=on 0pt off 2.5\pgflinewidth] (W1+)--(W1-);
\draw [line width=\edgewidth,line cap=round, dash pattern=on 0pt off 2.5\pgflinewidth] (W2+)--(W2-);
\draw [line width=\edgewidth,line cap=round, dash pattern=on 0pt off 2.5\pgflinewidth] (B1+)--(B1-);
\draw [line width=\edgewidth,line cap=round, dash pattern=on 0pt off 2.5\pgflinewidth] (B2+)--(B2-);

\draw [line width=\edgewidth] (B2-3)--(W1-3);
\path (B1-3) ++(-70:1.8cm) coordinate (B1-3-); \draw [line width=\edgewidth] (B1-3)--(B1-3-);
\path (W2-3) ++(110:1.8cm) coordinate (W2-3-); \draw [line width=\edgewidth] (W2-3)--(W2-3-);
\draw [line width=\nodewidth, fill=black] (B1) circle [radius=\noderad] ; \draw [line width=\nodewidth, fill=black] (B2) circle [radius=\noderad] ;
\draw [line width=\nodewidth, fill=black] (B1-3) circle [radius=\noderad] ; \draw [line width=\nodewidth, fill=black] (B2-3) circle [radius=\noderad] ;
\draw [line width=\nodewidth, fill=white] (W1) circle [radius=\noderad] ; \draw [line width=\nodewidth, fill=white] (W2) circle [radius=\noderad] ;
\draw [line width=\nodewidth, fill=white] (W1-3) circle [radius=\noderad] ; \draw [line width=\nodewidth, fill=white] (W2-3) circle [radius=\noderad] ;

\path (W1-3) ++(240:0.22cm) coordinate (W1-3z); \path (W2-3) ++(240:0.22cm) coordinate (W2-3z);
\path (W1-3z) ++(160:1.6cm) coordinate (B1-3z); \path (W2-3z) ++(160:1.6cm) coordinate (B2-3z);
\path (W2-3z) ++(110:2cm) coordinate (W2-3z+); \path (B1-3z) ++(-70:2cm) coordinate (B1-3z-);
\draw [->, rounded corners, line width=0.08cm, blue] (W2-3z+)--(W2-3z)--(B2-3z)--(W1-3z)--(B1-3z)--(B1-3z-);

\node at (2.6,1.5) {\color{blue}\LARGE$z_{i,1}$};
\end{tikzpicture}};

\node at (18,0) {
\begin{tikzpicture}
\draw [line width=\edgewidth] (B1)--(W1); \draw [line width=\edgewidth] (B2)--(W2) ;
\draw [line width=\edgewidth] (B1)--(B1e); \draw [line width=\edgewidth] (B1)--(B1s);
\draw [line width=\edgewidth] (W1)--(W1w); \draw [line width=\edgewidth] (W1)--(W1s);
\draw [line width=\edgewidth] (B2)--(B2e); \draw [line width=\edgewidth] (B2)--(B2s);
\draw [line width=\edgewidth] (W2)--(W2w); \draw [line width=\edgewidth] (W2)--(W2s);

\draw [line width=\edgewidth,line cap=round, dash pattern=on 0pt off 2.5\pgflinewidth] (W1+)--(W1-);
\draw [line width=\edgewidth,line cap=round, dash pattern=on 0pt off 2.5\pgflinewidth] (W2+)--(W2-);
\draw [line width=\edgewidth,line cap=round, dash pattern=on 0pt off 2.5\pgflinewidth] (B1+)--(B1-);
\draw [line width=\edgewidth,line cap=round, dash pattern=on 0pt off 2.5\pgflinewidth] (B2+)--(B2-);

\draw [line width=\edgewidth] (W1-1)--(B2-2);
\path (B1-2) ++(250:1.5cm) coordinate (B1-2-); \draw [line width=\edgewidth] (B1-2)--(B1-2-);
\path (W2-1) ++(70:1.5cm) coordinate (W2-1-); \draw [line width=\edgewidth] (W2-1)--(W2-1-);

\draw [line width=\nodewidth, fill=black] (B1) circle [radius=\noderad] ; \draw [line width=\nodewidth, fill=black] (B2) circle [radius=\noderad] ;
\draw [line width=\nodewidth, fill=black] (B1-1) circle [radius=\noderad] ; \draw [line width=\nodewidth, fill=black] (B1-2) circle [radius=\noderad] ;
\draw [line width=\nodewidth, fill=black] (B2-1) circle [radius=\noderad] ; \draw [line width=\nodewidth, fill=black] (B2-2) circle [radius=\noderad] ;
\draw [line width=\nodewidth, fill=white] (W1) circle [radius=\noderad] ; \draw [line width=\nodewidth, fill=white] (W2) circle [radius=\noderad] ;
\draw [line width=\nodewidth, fill=white] (W1-1) circle [radius=\noderad] ; \draw [line width=\nodewidth, fill=white] (W1-2) circle [radius=\noderad] ;
\draw [line width=\nodewidth, fill=white] (W2-1) circle [radius=\noderad] ; \draw [line width=\nodewidth, fill=white] (W2-2) circle [radius=\noderad] ;

\path (W1-1) ++(210:0.26cm) coordinate (W1-1z); \path (W2-1) ++(210:0.26cm) coordinate (W2-1z);
\path (W1-1z) ++(-20:0.8cm) coordinate (B1-2z); \path (W2-1z) ++(-20:0.8cm) coordinate (B2-2z);
\path (B1-2z) ++(250:1.3cm) coordinate (B1-2-z); \path (W2-1z) ++(70:2cm) coordinate (W2-1+z);
\draw [->, rounded corners, line width=0.08cm, red] (B1-2-z)--(B1-2z)--(W1-1z)--(B2-2z)--(W2-1z)--(W2-1+z);
\end{tikzpicture}};

\draw [->, line width=\edgewidth] (3.5,0)--(5.5,0) node[midway,xshift=0cm,yshift=0.6cm] {\LARGE $\nu^\zig_\bfp$} ;
\draw [->, line width=\edgewidth] (12.5,0)--(14.5,0) node[midway,xshift=0cm,yshift=0.6cm] {\LARGE $\nu^\zag_\bfq$} ;
\end{tikzpicture}
}
\caption{Transitions of $z_i$ via the composition of the zig-deformation and zag-deformation. (In the rightmost figure, we still do not apply (join). The zigzag path in this figure coincides with $z_i$ after applying (join).)}\label{fig_zigzag_recover}
\end{figure}

The next proposition follows from these observations.

\begin{Proposition}\label{prop_mut_inversecase}
We consider the situation as in Definition~{\rm \ref{def_deformation_data}}, and assume that $p_1=\cdots=p_r=1$,
in which case $r=h$.
Let $z_{i,p_i}=z_{i,1}$ $(i=1,\dots,r)$ be the zigzag path of $\nu^\zig_\bfp(\Gamma, \{z_1,\dots,z_r\})$ $($resp.\ $\nu^\zag_\bfp(\Gamma, \{z_1,\dots,z_r\})$$)$
created by modifying~$z_i$ as in Definition~{\rm \ref{def_deformation_zig}} $($resp.\ Definition~{\rm \ref{def_deformation_zag})}.
We take $\bfq=(1,\dots,1)$ as a deformation weight of $\{z_{1,1}\,,\dots,\,z_{r,1}\}$. Then, we respectively have
\begin{gather*}
\nu^\zag_\bfq\big(\nu^\zig_\bfp(\Gamma, \{z_1,\dots,z_r\}),\{z_{1,1}\,,\dots,\,z_{r,1}\}\big)=\Gamma,
\\
\nu^\zig_\bfq\big(\nu^\zag_\bfp(\Gamma, \{z_1,\dots,z_r\}),\{z_{1,1}\,,\dots,\,z_{r,1}\}\big)=\Gamma.
\end{gather*}
\end{Proposition}

This property depends on a choice of a deformation data,
thus in general these deformations are not mutually inverse as in Example~\ref{zigzag_counterEX} given in the next subsection.
Whereas, on the level of the PM polygons, they are mutually inverse as we will see in Corollary~\ref{cor_mut_inverse_on_polygon}.

\subsection{Examples of deformations of dimer models}\label{subsec_ex_deformation}

In this subsection we give several examples for the case of $r=1$,
in which the deformed dimer model is consistent (see Proposition~\ref{prop_preserve_consistent1}).

\begin{Example}\label{ex_deformation_4b}
Let $\Gamma$ be the dimer model given in Figure~\ref{ex_dimer_4b}.
We consider zigzag paths on $\Gamma$ given in Figure~\ref{zigzag_4b}.
We first collect the deformation data (see Definition~\ref{def_deformation_data}).
Let us choose the zigzag path $z_3$, and denote it by $z$. We see that $\ell(z)=6$, $v\coloneqq[z]=(-1,-1)$, and $\big|\calZ_v^\rmI(\Gamma)\big|=1$.
Since $\big|\calZ_v^\rmI(\Gamma)\big|=1$, we can take only $r=1$, in which case $h=\ell(z)/2-r=2$.
Thus, we consider the deformation weight $p=h=2$.
The zig-deformation $\nu^\zig_p(\Gamma, z)$ of $\Gamma$ at $z$ with $p=2$ is shown in Figure~\ref{ex_def_zig}.

\begin{figure}[h!]\centering

\begin{tikzpicture}
\node at (0,0)
{\scalebox{0.45}{
\begin{tikzpicture}
\newcommand{\edgewidth}{0.07cm}
\newcommand{\nodewidth}{0.07cm}
\newcommand{\noderad}{0.24} 
\newcommand{\arrowwidth}{0.08cm}
\basicdimerB

\draw[->, line width=0.15cm, rounded corners, color=red] (3,6)--(W3)--(B2)--(W2)--(0,3);
\draw[->, line width=0.15cm, rounded corners, color=red] (6,3)--(B3)--(W1)--(B1)--(3,0);
\end{tikzpicture}
} };

\draw[->,line width=0.03cm] (2,0)--(3.5,0);
\node at (2.4,0.85) {(zig-1)}; \node at (3,0.35) {-- (zig-3)};

\node at (5.5,0)
{\scalebox{0.45}{
\begin{tikzpicture}
\newcommand{\edgewidth}{0.065cm}
\newcommand{\nodewidth}{0.05cm}
\newcommand{\noderad}{0.18} 
\newcommand{\arrowwidth}{0.08cm}
\coordinate (W1) at (5,1); \coordinate (W2) at (1,2); \coordinate (W3) at (3,4);
\coordinate (B1) at (3,2); \coordinate (B2) at (1,5); \coordinate (B3) at (5,4);
\draw[line width=\edgewidth] (0,0) rectangle (6,6);
\draw[line width=\edgewidth] (B2)--(W2)--(B1)--(W3)--(B3)--(W1); \draw[line width=\edgewidth] (B1)--(W3);
\draw[line width=\edgewidth] (3,0)--(B1); \draw[line width=\edgewidth] (6,0)--(W1);
\draw[line width=\edgewidth] (3,6)--(W3); \draw[line width=\edgewidth] (0,6)--(B2);
\draw [line width=\nodewidth, fill=black] (B1) circle [radius=\noderad] ;
\draw [line width=\nodewidth, fill=black] (B2) circle [radius=\noderad] ;
\draw [line width=\nodewidth, fill=black] (B3) circle [radius=\noderad] ;
\draw [line width=\nodewidth, fill=white] (W1) circle [radius=\noderad] ;
\draw [line width=\nodewidth, fill=white] (W2) circle [radius=\noderad] ;
\draw [line width=\nodewidth, fill=white] (W3) circle [radius=\noderad] ;

\coordinate (Wa1) at (3,5.6); \coordinate (Wa2) at (3,1.2); \coordinate (Ba1) at (3,4.8); \coordinate (Ba2) at (3,0.4);
\coordinate (Wb1) at (1,3.2); \coordinate (Wb2) at (1,4.4); \coordinate (Bb1) at (1,2.6); \coordinate (Bb2) at (1,3.8);
\coordinate (Wc1) at (5,2.2); \coordinate (Wc2) at (5,3.4); \coordinate (Bc1) at (5,1.6); \coordinate (Bc2) at (5,2.8);

\draw[line width=\edgewidth] (Wa1)--(Bb1);
\draw[line width=\edgewidth] (Wc2)--(Ba2);

\draw[line width=\edgewidth] (Bb2)--(2.3,6);
\draw[line width=\edgewidth] (2.3,0)--(Wa2);

\draw[line width=\edgewidth] (Wc1)--(3.7,0); \draw[line width=\edgewidth] (Ba1)--(3.7,6);

\draw[line width=\edgewidth] (Wb1)--(0,2.4); \draw[line width=\edgewidth] (Bc1)--(6,2.4);
\draw[line width=\edgewidth] (Wb2)--(0,3.6); \draw[line width=\edgewidth] (Bc2)--(6,3.6);

\draw [line width=\nodewidth, fill=white] (Wa1) circle [radius=\noderad] ; \draw [line width=\nodewidth, fill=white] (Wa2) circle [radius=\noderad] ;
\draw [line width=\nodewidth, fill=black] (Ba1) circle [radius=\noderad] ; \draw [line width=\nodewidth, fill=black] (Ba2) circle [radius=\noderad] ;
\draw [line width=\nodewidth, fill=white] (Wb1) circle [radius=\noderad] ; \draw [line width=\nodewidth, fill=white] (Wb2) circle [radius=\noderad] ;
\draw [line width=\nodewidth, fill=black] (Bb1) circle [radius=\noderad] ; \draw [line width=\nodewidth, fill=black] (Bb2) circle [radius=\noderad] ;
\draw [line width=\nodewidth, fill=white] (Wc1) circle [radius=\noderad] ; \draw [line width=\nodewidth, fill=white] (Wc2) circle [radius=\noderad] ;
\draw [line width=\nodewidth, fill=black] (Bc1) circle [radius=\noderad] ; \draw [line width=\nodewidth, fill=black] (Bc2) circle [radius=\noderad] ;
\end{tikzpicture}
} };

\draw[->,line width=0.03cm] (7.5,0)--(9,0);
\node at (8.25,0.35) {(join)};

\node at (11,0)
{\scalebox{0.45}{
\begin{tikzpicture}
\newcommand{\edgewidth}{0.07cm}
\newcommand{\nodewidth}{0.06cm}
\newcommand{\noderad}{0.2} 
\newcommand{\arrowwidth}{0.08cm}
\coordinate (W1) at (2,1); \coordinate (W2) at (5,2); \coordinate (W3) at (1,3);
\coordinate (W4) at (4,3); \coordinate (W5) at (0.5,5); \coordinate (W6) at (3,5);
\coordinate (B1) at (3,1); \coordinate (B2) at (5.5,1); \coordinate (B3) at (2,2.5);
\coordinate (B4) at (5,3); \coordinate (B5) at (1.25,4.5); \coordinate (B6) at (4,5);
\draw[line width=\edgewidth] (0,0) rectangle (6,6);
\draw[line width=\edgewidth] (W1)--(B1); \draw[line width=\edgewidth] (W1)--(B3);
\draw[line width=\edgewidth] (W2)--(B2); \draw[line width=\edgewidth] (W2)--(B4);
\draw[line width=\edgewidth] (W3)--(B3); \draw[line width=\edgewidth] (W3)--(B5);
\draw[line width=\edgewidth] (W4)--(B1); \draw[line width=\edgewidth] (W4)--(B3);
\draw[line width=\edgewidth] (W4)--(B4); \draw[line width=\edgewidth] (W4)--(B6);
\draw[line width=\edgewidth] (W5)--(B5); \draw[line width=\edgewidth] (W6)--(B3);
\draw[line width=\edgewidth] (W6)--(B6);
\draw[line width=\edgewidth] (B1)--(3,0); \draw[line width=\edgewidth] (W6)--(3,6);
\draw[line width=\edgewidth] (B2)--(6,0); \draw[line width=\edgewidth] (W5)--(0,6);
\draw[line width=\edgewidth] (B2)--(6,1.666); \draw[line width=\edgewidth] (W3)--(0,1.666);
\draw[line width=\edgewidth] (B4)--(6,4.333); \draw[line width=\edgewidth] (W5)--(0,4.333);
\draw[line width=\edgewidth] (B5)--(1.7,6); \draw[line width=\edgewidth] (W1)--(1.7,0);
\draw[line width=\edgewidth] (B6)--(4.333,6); \draw[line width=\edgewidth] (W2)--(4.333,0);

\draw [line width=\nodewidth, fill=black] (B1) circle [radius=\noderad] ; \draw [line width=\nodewidth, fill=black] (B2) circle [radius=\noderad] ;
\draw [line width=\nodewidth, fill=black] (B3) circle [radius=\noderad] ;
\draw [line width=\nodewidth, fill=black] (B4) circle [radius=\noderad] ; \draw [line width=\nodewidth, fill=black] (B5) circle [radius=\noderad] ;
\draw [line width=\nodewidth, fill=black] (B6) circle [radius=\noderad] ;
\draw [line width=\nodewidth, fill=white] (W1) circle [radius=\noderad] ; \draw [line width=\nodewidth, fill=white] (W2) circle [radius=\noderad] ;
\draw [line width=\nodewidth, fill=white] (W3) circle [radius=\noderad] ;
\draw [line width=\nodewidth, fill=white] (W4) circle [radius=\noderad] ; \draw [line width=\nodewidth, fill=white] (W5) circle [radius=\noderad] ;
\draw [line width=\nodewidth, fill=white] (W6) circle [radius=\noderad] ;
\end{tikzpicture}
} };

\end{tikzpicture}
\caption{The zig-deformation $\nu^\zig_p(\Gamma, z)$ of $\Gamma$ at $z$.}\label{ex_def_zig}
\end{figure}

We give the zag-deformation $\nu^\zag_p(\Gamma, z)$ of $\Gamma$ at $z$ with $p=2$ as in Figure~\ref{ex_def_zag}.

\begin{figure}[h!]\centering

\begin{tikzpicture}
\node at (0,0)
{\scalebox{0.45}{
\begin{tikzpicture}
\newcommand{\edgewidth}{0.07cm}
\newcommand{\nodewidth}{0.07cm}
\newcommand{\noderad}{0.24} 
\newcommand{\arrowwidth}{0.08cm}
\basicdimerB

\draw[->, line width=0.15cm, rounded corners, color=red] (3,6)--(W3)--(B2)--(W2)--(0,3);
\draw[->, line width=0.15cm, rounded corners, color=red] (6,3)--(B3)--(W1)--(B1)--(3,0);
\end{tikzpicture}
} };

\draw[->,line width=0.03cm] (2,0)--(3.5,0);
\node at (2.4,0.85) {(zag-1)}; \node at (3,0.35) {-- (zag-3)};

\node at (5.5,0)
{\scalebox{0.45}{
\begin{tikzpicture}
\newcommand{\edgewidth}{0.065cm}
\newcommand{\nodewidth}{0.05cm}
\newcommand{\noderad}{0.18} 
\newcommand{\arrowwidth}{0.08cm}
\coordinate (W1) at (5,1); \coordinate (W2) at (4,5); \coordinate (W3) at (2,3);
\coordinate (B1) at (2,1); \coordinate (B2) at (1,5); \coordinate (B3) at (4,3);
\draw[line width=\edgewidth] (0,0) rectangle (6,6);

\draw[line width=\edgewidth] (B1)--(W2); \draw[line width=\edgewidth] (B2)--(W2); \draw[line width=\edgewidth] (B3)--(W2);
\draw[line width=\edgewidth] (B1)--(W3); \draw[line width=\edgewidth] (B1)--(W1);
\draw[line width=\edgewidth] (0,6)--(B2); \draw[line width=\edgewidth] (6,3)--(B3);
\draw[line width=\edgewidth] (0,3)--(W3); \draw[line width=\edgewidth] (6,0)--(W1);
\draw [line width=\nodewidth, fill=black] (B1) circle [radius=\noderad] ;
\draw [line width=\nodewidth, fill=black] (B2) circle [radius=\noderad] ;
\draw [line width=\nodewidth, fill=black] (B3) circle [radius=\noderad] ;
\draw [line width=\nodewidth, fill=white] (W1) circle [radius=\noderad] ;
\draw [line width=\nodewidth, fill=white] (W2) circle [radius=\noderad] ;
\draw [line width=\nodewidth, fill=white] (W3) circle [radius=\noderad] ;

\coordinate (W1a) at (2.6,1); \coordinate (W1b) at (3.8,1); \coordinate (B1a) at (3.2,1); \coordinate (B1b) at (4.4,1);
\coordinate (W2a) at (1.6,5); \coordinate (W2b) at (2.8,5); \coordinate (B2a) at (2.2,5); \coordinate (B2b) at (3.4,5);
\coordinate (W3a) at (0.4,3); \coordinate (W3b) at (4.8,3); \coordinate (B3a) at (1.2,3); \coordinate (B3b) at (5.6,3);
\draw[line width=\edgewidth] (B2b)--(W3a); \draw[line width=\edgewidth] (W1a)--(B3b);
\draw[line width=\edgewidth] (W2a)--(2.4,6); \draw[line width=\edgewidth] (B1a)--(2.4,0);
\draw[line width=\edgewidth] (W2b)--(3.6,6); \draw[line width=\edgewidth] (B1b)--(3.6,0);
\draw[line width=\edgewidth] (B3a)--(0,2.295); \draw[line width=\edgewidth] (W1b)--(6,2.295);
\draw[line width=\edgewidth] (B2a)--(0,3.706); \draw[line width=\edgewidth] (W3b)--(6,3.706);

\draw [line width=\nodewidth, fill=white] (W1a) circle [radius=\noderad] ; \draw [line width=\nodewidth, fill=white] (W1b) circle [radius=\noderad] ;
\draw [line width=\nodewidth, fill=black] (B1a) circle [radius=\noderad] ; \draw [line width=\nodewidth, fill=black] (B1b) circle [radius=\noderad] ;
\draw [line width=\nodewidth, fill=white] (W2a) circle [radius=\noderad] ; \draw [line width=\nodewidth, fill=white] (W2b) circle [radius=\noderad] ;
\draw [line width=\nodewidth, fill=black] (B2a) circle [radius=\noderad] ; \draw [line width=\nodewidth, fill=black] (B2b) circle [radius=\noderad] ;
\draw [line width=\nodewidth, fill=white] (W3a) circle [radius=\noderad] ; \draw [line width=\nodewidth, fill=white] (W3b) circle [radius=\noderad] ;
\draw [line width=\nodewidth, fill=black] (B3a) circle [radius=\noderad] ; \draw [line width=\nodewidth, fill=black] (B3b) circle [radius=\noderad] ;
\end{tikzpicture}
} };

\draw[->,line width=0.03cm] (7.5,0)--(9,0);
\node at (8.25,0.35) {(join)};

\node at (11,0)
{\scalebox{0.45}{
\begin{tikzpicture}
\newcommand{\edgewidth}{0.07cm}
\newcommand{\nodewidth}{0.06cm}
\newcommand{\noderad}{0.2} 
\newcommand{\arrowwidth}{0.08cm}
\coordinate (W1) at (5,1); \coordinate (W2) at (3.5,2); \coordinate (W3) at (1,3);
\coordinate (W4) at (4,4); \coordinate (W5) at (2.5,5); \coordinate (W6) at (0.5,5.5);
\coordinate (B1) at (5.5,0.5); \coordinate (B2) at (3.5,1); \coordinate (B3) at (2,2);
\coordinate (B4) at (5,3); \coordinate (B5) at (2.5,4); \coordinate (B6) at (1,5);
\draw[line width=\edgewidth] (0,0) rectangle (6,6);

\draw[line width=\edgewidth] (W1)--(B1); \draw[line width=\edgewidth] (W1)--(B2);
\draw[line width=\edgewidth] (W2)--(B2); \draw[line width=\edgewidth] (W2)--(B3);
\draw[line width=\edgewidth] (W2)--(B4);
\draw[line width=\edgewidth] (W3)--(B3); \draw[line width=\edgewidth] (W3)--(B5);
\draw[line width=\edgewidth] (W4)--(B3); \draw[line width=\edgewidth] (W4)--(B4); \draw[line width=\edgewidth] (W4)--(B5);
\draw[line width=\edgewidth] (W5)--(B5); \draw[line width=\edgewidth] (W5)--(B6); \draw[line width=\edgewidth] (W6)--(B6);
\draw[line width=\edgewidth] (W1)--(6,1.333); \draw[line width=\edgewidth] (B3)--(0,1.333);
\draw[line width=\edgewidth] (W3)--(0,3); \draw[line width=\edgewidth] (B4)--(6,3);
\draw[line width=\edgewidth] (W4)--(6,4.666); \draw[line width=\edgewidth] (B6)--(0,4.666);
\draw[line width=\edgewidth] (W5)--(4.5,6); \draw[line width=\edgewidth] (B1)--(4.5,0);
\draw[line width=\edgewidth] (W6)--(1.5,6); \draw[line width=\edgewidth] (B2)--(1.5,0);
\draw[line width=\edgewidth] (W6)--(0,6); \draw[line width=\edgewidth] (B1)--(6,0);
\draw [line width=\nodewidth, fill=black] (B1) circle [radius=\noderad] ; \draw [line width=\nodewidth, fill=black] (B2) circle [radius=\noderad] ;
\draw [line width=\nodewidth, fill=black] (B3) circle [radius=\noderad] ;
\draw [line width=\nodewidth, fill=black] (B4) circle [radius=\noderad] ; \draw [line width=\nodewidth, fill=black] (B5) circle [radius=\noderad] ;
\draw [line width=\nodewidth, fill=black] (B6) circle [radius=\noderad] ;
\draw [line width=\nodewidth, fill=white] (W1) circle [radius=\noderad] ; \draw [line width=\nodewidth, fill=white] (W2) circle [radius=\noderad] ;
\draw [line width=\nodewidth, fill=white] (W3) circle [radius=\noderad] ;
\draw [line width=\nodewidth, fill=white] (W4) circle [radius=\noderad] ; \draw [line width=\nodewidth, fill=white] (W5) circle [radius=\noderad] ;
\draw [line width=\nodewidth, fill=white] (W6) circle [radius=\noderad] ;
\end{tikzpicture}
} };

\end{tikzpicture}
\caption{The zag-deformation $\nu^\zag_p(\Gamma, z)$ of $\Gamma$ at $z$.}\label{ex_def_zag}
\end{figure}
\end{Example}

\begin{Example}\label{ex_not_isoradial}
We remark that even if $\Gamma$ is an isoradial dimer model, the deformed ones are not necessarily isoradial.
For example, the leftmost dimer model $\Gamma$ in Figure~\ref{ex_isoradial_defoem2} is isoradial.
We choose a type~I zigzag path $z$ whose slope is $v=(-1,0)$, in which case $\big|\calZ^\rmI_v(\Gamma)\big|=2$ and $\ell(z)=4$.
We fix $r=1$, and hence $h=\ell(z)/2-r=1$.
Applying the zig-deformation at~$z$ with $p=1$ to~$\Gamma$, we have the rightmost one in Figure~\ref{ex_isoradial_defoem2}.
One can check that the deformed dimer model is consistent but not isoradial.

\begin{figure}[h!]\centering
\begin{tikzpicture}
\node at (0,0)
{\scalebox{0.675}{
\begin{tikzpicture}
\newcommand{\edgewidth}{0.043cm} 
\newcommand{\nodewidth}{0.045cm} 
\newcommand{\noderad}{0.14} 
\coordinate (B1) at (0.5,1.5); \coordinate (B2) at (2.5,1.5); \coordinate (B3) at (1.5,3.5); \coordinate (B4) at (3.5,3.5);
\coordinate (W1) at (1.5,0.5); \coordinate (W2) at (3.5,0.5); \coordinate (W3) at (0.5,2.5); \coordinate (W4) at (2.5,2.5);
\draw[line width=\edgewidth] (0,0) rectangle (4,4);
\draw[line width=\edgewidth] (B1)--(W1); \draw[line width=\edgewidth] (B1)--(W3);
\draw[line width=\edgewidth] (B2)--(W1); \draw[line width=\edgewidth] (B2)--(W2);
\draw[line width=\edgewidth] (B2)--(W3); \draw[line width=\edgewidth] (B2)--(W4);
\draw[line width=\edgewidth] (B3)--(W3); \draw[line width=\edgewidth] (B3)--(W4); \draw[line width=\edgewidth] (B4)--(W4);
\draw[line width=\edgewidth] (W1)--(1.5,0); \draw[line width=\edgewidth] (B3)--(1.5,4);
\draw[line width=\edgewidth] (W2)--(3.5,0); \draw[line width=\edgewidth] (B4)--(3.5,4);
\draw[line width=\edgewidth] (W2)--(4,1); \draw[line width=\edgewidth] (B1)--(0,1);
\draw[line width=\edgewidth] (W2)--(2.5,0); \draw[line width=\edgewidth] (B3)--(2.5,4);
\draw[line width=\edgewidth] (W3)--(0,3); \draw[line width=\edgewidth] (B4)--(4,3);
\draw [line width=\nodewidth, fill=black] (B1) circle [radius=\noderad] ; \draw [line width=\nodewidth, fill=black] (B2) circle [radius=\noderad] ;
\draw [line width=\nodewidth, fill=black] (B3) circle [radius=\noderad] ; \draw [line width=\nodewidth, fill=black] (B4) circle [radius=\noderad] ;
\draw [line width=\nodewidth, fill=white] (W1) circle [radius=\noderad] ; \draw [line width=\nodewidth, fill=white] (W2) circle [radius=\noderad] ;
\draw [line width=\nodewidth, fill=white] (W3) circle [radius=\noderad] ; \draw [line width=\nodewidth, fill=white] (W4) circle [radius=\noderad] ;
\path (4,3) ++(90:0.2cm) coordinate (start); \path (0,3) ++(90:0.2cm) coordinate (target);
\path (W3) ++(90:0.2cm) coordinate (W3+); \path (W4) ++(90:0.2cm) coordinate (W4+);
\path (B3) ++(90:0.2cm) coordinate (B3+); \path (B4) ++(90:0.2cm) coordinate (B4+);
\draw [->, rounded corners, line width=0.1cm, red] (start)--(B4+)--(W4+)--(B3+)--(W3+)--(target) ;
\node [red] at (2.3,3.3) {\Large$z$};
\end{tikzpicture}
} };

\draw[->,line width=0.03cm] (2,0)--(3.5,0);
\node at (2.4,0.85) {(zig-1)}; \node at (3,0.35) {-- (zig-3)};

\node at (5.5,0)
{\scalebox{0.675}{
\begin{tikzpicture}
\newcommand{\edgewidth}{0.05cm} 
\newcommand{\nodewidth}{0.035cm} 
\newcommand{\noderad}{0.12} 
\coordinate (B5) at (0.83,2.83); \coordinate (B6) at (2.83,2.83);
\coordinate (W5) at (1.17,3.17); \coordinate (W6) at (3.17,3.17);
\draw[line width=\edgewidth] (0,0) rectangle (4,4);
\draw[line width=\edgewidth] (B1)--(W1); \draw[line width=\edgewidth] (B1)--(W3);
\draw[line width=\edgewidth] (B2)--(W1); \draw[line width=\edgewidth] (B2)--(W2);
\draw[line width=\edgewidth] (B2)--(W3); \draw[line width=\edgewidth] (B2)--(W4);
\draw[line width=\edgewidth] (B3)--(W3); \draw[line width=\edgewidth] (B4)--(W4);
\draw[line width=\edgewidth] (W1)--(1.5,0); \draw[line width=\edgewidth] (B3)--(1.5,4);
\draw[line width=\edgewidth] (W2)--(3.5,0); \draw[line width=\edgewidth] (B4)--(3.5,4);
\draw[line width=\edgewidth] (W2)--(4,1); \draw[line width=\edgewidth] (B1)--(0,1);
\draw[line width=\edgewidth] (W2)--(2.5,0); \draw[line width=\edgewidth] (B3)--(2.5,4);
\draw[line width=\edgewidth] (B5)--(W6); \draw[line width=\edgewidth] (W5)--(0,3); \draw[line width=\edgewidth] (B6)--(4,3);

\draw [line width=\nodewidth, fill=black] (B1) circle [radius=\noderad] ; \draw [line width=\nodewidth, fill=black] (B2) circle [radius=\noderad] ;
\draw [line width=\nodewidth, fill=black] (B3) circle [radius=\noderad] ; \draw [line width=\nodewidth, fill=black] (B4) circle [radius=\noderad] ;
\draw [line width=\nodewidth, fill=black] (B5) circle [radius=\noderad] ; \draw [line width=\nodewidth, fill=black] (B6) circle [radius=\noderad] ;
\draw [line width=\nodewidth, fill=white] (W1) circle [radius=\noderad] ; \draw [line width=\nodewidth, fill=white] (W2) circle [radius=\noderad] ;
\draw [line width=\nodewidth, fill=white] (W3) circle [radius=\noderad] ; \draw [line width=\nodewidth, fill=white] (W4) circle [radius=\noderad] ;
\draw [line width=\nodewidth, fill=white] (W5) circle [radius=\noderad] ; \draw [line width=\nodewidth, fill=white] (W6) circle [radius=\noderad] ;
\end{tikzpicture}
} };

\draw[->,line width=0.03cm] (7.5,0)--(9,0);
\node at (8.25,0.35) {(join)};

\node at (11,0)
{\scalebox{0.675}{
\begin{tikzpicture}
\newcommand{\edgewidth}{0.05cm} 
\newcommand{\nodewidth}{0.045cm} 
\newcommand{\noderad}{0.14} 
\coordinate (B1) at (1,1); \coordinate (B2) at (3,1); \coordinate (B3) at (1.5,2.5); \coordinate (B4) at (1.5,3.5);
\coordinate (W1) at (1.5,0.5); \coordinate (W2) at (1.5,1.5); \coordinate (W3) at (1,3); \coordinate (W4) at (3,3);
\draw[line width=\edgewidth] (0,0) rectangle (4,4);
\draw[line width=\edgewidth] (B1)--(W1); \draw[line width=\edgewidth] (B1)--(W2);
\draw[line width=\edgewidth] (B2)--(W1); \draw[line width=\edgewidth] (B2)--(W2); \draw[line width=\edgewidth] (B2)--(W4);
\draw[line width=\edgewidth] (B3)--(W2); \draw[line width=\edgewidth] (B3)--(W3); \draw[line width=\edgewidth] (B3)--(W4);
\draw[line width=\edgewidth] (B4)--(W3); \draw[line width=\edgewidth] (B4)--(W4);
\draw[line width=\edgewidth] (W1)--(1.5,0); \draw[line width=\edgewidth] (B4)--(1.5,4);
\draw[line width=\edgewidth] (W4)--(3,4); \draw[line width=\edgewidth] (B2)--(3,0);
\draw[line width=\edgewidth] (W4)--(4,4); \draw[line width=\edgewidth] (B1)--(0,0);
\draw[line width=\edgewidth] (W3)--(0,2); \draw[line width=\edgewidth] (B2)--(4,2);
\draw [line width=\nodewidth, fill=black] (B1) circle [radius=\noderad] ; \draw [line width=\nodewidth, fill=black] (B2) circle [radius=\noderad] ;
\draw [line width=\nodewidth, fill=black] (B3) circle [radius=\noderad] ; \draw [line width=\nodewidth, fill=black] (B4) circle [radius=\noderad] ;
\draw [line width=\nodewidth, fill=white] (W1) circle [radius=\noderad] ; \draw [line width=\nodewidth, fill=white] (W2) circle [radius=\noderad] ;
\draw [line width=\nodewidth, fill=white] (W3) circle [radius=\noderad] ; \draw [line width=\nodewidth, fill=white] (W4) circle [radius=\noderad] ;
\end{tikzpicture}
} };
\end{tikzpicture}
\caption{An example of the deformed dimer model that is not isoradial.}\label{ex_isoradial_defoem2}
\end{figure}
\end{Example}

\begin{Example}\label{zigzag_counterEX}
We here give an example showing that the zig-deformation and zag-deformation are not mutually inverse operations on the level of dimer models in general.

We consider the consistent dimer model $\Gamma$ that takes the form as in the top-left of Figure~\ref{ex_zigzag_non_involution}, and
the type I zigzag path $z$ on $\Gamma$. We see that $[z]=(-1,0)$ and $\ell(z)=4$. Let $r=1$, and hence $h=\ell(z)/2-r=1$.
Then, the zig-deformation $\nu_p^\zig(\Gamma,z)$ of~$\Gamma$ at $z$ with the deformation weight $p=h=1$ is the top-right of
Figure~\ref{ex_zigzag_non_involution}.
Through this process, the new zigzag path on $\nu_p^\zig(\Gamma,z)$ whose slope is~$(1,0)$ has been created.
If we apply the zag-deformation with the deformation weight $q=1$ to this new zigzag path, then we recover the original dimer model $\Gamma$
as we saw in Proposition~\ref{prop_mut_inversecase}.
On the other hand, there is another type I zigzag path on $\nu_p^\zig(\Gamma,z)$ whose slope is $(1,0)$, which is denoted by $z^\prime$.
We see that $\ell(z^\prime)=4$, and consider $r=1$, and hence $h=\ell(z^\prime)/2-r=1$.
Applying the zag deformation at $z^\prime$ with the deformation weight $q=h=1$ to $\nu_p^\zig(\Gamma,z)$,
we have the dimer model $\nu_q^\zag\big(\nu_p^\zig(\Gamma,z),z^\prime\big)$ shown in the bottom-right of Figure~\ref{ex_zigzag_non_involution},
which is not isomorphic to~$\Gamma$.
Nevertheless, $\Gamma$ and $\nu_q^\zag\big(\nu_p^\zig(\Gamma,z),z^\prime\big)$ are ``mutation-equivalent'' (see, e.g., \cite[Section~5.11]{Nak})
and hence their PM polygons are the same, see Appendix~\ref{app_mutationdimer} for more details of the mutation.

\begin{figure}[h!]\centering
\begin{tikzpicture}
\newcommand{\edgewidth}{0.07cm}
\newcommand{\nodewidth}{0.06cm}
\newcommand{\noderad}{0.26} 
\newcommand{\arrowwidth}{0.08cm}

\node at (-5,0) {};
\node at (-1.8,0) {$\Gamma$ : };
\node at (9.1,-4) {: $\nu_q^\zag\big(\nu_p^\zig(\Gamma,z),z^\prime\big)$};

\node at (0,0)
{\scalebox{0.35}{
\begin{tikzpicture}
\draw[line width=\edgewidth] (0,0) rectangle (8,8);

\foreach \n/\a/\b in {1/3/1,2/7/1,3/3/5,4/7/5} {
\coordinate (B\n) at (\a,\b);
};
\foreach \m\a/\b in {1/1/3,2/5/3,3/1/7,4/5/7} {
\coordinate (W\m) at (\a,\b);
};
\foreach \o\a/\b in {1/2/0,2/4/0,3/6/0,4/8/0,5/8/2,6/8/4,7/8/6,8/6/8,9/4/8,10/2/8,11/0/8,12/0/6,13/0/4,14/0/2} {
\coordinate (E\o) at (\a,\b);
};

\foreach \n/\m in {1/1,1/2,2/2,3/1,3/2,3/3,3/4,4/2,4/4} {
\draw[line width=\edgewidth] (B\n)--(W\m); };
\foreach \n/\o in {1/1,1/2,2/3,2/4,2/5,4/6,4/7} {
\draw[line width=\edgewidth] (B\n)--(E\o); };
\foreach \m/\o in {4/8,4/9,3/10,3/11,3/12,1/13,1/14} {
\draw[line width=\edgewidth] (W\m)--(E\o); };

\foreach \n in {1,2,3,4} {
\draw [line width=\nodewidth, fill=black] (B\n) circle [radius=\noderad] ;
};
\foreach \m in {1,2,3,4} {
\draw [line width=\nodewidth, fill=white] (W\m) circle [radius=\noderad] ;
};

\foreach \n in {3,4} {
\path (B\n) ++(-90:0.3cm) coordinate (B\n+); };
\foreach \n in {1,2} {
\path (W\n) ++(-90:0.3cm) coordinate (W\n+); };
\draw[->, line width=0.15cm, rounded corners, color=red] (8,3.7)--(B4+)--(W2+)--(B3+)--(W1+)--(0,3.7);
\node [red] at (4.2,4.5) {\scalebox{2.5}{$z$}};
\end{tikzpicture}
} };

\draw[->,line width=0.03cm] (2,0)--(4,0);
\node at (3,0.3) {\small$\nu^\zig_p(-,z)$};

\node at (6,0)
{\scalebox{0.35}{
\begin{tikzpicture}
\draw[line width=\edgewidth] (0,0) rectangle (8,8);

\foreach \n/\a/\b in {1/3/1,2/7/1,3/3/5,4/7/5, 5/1.666/3.666, 6/5.666/3.666} {
\coordinate (B\n) at (\a,\b);
};
\foreach \m\a/\b in {1/1/3,2/5/3,3/1/7,4/5/7, 5/2.333/4.333, 6/6.333/4.333} {
\coordinate (W\m) at (\a,\b);
};

\foreach \o\a/\b in {1/2/0,2/4/0,3/6/0,4/8/0,5/8/2,6/8/4,7/8/6,8/6/8,9/4/8,10/2/8,11/0/8,12/0/6,13/0/4,14/0/2} {
\coordinate (E\o) at (\a,\b);
};

\foreach \n/\m in {1/1,1/2,2/2,3/1,3/3,3/4,4/2,4/4,5/6} {
\draw[line width=\edgewidth] (B\n)--(W\m); };
\foreach \n/\o in {1/1,1/2,2/3,2/4,2/5,4/7,6/6} {
\draw[line width=\edgewidth] (B\n)--(E\o); };
\foreach \m/\o in {4/8,4/9,3/10,3/11,3/12,1/14,5/13} {
\draw[line width=\edgewidth] (W\m)--(E\o); };

\foreach \n in {1,2,3,4,5,6} {
\draw [line width=\nodewidth, fill=black] (B\n) circle [radius=\noderad] ;
};
\foreach \m in {1,2,3,4,5,6} {
\draw [line width=\nodewidth, fill=white] (W\m) circle [radius=\noderad] ;
};

\foreach \n in {1,2} {
\path (B\n) ++(-90:0.3cm) coordinate (B\n-); };
\foreach \n in {1,2} {
\path (W\n) ++(-90:0.3cm) coordinate (W\n-); };
\draw[->, line width=0.15cm, rounded corners, color=blue] (0,1.7)--(W1-)--(B1-)--(W2-)--(B2-)--(8,1.7);
\node [blue] at (4,2.8) {\scalebox{2.5}{$z^\prime$}};
\end{tikzpicture}
} };

\draw[->,line width=0.03cm] (4,-1)--(2,-3);
\filldraw[white] (2,-1.6) rectangle (3.8,-2.4);
\node at (2.6,-1.8) {\small (zag-1)}; \node at (3.1,-2.2) {\small -- (zag-3)};

\node at (0,-4)
{\scalebox{0.35}{
\begin{tikzpicture}
\draw[line width=\edgewidth] (0,0) rectangle (8,8);

\foreach \n/\a/\b in {1/3/1,2/7/1,3/3/5,4/7/5, 5/1.666/3.666, 6/5.666/3.666, 7/1.666/2.333, 8/5.666/2.333} {
\coordinate (B\n) at (\a,\b);
};
\foreach \m\a/\b in {1/1/3,2/5/3,3/1/7,4/5/7, 5/2.333/4.333, 6/6.333/4.333, 7/2.333/1.666, 8/6.333/1.666} {
\coordinate (W\m) at (\a,\b);
};

\foreach \o\a/\b in {1/2/0,2/4/0,3/6/0,4/8/0,5/8/2,6/8/4,7/8/6,8/6/8,9/4/8,10/2/8,11/0/8,12/0/6,13/0/4,14/0/2} {
\coordinate (E\o) at (\a,\b);
};

\foreach \n/\m in {1/1,2/2,3/1,3/3,3/4,4/2,4/4,5/6,7/8} {
\draw[line width=\edgewidth] (B\n)--(W\m); };
\foreach \n/\o in {1/1,1/2,2/3,2/4,4/7,6/6,8/5} {
\draw[line width=\edgewidth] (B\n)--(E\o); };
\foreach \m/\o in {4/8,4/9,3/10,3/11,3/12,5/13,7/14} {
\draw[line width=\edgewidth] (W\m)--(E\o); };

\foreach \n in {1,2,3,4,5,6,7,8} {
\draw [line width=\nodewidth, fill=black] (B\n) circle [radius=\noderad] ;
};
\foreach \m in {1,2,3,4,5,6,7,8} {
\draw [line width=\nodewidth, fill=white] (W\m) circle [radius=\noderad] ;
};

\end{tikzpicture}
} };

\draw[->,line width=0.03cm] (2,-4)--(4,-4);
\node at (3,-3.7) {\small (join)};

\node at (6,-4)
{\scalebox{0.35}{
\begin{tikzpicture}
\draw[line width=\edgewidth] (0,0) rectangle (8,8);

\foreach \n/\a/\b in {1/3/0.666,2/7/0.666,3/3/3.333,4/7/3.333,5/3/6, 6/7/6} {
\coordinate (B\n) at (\a,\b);
};
\foreach \m\a/\b in {1/1/2,2/5/2,3/1/4.666,4/5/4.666,5/1/7.333,6/5/7.333} {
\coordinate (W\m) at (\a,\b);
};
\foreach \o\a/\b in {1/2/0,2/4/0,3/6/0,4/8/0,5/8/2.666,6/8/4,7/8/6.666,8/6/8,9/4/8,10/2/8,11/0/8,12/0/6.666,13/0/4,14/0/2.666} {
\coordinate (E\o) at (\a,\b);
};

\foreach \n/\m in {1/1,2/2,3/1,3/2,3/3,3/4,4/2,4/4,5/3,5/5,5/6,6/4,6/6} {
\draw[line width=\edgewidth] (B\n)--(W\m); };
\foreach \n/\o in {1/1,1/2,2/3,2/4,4/5,4/6,6/7} {
\draw[line width=\edgewidth] (B\n)--(E\o); };
\foreach \m/\o in {6/8,6/9,5/10,5/11,5/12,3/13,1/14} {
\draw[line width=\edgewidth] (W\m)--(E\o); };

\foreach \n in {1,2,3,4,5,6} {
\draw [line width=\nodewidth, fill=black] (B\n) circle [radius=\noderad] ;
};
\foreach \m in {1,2,3,4,5,6} {
\draw [line width=\nodewidth, fill=white] (W\m) circle [radius=\noderad] ;
};
\end{tikzpicture}
} };
\end{tikzpicture}
\caption{An example showing that the zig-deformation and zag-deformation are not mutually inverse operations.}
\label{ex_zigzag_non_involution}
\end{figure}
\end{Example}

\section{Combinatorial mutations of perfect matching polygons}\label{sec_mutation_polygon1}
In this section, we discuss a relationship between deformations of consistent dimer models and combinatorial mutations of polygons.
We first define combinatorial mutations for lattice polytopes of any dimension, and then we mainly discuss the case of polygons.

\subsection{Preliminaries on combinatorial mutations of polytopes}\label{sec_pre_mutation_polygon}

Following \cite{ACGK}, we introduce combinatorial mutations of polytopes.
Let $N\cong\ZZ^d$ be a lattice of rank $d$, and $P\subset N_\RR\coloneqq N\otimes_\ZZ\RR$ be a convex lattice polytope.
We assume that $P$ contains the origin ${\bf 0}$.
We denote by $\calV(P)$ the set of vertices of $P$.
We say that two polytopes $P,Q\subset N_\RR$ are \emph{isomorphic} if they are transformed into each other by $\GL(d,\ZZ)$-transformations, in which case we denote $P\cong Q$.

We first prepare some notions used in the definition of combinatorial mutations of polytopes.

\begin{Definition}[mutation data]\label{def_mutation_data}
Let $P$ be a convex lattice polytope as above.
In order to define a combinatorial mutation of $P$, we fix the data consisting of a vector $w$ and the associated integers $h_{\max}(P,w)$,
$h_{\min}(P,w)$ and $\width(P,w)$ defined as follows.

Let $w\in M\coloneqq\Hom_\ZZ(N,\ZZ)\cong\ZZ^d$ be a primitive lattice vector.
The element $w\in M$ determines the linear map $\langle w,-\rangle\colon N_\RR\rightarrow \RR$, where $\langle-,- \rangle$ is the natural inner product.
We set
\begin{displaymath}
h_{\max}(P,w)\coloneqq{\max}\{\langle w,u\rangle \,|\, u\in P\} \qquad\text{and}\qquad
h_{\min}(P,w)\coloneqq{\min}\{\langle w,u\rangle \,|\, u\in P\},
\end{displaymath}
which are integers since $P$ is a lattice polytope.
When the situation is clear, we simply denote these by $h_{\max}$ and $h_{\min}$, respectively.
We define the \emph{width} of $P$ with respect to $w$ as
\begin{displaymath}
\width(P,w)\coloneqq h_{\max}(P,w)-h_{\min}(P,w).
\end{displaymath}
\end{Definition}

We note that if the origin ${\bf 0}$ is contained in the strict interior $P^\circ$ of $P$, then $h_{\min}<0$ and $h_{\max}>0$, in which case we have $\width(P,w)\ge2$.
We say that a lattice point $u\in N$ (resp.\ a~subset $F\in N_\RR$) is at \emph{height} $m$ with respect to $w$ if $\langle w,u\rangle=m$
(resp.\ $\langle w,u\rangle=m$ for any $u\in F$).

For each height $h\in\ZZ$, we let
\begin{displaymath}
w_h(P)\coloneqq\operatorname{conv}\{u\in P\cap N \,|\, \langle w,u\rangle=h\},
\end{displaymath}
which is the (possibly empty) convex hull of all lattice points in $P$ at height $h$.
By definition, $w_{h_{\min}}(P)$ and $w_{h_{\max}}(P)$ are faces of $P$.
Using these notions, a combinatorial mutation of $P$ is defined as follows.

\begin{Definition}\label{def_mutation_polygon}
Let the notation be the same as above.
We assume that there exists a lattice polytope $F\subset N_\RR$ such that $\langle w,u\rangle=0$ for any $u\in F$, and
for each negative height $h_{\min}\le h<0$ there exists a possibly empty lattice polytope $G_h\subset N_\RR$ satisfying
\begin{equation}\label{condition_Gh}
\{u\in\calV(P)\,|\, \langle w,u \rangle=h\}\subseteq G_h+(-h)F\subseteq w_h(P),
\end{equation}
where $+$ means the Minkowski sum, and we especially define $Q+\varnothing=\varnothing$ for any polytope $Q$.
We call $F$ a \emph{factor} of $P$ with respect to~$w$.
Then, we define the \emph{combinatorial mutation} of $P$ given by the vector $w$, factor $F$ and polytopes $\{G_h\}$ as
\begin{displaymath}
\mutation_w(P,F)\coloneqq\operatorname{conv}\bigg(\bigcup_{h=h_{\min}}^{-1}G_h\cup\bigcup_{h=0}^{h_{\max}}(w_h(P)+hF)\bigg).
\end{displaymath}
\end{Definition}

We note that the combinatorial mutation is independent of the choice of the polytopes $\{G_h\}$ (see \cite[Proposition~1]{ACGK}).
Also, a translation of the factor $F$ does not affect the combinatorial mutation; that is,
for any $u\in N$ with $\langle w,u\rangle=0$ we have $\mutation_w(P,F)\cong\mutation_w(P,u+F)$; see~\cite{ACGK} for more details.

\begin{Remark}[the combinatorial mutation for the case of $d=2$]\label{rem_def_mutation}
When $d=2$, we choose an edge $E$ of a lattice polygon~$P$, and take $w\in M\cong\ZZ^2$ as a primitive inner normal vector for~$E$.
By the choice of $w$, we see that $w_{h_{\min}}(P)=E$ and $w_{h_{\max}}(P)$ is either a vertex or an edge of~$P$.
Then, we take a primitive lattice element $u_E\in N$ satisfying $\langle w,u_E \rangle=0$, and define the line segment $F\coloneqq\operatorname{conv}\{{\bf 0},u_E\}$,
which is parallel to $E$ at height $0$ and has unit lattice length. Since $u_E$ is uniquely determined up to sign, so is $F$.
In this case, $P$ admits a combinatorial mutation with respect to $w$ (equivalently, there exit polytopes $\{G_h\}$ satisfying (\ref{condition_Gh})) if and only if $|E\cap N|-1\ge -h_{\min}$, see \cite[Lemma~1]{KNP}.
We note that the combinatorial mutation does not depend on the choice of $u_E$ (and hence $F$). That is,
$\mutation_w(P,F)\cong\mutation_w(P,-F)$, which means they are $\GL(2,\ZZ)$-equivalent.
\end{Remark}

\begin{Example}
\label{ex_mutation_4b_polygon}
We consider the polygon $P$ given in the left of Figure~\ref{ex_mut_poly_4b} below,
which coincides with the PM polygon of the dimer model given in Figure~\ref{ex_dimer_4b} (see also Figure~\ref{polygon_4b}).
Here, the double circle stands for the origin ${\bf 0}$.
We consider the edge $E$ whose primitive inner normal vector is $w=(1,1)$, in which case $h_{\min}=-1$ and $h_{\max}=2$.
We take $u_E=(1,-1)\in N$, which satisfies $\langle w,u_E\rangle=0$, and consider the line segment $F=\operatorname{conv}\{{\bf 0},u_E\}$.
The combinatorial mutation $\mut_w(P,F)$ of the polygon $P$ is shown in the upper part of Figure~\ref{ex_mut_poly_4b}.
On the other hand, if we consider the line segment $-F=\operatorname{conv}\{{\bf 0},-u_E\}$, the combinatorial mutation $\mut_w(P,-F)$ is as shown in the lower part of Figure~\ref{ex_mut_poly_4b}.

\begin{figure}[h!]\centering
\begin{tikzpicture}
\node at (0,0)
{\scalebox{0.5}{
\begin{tikzpicture}
\coordinate (00) at (0,0); \coordinate (10) at (1,0); \coordinate (01) at (0,1); \coordinate (-10) at (-1,0); \coordinate (11) at (1,1);
\coordinate (0-1) at (0,-1);

\draw [step=1, gray] (-2.3,-2.3) grid (2.3,2.3);
\draw [line width=0.06cm] (10)--(11) ; \draw [line width=0.06cm] (11)--(-10) ;
\draw [line width=0.06cm] (-10)--(0-1) ; \draw [line width=0.06cm] (0-1)--(10) ;

\draw [line width=0.025cm] (00) circle [radius=0.25] ;
\draw [line width=0.05cm, fill=black] (00) circle [radius=0.1] ; \draw [line width=0.05cm, fill=black] (10) circle [radius=0.1] ;
\draw [line width=0.05cm, fill=black] (11) circle [radius=0.1] ;
\draw [line width=0.05cm, fill=black] (-10) circle [radius=0.1] ; \draw [line width=0.05cm, fill=black] (0-1) circle [radius=0.1] ;
\end{tikzpicture}
}};

\draw[->,line width=0.03cm] (2,0)--(4,0);
\node at (3,0.5) {$\mut_w(-,F)$};

\node at (6.5,0)
{\scalebox{0.5}{
\begin{tikzpicture}
\coordinate (00) at (0,0); \coordinate (10) at (1,0); \coordinate (20) at (2,0);
\coordinate (11) at (1,1); \coordinate (2-1) at (2,-1); \coordinate (3-1) at (3,-1);
\coordinate (-10) at (-1,0);

\draw [step=1, gray] (-2.3,-2.3) grid (4.3,2.3);
\draw [line width=0.06cm] (-10)--(11) ; \draw [line width=0.06cm] (-10)--(2-1) ;
\draw [line width=0.06cm] (2-1)--(3-1) ; \draw [line width=0.06cm] (11)--(3-1) ;

\draw [line width=0.025cm] (00) circle [radius=0.25] ;
\draw [line width=0.05cm, fill=black] (00) circle [radius=0.1] ; \draw [line width=0.05cm, fill=black] (10) circle [radius=0.1] ;
\draw [line width=0.05cm, fill=black] (20) circle [radius=0.1] ;
\draw [line width=0.05cm, fill=black] (11) circle [radius=0.1] ; \draw [line width=0.05cm, fill=black] (2-1) circle [radius=0.1] ;
\draw [line width=0.05cm, fill=black] (3-1) circle [radius=0.1] ; \draw [line width=0.05cm, fill=black] (-10) circle [radius=0.1] ;
\end{tikzpicture}
}};

\node at (0,-3.5)
{\scalebox{0.5}{
\begin{tikzpicture}
\coordinate (00) at (0,0); \coordinate (10) at (1,0); \coordinate (01) at (0,1); \coordinate (-10) at (-1,0); \coordinate (11) at (1,1);
\coordinate (0-1) at (0,-1);

\draw [step=1, gray] (-2.3,-2.3) grid (2.3,2.3);
\draw [line width=0.06cm] (10)--(11) ; \draw [line width=0.06cm] (11)--(-10) ;
\draw [line width=0.06cm] (-10)--(0-1) ; \draw [line width=0.06cm] (0-1)--(10) ;

\draw [line width=0.025cm] (00) circle [radius=0.25] ;
\draw [line width=0.05cm, fill=black] (00) circle [radius=0.1] ; \draw [line width=0.05cm, fill=black] (10) circle [radius=0.1] ;
\draw [line width=0.05cm, fill=black] (11) circle [radius=0.1] ;
\draw [line width=0.05cm, fill=black] (-10) circle [radius=0.1] ; \draw [line width=0.05cm, fill=black] (0-1) circle [radius=0.1] ;
\end{tikzpicture}
}};

\draw[->,line width=0.03cm] (2,-3.5)--(4,-3.5);
\node at (3,-3) {$\mut_w(-,-F)$};

\node at (6.5,-3.5)
{\scalebox{0.5}{
\begin{tikzpicture}
\coordinate (00) at (0,0); \coordinate (10) at (1,0); \coordinate (0-1) at (0,-1);
\coordinate (11) at (1,1); \coordinate (01) at (0,1); \coordinate (02) at (0,2);
\coordinate (-13) at (-1,3);

\draw [step=1, gray] (-2.3,-2.3) grid (2.3,4.3);
\draw [line width=0.06cm] (0-1)--(10) ; \draw [line width=0.06cm] (10)--(11) ;
\draw [line width=0.06cm] (11)--(-13) ; \draw [line width=0.06cm] (-13)--(0-1) ;

\draw [line width=0.025cm] (00) circle [radius=0.25] ;
\draw [line width=0.05cm, fill=black] (00) circle [radius=0.1] ; \draw [line width=0.05cm, fill=black] (10) circle [radius=0.1] ;
\draw [line width=0.05cm, fill=black] (0-1) circle [radius=0.1] ;
\draw [line width=0.05cm, fill=black] (11) circle [radius=0.1] ; \draw [line width=0.05cm, fill=black] (01) circle [radius=0.1] ;
\draw [line width=0.05cm, fill=black] (02) circle [radius=0.1] ; \draw [line width=0.05cm, fill=black] (-13) circle [radius=0.1] ;
\end{tikzpicture}
}};
\end{tikzpicture}
\caption{Examples of the combinatorial mutations of $P$.}\label{ex_mut_poly_4b}
\end{figure}
\end{Example}

Now, we collect fundamental properties of this combinatorial mutation.

\begin{Proposition}[{see \cite[Lemma~2 and Proposition~2]{ACGK}}]\label{properties_mutation_polygon}
Let the notation be the same as above.
\begin{itemize}\itemsep=0pt
\item[\rm (1)] If $Q\coloneqq\mutation_w(P,F)$, then we have $P=\mutation_{(-w)}(Q,F)$.
\item[\rm (2)] $P$ is a Fano polytope if and only if $\mutation_w(P,F)$ is a Fano polytope.
\end{itemize}
\end{Proposition}

Here, we recall that a convex lattice polytope $P\subset N_\RR$ with $\dim P=d$ is called \emph{Fano} if the origin is contained in the strict interior of $P$,
and the vertices $\calV(P)$ of $P$ are primitive lattice points of $N$.

Then, we consider the combinatorial mutation of a lattice polytope $P$ in terms of the polar dual $P^*$ of $P$ in $M$.
To do this, we must first discuss the polar dual $P^*$ for a polyhedron $P$.
We consider the family of polyhedra (not necessarily convex polytopes) which are of the following form:
\begin{displaymath}
\calP^d \coloneqq \bigg\{ \bigcap_{v \in S} H_{v,\geq -k_v} \cap \bigcap_{v' \in T}H_{v',\geq 0} \subset N_\RR \,|\, S,T \subset M, \; |S|, |T| < \infty, \; k_v \in \ZZ_{>0} \bigg\},
\end{displaymath}
where $H_{v,\geq k}=\{u \in N_\RR \,|\, \langle v,u \rangle \geq k\}$ for $v \in M$ and $k \in \RR$.
We note that any lattice polytope containing the origin of~$N_\RR$ belongs to $\calP^d$ but the ones not containing the origin do not belong to~$\calP^d$
since one of the supporting hyperplanes of such a polytope is of the form $H_{v, \geq k}$ for some $v \in M$ and some positive integer~$k$.
Similarly, we define~$\calQ^d$ by swapping the roles of~$N_\RR$ and~$M_\RR$.

For a given $P \in \calP^d$, we consider the \emph{polar dual} $P^* \subset M_\RR$ of $P$ defined as
\begin{displaymath}
P^* \coloneqq \{v \in M_\RR\,|\,\langle v,u\rangle\ge-1 \text{ for all } u\in P\}\subset M_\RR.
\end{displaymath}
Then, we have the following statements.

\begin{Proposition}\label{prop_dual_polytope}
Let the notation be the same as above. Then, for any $P \in \calP^d$ we have
\begin{itemize}\itemsep=0pt
\item[$(i)$] $P^* \in \calQ^d$,
\item[$(ii)$] $(P^*)^*=P$.
\end{itemize}
\end{Proposition}

\begin{proof}(i) Let $P \in \calP^d$. Since $P$ is a polyhedron, there exist a polytope~$Q$ and a polyhedral cone~$C$ such that
$P=Q+C$, where $+$ denotes the Minkowski sum (see \cite[Corollary~7.1b]{Sch}).
Let $Q=\operatorname{conv}\big(\big\{\frac{1}{k_1}u_1,\dots,\frac{1}{k_p}u_p\big\}\big)$ and let $C=\operatorname{cone}(\{u_1',\dots,u_q'\})$.
Note that we can choose~$u_i$,~$u_j'$ from $N$ and $k_i \in \ZZ_{>0}$ because of the form of~$P$.
In what follows, we will show that
\begin{displaymath}
P^*=\bigcap_{i=1}^p H_{u_i,\geq -k_i} \cap \bigcap_{j=1}^qH_{u_j',\geq 0}.
\end{displaymath}

First, we take $v \in P^*$. Then, we have $\langle v,u \rangle \geq -1$ for any $u \in P$.
Since $\frac{1}{k_i}u_i \in Q+{\bf 0} \subset P$, where ${\bf 0} \in C$ denotes the origin,
we see that $\langle v, \frac{1}{k_i}u_i \rangle \geq -1$, i.e., $\langle v, u_i \rangle \geq -k_i$ for each $i$.
If there is $j$ with $\langle v,u_j' \rangle<0$, then $\langle v, u'+ru_j' \rangle < -1$ for some $u' \in Q$ and some sufficiently large~$r$.
Moreover, we have $u'+ru_j' \in Q+C=P$. This contradicts $v \in P^*$; thus $\langle v,u_j' \rangle\geq 0$ for each~$j$.
Therefore, $v \in \bigcap_{i=1}^p H_{u_i,\geq -k_i} \cap \bigcap_{j=1}^qH_{u_j',\geq 0}$.

On the other hand, we take $v \in \bigcap_{i=1}^p H_{u_i,\geq -k_i} \cap \bigcap_{j=1}^qH_{u_j',\geq 0}$.
For any $u \in P$, as mentioned above, there exist $u' \in Q$ and $u'' \in C$ such that $u=u'+u''$.
Let $u'=\sum_{i=1}^p \frac{r_i}{k_i}u_i$, where $r_i \geq 0$ with $\sum_{i=1}^pr_i=1$, and let $u''=\sum_{j=1}^q s_ju_j'$, where $s_j \geq 0$.
By using these expressions together with the inequalities $\langle v,u_i \rangle \geq -k_i$ for each $i$ and $\langle v,u_j' \rangle \geq 0$ for each $j$, we see that
\begin{align*}
\langle v,u \rangle = \langle v,u' \rangle+\langle v,u'' \rangle =
\sum_{i=1}^p \frac{r_i}{k_i}\langle v,u_i \rangle+\sum_{j=1}^q s_j\langle v,u_j' \rangle \geq -\sum_{i=1}^pr_i=-1,
\end{align*}
and thus we have $v \in P^*$.

(ii) For any $u \in P$, we have $\langle v,u\rangle\geq -1$ for any $v \in P^*$, which means that $P \subset (P^*)^*$.
For the other inclusion, we take $u \in N_\RR \setminus P$.
Let $P=\bigcap_{v \in S} H_{v,\geq -k_v} \cap \bigcap_{v' \in T}H_{v',\geq 0}$.
Then either $\langle v,u \rangle <-k_v$ for some $v \in S$ or $\langle v',u \rangle < 0$ for some $v' \in T$ holds.
In the former case, since $\langle v,u' \rangle \geq -k_v$ for any $u' \in P$, we have $\frac{1}{k_v}v \in P^*$.
This means that there is $v''\coloneqq\frac{1}{k_v}v \in P^*$ such that $\langle v'',u \rangle<-1$, and hence $u \not\in (P^*)^*$.
Similarly, in the latter case, since $\langle rv',u' \rangle \geq 0 \geq -1$ for any $u' \in P$ and $r \geq 0$, we have $rv' \in P^*$.
This implies that $u \not\in (P^*)^*$ for sufficiently large $r$, and hence $u \not\in (P^*)^*$. Therefore, we obtain $(P^*)^* \subset P$, as required.
\end{proof}

We next define a map
\[
\varphi_{w,F}\colon \ M_\RR\rightarrow M_\RR \qquad \text{as} \quad \varphi_{w,F}(v)\coloneqq v-v_{\min}w,
\]
where \mbox{$v_{\min}\coloneqq {\min}\{\langle v, u\rangle\,|\, u\in F\}$}.
In particular, when $d=2$ (see Remark~\ref{rem_def_mutation}), this map can be described as
\begin{equation}\label{mutation_M}
\varphi_{w,F}(v)=
\begin{cases}
v & \text{if $\langle v,u_E\rangle\ge0$}, \\
v-\langle v,u_E\rangle w&\text{if $\langle v,u_E\rangle<0$},
\end{cases}
\end{equation}
with $F=\operatorname{conv}\{{\bf 0},u_E\}$.
The next proposition is crucial to prove our main result Theorem~\ref{mutation=deformation}.

\begin{Proposition}
\label{prop_mutation_dual_polytope}
For any $P \in \calP^d$, we have
\begin{displaymath}\varphi_{w,F}(P^*)=\mut_w(P,F)^*.
\end{displaymath}
\end{Proposition}

\begin{proof}
Although this equality essentially follows from \cite[Proposition~4]{ACGK} and the discussions in \cite[p.~12]{ACGK}, we give a precise proof for completeness.

Let $\varphi=\varphi_{w,F}$ and $Q=\mut_w(P,F)$.
To show $\varphi(P^*)\subset Q^*$, we take $v \in P^*$ arbitrarily and consider $\varphi(v)=v - v_{\min}w \in \varphi(P^*)$.
We will show that $\langle v - v_{\min}w,u \rangle \geq -1$ for any $u \in Q$.
It suffices to show this for each vertex $u\in\calV(Q)$.
\begin{itemize}\itemsep=0pt
\item Let $u\in\calV(Q)$ with $\langle w,u \rangle \geq 0$. Then we can write $u=u_P+\langle w,u_P \rangle u_F$ for some $u_P \in \calV(P)$ and $u_F \in \calV(F)$.
In particular, we have
\begin{displaymath}
\langle v-v_{\min}w, u\rangle = \langle v,u_P \rangle + \langle w,u_P \rangle(\langle v,u_F \rangle-v_{\min}) \geq \langle v,u_P \rangle \geq -1.
\end{displaymath}

\item Let $u\in\calV(Q)$ with $\langle w,u \rangle < 0$. For any $u_F \in \calV(F)$, we have $u-\langle w,u \rangle u_F \in P$.
Hence, $\langle v,u-\langle w,u \rangle u_F \rangle \geq -1$. In particular, $\langle v,u \rangle \geq -1 + v_{\min} \langle w,u \rangle$. Thus, we see that
\begin{displaymath}
\langle v-v_{\min}w, u\rangle =\langle v, u\rangle - v_{\min} \langle w, u\rangle \geq -1.
\end{displaymath}
\end{itemize}

To show $Q^*\subset\varphi(P^*)$, we will show that for any $v \in Q^*$ there is $v' \in P^*$ such that $v=\varphi(v')$.
Let $\Delta_F$ be the normal fan of $F$ in $M_\RR$ and let $\sigma \in \Delta_F$ be a maximal cone in $\Delta_F$.
The discussions in \cite[p.~12]{ACGK} say that there exists $M_\sigma \in \GL_d(\ZZ)$ such that the map $\varphi$ is equal to $M_\sigma$,
i.e., $\varphi(v)=vM_\sigma$. Thus, we conclude that $\varphi(v')=v$ for $v'=vM_\sigma^{-1}$.
\end{proof}

\begin{Example}\label{ex_mutation_4b_polygon_dual}
Let $P$ be the polygon used in Example~\ref{ex_mutation_4b_polygon}.
The polar dual~$P^*$ takes the form as in the left of Figure~\ref{fig_mut_poly_4b_dual} below.
We take $w=(1,1)$, $u_E=(1,-1)$ and consider the line segment $F=\operatorname{conv}\{{\bf 0},u_E\}$ just like Example~\ref{ex_mutation_4b_polygon}.
Then, applying the piecewise-linear map \eqref{mutation_M} with $u_E$ (resp.~$-u_E$) to the polar dual~$P^*$,
we have the new polygon shown in the upper (resp.\ lower) part of Figure~\ref{fig_mut_poly_4b_dual}.
In this figure, the red line imply the points $v\in\RR^2$ satisfying $\langle v,u_E \rangle =0$.
We see that the resulting polygons are the polar dual of~$\mut_w(P,F)$ and~$\mut_w(P,-F)$ given in Example~\ref{ex_mutation_4b_polygon}
as shown in Proposition~\ref{prop_mutation_dual_polytope}.
\end{Example}

\begin{figure}[h!]\centering
\begin{tikzpicture}
\node at (0,0)
{\scalebox{0.5}{
\begin{tikzpicture}
\foreach \n/\a/\b in {00/0/0,11/1/1,10/1/0,1-1/1/-1,1-2/1/-2,0-1/0/-1,-10/-1/0,-11/-1/1,01/0/1} {
\coordinate (V\n) at (\a,\b);
};

\draw [step=1, gray] (-2.3,-3.3) grid (2.3,2.3);
\draw [line width=0.06cm, red] (-2,-2)--(2,2);
\draw [line width=0.06cm] (V11)--(V1-2)--(V-10)--(V-11)--(V11) ;

\draw [line width=0.025cm] (00) circle [radius=0.25] ;
\foreach \n in {00,11,10,1-1,1-2,0-1,-10,-11,01}{
\draw [line width=0.05cm, fill=black] (V\n) circle [radius=0.1] ;
};
\end{tikzpicture}
}};

\draw[->,line width=0.03cm] (2,0)--(4,0);
\node at (3,0.4) {$\varphi_{w,F}$};

\node at (6.5,0)
{\scalebox{0.5}{
\begin{tikzpicture}
\foreach \n/\a/\b in {00/0/0,13/1/3,12/1/2,11/1/1,10/1/0,1-1/1/-1,1-2/1/-2,0-1/0/-1,01/0/1} {
\coordinate (V\n) at (\a,\b);
};

\draw [step=1, gray] (-2.3,-3.3) grid (2.3,4.3);
\draw [line width=0.06cm, red] (-2,-2)--(2,2);
\draw [line width=0.06cm] (V13)--(V1-2)--(-1/2,-1/2)--(V01)--(V13);

\draw [line width=0.025cm] (00) circle [radius=0.25] ;
\foreach \n in {00,13,12,11,10,1-1,1-2,0-1,01}{
\draw [line width=0.05cm, fill=black] (V\n) circle [radius=0.1] ;
};

\end{tikzpicture}
}};
\node at (0,-3.5)
{\scalebox{0.5}{
\begin{tikzpicture}
\foreach \n/\a/\b in {00/0/0,11/1/1,10/1/0,1-1/1/-1,1-2/1/-2,0-1/0/-1,-10/-1/0,-11/-1/1,01/0/1} {
\coordinate (V\n) at (\a,\b);
};

\draw [step=1, gray] (-2.3,-3.3) grid (2.3,2.3);
\draw [line width=0.06cm, red] (-2,-2)--(2,2);
\draw [line width=0.06cm] (V11)--(V1-2)--(V-10)--(V-11)--(V11) ;

\draw [line width=0.025cm] (00) circle [radius=0.25] ;
\foreach \n in {00,11,10,1-1,1-2,0-1,-10,-11,01}{
\draw [line width=0.05cm, fill=black] (V\n) circle [radius=0.1] ;
};
\end{tikzpicture}
}};

\draw[->,line width=0.03cm] (2,-3.5)--(4,-3.5);
\node at (3,-3.1) {$\varphi_{w,-F}$};

\node at (6.5,-3.5)
{\scalebox{0.5}{
\begin{tikzpicture}

\foreach \n/\a/\b in {00/0/0,10/1/0,-10/-1/0,41/4/1,31/3/1,21/2/1,11/1/1,01/0/1,-11/-1/1} {
\coordinate (V\n) at (\a,\b);
};

\draw [step=1, gray] (-2.3,-1.3) grid (5.3,2.3);
\draw [line width=0.06cm, red] (-1,-1)--(2,2);
\draw [line width=0.06cm] (V41)--(V10)--(-1/2,-1/2)--(V-10)--(V-11)--(V41) ;

\draw [line width=0.025cm] (00) circle [radius=0.25] ;
\foreach \n in {00,10,-10,41,31,21,11,01,-11}{
\draw [line width=0.05cm, fill=black] (V\n) circle [radius=0.1] ;
};

\end{tikzpicture}
}};
\end{tikzpicture}
\caption{Examples of the image of $P^*$ under the piecewise-linear map \eqref{mutation_M}.}
\label{fig_mut_poly_4b_dual}
\end{figure}

\subsection{The perfect matching polygons of deformed dimer models}\label{subsec_behave_PMpolygon}

We here consider the PM polygons of deformed dimer models.

\begin{Example}\label{ex_deformation_4b_polygon}
Let $\Gamma$ be the consistent dimer model given in Figure~\ref{ex_dimer_4b}, and let
$\nu^\zig_p(\Gamma, z)$ (resp.~$\nu^\zag_p(\Gamma, z)$) be the zig-deformed (resp.\ zag-deformed) dimer model at~$z$ as shown in
Figure~\ref{ex_def_zig} (resp.\ Figure~\ref{ex_def_zag}) in Example~\ref{ex_deformation_4b}.
Here, we note the changes of zigzag paths on $\Gamma$ under these deformations.
First, we recall the zigzag paths on $\Gamma$ given in Figure~\ref{zigzag_4b}.

\begin{center}
\newcommand{\edgewidth}{0.07cm}
\newcommand{\nodewidth}{0.07cm}
\newcommand{\noderad}{0.24} 

{\scalebox{1}{
\begin{tikzpicture}

\node at (0,-1.6) {$z_1$}; \node at (3.5,-1.6) {$z_2$}; \node at (7,-1.6) {$z=z_3$}; \node at (10.5,-1.6) {$z_4$};

\node (ZZ1) at (0,0)
{\scalebox{0.4}{
\begin{tikzpicture}
\basicdimerB
\draw[->, line width=0.2cm, rounded corners, color=green] (0,3)--(W2)--(B1)--(W3)--(B3)--(6,3);
\end{tikzpicture} }};

\node (ZZ2) at (3.5,0)
{\scalebox{0.4}{
\begin{tikzpicture}
\basicdimerB
\draw[->, line width=0.2cm, rounded corners, color=blue] (3,0)--(B1)--(W2)--(B2)--(0,6);
\draw[->, line width=0.2cm, rounded corners, color=blue] (6,0)--(W1)--(B3)--(W3)--(3,6);
\end{tikzpicture} }} ;

\node (ZZ3) at (7,0)
{\scalebox{0.4}{
\begin{tikzpicture}
\basicdimerB
\draw[->, line width=0.2cm, rounded corners, color=red] (3,6)--(W3)--(B2)--(W2)--(0,3);
\draw[->, line width=0.2cm, rounded corners, color=red] (6,3)--(B3)--(W1)--(B1)--(3,0);
\end{tikzpicture} }};

\node (ZZ4) at (10.5,0)
{\scalebox{0.4}{
\begin{tikzpicture}
\basicdimerB
\draw[->, line width=0.2cm, rounded corners, color=orange] (0,6)--(B2)--(W3)--(B1)--(W1)--(6,0);
\end{tikzpicture} }} ;

\end{tikzpicture}
}}
\end{center}

Then, we observe zigzag paths on $\nu^\zig_p(\Gamma, z)$ and $\nu^\zag_p(\Gamma, z)$,
but it would be more convenient to see the zigzag paths on the dimer model shown in the middle of Figures~\ref{ex_def_zig} and~\ref{ex_def_zag}
for observing the changes of zigzag paths as in Figures~\ref{fig_change_zzpath} and~\ref{fig_change_zzpath2}.
In both cases, the zigzag path $z=z_3$ on $\Gamma$ is removed and the new zigzag paths, which are colored by red, are created.
The slopes of these new zigzag paths are~$-[z]$.
Oh the other hand, some slopes of zigzag paths on $\Gamma$ are preserved even if we apply the deformations.
For example, we have the zigzag path whose slope is~$[z_2]$ on~$\nu^\zig_p(\Gamma, z)$ which is colored by blue in Figure~\ref{fig_change_zzpath}.
Also, we have the zigzag paths whose slopes are~$[z_1]$,~$[z_4]$ on~$\nu^\zag_p(\Gamma, z)$ which are respectively colored by green and orange in Figure~\ref{fig_change_zzpath2}.
We will see these phenomena for more general situations in Propositions~\ref{deform_vector} and \ref{zigzag_afterdeform1}.

\begin{figure}[h!]\centering

\begin{tikzpicture}
\newcommand{\edgewidth}{0.065cm}
\newcommand{\nodewidth}{0.05cm}
\newcommand{\noderad}{0.18} 
\newcommand{\arrowwidth}{0.08cm}
\node at (0,0)
{\scalebox{0.4}{
\begin{tikzpicture}
\coordinate (W1) at (5,1); \coordinate (W2) at (1,2); \coordinate (W3) at (3,4);
\coordinate (B1) at (3,2); \coordinate (B2) at (1,5); \coordinate (B3) at (5,4);
\draw[line width=\edgewidth] (0,0) rectangle (6,6);
\draw[line width=\edgewidth] (B2)--(W2)--(B1)--(W3)--(B3)--(W1); \draw[line width=\edgewidth] (B1)--(W3);
\draw[line width=\edgewidth] (3,0)--(B1); \draw[line width=\edgewidth] (6,0)--(W1);
\draw[line width=\edgewidth] (3,6)--(W3); \draw[line width=\edgewidth] (0,6)--(B2);
\draw [line width=\nodewidth, fill=black] (B1) circle [radius=\noderad] ;
\draw [line width=\nodewidth, fill=black] (B2) circle [radius=\noderad] ;
\draw [line width=\nodewidth, fill=black] (B3) circle [radius=\noderad] ;
\draw [line width=\nodewidth, fill=white] (W1) circle [radius=\noderad] ;
\draw [line width=\nodewidth, fill=white] (W2) circle [radius=\noderad] ;
\draw [line width=\nodewidth, fill=white] (W3) circle [radius=\noderad] ;

\coordinate (Wa1) at (3,5.6); \coordinate (Wa2) at (3,1.2); \coordinate (Ba1) at (3,4.8); \coordinate (Ba2) at (3,0.4);
\coordinate (Wb1) at (1,3.2); \coordinate (Wb2) at (1,4.4); \coordinate (Bb1) at (1,2.6); \coordinate (Bb2) at (1,3.8);
\coordinate (Wc1) at (5,2.2); \coordinate (Wc2) at (5,3.4); \coordinate (Bc1) at (5,1.6); \coordinate (Bc2) at (5,2.8);

\draw[line width=\edgewidth] (Wa1)--(Bb1); \draw[line width=\edgewidth] (Wc2)--(Ba2);
\draw[line width=\edgewidth] (Bb2)--(2.3,6); \draw[line width=\edgewidth] (2.3,0)--(Wa2);
\draw[line width=\edgewidth] (Wc1)--(3.7,0); \draw[line width=\edgewidth] (Ba1)--(3.7,6);
\draw[line width=\edgewidth] (Wb1)--(0,2.4); \draw[line width=\edgewidth] (Bc1)--(6,2.4);
\draw[line width=\edgewidth] (Wb2)--(0,3.6); \draw[line width=\edgewidth] (Bc2)--(6,3.6);

\draw [line width=\nodewidth, fill=white] (Wa1) circle [radius=\noderad] ; \draw [line width=\nodewidth, fill=white] (Wa2) circle [radius=\noderad] ;
\draw [line width=\nodewidth, fill=black] (Ba1) circle [radius=\noderad] ; \draw [line width=\nodewidth, fill=black] (Ba2) circle [radius=\noderad] ;
\draw [line width=\nodewidth, fill=white] (Wb1) circle [radius=\noderad] ; \draw [line width=\nodewidth, fill=white] (Wb2) circle [radius=\noderad] ;
\draw [line width=\nodewidth, fill=black] (Bb1) circle [radius=\noderad] ; \draw [line width=\nodewidth, fill=black] (Bb2) circle [radius=\noderad] ;
\draw [line width=\nodewidth, fill=white] (Wc1) circle [radius=\noderad] ; \draw [line width=\nodewidth, fill=white] (Wc2) circle [radius=\noderad] ;
\draw [line width=\nodewidth, fill=black] (Bc1) circle [radius=\noderad] ; \draw [line width=\nodewidth, fill=black] (Bc2) circle [radius=\noderad] ;
\draw[->, line width=0.17cm, rounded corners, color=blue] (3,0)--(B1)--(W2)--(B2)--(0,6);
\draw[->, line width=0.17cm, rounded corners, color=blue] (6,0)--(W1)--(B3)--(W3)--(3,6);
\end{tikzpicture}
} };

\node at (4,0)
{\scalebox{0.4}{
\begin{tikzpicture}
\coordinate (W1) at (5,1); \coordinate (W2) at (1,2); \coordinate (W3) at (3,4);
\coordinate (B1) at (3,2); \coordinate (B2) at (1,5); \coordinate (B3) at (5,4);
\draw[line width=\edgewidth] (0,0) rectangle (6,6);
\draw[line width=\edgewidth] (B2)--(W2)--(B1)--(W3)--(B3)--(W1); \draw[line width=\edgewidth] (B1)--(W3);
\draw[line width=\edgewidth] (3,0)--(B1); \draw[line width=\edgewidth] (6,0)--(W1);
\draw[line width=\edgewidth] (3,6)--(W3); \draw[line width=\edgewidth] (0,6)--(B2);
\draw [line width=\nodewidth, fill=black] (B1) circle [radius=\noderad] ;
\draw [line width=\nodewidth, fill=black] (B2) circle [radius=\noderad] ;
\draw [line width=\nodewidth, fill=black] (B3) circle [radius=\noderad] ;
\draw [line width=\nodewidth, fill=white] (W1) circle [radius=\noderad] ;
\draw [line width=\nodewidth, fill=white] (W2) circle [radius=\noderad] ;
\draw [line width=\nodewidth, fill=white] (W3) circle [radius=\noderad] ;

\coordinate (Wa1) at (3,5.6); \coordinate (Wa2) at (3,1.2); \coordinate (Ba1) at (3,4.8); \coordinate (Ba2) at (3,0.4);
\coordinate (Wb1) at (1,3.2); \coordinate (Wb2) at (1,4.4); \coordinate (Bb1) at (1,2.6); \coordinate (Bb2) at (1,3.8);
\coordinate (Wc1) at (5,2.2); \coordinate (Wc2) at (5,3.4); \coordinate (Bc1) at (5,1.6); \coordinate (Bc2) at (5,2.8);

\draw[line width=\edgewidth] (Wa1)--(Bb1); \draw[line width=\edgewidth] (Wc2)--(Ba2);
\draw[line width=\edgewidth] (Bb2)--(2.3,6); \draw[line width=\edgewidth] (2.3,0)--(Wa2);
\draw[line width=\edgewidth] (Wc1)--(3.7,0); \draw[line width=\edgewidth] (Ba1)--(3.7,6);
\draw[line width=\edgewidth] (Wb1)--(0,2.4); \draw[line width=\edgewidth] (Bc1)--(6,2.4);
\draw[line width=\edgewidth] (Wb2)--(0,3.6); \draw[line width=\edgewidth] (Bc2)--(6,3.6);

\draw [line width=\nodewidth, fill=white] (Wa1) circle [radius=\noderad] ; \draw [line width=\nodewidth, fill=white] (Wa2) circle [radius=\noderad] ;
\draw [line width=\nodewidth, fill=black] (Ba1) circle [radius=\noderad] ; \draw [line width=\nodewidth, fill=black] (Ba2) circle [radius=\noderad] ;
\draw [line width=\nodewidth, fill=white] (Wb1) circle [radius=\noderad] ; \draw [line width=\nodewidth, fill=white] (Wb2) circle [radius=\noderad] ;
\draw [line width=\nodewidth, fill=black] (Bb1) circle [radius=\noderad] ; \draw [line width=\nodewidth, fill=black] (Bb2) circle [radius=\noderad] ;
\draw [line width=\nodewidth, fill=white] (Wc1) circle [radius=\noderad] ; \draw [line width=\nodewidth, fill=white] (Wc2) circle [radius=\noderad] ;
\draw [line width=\nodewidth, fill=black] (Bc1) circle [radius=\noderad] ; \draw [line width=\nodewidth, fill=black] (Bc2) circle [radius=\noderad] ;
\draw[->, line width=0.17cm, rounded corners, color=red] (0,2.4)--(Wb1)--(Bb1)--(Wa1)--(Ba1)--(3.7,6);
\draw[->, line width=0.17cm, rounded corners, color=red] (3.7,0)--(Wc1)--(Bc1)--(6,2.4);
\draw[->, line width=0.17cm, rounded corners, color=red] (0,3.6)--(Wb2)--(Bb2)--(2.3,6);
\draw[->, line width=0.17cm, rounded corners, color=red] (2.3,0)--(Wa2)--(Ba2)--(Wc2)--(Bc2)--(6,3.6);
\end{tikzpicture}
} };
\end{tikzpicture}
\caption{Some zigzag paths on the dimer model shown in the middle of Figure~\ref{ex_def_zig} which can be reduced to $\nu^\zig_p(\Gamma, z)$.}
\label{fig_change_zzpath}
\end{figure}

\begin{figure}[h!]\centering

\begin{tikzpicture}
\newcommand{\edgewidth}{0.065cm}
\newcommand{\nodewidth}{0.05cm}
\newcommand{\noderad}{0.18} 
\newcommand{\arrowwidth}{0.08cm}
\node at (0,0)
{\scalebox{0.4}{
\begin{tikzpicture}
\coordinate (W1) at (5,1); \coordinate (W2) at (4,5); \coordinate (W3) at (2,3);
\coordinate (B1) at (2,1); \coordinate (B2) at (1,5); \coordinate (B3) at (4,3);
\draw[line width=\edgewidth] (0,0) rectangle (6,6);

\draw[line width=\edgewidth] (B1)--(W2); \draw[line width=\edgewidth] (B2)--(W2); \draw[line width=\edgewidth] (B3)--(W2);
\draw[line width=\edgewidth] (B1)--(W3); \draw[line width=\edgewidth] (B1)--(W1);
\draw[line width=\edgewidth] (0,6)--(B2); \draw[line width=\edgewidth] (6,3)--(B3);
\draw[line width=\edgewidth] (0,3)--(W3); \draw[line width=\edgewidth] (6,0)--(W1);
\draw [line width=\nodewidth, fill=black] (B1) circle [radius=\noderad] ;
\draw [line width=\nodewidth, fill=black] (B2) circle [radius=\noderad] ;
\draw [line width=\nodewidth, fill=black] (B3) circle [radius=\noderad] ;
\draw [line width=\nodewidth, fill=white] (W1) circle [radius=\noderad] ;
\draw [line width=\nodewidth, fill=white] (W2) circle [radius=\noderad] ;
\draw [line width=\nodewidth, fill=white] (W3) circle [radius=\noderad] ;

\coordinate (W1a) at (2.6,1); \coordinate (W1b) at (3.8,1); \coordinate (B1a) at (3.2,1); \coordinate (B1b) at (4.4,1);
\coordinate (W2a) at (1.6,5); \coordinate (W2b) at (2.8,5); \coordinate (B2a) at (2.2,5); \coordinate (B2b) at (3.4,5);
\coordinate (W3a) at (0.4,3); \coordinate (W3b) at (4.8,3); \coordinate (B3a) at (1.2,3); \coordinate (B3b) at (5.6,3);
\draw[line width=\edgewidth] (B2b)--(W3a); \draw[line width=\edgewidth] (W1a)--(B3b);
\draw[line width=\edgewidth] (W2a)--(2.4,6); \draw[line width=\edgewidth] (B1a)--(2.4,0);
\draw[line width=\edgewidth] (W2b)--(3.6,6); \draw[line width=\edgewidth] (B1b)--(3.6,0);
\draw[line width=\edgewidth] (B3a)--(0,2.295); \draw[line width=\edgewidth] (W1b)--(6,2.295);
\draw[line width=\edgewidth] (B2a)--(0,3.706); \draw[line width=\edgewidth] (W3b)--(6,3.706);

\draw [line width=\nodewidth, fill=white] (W1a) circle [radius=\noderad] ; \draw [line width=\nodewidth, fill=white] (W1b) circle [radius=\noderad] ;
\draw [line width=\nodewidth, fill=black] (B1a) circle [radius=\noderad] ; \draw [line width=\nodewidth, fill=black] (B1b) circle [radius=\noderad] ;
\draw [line width=\nodewidth, fill=white] (W2a) circle [radius=\noderad] ; \draw [line width=\nodewidth, fill=white] (W2b) circle [radius=\noderad] ;
\draw [line width=\nodewidth, fill=black] (B2a) circle [radius=\noderad] ; \draw [line width=\nodewidth, fill=black] (B2b) circle [radius=\noderad] ;
\draw [line width=\nodewidth, fill=white] (W3a) circle [radius=\noderad] ; \draw [line width=\nodewidth, fill=white] (W3b) circle [radius=\noderad] ;
\draw [line width=\nodewidth, fill=black] (B3a) circle [radius=\noderad] ; \draw [line width=\nodewidth, fill=black] (B3b) circle [radius=\noderad] ;
\draw[->, line width=0.17cm, rounded corners, color=green] (0,3)--(W3)--(B1)--(W2)--(B3)--(6,3);
\end{tikzpicture}
} };

\node at (4,0)
{\scalebox{0.4}{
\begin{tikzpicture}
\coordinate (W1) at (5,1); \coordinate (W2) at (4,5); \coordinate (W3) at (2,3);
\coordinate (B1) at (2,1); \coordinate (B2) at (1,5); \coordinate (B3) at (4,3);
\draw[line width=\edgewidth] (0,0) rectangle (6,6);

\draw[line width=\edgewidth] (B1)--(W2); \draw[line width=\edgewidth] (B2)--(W2); \draw[line width=\edgewidth] (B3)--(W2);
\draw[line width=\edgewidth] (B1)--(W3); \draw[line width=\edgewidth] (B1)--(W1);
\draw[line width=\edgewidth] (0,6)--(B2); \draw[line width=\edgewidth] (6,3)--(B3);
\draw[line width=\edgewidth] (0,3)--(W3); \draw[line width=\edgewidth] (6,0)--(W1);
\draw [line width=\nodewidth, fill=black] (B1) circle [radius=\noderad] ;
\draw [line width=\nodewidth, fill=black] (B2) circle [radius=\noderad] ;
\draw [line width=\nodewidth, fill=black] (B3) circle [radius=\noderad] ;
\draw [line width=\nodewidth, fill=white] (W1) circle [radius=\noderad] ;
\draw [line width=\nodewidth, fill=white] (W2) circle [radius=\noderad] ;
\draw [line width=\nodewidth, fill=white] (W3) circle [radius=\noderad] ;

\coordinate (W1a) at (2.6,1); \coordinate (W1b) at (3.8,1); \coordinate (B1a) at (3.2,1); \coordinate (B1b) at (4.4,1);
\coordinate (W2a) at (1.6,5); \coordinate (W2b) at (2.8,5); \coordinate (B2a) at (2.2,5); \coordinate (B2b) at (3.4,5);
\coordinate (W3a) at (0.4,3); \coordinate (W3b) at (4.8,3); \coordinate (B3a) at (1.2,3); \coordinate (B3b) at (5.6,3);
\draw[line width=\edgewidth] (B2b)--(W3a); \draw[line width=\edgewidth] (W1a)--(B3b);
\draw[line width=\edgewidth] (W2a)--(2.4,6); \draw[line width=\edgewidth] (B1a)--(2.4,0);
\draw[line width=\edgewidth] (W2b)--(3.6,6); \draw[line width=\edgewidth] (B1b)--(3.6,0);
\draw[line width=\edgewidth] (B3a)--(0,2.295); \draw[line width=\edgewidth] (W1b)--(6,2.295);
\draw[line width=\edgewidth] (B2a)--(0,3.706); \draw[line width=\edgewidth] (W3b)--(6,3.706);

\draw [line width=\nodewidth, fill=white] (W1a) circle [radius=\noderad] ; \draw [line width=\nodewidth, fill=white] (W1b) circle [radius=\noderad] ;
\draw [line width=\nodewidth, fill=black] (B1a) circle [radius=\noderad] ; \draw [line width=\nodewidth, fill=black] (B1b) circle [radius=\noderad] ;
\draw [line width=\nodewidth, fill=white] (W2a) circle [radius=\noderad] ; \draw [line width=\nodewidth, fill=white] (W2b) circle [radius=\noderad] ;
\draw [line width=\nodewidth, fill=black] (B2a) circle [radius=\noderad] ; \draw [line width=\nodewidth, fill=black] (B2b) circle [radius=\noderad] ;
\draw [line width=\nodewidth, fill=white] (W3a) circle [radius=\noderad] ; \draw [line width=\nodewidth, fill=white] (W3b) circle [radius=\noderad] ;
\draw [line width=\nodewidth, fill=black] (B3a) circle [radius=\noderad] ; \draw [line width=\nodewidth, fill=black] (B3b) circle [radius=\noderad] ;
\draw[->, line width=0.17cm, rounded corners, color=red] (0,2.295)--(B3a)--(W3a)--(B2b)--(W2b)--(3.6,6);
\draw[->, line width=0.17cm, rounded corners, color=red] (3.6,0)--(B1b)--(W1b)--(6,2.295);
\draw[->, line width=0.17cm, rounded corners, color=red] (0,3.706)--(B2a)--(W2a)--(2.4,6);
\draw[->, line width=0.17cm, rounded corners, color=red] (2.4,0)--(B1a)--(W1a)--(B3b)--(W3b)--(6,3.706);
\end{tikzpicture}
} };

\node at (8,0)
{\scalebox{0.4}{
\begin{tikzpicture}
\coordinate (W1) at (5,1); \coordinate (W2) at (4,5); \coordinate (W3) at (2,3);
\coordinate (B1) at (2,1); \coordinate (B2) at (1,5); \coordinate (B3) at (4,3);
\draw[line width=\edgewidth] (0,0) rectangle (6,6);

\draw[line width=\edgewidth] (B1)--(W2); \draw[line width=\edgewidth] (B2)--(W2); \draw[line width=\edgewidth] (B3)--(W2);
\draw[line width=\edgewidth] (B1)--(W3); \draw[line width=\edgewidth] (B1)--(W1);
\draw[line width=\edgewidth] (0,6)--(B2); \draw[line width=\edgewidth] (6,3)--(B3);
\draw[line width=\edgewidth] (0,3)--(W3); \draw[line width=\edgewidth] (6,0)--(W1);
\draw [line width=\nodewidth, fill=black] (B1) circle [radius=\noderad] ;
\draw [line width=\nodewidth, fill=black] (B2) circle [radius=\noderad] ;
\draw [line width=\nodewidth, fill=black] (B3) circle [radius=\noderad] ;
\draw [line width=\nodewidth, fill=white] (W1) circle [radius=\noderad] ;
\draw [line width=\nodewidth, fill=white] (W2) circle [radius=\noderad] ;
\draw [line width=\nodewidth, fill=white] (W3) circle [radius=\noderad] ;

\coordinate (W1a) at (2.6,1); \coordinate (W1b) at (3.8,1); \coordinate (B1a) at (3.2,1); \coordinate (B1b) at (4.4,1);
\coordinate (W2a) at (1.6,5); \coordinate (W2b) at (2.8,5); \coordinate (B2a) at (2.2,5); \coordinate (B2b) at (3.4,5);
\coordinate (W3a) at (0.4,3); \coordinate (W3b) at (4.8,3); \coordinate (B3a) at (1.2,3); \coordinate (B3b) at (5.6,3);
\draw[line width=\edgewidth] (B2b)--(W3a); \draw[line width=\edgewidth] (W1a)--(B3b);
\draw[line width=\edgewidth] (W2a)--(2.4,6); \draw[line width=\edgewidth] (B1a)--(2.4,0);
\draw[line width=\edgewidth] (W2b)--(3.6,6); \draw[line width=\edgewidth] (B1b)--(3.6,0);
\draw[line width=\edgewidth] (B3a)--(0,2.295); \draw[line width=\edgewidth] (W1b)--(6,2.295);
\draw[line width=\edgewidth] (B2a)--(0,3.706); \draw[line width=\edgewidth] (W3b)--(6,3.706);

\draw [line width=\nodewidth, fill=white] (W1a) circle [radius=\noderad] ; \draw [line width=\nodewidth, fill=white] (W1b) circle [radius=\noderad] ;
\draw [line width=\nodewidth, fill=black] (B1a) circle [radius=\noderad] ; \draw [line width=\nodewidth, fill=black] (B1b) circle [radius=\noderad] ;
\draw [line width=\nodewidth, fill=white] (W2a) circle [radius=\noderad] ; \draw [line width=\nodewidth, fill=white] (W2b) circle [radius=\noderad] ;
\draw [line width=\nodewidth, fill=black] (B2a) circle [radius=\noderad] ; \draw [line width=\nodewidth, fill=black] (B2b) circle [radius=\noderad] ;
\draw [line width=\nodewidth, fill=white] (W3a) circle [radius=\noderad] ; \draw [line width=\nodewidth, fill=white] (W3b) circle [radius=\noderad] ;
\draw [line width=\nodewidth, fill=black] (B3a) circle [radius=\noderad] ; \draw [line width=\nodewidth, fill=black] (B3b) circle [radius=\noderad] ;
\draw[->, line width=0.17cm, rounded corners, color=orange] (0,6)--(B2)--(W2)--(B1)--(W1)--(6,0);
\end{tikzpicture}
} };

\end{tikzpicture}
\caption{Some zigzag paths on the dimer model shown in the middle of Figure~\ref{ex_def_zag} which can be reduced to $\nu^\zag_p(\Gamma, z)$.}
\label{fig_change_zzpath2}
\end{figure}

Then, we respectively have the PM polygons as in Figure~\ref{change_PM_zigzag}.
We see that the shape of the PM polygon $\Delta_{\nu^\zig_p(\Gamma, z)}$ (resp.~$\Delta_{\nu^\zag_p(\Gamma, z)}$) coincides with
the combinatorial mutation $\mut_w(P,F)$ (resp.~$\mut_w(P,-F)$) given in Example~\ref{ex_mutation_4b_polygon}.

\begin{figure}[h!]\centering
\begin{tikzpicture}[sarrow/.style={-latex, very thick}]
\node at (0,0)
{\scalebox{0.5}{
\begin{tikzpicture}
\coordinate (00) at (0,0); \coordinate (10) at (1,0); \coordinate (01) at (0,1); \coordinate (-10) at (-1,0); \coordinate (11) at (1,1);
\coordinate (0-1) at (0,-1);

\draw [step=1, gray] (-2.3,-2.3) grid (2.3,2.3);
\draw [line width=0.06cm] (10)--(11) ; \draw [line width=0.06cm] (11)--(-10) ;
\draw [line width=0.06cm] (-10)--(0-1) ; \draw [line width=0.06cm] (0-1)--(10) ;

\draw [line width=0.05cm, fill=black] (00) circle [radius=0.1] ; \draw [line width=0.05cm, fill=black] (10) circle [radius=0.1] ;
\draw [line width=0.05cm, fill=black] (11) circle [radius=0.1] ;
\draw [line width=0.05cm, fill=black] (-10) circle [radius=0.1] ; \draw [line width=0.05cm, fill=black] (0-1) circle [radius=0.1] ;
\draw [sarrow, line width=0.08cm, red] (-0.53,-0.53)--(-1.5,-1.5); \node[red] at (-0.6,-1.4) {\Large $[z]$};
\draw[sarrow, line width=0.08cm, blue] (-0.03,0.53)--(-1,2.5); \node[blue] at (-1,1) {\Large $[z_2]$};
\end{tikzpicture}
}};

\draw[->,line width=0.03cm] (2,0)--(4,0);
\node at (3,0.5) {$\nu^\zig_p$};

\node at (6.5,0)
{\scalebox{0.5}{
\begin{tikzpicture}
\coordinate (00) at (0,0); \coordinate (10) at (1,0); \coordinate (20) at (2,0);
\coordinate (11) at (1,1); \coordinate (2-1) at (2,-1); \coordinate (3-1) at (3,-1);
\coordinate (-10) at (-1,0);

\draw [step=1, gray] (-2.3,-2.3) grid (4.3,2.3);
\draw [line width=0.06cm] (-10)--(11) ; \draw [line width=0.06cm] (-10)--(2-1) ;
\draw [line width=0.06cm] (2-1)--(3-1) ; \draw [line width=0.06cm] (11)--(3-1) ;

\draw [line width=0.05cm, fill=black] (00) circle [radius=0.1] ; \draw [line width=0.05cm, fill=black] (10) circle [radius=0.1] ;
\draw [line width=0.05cm, fill=black] (20) circle [radius=0.1] ;
\draw [line width=0.05cm, fill=black] (11) circle [radius=0.1] ; \draw [line width=0.05cm, fill=black] (2-1) circle [radius=0.1] ;
\draw [line width=0.05cm, fill=black] (3-1) circle [radius=0.1] ; \draw [line width=0.05cm, fill=black] (-10) circle [radius=0.1] ;
\draw [sarrow, line width=0.08cm, red] (1.53,0.53)--(2.5,1.5);
\draw [sarrow, line width=0.08cm, red] (2.53,-0.47)--(3.5,0.5); \node[red] at (2.5,0.55) {\Large $-[z]$};
\draw[sarrow, line width=0.08cm, blue] (-0.03,0.53)--(-1,2.5); \node[blue] at (-1,1) {\Large $[z_2]$};
\end{tikzpicture}
}};

\node at (0,-3.5)
{\scalebox{0.5}{
\begin{tikzpicture}
\coordinate (00) at (0,0); \coordinate (10) at (1,0); \coordinate (01) at (0,1); \coordinate (-10) at (-1,0); \coordinate (11) at (1,1);
\coordinate (0-1) at (0,-1);

\draw [step=1, gray] (-2.3,-2.3) grid (2.3,2.3);
\draw [line width=0.06cm] (10)--(11) ; \draw [line width=0.06cm] (11)--(-10) ;
\draw [line width=0.06cm] (-10)--(0-1) ; \draw [line width=0.06cm] (0-1)--(10) ;

\draw [line width=0.05cm, fill=black] (00) circle [radius=0.1] ; \draw [line width=0.05cm, fill=black] (10) circle [radius=0.1] ;
\draw [line width=0.05cm, fill=black] (11) circle [radius=0.1] ;
\draw [line width=0.05cm, fill=black] (-10) circle [radius=0.1] ; \draw [line width=0.05cm, fill=black] (0-1) circle [radius=0.1] ;
\draw [sarrow, line width=0.08cm, red] (-0.53,-0.53)--(-1.5,-1.5); \node[red] at (-0.6,-1.4) {\Large $[z]$};
\draw [sarrow, line width=0.08cm, green] (1.03,0.5)--(2,0.5); \node[green] at (1.6,0.9) {\Large $[z_1]$};
\draw [sarrow, line width=0.08cm, orange] (0.53,-0.53)--(1.5,-1.5); \node[orange] at (1.4,-0.6) {\Large $[z_4]$};
\end{tikzpicture}
}};

\draw[->,line width=0.03cm] (2,-3.5)--(4,-3.5);
\node at (3,-3) {$\nu^\zag_p$};

\node at (6.5,-3.5)
{\scalebox{0.5}{
\begin{tikzpicture}
\coordinate (00) at (0,0); \coordinate (10) at (1,0); \coordinate (0-1) at (0,-1);
\coordinate (11) at (1,1); \coordinate (01) at (0,1); \coordinate (02) at (0,2);
\coordinate (-13) at (-1,3);

\draw [step=1, gray] (-2.3,-2.3) grid (2.3,4.3);
\draw [line width=0.06cm] (0-1)--(10) ; \draw [line width=0.06cm] (10)--(11) ;
\draw [line width=0.06cm] (11)--(-13) ; \draw [line width=0.06cm] (-13)--(0-1) ;

\draw [line width=0.05cm, fill=black] (00) circle [radius=0.1] ; \draw [line width=0.05cm, fill=black] (10) circle [radius=0.1] ;
\draw [line width=0.05cm, fill=black] (0-1) circle [radius=0.1] ;
\draw [line width=0.05cm, fill=black] (11) circle [radius=0.1] ; \draw [line width=0.05cm, fill=black] (01) circle [radius=0.1] ;
\draw [line width=0.05cm, fill=black] (02) circle [radius=0.1] ; \draw [line width=0.05cm, fill=black] (-13) circle [radius=0.1] ;

\draw [sarrow, line width=0.08cm, red] (-0.47,2.53)--(0.5,3.5);
\draw [sarrow, line width=0.08cm, red] (0.53,1.53)--(1.5,2.5); \node[red] at (0.5,2.5) {\Large $-[z]$};
\draw [sarrow, line width=0.08cm, green] (1.03,0.5)--(2,0.5); \node[green] at (1.6,0.9) {\Large $[z_1]$};
\draw [sarrow, line width=0.08cm, orange] (0.53,-0.53)--(1.5,-1.5); \node[orange] at (1.4,-0.6) {\Large $[z_4]$};
\end{tikzpicture}
}};
\end{tikzpicture}

\caption{The changes of the associated PM polygon via the zig-deformation (upper) and the zag-deformation (lower).}\label{change_PM_zigzag}
\end{figure}
\end{Example}

Example~\ref{ex_deformation_4b_polygon} indicates a close connection between
the PM polygon of a deformed dimer model and the combinatorial mutation of the PM polygon of the original dimer model.
In Section~\ref{mutation_vs_deformation}, we will show this phenomenon for a general situation
by considering the \emph{extended deformations} introduced in Section~\ref{sec_def_exdeform}.
In particular, as a special case of Theorem~\ref{mutation=deformation},
we have Theorem~\ref{mutation=deformation1} below (see also Remark~\ref{r=1case} and Proposition~\ref{skip_hexagonal_square}),
and this explains the phenomenon observed in Example~\ref{ex_deformation_4b_polygon}.

\begin{Theorem}\label{mutation=deformation1}
Let the notation be the same as in Definitions {\rm \ref{def_deformation_data}}, {\rm \ref{def_deformation_zig}} and {\rm \ref{def_deformation_zag}}.
We assume that either one of the following conditions is satisfied:
\begin{itemize}\itemsep=0pt
\item[$(i)$] $r=1$ or
\item[$(ii)$] $\Gamma$ is a hexagonal or rectangular dimer model $($see Definition~{\rm \ref{def_hexagonal_square})}.
\end{itemize}
Then we have
\begin{gather*}
\mutation_{w}(\Delta_\Gamma,F)=\Delta_{\nu^\zig_\bfp(\Gamma,\{z_1,\dots,z_r\})},
\\
\mutation_{w}(\Delta_\Gamma,-F)=\Delta_{\nu^\zag_\bfp(\Gamma,\{z_1,\dots,z_r\})}
\end{gather*}
for certain choices of the deformation data and the mutation data $($see Setting~{\rm \ref{mutation_deformation_setting})}.
\end{Theorem}

Combining this result with some properties of the combinatorial mutation given in Section~\ref{sec_pre_mutation_polygon},
we have the following corollary, which will be generalized in Section~\ref{mutation_vs_deformation}
(see Corollaries~\ref{cor_unimodular_deformedDM} and~\ref{cor_Fano_deformedDM}).

\begin{Corollary}Using the same setting as in Theorem~{\rm \ref{mutation=deformation1}}, we see that
\begin{itemize}\itemsep=0pt
\item $\Delta_{\nu^\zig_\bfp(\Gamma,\{z_1,\dots,z_r\})}\cong \Delta_{\nu^\zag_\bfp(\Gamma,\{z_1,\dots,z_r\})}$
\item $\Delta_\Gamma$ is Fano if and only if $\Delta_{\nu^\zig_\bfp(\Gamma,\{z_1,\dots,z_r\})}$
$($resp.\ $\Delta_{\nu^\zag_\bfp(\Gamma,\{z_1,\dots,z_r\})}$$)$ is Fano.
\end{itemize}
\end{Corollary}

\section{Extended deformations of dimer models}\label{sec_def_exdeform}

In this section, we add some procedures to the deformations given in Definitions~\ref{def_deformation_zig} and \ref{def_deformation_zag}
to construct a consistent dimer model having the same properties as Theorem~\ref{mutation=deformation1} in a more general situation.
The operations which will be introduced in this section might be complicated, thus we refer the reader to Figures~\ref{def_leftright}--\ref{fig_observation_zigzag4}
in the next section and Appendix~\ref{app_large_example} for recognizing their points.

\begin{setting}\label{def_deformation_exdata}
Let $\Gamma$ be a reduced consistent dimer model.
As in Definition~\ref{def_deformation_data}, we choose a~type~I zigzag path $z$, and we let $2n\coloneqq\ell(z)$ and $v\coloneqq [z]$.
We fix positive integers $r, h$ such that $r\le \big|\calZ_v^\rmI(\Gamma)\big|$ and $n=r+h$,
and take a subset $\{z_1,\dots,z_r\}\subset\calZ_v^\rmI(\Gamma)$ of type I zigzag paths.
We recall that we have $2n=\ell(z_1)=\cdots=\ell(z_r)$ by Lemma~\ref{lem_same_length}, and hence we write $z_i$ as
\begin{displaymath}
z_i=z_i[1]z_i[2]\cdots z_i[2n-1]z_i[2n].
\end{displaymath}

Then we prepare the additional data as follows.
\begin{itemize}\itemsep=0pt
\item[(1)] We consider all zigzag paths $x_1,\dots, x_s$ (resp.~$y_1,\dots, y_t$) intersecting with $z$ at some zags (resp.\ zigs) of $z$.
In this case, each of $x_1,\dots, x_s$ (resp.\ $y_1,\dots, y_t$) intersects with any $z_i$ at some zags (resp.\ zigs) of $z_i$ for all $i=1,\dots,r$ by Lemma~\ref{intersect_zigorzag2}.
We may assume that $z_1,\dots,z_r$ are ordered cyclically in the sense that if $x_j$ (resp.~$y_k$) intersects with $z_i$,
then it also intersects with $z_{i-1}$ (resp.~$z_{i+1}$).
\item[(2)]
We recall that $|x_j\cap z_i|$ (resp.~$|y_k\cap z_i|$) is the same number for any $i=1,\dots,r$ by Lemma~\ref{intersect_zigorzag2}.
Thus, for $j\in\{1,\dots, s\}$ let $m_j\coloneqq|x_j\cap z_i|$ be this constant, and for $k\in\{1,\dots, t\}$ let $m^\prime_k\coloneqq|y_k\cap z_i|$.
We note that $n=m_1+\cdots+m_s=m^\prime_1+\cdots+m^\prime_t$.
\item[(3)] Then, we divide each zigzag path $x_j$ into $m_j$ parts $x_j^{(1)},\dots,x_j^{(m_j)}$ as follows.
We first fix one of the intersections of $z_r$ and $x_j$ as the starting edge of $x_j^{(1)}$,
and tracing along $x_j$ we will arrive at another intersection of~$z_r$ and~$x_j$.
We consider the edge of $x_j$ just before this intersection as the ending edge of $x_j^{(1)}$, and consider the second intersection as the starting edge of $x_j^{(2)}$.
Repeating this procedure, we obtain $x_j^{(1)},\dots,x_j^{(m_j)}$ and the set of sub-zigzag paths:
\begin{equation}\label{subzigzag_x}
\big\{x_1^{(1)},\dots,x_1^{(m_1)}, x_2^{(1)},\dots,x_2^{(m_2)}, \dots, x_s^{(1)},\dots,x_s^{(m_s)}\big\}.
\end{equation}

Similarly, we also divide each zigzag path $y_k$ into $m^\prime_k$ parts $y_k^{(1)},\dots,y_k^{(m^\prime_k)}$
by considering one of the intersections of $z_1$ and $y_k$ as the starting edge of $y_k^{(1)}$, and obtain the set of sub-zigzag paths:
\begin{equation}\label{subzigzag_y}
\big\{y_1^{(1)},\dots,y_1^{(m^\prime_1)}, y_2^{(1)},\dots,y_2^{(m^\prime_2)}, \dots, y_t^{(1)},\dots,y_t^{(m^\prime_t)}\big\}.
\end{equation}
\item[(4)]
We then assign one of $\{z_1, \dots, z_r\}$ to $x_j^{(a_j)}$ for $j=1, \dots, s$ and $a_j=1, \dots, m_j$.
Then, we define $X_i$ as the set of edges consisting of the intersections between $z_i$ and the sub-zigzag paths in~(\ref{subzigzag_x}) that are assigned with~$z_i$.
We can make this assignment so that $|X_i|\ge1$, and then set $p_i\coloneqq|X_i|-1$ for $i=1,\dots,r$.
We call $\calX\coloneqq\{X_1,\dots,X_r\}$ a \emph{zig-deformation parameter} with respect to $z_1,\dots,z_r$ and call the tuple $\bfp=(p_1,\dots,p_r)\in\ZZ^r_{\ge0}$ the \emph{weight} of $\calX$.
Similarly, we assign one of $\{z_1, \dots, z_r\}$ to $y_k^{(b_k)}$ for $k=1, \dots, t$ and $b_k=1, \dots, m^\prime_k$.
Then, we define $Y_i$ as the set of edges consisting of the intersections between $z_i$ and the sub-zigzag paths in~(\ref{subzigzag_y}) that are assigned with~$z_i$.
We can make this assignment so that $|Y_i|\ge1$, and then set $q_i\coloneqq|Y_i|-1$ for $i=1,\dots,r$.
We call $\calY\coloneqq\{Y_1,\dots,Y_r\}$ a \emph{zag-deformation parameter} with respect to $z_1,\dots,z_r$ and call the tuple $\bfq=(q_1,\dots,q_r)\in\ZZ^r_{\ge0}$ the \emph{weight} of $\calY$.
We remark that $p_1+\cdots+p_r=m_1+\cdots+m_s-r=n-r=h$, and also that $q_1+\cdots+q_r=h$.
\end{itemize}
\end{setting}

\begin{Definition}[extended zig-deformation]\label{def_exdeformation_zig}
Let the notation be the same as in Setting~\ref{def_deformation_exdata}.
For a zig-deformation parameter $\calX=\{X_1,\dots,X_r\}$ of weight $\bfp=(p_1,\dots,p_r)$, we consider
the operations (zig-1), (zig-2), (zig-3) defined in Definition~\ref{def_deformation_zig} and use the same notation.
Then we conduct the following procedures:
\begin{itemize}\itemsep=0pt
\setlength{\parskip}{0pt}
\setlength{\leftskip}{0.3cm}
\item[(\text{zig}-4)]
For $m=1,\dots,n$ and $i=1, \dots, r$, if the zag $z_i[2m]$ of the original zigzag path~$z_i$ on~$\Gamma$ is not contained in~$X_i$,
then we add edges, which we call \emph{bypasses}
(since these edges provide a new route that connect edges of the zigzag path $x_j$ contained in non-deformed parts, see Figure~\ref{fig_observation_zigzag4}),
connecting the following pairs of black and white nodes:
\begin{align*}
(w_{i,1}[2m-1], &\, b_i[2m+1]), \,(w_{i,2}[2m-1],b_{i,1}[2m+1]), \\ &\dots, (w_{i,p_i}[2m-1],b_{i,p_i-1}[2m+1]),\, (w_i[2m-1],b_{i,p_i}[2m+1]).
\end{align*}
We denote the resulting dimer mode by $\overline{\nu}^\zig_\calX(\Gamma, \{z_1,\dots,z_r\})$.

We note that $\overline{\nu}^\zig_\calX(\Gamma, \{z_1,\dots,z_r\})$ is non-degenerate by Proposition~\ref{prop_nondegenerate}.
\item[(\text{zig}-5)]
Then, we make the dimer model $\overline{\nu}^\zig_\calX(\Gamma, \{z_1,\dots,z_r\})$ consistent using the method given in the proof of \cite[Theorem~1.1]{BIU} (see Operation~\ref{operation_BIU} and Proposition~\ref{prop_make_consistent}).
\end{itemize}

At the end, we perform the operation (join) if the resulting dimer model contains $2$-valent nodes.
We denote the resulting dimer model by $\nu^\zig_\calX(\Gamma, \{z_1,\dots,z_r\})$,
and call it the \emph{extended zig-deformation of $\Gamma$ at $\{z_1,\dots,z_r\}$ with respect to the zig-deformation parameter $\calX$}.
If the situation is clear, we simply denote this by $\nu^\zig_\calX(\Gamma)$.
\end{Definition}

\begin{figure}[h!]\centering\vspace*{-10mm}
\scalebox{0.55}{
\begin{tikzpicture}
\newcommand{\edgewidth}{0.05cm} 
\newcommand{\nodewidth}{0.05cm} 
\newcommand{\noderad}{0.16} 
\coordinate (W1) at (0,1.2); \coordinate (W2) at (0,3.6);
\path (W1) ++(-20:4cm) coordinate (B1); \path (W2) ++(-20:4cm) coordinate (B2);
\path (B1) ++(200:2cm) coordinate (B0); \path (W2) ++(20:2cm) coordinate (W3);

\path (W1) ++(-20:0.8cm) coordinate (B1-1); \path (W1) ++(-20:1.6cm) coordinate (W1-1);
\path (W1) ++(-20:2.4cm) coordinate (B1-2); \path (W1) ++(-20:3.2cm) coordinate (W1-2);
\path (W2) ++(-20:0.8cm) coordinate (B2-1); \path (W2) ++(-20:1.6cm) coordinate (W2-1);
\path (W2) ++(-20:2.4cm) coordinate (B2-2); \path (W2) ++(-20:3.2cm) coordinate (W2-2);

\path (B1) ++(30:1cm) coordinate (B1e); \path (B1) ++(330:1cm) coordinate (B1s);
\path (W1) ++(150:1cm) coordinate (W1w); \path (W1) ++(210:1cm) coordinate (W1s);
\path (B2) ++(30:1cm) coordinate (B2e); \path (B2) ++(330:1cm) coordinate (B2s);
\path (W2) ++(150:1cm) coordinate (W2w); \path (W2) ++(210:1cm) coordinate (W2s);

\path (W1) ++(160:0.7cm) coordinate (W1+); \path (W1) ++(200:0.7cm) coordinate (W1-);
\path (W2) ++(160:0.7cm) coordinate (W2+); \path (W2) ++(200:0.7cm) coordinate (W2-);
\path (B1) ++(20:0.7cm) coordinate (B1+); \path (B1) ++(340:0.7cm) coordinate (B1-);
\path (B2) ++(20:0.7cm) coordinate (B2+); \path (B2) ++(340:0.7cm) coordinate (B2-);

\node (original) at (0,0) {
\begin{tikzpicture}
\draw [line width=\edgewidth] (B1)--(W1); \draw [line width=\edgewidth] (B2)--(W1) ; \draw [line width=\edgewidth] (B2)--(W2) ;
\draw [line width=\edgewidth] (B1)--(B0) ; \draw [line width=\edgewidth] (W2)--(W3) ;
\draw [line width=\edgewidth] (B1)--(B1e); \draw [line width=\edgewidth] (B1)--(B1s);
\draw [line width=\edgewidth] (W1)--(W1w); \draw [line width=\edgewidth] (W1)--(W1s);
\draw [line width=\edgewidth] (B2)--(B2e); \draw [line width=\edgewidth] (B2)--(B2s);
\draw [line width=\edgewidth] (W2)--(W2w); \draw [line width=\edgewidth] (W2)--(W2s);

\draw [line width=\edgewidth,line cap=round, dash pattern=on 0pt off 2.5\pgflinewidth] (W1+)--(W1-);
\draw [line width=\edgewidth,line cap=round, dash pattern=on 0pt off 2.5\pgflinewidth] (W2+)--(W2-);
\draw [line width=\edgewidth,line cap=round, dash pattern=on 0pt off 2.5\pgflinewidth] (B1+)--(B1-);
\draw [line width=\edgewidth,line cap=round, dash pattern=on 0pt off 2.5\pgflinewidth] (B2+)--(B2-);
\draw [line width=\nodewidth, fill=black] (B1) circle [radius=\noderad] ; \draw [line width=\nodewidth, fill=black] (B2) circle [radius=\noderad] ;
\draw [line width=\nodewidth, fill=white] (W1) circle [radius=\noderad] ; \draw [line width=\nodewidth, fill=white] (W2) circle [radius=\noderad] ;

\coordinate (W1z) at (0,1.2); \coordinate (W2z) at (0,3.6);
\path (W1) ++(90:0.2cm) coordinate (W1z); \path (W2) ++(90:0.2cm) coordinate (W2z);
\path (W1z) ++(-20:4cm) coordinate (B1z); \path (W2z) ++(-20:4cm) coordinate (B2z);
\path (B1z) ++(200:2cm) coordinate (B0z); \path (W2z) ++(20:2cm) coordinate (W3z);
\path (B1z) ++(200:2cm) coordinate (B0z); \path (W2z) ++(20:2.2cm) coordinate (W3z);

\path (B2) ++(-90:0.2cm) coordinate (B2zz); \path (B2zz) ++(330:1cm) coordinate (B2szz);
\path (W1) ++(-90:0.2cm) coordinate (W1zz); \path (W1zz) ++(150:1.5cm) coordinate (W1wzz);

\draw [->, rounded corners, line width=0.08cm, blue] (B2szz)--(B2zz)--(W1zz)--(W1wzz);
\draw [->, rounded corners, line width=0.08cm, red] (B0z)--(B1z)--(W1z)--(B2z)--(W2z)--(W3z);

\node at (2.5,3.5) {\color{red}\large$z_i[2m+1]$};
\node at (1.2,2.2) {\color{red}\large$z_i[2m]$};
\node at (2.5,1.1) {\color{red}\large$z_i[2m-1]$};
\node at (-1.2,2.2) {\color{blue}\large$x_j$};
\end{tikzpicture}};

\node (before) at (10,0) {
\begin{tikzpicture}
\draw [line width=\edgewidth] (B1)--(W1); \draw [line width=\edgewidth] (B2)--(W2) ;
\draw [line width=\edgewidth] (B1)--(B1e); \draw [line width=\edgewidth] (B1)--(B1s);
\draw [line width=\edgewidth] (W1)--(W1w); \draw [line width=\edgewidth] (W1)--(W1s);
\draw [line width=\edgewidth] (B2)--(B2e); \draw [line width=\edgewidth] (B2)--(B2s);
\draw [line width=\edgewidth] (W2)--(W2w); \draw [line width=\edgewidth] (W2)--(W2s);

\draw [line width=\edgewidth,line cap=round, dash pattern=on 0pt off 2.5\pgflinewidth] (W1+)--(W1-);
\draw [line width=\edgewidth,line cap=round, dash pattern=on 0pt off 2.5\pgflinewidth] (W2+)--(W2-);
\draw [line width=\edgewidth,line cap=round, dash pattern=on 0pt off 2.5\pgflinewidth] (B1+)--(B1-);
\draw [line width=\edgewidth,line cap=round, dash pattern=on 0pt off 2.5\pgflinewidth] (B2+)--(B2-);

\draw [line width=\edgewidth] (B2-1)--(W1-1); \draw [line width=\edgewidth] (B2-2)--(W1-2);
\path (B1-1) ++(-70:1.5cm) coordinate (B1-1-); \draw [line width=\edgewidth] (B1-1)--(B1-1-);
\path (B1-2) ++(-70:1.5cm) coordinate (B1-2-); \draw [line width=\edgewidth] (B1-2)--(B1-2-);
\path (W2-1) ++(110:1.5cm) coordinate (W2-1-); \draw [line width=\edgewidth] (W2-1)--(W2-1-);
\path (W2-2) ++(110:1.5cm) coordinate (W2-2-); \draw [line width=\edgewidth] (W2-2)--(W2-2-);
\draw [line width=\nodewidth, fill=black] (B1) circle [radius=\noderad] ; \draw [line width=\nodewidth, fill=black] (B2) circle [radius=\noderad] ;
\draw [line width=\nodewidth, fill=black] (B1-1) circle [radius=\noderad] ; \draw [line width=\nodewidth, fill=black] (B1-2) circle [radius=\noderad] ;
\draw [line width=\nodewidth, fill=black] (B2-1) circle [radius=\noderad] ; \draw [line width=\nodewidth, fill=black] (B2-2) circle [radius=\noderad] ;
\draw [line width=\nodewidth, fill=white] (W1) circle [radius=\noderad] ; \draw [line width=\nodewidth, fill=white] (W2) circle [radius=\noderad] ;
\draw [line width=\nodewidth, fill=white] (W1-1) circle [radius=\noderad] ; \draw [line width=\nodewidth, fill=white] (W1-2) circle [radius=\noderad] ;
\draw [line width=\nodewidth, fill=white] (W2-1) circle [radius=\noderad] ; \draw [line width=\nodewidth, fill=white] (W2-2) circle [radius=\noderad] ;
\end{tikzpicture}};

\node (after) at (20,0) {
\begin{tikzpicture}
\draw [line width=\edgewidth] (B1)--(W1); \draw [line width=\edgewidth] (B2)--(W2) ;
\draw [line width=\edgewidth] (B1)--(B1e); \draw [line width=\edgewidth] (B1)--(B1s);
\draw [line width=\edgewidth] (W1)--(W1w); \draw [line width=\edgewidth] (W1)--(W1s);
\draw [line width=\edgewidth] (B2)--(B2e); \draw [line width=\edgewidth] (B2)--(B2s);
\draw [line width=\edgewidth] (W2)--(W2w); \draw [line width=\edgewidth] (W2)--(W2s);

\draw [line width=\edgewidth,line cap=round, dash pattern=on 0pt off 2.5\pgflinewidth] (W1+)--(W1-);
\draw [line width=\edgewidth,line cap=round, dash pattern=on 0pt off 2.5\pgflinewidth] (W2+)--(W2-);
\draw [line width=\edgewidth,line cap=round, dash pattern=on 0pt off 2.5\pgflinewidth] (B1+)--(B1-);
\draw [line width=\edgewidth,line cap=round, dash pattern=on 0pt off 2.5\pgflinewidth] (B2+)--(B2-);

\draw [line width=\edgewidth] (B2-1)--(W1-1); \draw [line width=\edgewidth] (B2-2)--(W1-2);
\path (B1-1) ++(-70:1.5cm) coordinate (B1-1-); \draw [line width=\edgewidth] (B1-1)--(B1-1-);
\path (B1-2) ++(-70:1.5cm) coordinate (B1-2-); \draw [line width=\edgewidth] (B1-2)--(B1-2-);
\path (W2-1) ++(110:1.5cm) coordinate (W2-1-); \draw [line width=\edgewidth] (W2-1)--(W2-1-);
\path (W2-2) ++(110:1.5cm) coordinate (W2-2-); \draw [line width=\edgewidth] (W2-2)--(W2-2-);

\draw [line width=\edgewidth] (B2-1)--(W1); \draw [line width=\edgewidth] (B2-2)--(W1-1); \draw [line width=\edgewidth] (B2)--(W1-2);
\draw [line width=\nodewidth, fill=black] (B1) circle [radius=\noderad] ; \draw [line width=\nodewidth, fill=black] (B2) circle [radius=\noderad] ;
\draw [line width=\nodewidth, fill=black] (B1-1) circle [radius=\noderad] ; \draw [line width=\nodewidth, fill=black] (B1-2) circle [radius=\noderad] ;
\draw [line width=\nodewidth, fill=black] (B2-1) circle [radius=\noderad] ; \draw [line width=\nodewidth, fill=black] (B2-2) circle [radius=\noderad] ;
\draw [line width=\nodewidth, fill=white] (W1) circle [radius=\noderad] ; \draw [line width=\nodewidth, fill=white] (W2) circle [radius=\noderad] ;
\draw [line width=\nodewidth, fill=white] (W1-1) circle [radius=\noderad] ; \draw [line width=\nodewidth, fill=white] (W1-2) circle [radius=\noderad] ;
\draw [line width=\nodewidth, fill=white] (W2-1) circle [radius=\noderad] ; \draw [line width=\nodewidth, fill=white] (W2-2) circle [radius=\noderad] ;
\end{tikzpicture}};

\path (original) ++(0:3.5cm) coordinate (original+); \path (before) ++(180:3.5cm) coordinate (before+);
\draw [->, line width=\edgewidth] (original+)--(before+) node[midway,xshift=-0.5cm,yshift=1.2cm] {\LARGE(zig-1)}
node[midway,xshift=0.5cm,yshift=0.5cm] {\LARGE--\,(zig-3)};

\path (before) ++(0:3.5cm) coordinate (before-); \path (after) ++(180:3.5cm) coordinate (after+);
\draw [->, line width=\edgewidth] (before-)--(after+) node[midway,xshift=0cm,yshift=0.5cm] {\LARGE(zig-4)} ;

\end{tikzpicture}
}
\caption{The operations (zig-1)--(zig-4) at $z_i$ for the case $p_i=|X_i|-1=2$ and $z_i[m]\not\in X_i$.}\label{fig_exzigdeform_p2}
\end{figure}

Similarly, we can define the ``zag version'' of this extended deformation as follows.

\begin{Definition}[extended zag-deformation]\label{def_exdeformation_zag}
Let the notation be the same as in Setting~\ref{def_deformation_exdata}.
For a zag-deformation parameter $\calY=\{Y_1,\dots,Y_r\}$ of weight $\bfq=(q_1,\dots,q_r)$, we consider
the operations (zag-1), (zag-2), (zag-3) defined in Definition~\ref{def_deformation_zag} and use the same notation.
Then we conduct the following procedures:

\begin{itemize}\itemsep=0pt
\setlength{\parskip}{0pt}
\setlength{\leftskip}{0.4cm}
\item[(\text{zag}-4)]
For $m=1,\dots,n$ and $i=1, \dots, r$, if the zig $z_i[2m-1]$ of the original zigzag path $z_i$ on~$\Gamma$ is not contained in $Y_i$,
then we add edges, which we call \emph{bypasses}, connecting the following pairs of black and white nodes:
\begin{align*}
(b_{i,1}[2m],w_i[2m+2]), \,& (b_{i,2}[2m],w_{i,1}[2m+2]), \\ &\dots, (b_{i,q_i}[2m],w_{i,q_i-1}[2m+2]), \, (b_i[2m],w_{i,q_i}[2m+2]).
\end{align*}
We denote the resulting dimer mode by $\overline{\nu}^\zag_\calY(\Gamma, \{z_1,\dots,z_r\})$.

We note that $\overline{\nu}^\zag_\calY(\Gamma, \{z_1,\dots,z_r\})$ is non-degenerate by Proposition~\ref{prop_nondegenerate}.
\item[(\text{zag}-5)]
Then, we make the dimer model $\overline{\nu}^\zag_\calY(\Gamma, \{z_1,\dots,z_r\})$ consistent using the method given in the proof of \cite[Theorem~1.1]{BIU} (see Operation~\ref{operation_BIU} and Proposition~\ref{prop_make_consistent}).
\end{itemize}

At the end, we perform the operation (join) if the resulting dimer model contains $2$-valent nodes.
We denote the resulting dimer model by $\nu^\zag_\calY(\Gamma, \{z_1,\dots,z_r\})$, and call it the \emph{extended zag-deformation of $\Gamma$ at $\{z_1,\dots,z_r\}$ with respect to the zag-deformation parameter $\calY$}.
If the situation is clear, we simply denote this by $\nu^\zag_\calY(\Gamma)$.
\end{Definition}

\begin{figure}[h!]\centering\vspace*{-8mm}
\scalebox{0.55}{
\begin{tikzpicture}
\newcommand{\edgewidth}{0.05cm} 
\newcommand{\nodewidth}{0.05cm} 
\newcommand{\noderad}{0.16} 
\coordinate (W1) at (0,0); \coordinate (W2) at (0,2.4);
\path (W1) ++(20:4cm) coordinate (B1); \path (W2) ++(20:4cm) coordinate (B2);
\path (W1) ++(340:2cm) coordinate (W0); \path (B2) ++(160:2cm) coordinate (B3);

\path (W1) ++(20:0.8cm) coordinate (B1-1); \path (W1) ++(20:1.6cm) coordinate (W1-1);
\path (W1) ++(20:2.4cm) coordinate (B1-2); \path (W1) ++(20:3.2cm) coordinate (W1-2);
\path (W2) ++(20:0.8cm) coordinate (B2-1); \path (W2) ++(20:1.6cm) coordinate (W2-1);
\path (W2) ++(20:2.4cm) coordinate (B2-2); \path (W2) ++(20:3.2cm) coordinate (W2-2);

\path (B1) ++(30:1cm) coordinate (B1e); \path (B1) ++(330:1cm) coordinate (B1s);
\path (W1) ++(150:1cm) coordinate (W1w); \path (W1) ++(210:1cm) coordinate (W1s);
\path (B2) ++(30:1cm) coordinate (B2e); \path (B2) ++(330:1cm) coordinate (B2s);
\path (W2) ++(150:1cm) coordinate (W2w); \path (W2) ++(210:1cm) coordinate (W2s);

\path (W1) ++(160:0.7cm) coordinate (W1+); \path (W1) ++(200:0.7cm) coordinate (W1-);
\path (W2) ++(160:0.7cm) coordinate (W2+); \path (W2) ++(200:0.7cm) coordinate (W2-);
\path (B1) ++(20:0.7cm) coordinate (B1+); \path (B1) ++(340:0.7cm) coordinate (B1-);
\path (B2) ++(20:0.7cm) coordinate (B2+); \path (B2) ++(340:0.7cm) coordinate (B2-);

\node (original) at (0,0) {
\begin{tikzpicture}
\draw [line width=\edgewidth] (W1)--(B1); \draw [line width=\edgewidth] (B1)--(W2) ; \draw [line width=\edgewidth] (W2)--(B2) ;
\draw [line width=\edgewidth] (W0)--(W1) ; \draw [line width=\edgewidth] (B2)--(B3) ;
\draw [line width=\edgewidth] (B1)--(B1e); \draw [line width=\edgewidth] (B1)--(B1s);
\draw [line width=\edgewidth] (W1)--(W1w); \draw [line width=\edgewidth] (W1)--(W1s);
\draw [line width=\edgewidth] (B2)--(B2e); \draw [line width=\edgewidth] (B2)--(B2s);
\draw [line width=\edgewidth] (W2)--(W2w); \draw [line width=\edgewidth] (W2)--(W2s);

\draw [line width=\edgewidth,line cap=round, dash pattern=on 0pt off 2.5\pgflinewidth] (W1+)--(W1-);
\draw [line width=\edgewidth,line cap=round, dash pattern=on 0pt off 2.5\pgflinewidth] (W2+)--(W2-);
\draw [line width=\edgewidth,line cap=round, dash pattern=on 0pt off 2.5\pgflinewidth] (B1+)--(B1-);
\draw [line width=\edgewidth,line cap=round, dash pattern=on 0pt off 2.5\pgflinewidth] (B2+)--(B2-);
\draw [line width=\nodewidth, fill=black] (B1) circle [radius=\noderad] ; \draw [line width=\nodewidth, fill=black] (B2) circle [radius=\noderad] ;
\draw [line width=\nodewidth, fill=white] (W1) circle [radius=\noderad] ; \draw [line width=\nodewidth, fill=white] (W2) circle [radius=\noderad] ;

\coordinate (W1z) at (0,1.2); \coordinate (W2z) at (0,3.6);
\path (W1) ++(90:0.2cm) coordinate (W1z); \path (W1z) ++(-20:2cm) coordinate (W0z);
\path (B1) ++(90:0.2cm) coordinate (B1z); \path (W2) ++(90:0.2cm) coordinate (W2z);
\path (B2) ++(90:0.2cm) coordinate (B2z); \path (B2z) ++(160:2.2cm) coordinate (B3z);

\path (W2) ++(-90:0.2cm) coordinate (W2zz); \path (W2zz) ++(210:1cm) coordinate (W2wzz);
\path (B1) ++(-90:0.2cm) coordinate (B1zz); \path (B1zz) ++(30:1.5cm) coordinate (B1ezz);

\draw [->, rounded corners, line width=0.08cm, blue] (W2wzz)--(W2zz)--(B1zz)--(B1ezz);
\draw [->, rounded corners, line width=0.08cm, red] (W0z)--(W1z)--(B1z)--(W2z)--(B2z)--(B3z);

\node at (1.4,3.7) {\color{red}\large$z_i[2m+2]$};
\node at (2.4,2.4) {\color{red}\large$z_i[2m+1]$};
\node at (1.4,1.2) {\color{red}\large$z_i[2m]$};
\node at (5,2.3) {\color{blue}\large$y_k$};
\end{tikzpicture}};

\node (before) at (10,0) {
\begin{tikzpicture}
\draw [line width=\edgewidth] (W1)--(B1); \draw [line width=\edgewidth] (W2)--(B2) ;
\draw [line width=\edgewidth] (B1)--(B1e); \draw [line width=\edgewidth] (B1)--(B1s);
\draw [line width=\edgewidth] (W1)--(W1w); \draw [line width=\edgewidth] (W1)--(W1s);
\draw [line width=\edgewidth] (B2)--(B2e); \draw [line width=\edgewidth] (B2)--(B2s);
\draw [line width=\edgewidth] (W2)--(W2w); \draw [line width=\edgewidth] (W2)--(W2s);

\draw [line width=\edgewidth] (B1-1)--(W2-1); \draw [line width=\edgewidth] (B1-2)--(W2-2);

\path (B2-1) ++(70:1.5cm) coordinate (B2-1-); \draw [line width=\edgewidth] (B2-1)--(B2-1-);
\path (B2-2) ++(70:1.5cm) coordinate (B2-2-); \draw [line width=\edgewidth] (B2-2)--(B2-2-);
\path (W1-1) ++(-110:1.5cm) coordinate (W1-1-); \draw [line width=\edgewidth] (W1-1)--(W1-1-);
\path (W1-2) ++(-110:1.5cm) coordinate (W1-2-); \draw [line width=\edgewidth] (W1-2)--(W1-2-);

\draw [line width=\edgewidth,line cap=round, dash pattern=on 0pt off 2.5\pgflinewidth] (W1+)--(W1-);
\draw [line width=\edgewidth,line cap=round, dash pattern=on 0pt off 2.5\pgflinewidth] (W2+)--(W2-);
\draw [line width=\edgewidth,line cap=round, dash pattern=on 0pt off 2.5\pgflinewidth] (B1+)--(B1-);
\draw [line width=\edgewidth,line cap=round, dash pattern=on 0pt off 2.5\pgflinewidth] (B2+)--(B2-);
\draw [line width=\nodewidth, fill=black] (B1) circle [radius=\noderad] ; \draw [line width=\nodewidth, fill=black] (B2) circle [radius=\noderad] ;
\draw [line width=\nodewidth, fill=black] (B1-1) circle [radius=\noderad] ; \draw [line width=\nodewidth, fill=black] (B1-2) circle [radius=\noderad] ;
\draw [line width=\nodewidth, fill=black] (B2-1) circle [radius=\noderad] ; \draw [line width=\nodewidth, fill=black] (B2-2) circle [radius=\noderad] ;
\draw [line width=\nodewidth, fill=white] (W1) circle [radius=\noderad] ; \draw [line width=\nodewidth, fill=white] (W2) circle [radius=\noderad] ;
\draw [line width=\nodewidth, fill=white] (W1-1) circle [radius=\noderad] ; \draw [line width=\nodewidth, fill=white] (W1-2) circle [radius=\noderad] ;
\draw [line width=\nodewidth, fill=white] (W2-1) circle [radius=\noderad] ; \draw [line width=\nodewidth, fill=white] (W2-2) circle [radius=\noderad] ;
\end{tikzpicture}};

\node (after) at (20,0) {
\begin{tikzpicture}
\draw [line width=\edgewidth] (W1)--(B1); \draw [line width=\edgewidth] (W2)--(B2) ;
\draw [line width=\edgewidth] (B1)--(B1e); \draw [line width=\edgewidth] (B1)--(B1s);
\draw [line width=\edgewidth] (W1)--(W1w); \draw [line width=\edgewidth] (W1)--(W1s);
\draw [line width=\edgewidth] (B2)--(B2e); \draw [line width=\edgewidth] (B2)--(B2s);
\draw [line width=\edgewidth] (W2)--(W2w); \draw [line width=\edgewidth] (W2)--(W2s);

\draw [line width=\edgewidth] (B1-1)--(W2-1); \draw [line width=\edgewidth] (B1-2)--(W2-2);
\path (B2-1) ++(70:1.5cm) coordinate (B2-1-); \draw [line width=\edgewidth] (B2-1)--(B2-1-);
\path (B2-2) ++(70:1.5cm) coordinate (B2-2-); \draw [line width=\edgewidth] (B2-2)--(B2-2-);
\path (W1-1) ++(-110:1.5cm) coordinate (W1-1-); \draw [line width=\edgewidth] (W1-1)--(W1-1-);
\path (W1-2) ++(-110:1.5cm) coordinate (W1-2-); \draw [line width=\edgewidth] (W1-2)--(W1-2-);

\draw [line width=\edgewidth] (B1-1)--(W2); \draw [line width=\edgewidth] (B1-2)--(W2-1); \draw [line width=\edgewidth] (B1)--(W2-2);

\draw [line width=\edgewidth,line cap=round, dash pattern=on 0pt off 2.5\pgflinewidth] (W1+)--(W1-);
\draw [line width=\edgewidth,line cap=round, dash pattern=on 0pt off 2.5\pgflinewidth] (W2+)--(W2-);
\draw [line width=\edgewidth,line cap=round, dash pattern=on 0pt off 2.5\pgflinewidth] (B1+)--(B1-);
\draw [line width=\edgewidth,line cap=round, dash pattern=on 0pt off 2.5\pgflinewidth] (B2+)--(B2-);
\draw [line width=\nodewidth, fill=black] (B1) circle [radius=\noderad] ; \draw [line width=\nodewidth, fill=black] (B2) circle [radius=\noderad] ;
\draw [line width=\nodewidth, fill=black] (B1-1) circle [radius=\noderad] ; \draw [line width=\nodewidth, fill=black] (B1-2) circle [radius=\noderad] ;
\draw [line width=\nodewidth, fill=black] (B2-1) circle [radius=\noderad] ; \draw [line width=\nodewidth, fill=black] (B2-2) circle [radius=\noderad] ;
\draw [line width=\nodewidth, fill=white] (W1) circle [radius=\noderad] ; \draw [line width=\nodewidth, fill=white] (W2) circle [radius=\noderad] ;
\draw [line width=\nodewidth, fill=white] (W1-1) circle [radius=\noderad] ; \draw [line width=\nodewidth, fill=white] (W1-2) circle [radius=\noderad] ;
\draw [line width=\nodewidth, fill=white] (W2-1) circle [radius=\noderad] ; \draw [line width=\nodewidth, fill=white] (W2-2) circle [radius=\noderad] ;
\end{tikzpicture}};

\path (original) ++(0:3.5cm) coordinate (original+); \path (before) ++(180:3.5cm) coordinate (before+);
\draw [->, line width=\edgewidth] (original+)--(before+) node[midway,xshift=-0.5cm,yshift=1.2cm] {\LARGE(zag-1)}
node[midway,xshift=0.5cm,yshift=0.5cm] {\LARGE --\,(zag-3)};

\path (before) ++(0:3.5cm) coordinate (before-); \path (after) ++(180:3.5cm) coordinate (after+);
\draw [->, line width=\edgewidth] (before-)--(after+) node[midway,xshift=0cm,yshift=0.5cm] {\LARGE(zag-4)} ;

\end{tikzpicture}
}
\caption{The operations (zag-1)--(zag-4) at $z_i$ for the case $q_i=|Y_i|-1=2$ and $z_i[2m+1]\not\in Y_i$.}
\label{fig_exzagdeform_p2}
\end{figure}

\begin{Remark}\label{r=1case}
If $r=1$, then we consider deformation parameters $\calX=\{X_1\}$ and $\calY=\{Y_1\}$ in Setting~\ref{def_deformation_exdata},
where~$X_1$ (resp.~$Y_1$) is the set of intersections between a chosen type~I zigzag path~$z$ and $x_1,\dots,x_s$ (resp.\ $y_1,\dots,y_t$).
In particular, $X_1$ (resp.~$Y_1$) coincides with the set of zags (resp.~zigs) of~$z$, and hence they are determined uniquely.
Thus, in this case we may skip the operations~(zig-4) (resp.~(zag-4)), in which case we may also skip (zig-5) (resp.~(zag-5)), since there are no bypasses
(see also Observations~\ref{obs_deformed_part2},~\ref{obs_deformed_part3} and Lemma~\ref{bypass_removable}).
Thus, in this case, the extended deformations coincide with the usual deformations:
\begin{displaymath}
\nu_\calX^\zig(\Gamma,\{z\})=\nu_p^\zig(\Gamma,\{z\})\qquad \text{and} \qquad
\nu_\calY^\zag(\Gamma,\{z\})=\nu_q^\zag(\Gamma,\{z\}),
\end{displaymath}
where $p\coloneqq |X_1|-1=\ell(z)/2-1$ and $q\coloneqq |Y_1|-1=\ell(z)/2-1$.
\end{Remark}

\begin{Remark}\label{rem_def_deform_not_unique}
The non-degenerate dimer model $\overline{\nu}^\zig_\calX(\Gamma, \{z_1,\dots,z_r\})$ is determined uniquely for given deformation data, but in the operation (zig-5), the procedure for removing edges is not unique.
Therefore, the resulting consistent dimer model is not unique, whereas, since the set of slopes of zigzag paths is the same for all possible consistent dimer models (see Proposition~\ref{prop_make_consistent}(2)), the associated PM polygon is the same by Proposition~\ref{char_bound}.
In addition, it is generally believed that all consistent dimer models associated with the same lattice polygon are transformed into each other by the \emph{mutations of dimer models} (see Appendix~\ref{app_mutationdimer}).
Thus, we expect that the extended deformation of a consistent dimer model is determined uniquely up to ``\emph{mutation equivalence}".
(We encounter the same situation for the extended zag-deformation.)
\end{Remark}

We observe that under the operations (zig-1)--(zig-4) (resp.~(zag-1)--(zag-4)), (self-)in\-ter\-sec\-tions of zigzag paths may appear in the universal cover.
We follow the strategy~\cite{BIU} to deal with this as we explain now.

\begin{operation}\label{operation_BIU}
We note the operation given in the proof of \cite[Theorem~1.1]{BIU}, which we use in (zig-5) and (zag-5).
\begin{itemize}\itemsep=0pt
\item[\rm (a)] The dimer model $\overline{\nu}^\zig_\calX(\Gamma, \{z_1,\dots,z_r\})$ (resp.\ $\overline{\nu}^\zag_\calY(\Gamma, \{z_1,\dots,z_r\})$), which is obtained by
applying the operations (zig-1)--(zig-4) (resp.~(zag-1)--(zag-4)) to the reduced consistent dimer model $\Gamma$, sometimes contains zigzag paths having self-intersections in the universal cover.
In this case, we use the operation given in the proof of \cite[Theorem~1.1]{BIU}; that is, we remove all edges at self-intersections (see Figure~\ref{fig_remove_self}).
We note that this operation does not change the slope of the argued zigzag path.
\begin{figure}[h!]\centering
\scalebox{0.5}{
\begin{tikzpicture}
\newcommand{\edgewidth}{0.05cm} 
\newcommand{\nodewidth}{0.05cm} 
\newcommand{\noderad}{0.16} 

\coordinate (B1) at (0,0); \coordinate (B2) at (2,3);
\coordinate (W1) at (0,2); \coordinate (W2) at (2,-1);
\path (W1) ++(135:1.2cm) coordinate (W1a); \path (B1) ++(225:1.2cm) coordinate (B1a);
\path (W2) ++(45:1.2cm) coordinate (W2a); \path (W2) ++(0:1.2cm) coordinate (W2b);
\path (B2) ++(0:1.2cm) coordinate (B2a); \path (B2) ++(315:1.2cm) coordinate (B2b);

\node (original) at (0,0) {
\begin{tikzpicture}
\draw [line width=\edgewidth] (B1)--(B1a); \draw [line width=\edgewidth] (W1)--(W1a);
\draw [line width=\edgewidth] (B1)--(W1); \draw [line width=\edgewidth] (B1)--(W2); \draw [line width=\edgewidth] (B2)--(W1);
\draw [line width=\edgewidth] (W2)--(W2a); \draw [line width=\edgewidth] (W2)--(W2b);
\draw [line width=\edgewidth] (B2)--(B2a); \draw [line width=\edgewidth] (B2)--(B2b);
\draw [line width=\nodewidth, fill=white] (W1) circle [radius=\noderad] ; \draw [line width=\nodewidth, fill=white] (W2) circle [radius=\noderad] ;
\draw [line width=\nodewidth, fill=black] (B1) circle [radius=\noderad] ; \draw [line width=\nodewidth, fill=black] (B2) circle [radius=\noderad] ;
\path (B1) ++(135:0.25cm) coordinate (B1+); \path (B1+) ++(225:1.3cm) coordinate (B1++);
\path (W1) ++(225:0.25cm) coordinate (W1+); \path (W1+) ++(135:1.3cm) coordinate (W1++);
\path (B1) ++(45:0.25cm) coordinate (B1-);
\path (W1) ++(315:0.25cm) coordinate (W1-);
\path (B2) ++(270:0.25cm) coordinate (B2+); \path (B2+) ++(0:1.6cm) coordinate (B2++);
\path (W2) ++(90:0.25cm) coordinate (W2+); \path (W2+) ++(0:1.6cm) coordinate (W2++);
\draw [->, rounded corners, line width=0.08cm, red] (B1++)--(B1+)--(W1-)--(B2+)--(B2++);
\draw [->, rounded corners, line width=0.08cm, red] (W1++)--(W1+)--(B1-)--(W2+)--(W2++);
\end{tikzpicture}};

\node (remove) at (10,0) {
\begin{tikzpicture}
\draw [line width=\edgewidth] (B1)--(B1a); \draw [line width=\edgewidth] (W1)--(W1a);
\draw [line width=\edgewidth] (B1)--(W2); \draw [line width=\edgewidth] (B2)--(W1);
\draw [line width=\edgewidth] (W2)--(W2a); \draw [line width=\edgewidth] (W2)--(W2b);
\draw [line width=\edgewidth] (B2)--(B2a); \draw [line width=\edgewidth] (B2)--(B2b);
\draw [line width=\nodewidth, fill=white] (W1) circle [radius=\noderad] ; \draw [line width=\nodewidth, fill=white] (W2) circle [radius=\noderad] ;
\draw [line width=\nodewidth, fill=black] (B1) circle [radius=\noderad] ; \draw [line width=\nodewidth, fill=black] (B2) circle [radius=\noderad] ;
\path (B2) ++(270:0.3cm) coordinate (B2+); \path (W2) ++(90:0.3cm) coordinate (W2+);
\path (B2+) ++(0:1.6cm) coordinate (B2++); \path (W2+) ++(0:1.6cm) coordinate (W2++);
\path (B1) ++(90:0.3cm) coordinate (B1n); \path (B1n) ++(225:1.5cm) coordinate (B1n+);
\path (W1) ++(270:0.3cm) coordinate (W1s); \path (W1s) ++(135:1.5cm) coordinate (W1s+);
\draw [->, rounded corners, line width=0.08cm, red] (B1n+)--(B1n)--(W2+)--(W2++);
\draw [->, rounded corners, line width=0.08cm, red] (W1s+)--(W1s)--(B2+)--(B2++);
\end{tikzpicture}};

\path (original) ++(0:4cm) coordinate (original+); \path (remove) ++(180:4cm) coordinate (remove+);
\draw [->, line width=\edgewidth] (original+)--(remove+);
\end{tikzpicture}
}
\caption{An example of removing a self-intersection of a zigzag path.}\label{fig_remove_self}
\end{figure}

After this process, there might be a connected component of the resulting bipartite graph that is contained in a simply-connected domain in $\TT$.
In that case, we remove such a~connected component.
We note that this removal does not affect our purpose, because our main concern is the PM polygon which is recovered from the slopes of zigzag paths,
and the slope of the zigzag path corresponding to the argued connected component is trivial.

\item[\rm (b)] On the other hand, the dimer model $\overline{\nu}^\zig_\calX(\Gamma, \{z_1,\dots,z_r\})$ (resp.\ $\overline{\nu}^\zag_\calY(\Gamma, \{z_1,\dots,z_r\})$)
might have a pair of zigzag paths on the universal cover that intersect with each other in the same direction more than once.
In this case, we use another operation given in the proof of \cite[Theorem~1.1]{BIU}, that is,
we choose any such pair of zigzag paths and remove pairs of consecutive intersections of this pair of zigzag paths (see Figure~\ref{fig_remove_pair}).
We note that this operation does not change the slopes of zigzag paths and the resulting bipartite graph is also a dimer model
because $\overline{\nu}^\zig_\calX(\Gamma, \{z_1,\dots,z_r\})$ (resp.\ $\overline{\nu}^\zag_\calY(\Gamma, \{z_1,\dots,z_r\})$) satisfies the strong marriage condition, as we will see in Proposition~\ref{prop_nondegenerate}.
\begin{figure}[h!]\centering
\scalebox{0.5}{
\begin{tikzpicture}
\newcommand{\edgewidth}{0.05cm} 
\newcommand{\nodewidth}{0.05cm} 
\newcommand{\noderad}{0.16} 

\coordinate (B1) at (0,0); \coordinate (B2) at (2,2); \coordinate (B3) at (4,0);
\coordinate (W1) at (0,2); \coordinate (W2) at (2,0); \coordinate (W3) at (4,2);

\path (W1) ++(135:1.2cm) coordinate (W1a); \path (B1) ++(225:1.2cm) coordinate (B1a);
\path (W3) ++(45:1.2cm) coordinate (W3a); \path (B3) ++(315:1.2cm) coordinate (B3a);

\node (original) at (0,0) {
\begin{tikzpicture}
\draw [line width=\edgewidth] (B1)--(W1)--(B2)--(W3)--(B3)--(W2)--(B1); \draw [line width=\edgewidth] (B2)--(W2);
\draw [line width=\edgewidth] (B1)--(B1a); \draw [line width=\edgewidth] (W1)--(W1a);
\draw [line width=\edgewidth] (B3)--(B3a); \draw [line width=\edgewidth] (W3)--(W3a);
\draw [line width=\nodewidth, fill=white] (W1) circle [radius=\noderad] ; \draw [line width=\nodewidth, fill=white] (W2) circle [radius=\noderad] ;
\draw [line width=\nodewidth, fill=white] (W3) circle [radius=\noderad] ;
\draw [line width=\nodewidth, fill=black] (B1) circle [radius=\noderad] ; \draw [line width=\nodewidth, fill=black] (B2) circle [radius=\noderad] ;
\draw [line width=\nodewidth, fill=black] (B3) circle [radius=\noderad] ;

\path (W1) ++(45:0.25cm) coordinate (W1+); \path (W1+) ++(135:1.3cm) coordinate (W1++);
\path (W1) ++(225:0.25cm) coordinate (W1-);
\path (W3) ++(135:0.25cm) coordinate (W3+); \path (W3+) ++(45:1.3cm) coordinate (W3++);
\path (W3) ++(315:0.25cm) coordinate (W3-);
\path (B1) ++(45:0.25cm) coordinate (B1+);
\path (B1) ++(135:0.25cm) coordinate (B1-); \path (B1-) ++(225:1.3cm) coordinate (B1--);
\path (B3) ++(135:0.25cm) coordinate (B3+);
\path (B3) ++(45:0.25cm) coordinate (B3-); \path (B3-) ++(315:1.3cm) coordinate (B3--);
\draw [->, rounded corners, line width=0.08cm, blue] (B1--)--(B1-)--(W1-)--(W3-)--(B3-)--(B3--);
\draw [->, rounded corners, line width=0.08cm, red] (W1++)--(W1+)--(B1+)--(B3+)--(W3+)--(W3++);
\end{tikzpicture}
};

\node (remove) at (10,0) {
\begin{tikzpicture}
\draw [line width=\edgewidth] (W1)--(B2)--(W3); \draw [line width=\edgewidth] (B1)--(W2)--(B3);
\draw [line width=\edgewidth] (B2)--(W2);
\draw [line width=\edgewidth] (B1)--(B1a); \draw [line width=\edgewidth] (W1)--(W1a);
\draw [line width=\edgewidth] (B3)--(B3a); \draw [line width=\edgewidth] (W3)--(W3a);
\draw [line width=\nodewidth, fill=white] (W1) circle [radius=\noderad] ; \draw [line width=\nodewidth, fill=white] (W2) circle [radius=\noderad] ;
\draw [line width=\nodewidth, fill=white] (W3) circle [radius=\noderad] ;
\draw [line width=\nodewidth, fill=black] (B1) circle [radius=\noderad] ; \draw [line width=\nodewidth, fill=black] (B2) circle [radius=\noderad] ;
\draw [line width=\nodewidth, fill=black] (B3) circle [radius=\noderad] ;
\path (B1) ++(315:0.25cm) coordinate (B1-a); \path (B1-a) ++(225:1.3cm) coordinate (B1--a);
\path (B3) ++(225:0.25cm) coordinate (B3-a); \path (B3-a) ++(315:1.3cm) coordinate (B3--a);
\draw [->, rounded corners, line width=0.08cm, red] (W1++)--(W1+)--(W3+)--(W3++);
\draw [->, rounded corners, line width=0.08cm, blue] (B1--a)--(B1-a)--(B3-a)--(B3--a);
\end{tikzpicture}
};

\path (original) ++(0:4cm) coordinate (original+); \path (remove) ++(180:4cm) coordinate (remove+);
\draw [->, line width=\edgewidth] (original+)--(remove+);
\end{tikzpicture}}
\caption{An example of removing a pair of consecutive intersections of zigzag paths.}\label{fig_remove_pair}
\end{figure}
\end{itemize}
\end{operation}

Since the dimer model $\overline{\nu}^\zig_\calX(\Gamma, \{z_1,\dots,z_r\})$ $($resp.\ $\overline{\nu}^\zag_\calY(\Gamma, \{z_1,\dots,z_r\})$$)$ is non-degenerate by Proposition~\ref{prop_nondegenerate}
and since it does not contain a homologically trivial zigzag path (see the proofs of Propositions~\ref{deform_vector}, \ref{zigzag_afterdeform1} and~\ref{zigzag_afterdeform2}),
we can produce a dimer model satisfying the conditions in Definition~\ref{def_consistent}
from $\overline{\nu}^\zig_\calX(\Gamma, \{z_1,\dots,z_r\})$ (resp.\ $\overline{\nu}^\zag_\calY(\Gamma, \{z_1,\dots,z_r\})$)
by iterated application of Operation~\ref{operation_BIU}.
Thus, $\nu^\zig_\calX(\Gamma, \{z_1,\dots,z_r\})$ and $\nu^\zag_\calY(\Gamma, \{z_1,\dots,z_r\})$ are consistent dimer models, but those are not necessarily isoradial even if $\Gamma$ is isoradial (see Example~\ref{ex_not_isoradial}).
Furthermore, since the operation (join) does not change the slopes of zigzag paths, this proves the following proposition.

\begin{Proposition}\label{prop_make_consistent}
Let the notation be the same as Definitions {\rm \ref{def_exdeformation_zig}} and {\rm \ref{def_exdeformation_zag}}.
Then, we have the following.
\begin{itemize}\itemsep=0pt
\item[\rm (1)] The dimer models $\nu^\zig_\calX(\Gamma, \{z_1,\dots,z_r\})$ and $\nu^\zag_\calY(\Gamma, \{z_1,\dots,z_r\})$ are consistent.
\item[\rm (2)] The set of slopes of zigzag paths on $\nu^\zig_\calX(\Gamma, \{z_1,\dots,z_r\})$ $($resp.\ $\nu^\zag_\calY(\Gamma, \{z_1,\dots,z_r\})$$)$
is the same as that of $\overline{\nu}^\zig_\calX(\Gamma, \{z_1,\dots,z_r\})$ $($resp.\ $\overline{\nu}^\zag_\calY(\Gamma, \{z_1,\dots,z_r\})$$)$.
\end{itemize}
\end{Proposition}

\begin{Remark}\label{rem_def_deform}
We note several additional points concerning the definition of the extended deformations.
\begin{itemize}\itemsep=0pt
\item[(1)] The join move does not change the slopes of zigzag paths, and hence it does not affect the associated PM polygon.
Thus, when we are interested in only the PM polygon, we may skip (join).
\item[(2)]
Even if we choose different sets of intersections $X_1^\prime,\dots,X_r^\prime$ (resp.~$Y_1^\prime,\dots,Y_r^\prime$) in Definition~\ref{def_deformation_data}(7),
the PM polygon of the deformed dimer model is the same as that of $\nu^\zig_\calX(\Gamma)$ (resp.~$\nu^\zag_\calY(\Gamma)$), as we will show in Proposition~\ref{other_weight_PMpolygon}.
\item[(3)] If a dimer model $\Gamma$ is hexagonal or rectangular (see Definition~\ref{def_hexagonal_square}),
we may skip the operations (zig-4), (zig-5), (zag-4) and (zag-5) when we define $\nu^\zig_\calX(\Gamma, \{z_1,\dots,z_r\})$ and $\nu^\zag_\calY(\Gamma, \{z_1,\dots,z_r\})$, as we will show in Proposition~\ref{skip_hexagonal_square}.
\end{itemize}
\end{Remark}

\section{Foundations of extended deformations of dimer models}\label{sec_proof}

In this section, we observe the fundamental properties of extended zig-deformations and zag-deformations.
In particular, we will pursue the change of zigzag paths under these deformations
in Sections~\ref{subsec_behave_zigzag} and~\ref{subsec_property_zigzag}.
We refer the reader to Appendix~\ref{app_large_example} for understanding those observation in a concrete example.

Throughout this section, we keep the notation of Sections~\ref{sec_def_deform} and~\ref{sec_def_exdeform} unless otherwise stated.

\subsection{The proof of the non-degeneracy}\label{subsec_nondegenerate_deformation}

In this subsection, we prove the non-degeneracy of the dimer models $\overline{\nu}^\zig_\calX(\Gamma, \{z_1,\dots,z_r\})$ and $\overline{\nu}^\zag_\calY(\Gamma, \{z_1,\dots,z_r\})$.

\begin{Proposition}\label{prop_nondegenerate}
The dimer models $\overline{\nu}^\zig_\calX(\Gamma, \{z_1,\dots,z_r\})$, and $\overline{\nu}^\zag_\calY(\Gamma, \{z_1,\dots,z_r\})$ are non-degenerate.
\end{Proposition}

\begin{proof}
We prove this for $\overline{\nu}^\zig_\calX(\Gamma)=\overline{\nu}^\zig_\calX(\Gamma, \{z_1,\dots,z_r\})$, the other case is similar.
Let $\Gamma^\prime$ be the dimer model obtained by applying the operations (zig-1)--(zig-3) to $\Gamma$.

{\it The first step.}
We recall that the non-degeneracy condition is equivalent to the strong marriage condition; that is, a dimer model has equal numbers of black and white nodes and every proper subset $S$ of the black nodes satisfies the condition that~$S$ is connected to at least $|S|+1$ white nodes.

Suppose that $\Gamma^\prime$ is non-degenerate. Then $\Gamma^\prime$ satisfies the strong marriage condition.
By applying the operation (zig-4), we obtain the dimer model $\overline{\nu}^\zig_\calX(\Gamma, \{z_1,\dots,z_r\})$.
Since (zig-4) is the operation that adds new edges, it also satisfies the strong marriage condition, and hence it is non-degenerate.
Therefore, it is enough to show that $\Gamma^\prime$ is non-degenerate.

{\it The second step.}
We next consider a sub-dimer motel $\Gamma^{\prime \prime}$ of $\Gamma$ satisfying the condition $(*)$ below.
Note that $\Gamma^{\prime\prime}$ is said to be a \emph{sub-dimer model} of $\Gamma$
if the set of the nodes coincide and the set of edges in $\Gamma^{\prime\prime}$ is a subset of edges in $\Gamma$.

Condition $(*)$: For any given edge $e$ in $\Gamma^{\prime\prime}$,
let $z'$ and $z''$ be the two different zigzag paths on~$\Gamma^{\prime\prime}$ containing $e$.
Then either $(*1)$ or $(*2)$ holds:
\begin{itemize}
\setlength{\parskip}{0pt}
\itemsep=0pt
\setlength{\leftskip}{0.3cm}
\item[$(*1)$] Either $[z'] \in \{[z_i],-[z_i]\}$ or $[z''] \in \{[z_i],-[z_i]\}$ holds;
\item[$(*2)$] For any $i$, if $z'$ intersects with $z_i$ in $\Zig(z_i)$ (resp.~$\Zag(z_i)$), then $z''$ intersects with $z_i$ in~$\Zag(z_i)$ (resp.~$\Zig(z_i)$).
\end{itemize}

A sub-dimer model $\Gamma^{\prime \prime}$ of $\Gamma$ satisfying $(*)$ can be constructed using
the algorithm developed in~\cite{Gul,IU2} for proving Theorem~\ref{existence_dimer}.
We adapt it to our situation.

First, let us consider the original dimer model $\Gamma$ and let $E_1,E_2,\dots,E_m$ be all edges of $\Delta_\Gamma$, where the primitive outer normal vector for $E_1$ is the slope $[z_1]=\cdots=[z_r]$.
We assume that these edges are ordered cyclically in the anti-clockwise direction (see Figure~\ref{PMpolygon_normalvec} as a reference).
Also, let $v_j$ be the primitive outer normal vector corresponding to~$E_j$, for $1 \leq j \leq m$.
Let~$a$ be the index such that $v_a=-v_1$ if there exists such an edge among $E_2,\dots,E_m$; that is, $E_a$~is parallel to~$E_1$.
If there is no such edge, then let $E_a= \varnothing$ for simplicity of notation.
We recall that by Proposition~\ref{zigzag_sidepolygon}, for each $E_j\neq\varnothing$ there exist zigzag paths on $\Gamma$ such that the associated slopes coincide with $v_j$, and the set of such zigzag paths is denoted by $\calZ_{v_j}=\calZ_{v_j}(\Gamma)$.
By our assumption, each slope of the zigzag paths in $\calZ_1\coloneqq\calZ_{v_2}\cup\cdots\cup\calZ_{v_{a-1}}$ and $v_1$ are linearly independent.
Thus, the zigzag paths in $\calZ_1$ intersect with a type~I zigzag path $z_i$ precisely once in the universal cover (see Lemma~\ref{slope_linearly_independent}).
By definition of $E_2,\dots,E_{a-1}$, such an intersection is given from the right of $z_i$ to the left of $z_i$,
and hence zigzag paths in $\calZ_1$ intersect with~$z_i$ in~$\Zag(z_i)$.
Similarly, we find that the zigzag paths in $\calZ_2\coloneqq\calZ_{v_{a+1}}\cup\cdots\cup\calZ_{v_m}$ intersect with $z_i$ precisely once in the universal cover, and specifically, they intersect with $z_i$ in $\Zig(z_i)$.

If there are at least two edges between $E_2$ and $E_{a-1}$ (i.e., $a \geq 4$),
then we take two adjacent edges, say, $E_2$ and~$E_3$.
Since $v_2$ and $v_3$ are linearly independent, the zigzag paths $z_2'\in\calZ_{v_2}$ and $z_3'\in\calZ_{v_3}$ intersect at some edge of $\Gamma$.
Clearly, such an intersection $z_2' \cap z_3'$ is neither any edge constituting any zigzag path whose slope is $[z_i]$ nor $-[z_i]$.
Now, remove an edge in $z_2' \cap z_3'$. This operation merges $z_2'$ and $z_3'$, in which case the resulting dimer model stays consistent and the associated PM polygon
changes into the polygon obtained by cutting the corner of the original PM polygon consisting of edges whose outer normal vectors are $[z_2']$ and $[z_3']$
(see \cite[Sections~5 and 6]{Gul} for more details).
Furthermore, since $z_2' \cap z_i \subset \Zag(z_i)$ and $z_3' \cap z_i \subset \Zag(z_i)$ for each~$i$, edges in $z_2' \cap z_3'$ do not share a node with $z_i$ for any $i$.
Thus, $z_i$ is still type I even if we apply this operation and the merged zigzag path intersects with $z_i$ in $\Zag(z_i)$ .
We repeat this procedure until there are no two edges between $E_2$ and $E_{a-1}$.

Similarly, if there are at least two edges between $E_{a+1}$ and $E_m$ (i.e., $m-a \geq 2$), then we do the same procedures as above until there are no two edges between~$E_{a+1}$ and~$E_m$.
After removing all suitable edges from $\Gamma$, we get a consistent dimer model, which is clearly a sub-dimer model of~$\Gamma$,
and we denote this by~$\Gamma_{\sf sub}$.
Since we do not remove edges contained in a zigzag path whose slope is~$\pm[z_i]$ in the above arguments,
the edges~$E_1$ and~$E_a$ (if this is not empty) of $\Delta_\Gamma$ are preserved on $\Delta_{\Gamma_{\sf sub}}$ (and hence we will use the same notation).
Also, the edges $E_2,\dots,E_{a-1}$ (resp.~$E_{a+1},\dots,E_m$) of $\Delta_\Gamma$ are substituted by a single edge in $\Delta_{\Gamma_{\sf sub}}$.
We denote such an edge by $E_2^\prime$ (resp.~$E_m^\prime$).
We note that zigzag paths corresponding to $E_2^\prime$ (resp.~$E_m^\prime$) are obtained by merging the ones corresponding to $E_2,\dots,E_{a-1}$ (resp.~$E_{a+1},\dots,E_m$).
In particular, the PM polygon $\Delta_{\Gamma_{\sf sub}}$ is formed by $E_1$, $E_2^\prime$, $E_a$, $E_m^\prime$, in which case $\Delta_{\Gamma_{\sf sub}}$ is a triangle or a trapezoid.
Then, it follows from the construction of $\Gamma_{\sf sub}$ that
\begin{itemize}\itemsep=0pt
\item[--] the zigzag paths $z_1,\dots,z_r$ on $\Gamma$ are preserved on $\Gamma_{\sf sub}$ and they are type~I;
\item[--] the zigzag paths on $\Gamma$ corresponding to $E_a$ (if this is not empty) are preserved on $\Gamma_{\sf sub}$;
\item[--] the zigzag paths corresponding to $E_2^\prime$ (resp.~$E_m^\prime$) are intersected with $z_i$ in $\Zag(z_i)$ (resp.\ $\Zig(z_i)$)
(see also the argument in the proof of Lemma~\ref{intersect_zigorzag}).
\end{itemize}
From these facts, it is easy to verify that $\Gamma_{\sf sub}$ is a sub-dimer model $\Gamma^{\prime\prime}$ of $\Gamma$ satisfying the condition $(*)$.

{\it The third step.}
Now, we apply the operations (zig-1)--(zig-3) in Definition \ref{def_deformation_zig} to $\Gamma_{\sf sub}$,
which is possible since $z_1,\dots,z_r$ are preserved on~$\Gamma_{\sf sub}$.
We denote the resulting dimer model by~$\Gamma^\prime_{\sf sub}$.
By construction, $\Gamma^\prime$ can be obtained by adding some edges to $\Gamma^\prime_{\sf sub}$. Thus, similar to the discussion in the first step, it is enough to show that $\Gamma^\prime_{\sf sub}$ is non-degenerate to prove the non-degeneracy of~$\Gamma^\prime$.

Thus, we now show that $\Gamma^\prime_{\sf sub}$ is non-degenerate.
To do this, we prove the existence of a~perfect matching that contains a given edge $e$ of $\Gamma^\prime_{\sf sub}$.
We divide the set of edges into four cases~(i)--(iv):
\begin{itemize}\itemsep=0pt
\item[(i)] $e$ is of the form $(b_{i,j-1}[2m-1], w_{i,j}[2m-1])$ for some $1 \leq i \leq r, 1 \leq j \leq p_i+1$ and $1 \leq m \leq n$,
where we let $b_{i,0}[2m-1]=b_i[2m-1]$ and $w_{i,p_i+1}[2m-1]=w_i[2m-1]$;
\item[(ii)] $e$ is of the form $(w_{i,j}[2m-1], b_{i,j}[2m-1])$ for some $1 \leq i \leq r, 1 \leq j \leq p_i$ and $1 \leq m \leq n$;
\item[(iii)] $e$ is of the form $(b_{i,j}[2m-1],w_{i,j}[2m-3])$ for some $1 \leq i \leq r, 1 \leq j \leq p_i$ and $1 \leq m \leq n$;
\item[(iv)] $e$ is of a form other than those in (i)--(iii).
\end{itemize}
Namely, (i) and (ii) are the edges emanating from the process (zig-1), (iii) is an edge added in the process (zig-3), and (iv) is an edge that is invariant between $\Gamma^\prime_{\sf sub}$ and $\Gamma_{\sf sub}$.

Since $\Gamma_{\sf sub}$ is consistent and contains the type I zigzag paths $z_1,\dots,z_r$,
there exist corner perfect matchings $\sfP$ and $\sfP^\prime$ on $\Gamma_{\sf sub}$ that are adjacent and satisfy $\sfP \cap z_i=\Zig(z_i)$ and $\sfP^\prime \cap z_i=\Zag(z_i)$ for $1 \leq i \leq r$ (see Section~\ref{subsec_relation_zigzag_pm}).
Similarly, let~$\sfQ$ be a corner perfect matching on $\Gamma_{\sf sub}$ with $\sfQ \cap z_i = \varnothing$ for $1 \leq i \leq r$.
The existence of such~$\sfQ$ is guaranteed by Lemma~\ref{lem_existence_pm}.
We will use these $\sfP$, $\sfP^\prime$ and $\sfQ$ in order to find a suitable perfect matching on $\Gamma^\prime_{\sf sub}$ containing a given edge $e$.
We divide our discussions into the following cases (i)--(iv) that correspond to the above division of edges.

Case (i): Let \begin{gather}\label{PM_case(i)}
\sfP^{\prime\prime}=\bigg(\sfP \setminus \bigcup_{i=1}^r \Zig(z_i)\bigg) \cup \bigg(\bigcup_{i=1}^r \bigcup_{j=1}^{p_i+1} \bigcup_{m=1}^n(b_{i,j-1}[2m-1],w_{i,j}[2m-1])\bigg);
\end{gather}
see Figure~\ref{deformation_PM_i}. It is easy to see that $\sfP^{\prime\prime}$ is a perfect matching on $\Gamma^\prime_{\sf sub}$ containing $e$ in the case~(i).

Case (ii): Let \begin{gather*}
\sfP^{\prime\prime}=\sfQ \cup \bigg(\bigcup_{i=1}^r \bigcup_{j=1}^{p_i} \bigcup_{m=1}^n(w_{i,j}[2m-1],b_{i,j}[2m-1])\bigg);
\end{gather*}
see Figure~\ref{deformation_PM_ii}. Then, we see that $\sfP^{\prime\prime}$ is a perfect matching on $\Gamma^\prime_{\sf sub}$ containing $e$ in the case~(ii).

Case (iii): Let \begin{displaymath}
\sfP^{\prime\prime}=\sfQ \cup \bigg(\bigcup_{i=1}^r \bigcup_{j=1}^{p_i} \bigcup_{m=1}^n(b_{i,j}[2m-1],w_{i,j}[2m-3])\bigg); \end{displaymath}
see Figure~\ref{deformation_PM_iii}. Then, we see that $\sfP^{\prime\prime}$ is a perfect matching on $\Gamma^\prime_{\sf sub}$ containing~$e$ in the case~(iii).

Case (iv): We note that an edge $e$ in the case (iv) also appears in $\Gamma_{\sf sub}$, since~$e$ is unchanged even if we apply (zig-1)--(zig-3).
Thus, we can regard $e$ as an edge of $\Gamma_{\sf sub}$.
Since $\Gamma_{\sf sub}$ satisfies the condition $(*)$, the zigzag paths $z^\prime$, $z^{\prime\prime}$ on~$\Gamma_{\sf sub}$ that contain $e$ satisfy
either $(*1)$ or $(*2)$.
\begin{itemize}\itemsep=0pt
\item We assume that $z^\prime$ and $z^{\prime\prime}$ satisfy $(*1)$.
\begin{itemize}\itemsep=0pt
\item Let, say, $[z^\prime]=[z_i]$.
Since zigzag paths having the same slopes are obtained as the difference of adjacent corner perfect matchings, either $\sfP$ or $\sfP^\prime$ contains $e$.
If $e \in \sfP$, then we let $\sfP^{\prime\prime}$ be as in~\eqref{PM_case(i)}.
Then, $\sfP^{\prime\prime}$ is a perfect matching on $\Gamma^\prime_{\sf sub}$ containing $e$.
Even if $e \in \sfP^\prime$, we have the same conclusion by letting
\begin{displaymath}
\sfP^{\prime\prime}=\bigg(\sfP^\prime \setminus \bigcup_{i=1}^r \Zag(z_i)\bigg) \cup \bigg(\bigcup_{i=1}^r \bigcup_{j=1}^{p_i+1} \bigcup_{m=1}^n(b_{i,j-1}[2m-1],w_{i,j}[2m-1])\bigg).
\end{displaymath}
\item Let, say, $[z^\prime]=-[z_i]$, in which case $E_a\neq \varnothing$ and $z^\prime$ corresponds to~$E_a$.
Let $\sfQ^\prime$ and $\sfQ^{\prime\prime}$ be the corner perfect matchings on~$\Gamma_{\sf sub}$ whose difference forms $z'$. Let $e \in \sfQ^\prime$.
Since $h(\sfQ^\prime,\sfP_0)$ lies on $E_a$, where $\sfP_0$ is the reference perfect matching, we have $\sfQ^\prime\cap z_i=\varnothing$ for any $i$ by Lemma~\ref{zigzag_lem1}.
Thus, we let
\begin{displaymath}
\sfP^{\prime\prime}=\sfQ^\prime \cup \bigg(\bigcup_{i=1}^r \bigcup_{j=1}^{p_i} \bigcup_{m=1}^n(w_{i,j}[2m-1],b_{i,j}[2m-1])\bigg),
\end{displaymath}
and see that $\sfP^{\prime\prime}$ is a perfect matching on $\Gamma^\prime_{\sf sub}$ containing $e$.
\end{itemize}
\item We assume that $z^\prime$ and $z^{\prime\prime}$ satisfy $(*2)$.

Let, say, $z'$ intersect with each $z_i$ in $\Zig(z_i)$. Let $\sfQ^\prime$ and $\sfQ^{\prime\prime}$ be the corner perfect matchings on $\Gamma_{\sf sub}$ whose difference forms $z'$. Let $e \in \sfQ^\prime$.
As noted above, we see that all zigzag paths in $\Gamma_{\sf sub}$ intersecting with $z_i$ at some zig of $z_i$ have the same slopes.
This implies that $\sfQ^\prime$ contains all zigs of~$z_i$, i.e., $\sfQ^\prime \cap z_i=\Zig(z_i)$. This also means that $\sfQ^\prime=\sfP$.
Hence, we let~$\sfP^{\prime\prime}$ be the same as~\eqref{PM_case(i)} and see that $\sfP^{\prime\prime}$ is a perfect matching containing~$e$.\hfill\qed
\end{itemize}
\renewcommand{\qed}{}
\end{proof}

\begin{figure}[h!]\centering\vspace*{-10mm}
\scalebox{0.55}{
\begin{tikzpicture}
\newcommand{\edgewidth}{0.05cm} 
\newcommand{\nodewidth}{0.05cm} 
\newcommand{\noderad}{0.16} 

\newcommand{\pmwidth}{0.35cm}
\newcommand{\pmcolor}{red}

\coordinate (W1) at (0,1.2); \coordinate (W2) at (0,3.6);
\path (W1) ++(-20:4cm) coordinate (B1); \path (W2) ++(-20:4cm) coordinate (B2);
\path (B1) ++(200:2cm) coordinate (B0); \path (W2) ++(20:2cm) coordinate (W3);

\path (W1) ++(-20:0.8cm) coordinate (B1-1); \path (W1) ++(-20:1.6cm) coordinate (W1-1);
\path (W1) ++(-20:2.4cm) coordinate (B1-2); \path (W1) ++(-20:3.2cm) coordinate (W1-2);
\path (W2) ++(-20:0.8cm) coordinate (B2-1); \path (W2) ++(-20:1.6cm) coordinate (W2-1);
\path (W2) ++(-20:2.4cm) coordinate (B2-2); \path (W2) ++(-20:3.2cm) coordinate (W2-2);

\path (B1) ++(40:1cm) coordinate (B1e); \path (B1) ++(320:1cm) coordinate (B1s); \path (B1) ++(20:0.9cm) coordinate (B1es);
\path (W1) ++(140:1cm) coordinate (W1w); \path (W1) ++(220:1cm) coordinate (W1s); \path (W1) ++(200:0.9cm) coordinate (W1ws);
\path (B2) ++(40:1cm) coordinate (B2e); \path (B2) ++(320:1cm) coordinate (B2s); \path (B2) ++(20:0.9cm) coordinate (B2es);
\path (W2) ++(140:1cm) coordinate (W2w); \path (W2) ++(220:1cm) coordinate (W2s); \path (W2) ++(200:0.9cm) coordinate (W2ws);

\path (W1) ++(150:0.7cm) coordinate (W1+); \path (W2) ++(150:0.7cm) coordinate (W2+);
\path (B1) ++(330:0.7cm) coordinate (B1-); \path (B2) ++(330:0.7cm) coordinate (B2-);

\node (original) at (0,0) {
\begin{tikzpicture}
\draw [\pmcolor, line width=\pmwidth] (B1)--(W1);
\draw [\pmcolor, line width=\pmwidth] (B2)--(W2);
\draw [line width=\edgewidth] (B1)--(W1); \draw [line width=\edgewidth] (B2)--(W1) ; \draw [line width=\edgewidth] (B2)--(W2) ;
\draw [line width=\edgewidth] (B1)--(B0) ; \draw [line width=\edgewidth] (W2)--(W3) ;
\draw [line width=\edgewidth] (B1)--(B1e); \draw [line width=\edgewidth] (B1)--(B1s); \draw [line width=\edgewidth] (B1)--(B1es);
\draw [line width=\edgewidth] (W1)--(W1w); \draw [line width=\edgewidth] (W1)--(W1s); \draw [line width=\edgewidth] (W1)--(W1ws);
\draw [line width=\edgewidth] (B2)--(B2e); \draw [line width=\edgewidth] (B2)--(B2s); \draw [line width=\edgewidth] (B2)--(B2es);
\draw [line width=\edgewidth] (W2)--(W2w); \draw [line width=\edgewidth] (W2)--(W2s); \draw [line width=\edgewidth] (W2)--(W2ws);

\draw [line width=\edgewidth,line cap=round, dash pattern=on 0pt off 2.5\pgflinewidth] (W1+)-- ++(-90:0.5cm);
\draw [line width=\edgewidth,line cap=round, dash pattern=on 0pt off 2.5\pgflinewidth] (W2+)-- ++(-90:0.5cm);
\draw [line width=\edgewidth,line cap=round, dash pattern=on 0pt off 2.5\pgflinewidth] (B1-)-- ++(90:0.5cm);
\draw [line width=\edgewidth,line cap=round, dash pattern=on 0pt off 2.5\pgflinewidth] (B2-)-- ++(90:0.5cm);

\draw [line width=\nodewidth, fill=black] (B1) circle [radius=\noderad] ; \draw [line width=\nodewidth, fill=black] (B2) circle [radius=\noderad] ;
\draw [line width=\nodewidth, fill=white] (W1) circle [radius=\noderad] ; \draw [line width=\nodewidth, fill=white] (W2) circle [radius=\noderad] ;

\end{tikzpicture}};

\node (predeformed) at (12,0) {
\begin{tikzpicture}
\path (B1-1) ++(-70:1.5cm) coordinate (B1-1-); \path (B1-2) ++(-70:1.5cm) coordinate (B1-2-);
\path (W2-1) ++(110:1.5cm) coordinate (W2-1-); \path (W2-2) ++(110:1.5cm) coordinate (W2-2-);

\draw [\pmcolor, line width=\pmwidth] (W1)--(B1-1); \draw [\pmcolor, line width=\pmwidth] (W1-1)--(B1-2); \draw [\pmcolor, line width=\pmwidth] (W1-2)--(B1);
\draw [\pmcolor, line width=\pmwidth] (W2)--(B2-1); \draw [\pmcolor, line width=\pmwidth] (W2-1)--(B2-2); \draw [\pmcolor, line width=\pmwidth] (W2-2)--(B2);
\draw [line width=\edgewidth] (B1)--(W1); \draw [line width=\edgewidth] (B2)--(W2) ;
\draw [line width=\edgewidth] (B1)--(B1e); \draw [line width=\edgewidth] (B1)--(B1s); \draw [line width=\edgewidth] (B1)--(B1es);
\draw [line width=\edgewidth] (W1)--(W1w); \draw [line width=\edgewidth] (W1)--(W1s); \draw [line width=\edgewidth] (W1)--(W1ws);
\draw [line width=\edgewidth] (B2)--(B2e); \draw [line width=\edgewidth] (B2)--(B2s); \draw [line width=\edgewidth] (B2)--(B2es);
\draw [line width=\edgewidth] (W2)--(W2w); \draw [line width=\edgewidth] (W2)--(W2s); \draw [line width=\edgewidth] (W2)--(W2ws);

\draw [line width=\edgewidth,line cap=round, dash pattern=on 0pt off 2.5\pgflinewidth] (W1+)-- ++(-90:0.5cm);
\draw [line width=\edgewidth,line cap=round, dash pattern=on 0pt off 2.5\pgflinewidth] (W2+)-- ++(-90:0.5cm);
\draw [line width=\edgewidth,line cap=round, dash pattern=on 0pt off 2.5\pgflinewidth] (B1-)-- ++(90:0.5cm);
\draw [line width=\edgewidth,line cap=round, dash pattern=on 0pt off 2.5\pgflinewidth] (B2-)-- ++(90:0.5cm);

\draw [line width=\edgewidth] (B2-1)--(W1-1); \draw [line width=\edgewidth] (B2-2)--(W1-2);
\draw [line width=\edgewidth] (B1-1)--(B1-1-); \draw [line width=\edgewidth] (B1-2)--(B1-2-);
\draw [line width=\edgewidth] (W2-1)--(W2-1-); \draw [line width=\edgewidth] (W2-2)--(W2-2-);

\draw [line width=\nodewidth, fill=black] (B1) circle [radius=\noderad] ; \draw [line width=\nodewidth, fill=black] (B2) circle [radius=\noderad] ;
\draw [line width=\nodewidth, fill=black] (B1-1) circle [radius=\noderad] ; \draw [line width=\nodewidth, fill=black] (B1-2) circle [radius=\noderad] ;
\draw [line width=\nodewidth, fill=black] (B2-1) circle [radius=\noderad] ; \draw [line width=\nodewidth, fill=black] (B2-2) circle [radius=\noderad] ;
\draw [line width=\nodewidth, fill=white] (W1) circle [radius=\noderad] ; \draw [line width=\nodewidth, fill=white] (W2) circle [radius=\noderad] ;
\draw [line width=\nodewidth, fill=white] (W1-1) circle [radius=\noderad] ; \draw [line width=\nodewidth, fill=white] (W1-2) circle [radius=\noderad] ;
\draw [line width=\nodewidth, fill=white] (W2-1) circle [radius=\noderad] ; \draw [line width=\nodewidth, fill=white] (W2-2) circle [radius=\noderad] ;
\end{tikzpicture}};

\path (original) ++(0:4cm) coordinate (original+); \path (predeformed) ++(180:4cm) coordinate (predeformed+);
\draw [->, line width=\edgewidth] (original+)--(predeformed+) node[midway,xshift=0cm,yshift=0.5cm] {\Large (zig-1)--(zig-3)} ;

\end{tikzpicture}
}
\caption{The perfect matchings $\sfP$ on $\Gamma_{\sf sub}$ (left) and $\sfP^{\prime\prime}$ on $\Gamma^\prime_{\sf sub}$ (right) for the case (i).}
\label{deformation_PM_i}
\end{figure}

\begin{figure}[h!]\centering
\scalebox{0.55}{
\begin{tikzpicture}
\newcommand{\edgewidth}{0.05cm} 
\newcommand{\nodewidth}{0.05cm} 
\newcommand{\noderad}{0.16} 

\newcommand{\pmwidth}{0.35cm}
\newcommand{\pmcolor}{red}

\node (original) at (0,0) {
\begin{tikzpicture}
\draw [\pmcolor, line width=\pmwidth] (B1)--(B1es); \draw [\pmcolor, line width=\pmwidth] (B2)--(B2es);
\draw [\pmcolor, line width=\pmwidth] (W1)--(W1ws); \draw [\pmcolor, line width=\pmwidth] (W2)--(W2ws);
\draw [line width=\edgewidth] (B1)--(W1); \draw [line width=\edgewidth] (B2)--(W1) ; \draw [line width=\edgewidth] (B2)--(W2) ;
\draw [line width=\edgewidth] (B1)--(B0) ; \draw [line width=\edgewidth] (W2)--(W3) ;
\draw [line width=\edgewidth] (B1)--(B1e); \draw [line width=\edgewidth] (B1)--(B1s); \draw [line width=\edgewidth] (B1)--(B1es);
\draw [line width=\edgewidth] (W1)--(W1w); \draw [line width=\edgewidth] (W1)--(W1s); \draw [line width=\edgewidth] (W1)--(W1ws);
\draw [line width=\edgewidth] (B2)--(B2e); \draw [line width=\edgewidth] (B2)--(B2s); \draw [line width=\edgewidth] (B2)--(B2es);
\draw [line width=\edgewidth] (W2)--(W2w); \draw [line width=\edgewidth] (W2)--(W2s); \draw [line width=\edgewidth] (W2)--(W2ws);

\draw [line width=\edgewidth,line cap=round, dash pattern=on 0pt off 2.5\pgflinewidth] (W1+)-- ++(-90:0.5cm);
\draw [line width=\edgewidth,line cap=round, dash pattern=on 0pt off 2.5\pgflinewidth] (W2+)-- ++(-90:0.5cm);
\draw [line width=\edgewidth,line cap=round, dash pattern=on 0pt off 2.5\pgflinewidth] (B1-)-- ++(90:0.5cm);
\draw [line width=\edgewidth,line cap=round, dash pattern=on 0pt off 2.5\pgflinewidth] (B2-)-- ++(90:0.5cm);

\draw [line width=\nodewidth, fill=black] (B1) circle [radius=\noderad] ; \draw [line width=\nodewidth, fill=black] (B2) circle [radius=\noderad] ;
\draw [line width=\nodewidth, fill=white] (W1) circle [radius=\noderad] ; \draw [line width=\nodewidth, fill=white] (W2) circle [radius=\noderad] ;
\end{tikzpicture}};

\node (predeformed) at (12,0) {
\begin{tikzpicture}
\path (B1-1) ++(-70:1.5cm) coordinate (B1-1-); \path (B1-2) ++(-70:1.5cm) coordinate (B1-2-);
\path (W2-1) ++(110:1.5cm) coordinate (W2-1-); \path (W2-2) ++(110:1.5cm) coordinate (W2-2-);

\draw [\pmcolor, line width=\pmwidth] (B1)--(B1es); \draw [\pmcolor, line width=\pmwidth] (B2)--(B2es);
\draw [\pmcolor, line width=\pmwidth] (W1)--(W1ws); \draw [\pmcolor, line width=\pmwidth] (W2)--(W2ws);
\draw [\pmcolor, line width=\pmwidth] (W1-1)--(B1-1); \draw [\pmcolor, line width=\pmwidth] (W1-2)--(B1-2);
\draw [\pmcolor, line width=\pmwidth] (W2-1)--(B2-1); \draw [\pmcolor, line width=\pmwidth] (W2-2)--(B2-2);
\draw [line width=\edgewidth] (B1)--(W1); \draw [line width=\edgewidth] (B2)--(W2) ;
\draw [line width=\edgewidth] (B1)--(B1e); \draw [line width=\edgewidth] (B1)--(B1s); \draw [line width=\edgewidth] (B1)--(B1es);
\draw [line width=\edgewidth] (W1)--(W1w); \draw [line width=\edgewidth] (W1)--(W1s); \draw [line width=\edgewidth] (W1)--(W1ws);
\draw [line width=\edgewidth] (B2)--(B2e); \draw [line width=\edgewidth] (B2)--(B2s); \draw [line width=\edgewidth] (B2)--(B2es);
\draw [line width=\edgewidth] (W2)--(W2w); \draw [line width=\edgewidth] (W2)--(W2s); \draw [line width=\edgewidth] (W2)--(W2ws);

\draw [line width=\edgewidth,line cap=round, dash pattern=on 0pt off 2.5\pgflinewidth] (W1+)-- ++(-90:0.5cm);
\draw [line width=\edgewidth,line cap=round, dash pattern=on 0pt off 2.5\pgflinewidth] (W2+)-- ++(-90:0.5cm);
\draw [line width=\edgewidth,line cap=round, dash pattern=on 0pt off 2.5\pgflinewidth] (B1-)-- ++(90:0.5cm);
\draw [line width=\edgewidth,line cap=round, dash pattern=on 0pt off 2.5\pgflinewidth] (B2-)-- ++(90:0.5cm);

\draw [line width=\edgewidth] (B2-1)--(W1-1); \draw [line width=\edgewidth] (B2-2)--(W1-2);
\draw [line width=\edgewidth] (B1-1)--(B1-1-); \draw [line width=\edgewidth] (B1-2)--(B1-2-);
\draw [line width=\edgewidth] (W2-1)--(W2-1-); \draw [line width=\edgewidth] (W2-2)--(W2-2-);

\draw [line width=\nodewidth, fill=black] (B1) circle [radius=\noderad] ; \draw [line width=\nodewidth, fill=black] (B2) circle [radius=\noderad] ;
\draw [line width=\nodewidth, fill=black] (B1-1) circle [radius=\noderad] ; \draw [line width=\nodewidth, fill=black] (B1-2) circle [radius=\noderad] ;
\draw [line width=\nodewidth, fill=black] (B2-1) circle [radius=\noderad] ; \draw [line width=\nodewidth, fill=black] (B2-2) circle [radius=\noderad] ;
\draw [line width=\nodewidth, fill=white] (W1) circle [radius=\noderad] ; \draw [line width=\nodewidth, fill=white] (W2) circle [radius=\noderad] ;
\draw [line width=\nodewidth, fill=white] (W1-1) circle [radius=\noderad] ; \draw [line width=\nodewidth, fill=white] (W1-2) circle [radius=\noderad] ;
\draw [line width=\nodewidth, fill=white] (W2-1) circle [radius=\noderad] ; \draw [line width=\nodewidth, fill=white] (W2-2) circle [radius=\noderad] ;
\end{tikzpicture}};

\path (original) ++(0:4cm) coordinate (original+); \path (predeformed) ++(180:4cm) coordinate (predeformed+);
\draw [->, line width=\edgewidth] (original+)--(predeformed+) node[midway,xshift=0cm,yshift=0.5cm] {\Large (zig-1)--(zig-3)} ;

\end{tikzpicture}
}
\caption{The perfect matchings $\sfQ$ on $\Gamma_{\sf sub}$ (left) and $\sfP^{\prime\prime}$ on $\Gamma^\prime_{\sf sub}$ (right) for the case (ii).}\label{deformation_PM_ii}
\end{figure}

\begin{figure}[h!]\centering
\scalebox{0.55}{
\begin{tikzpicture}
\newcommand{\edgewidth}{0.05cm} 
\newcommand{\nodewidth}{0.05cm} 
\newcommand{\noderad}{0.16} 

\newcommand{\pmwidth}{0.35cm}
\newcommand{\pmcolor}{red}

\node (original) at (0,0) {
\begin{tikzpicture}
\draw [\pmcolor, line width=\pmwidth] (B1)--(B1es); \draw [\pmcolor, line width=\pmwidth] (B2)--(B2es);
\draw [\pmcolor, line width=\pmwidth] (W1)--(W1ws); \draw [\pmcolor, line width=\pmwidth] (W2)--(W2ws);
\draw [line width=\edgewidth] (B1)--(W1); \draw [line width=\edgewidth] (B2)--(W1) ; \draw [line width=\edgewidth] (B2)--(W2) ;
\draw [line width=\edgewidth] (B1)--(B0) ; \draw [line width=\edgewidth] (W2)--(W3) ;
\draw [line width=\edgewidth] (B1)--(B1e); \draw [line width=\edgewidth] (B1)--(B1s); \draw [line width=\edgewidth] (B1)--(B1es);
\draw [line width=\edgewidth] (W1)--(W1w); \draw [line width=\edgewidth] (W1)--(W1s); \draw [line width=\edgewidth] (W1)--(W1ws);
\draw [line width=\edgewidth] (B2)--(B2e); \draw [line width=\edgewidth] (B2)--(B2s); \draw [line width=\edgewidth] (B2)--(B2es);
\draw [line width=\edgewidth] (W2)--(W2w); \draw [line width=\edgewidth] (W2)--(W2s); \draw [line width=\edgewidth] (W2)--(W2ws);

\draw [line width=\edgewidth,line cap=round, dash pattern=on 0pt off 2.5\pgflinewidth] (W1+)-- ++(-90:0.5cm);
\draw [line width=\edgewidth,line cap=round, dash pattern=on 0pt off 2.5\pgflinewidth] (W2+)-- ++(-90:0.5cm);
\draw [line width=\edgewidth,line cap=round, dash pattern=on 0pt off 2.5\pgflinewidth] (B1-)-- ++(90:0.5cm);
\draw [line width=\edgewidth,line cap=round, dash pattern=on 0pt off 2.5\pgflinewidth] (B2-)-- ++(90:0.5cm);

\draw [line width=\nodewidth, fill=black] (B1) circle [radius=\noderad] ; \draw [line width=\nodewidth, fill=black] (B2) circle [radius=\noderad] ;
\draw [line width=\nodewidth, fill=white] (W1) circle [radius=\noderad] ; \draw [line width=\nodewidth, fill=white] (W2) circle [radius=\noderad] ;
\end{tikzpicture}};

\node (predeformed) at (12,0) {
\begin{tikzpicture}
\path (B1-1) ++(-70:1.5cm) coordinate (B1-1-); \path (B1-2) ++(-70:1.5cm) coordinate (B1-2-);
\path (W2-1) ++(110:1.5cm) coordinate (W2-1-); \path (W2-2) ++(110:1.5cm) coordinate (W2-2-);

\draw [\pmcolor, line width=\pmwidth] (B1)--(B1es); \draw [\pmcolor, line width=\pmwidth] (B2)--(B2es);
\draw [\pmcolor, line width=\pmwidth] (W1)--(W1ws); \draw [\pmcolor, line width=\pmwidth] (W2)--(W2ws);
\draw [\pmcolor, line width=\pmwidth] (B1-1)--(B1-1-); \draw [\pmcolor, line width=\pmwidth] (B1-2)--(B1-2-);
\draw [\pmcolor, line width=\pmwidth] (B2-1)--(W1-1); \draw [\pmcolor, line width=\pmwidth] (B2-2)--(W1-2);
\draw [\pmcolor, line width=\pmwidth] (W2-1)--(W2-1-); \draw [\pmcolor, line width=\pmwidth] (W2-2)--(W2-2-);
\draw [line width=\edgewidth] (B1)--(W1); \draw [line width=\edgewidth] (B2)--(W2) ;
\draw [line width=\edgewidth] (B1)--(B1e); \draw [line width=\edgewidth] (B1)--(B1s); \draw [line width=\edgewidth] (B1)--(B1es);
\draw [line width=\edgewidth] (W1)--(W1w); \draw [line width=\edgewidth] (W1)--(W1s); \draw [line width=\edgewidth] (W1)--(W1ws);
\draw [line width=\edgewidth] (B2)--(B2e); \draw [line width=\edgewidth] (B2)--(B2s); \draw [line width=\edgewidth] (B2)--(B2es);
\draw [line width=\edgewidth] (W2)--(W2w); \draw [line width=\edgewidth] (W2)--(W2s); \draw [line width=\edgewidth] (W2)--(W2ws);

\draw [line width=\edgewidth,line cap=round, dash pattern=on 0pt off 2.5\pgflinewidth] (W1+)-- ++(-90:0.5cm);
\draw [line width=\edgewidth,line cap=round, dash pattern=on 0pt off 2.5\pgflinewidth] (W2+)-- ++(-90:0.5cm);
\draw [line width=\edgewidth,line cap=round, dash pattern=on 0pt off 2.5\pgflinewidth] (B1-)-- ++(90:0.5cm);
\draw [line width=\edgewidth,line cap=round, dash pattern=on 0pt off 2.5\pgflinewidth] (B2-)-- ++(90:0.5cm);

\draw [line width=\edgewidth] (B2-1)--(W1-1); \draw [line width=\edgewidth] (B2-2)--(W1-2);
\draw [line width=\edgewidth] (B1-1)--(B1-1-); \draw [line width=\edgewidth] (B1-2)--(B1-2-);
\draw [line width=\edgewidth] (W2-1)--(W2-1-); \draw [line width=\edgewidth] (W2-2)--(W2-2-);
\draw [line width=\nodewidth, fill=black] (B1) circle [radius=\noderad] ; \draw [line width=\nodewidth, fill=black] (B2) circle [radius=\noderad] ;
\draw [line width=\nodewidth, fill=black] (B1-1) circle [radius=\noderad] ; \draw [line width=\nodewidth, fill=black] (B1-2) circle [radius=\noderad] ;
\draw [line width=\nodewidth, fill=black] (B2-1) circle [radius=\noderad] ; \draw [line width=\nodewidth, fill=black] (B2-2) circle [radius=\noderad] ;
\draw [line width=\nodewidth, fill=white] (W1) circle [radius=\noderad] ; \draw [line width=\nodewidth, fill=white] (W2) circle [radius=\noderad] ;
\draw [line width=\nodewidth, fill=white] (W1-1) circle [radius=\noderad] ; \draw [line width=\nodewidth, fill=white] (W1-2) circle [radius=\noderad] ;
\draw [line width=\nodewidth, fill=white] (W2-1) circle [radius=\noderad] ; \draw [line width=\nodewidth, fill=white] (W2-2) circle [radius=\noderad] ;
\end{tikzpicture}};

\path (original) ++(0:4cm) coordinate (original+); \path (predeformed) ++(180:4cm) coordinate (predeformed+);
\draw [->, line width=\edgewidth] (original+)--(predeformed+) node[midway,xshift=0cm,yshift=0.5cm] {\Large (zig-1)--(zig-3)} ;

\end{tikzpicture}
}
\caption{The perfect matchings $\sfQ$ on $\Gamma_{\sf sub}$ (left) and $\sfP^{\prime\prime}$ on $\Gamma^\prime_{\sf sub}$ (right) for the case (iii).}
\label{deformation_PM_iii}
\end{figure}

\subsection{Behaviors of zigzag paths after extended deformations}\label{subsec_behave_zigzag}

In this subsection, we study zigzag paths of the deformed dimer models and their slopes.
We mainly discuss the extended zig-deformation, but the same assertions hold for the extended zag-deformation by a similar argument.
We will work with the notation in Definition~\ref{def_deformation_data} and Setting~\ref{def_deformation_exdata}.
We consider the extended deformation $\nu_\calX^\zig(\Gamma,\{z_1,\dots,z_r\})$ of $\Gamma$ (see Definitions~\ref{def_deformation_zig} and~\ref{def_exdeformation_zig}).

First, we observe zigzag paths of $\Gamma$ and fix the notation which we will use throughout this section.

\begin{observation}\label{obs_deformed_part}
Let $z_1,\dots,z_r$ be type I zigzag paths of $\Gamma$ with $[z_1]=\cdots=[z_r]$.
These zigzag paths are ordered cyclically along the subscript $i=1,\dots,r$.
For any $\alpha\in\ZZ$ and $i=1,\dots,r$, let~$\widetilde{z_i}(\alpha)$ be a zigzag path on the universal cover $\widetilde{\Gamma}$ whose projection on $\Gamma$ is~$z_i$.
Each $\widetilde{z_i}(\alpha)$ divides~$\RR^2$ into two parts, and thus it makes sense to consider the left of $\widetilde{z_i}(\alpha)$ and the right of~$\widetilde{z_i}(\alpha)$.
Then, we can write a straight line $\ell_{i,\alpha}^L$ (resp.~$\ell_{i,\alpha}^R$) on the left (resp.\ right) of $\widetilde{z_i}(\alpha)$ such that
the gradient of $\ell_{i,\alpha}^L$ (resp.~$\ell_{i,\alpha}^R$) is $v=[z_i]$ and the nodes contained in the region obtained as the intersection of the right of $\ell_{i,\alpha}^L$ and the left of $\ell_{i,\alpha}^R$ are precisely the nodes located on $\widetilde{z_i}(\alpha)$.
We will call such a region the \emph{$(i,\alpha)$-th deformed part} (see Figure~\ref{def_leftright}).

\begin{figure}[h!]\centering
{\scalebox{0.6}{
\begin{tikzpicture}[sarrow/.style={black, -latex, very thick},tarrow/.style={black, latex-, very thick}]
\newcommand{\edgewidth}{0.05cm} 
\newcommand{\nodewidth}{0.05cm} 
\newcommand{\noderad}{0.16} 

\coordinate (B1) at (0,0); \coordinate (B2) at (3,0); \coordinate (B3) at (6,0);
\coordinate (W1) at (1.5,1.5); \coordinate (W2) at (4.5,1.5); \coordinate (W3) at (7.5,1.5);

\path (B1) ++(135:1.5cm) coordinate (B1w); \path (B1) ++(225:1.2cm) coordinate (B1s); \path (B1) ++(315:1.2cm) coordinate (B1e);
\path (W1) ++(45:1.2cm) coordinate (W1w); \path (W1) ++(135:1.2cm) coordinate (W1s);
\path (B2) ++(225:1.2cm) coordinate (B2s); \path (B2) ++(315:1.2cm) coordinate (B2e);
\path (W2) ++(45:1.2cm) coordinate (W2w); \path (W2) ++(135:1.2cm) coordinate (W2s);
\path (B3) ++(225:1.2cm) coordinate (B3s); \path (B3) ++(315:1.2cm) coordinate (B3e);
\path (W3) ++(315:1.5cm) coordinate (W3n);
\path (W3) ++(45:1.2cm) coordinate (W3w); \path (W3) ++(135:1.2cm) coordinate (W3s);

\path (B1) ++(245:0.5cm) coordinate (B1ss); \path (B1) ++(295:0.5cm) coordinate (B1ee);
\path (W1) ++(65:0.5cm) coordinate (W1ww); \path (W1) ++(115:0.5cm) coordinate (W1ss);
\path (B2) ++(245:0.5cm) coordinate (B2ss); \path (B2) ++(295:0.5cm) coordinate (B2ee);
\path (W2) ++(65:0.5cm) coordinate (W2ww); \path (W2) ++(115:0.5cm) coordinate (W2ss);
\path (B3) ++(245:0.5cm) coordinate (B3ss); \path (B3) ++(295:0.5cm) coordinate (B3ee);
\path (W3) ++(65:0.5cm) coordinate (W3ww); \path (W3) ++(115:0.5cm) coordinate (W3ss);

\filldraw[blue!15, fill=blue!15] (-1.3,2.1)--(9,2.1)--(9,3.5)--(-1.3,3.5);
\filldraw[blue!15, fill=blue!15] (-1.3,-0.6)--(9,-0.6)--(9,-2)--(-1.3,-2);
\draw [line width=\edgewidth] (B1w)--(B1)--(W1)--(B2)--(W2)--(B3)--(W3)--(W3n);
\draw [line width=\edgewidth] (B1s)--(B1)--(B1e); \draw[line width=\edgewidth] (B2s)--(B2)--(B2e); \draw [line width=\edgewidth] (B3s)--(B3)--(B3e);
\draw [line width=\edgewidth] (W1w)--(W1)--(W1s); \draw[line width=\edgewidth] (W2w)--(W2)--(W2s); \draw [line width=\edgewidth] (W3w)--(W3)--(W3s);

\draw [line width=\edgewidth,line cap=round, dash pattern=on 0pt off 2.5\pgflinewidth] (B1ss)--(B1ee);
\draw [line width=\edgewidth,line cap=round, dash pattern=on 0pt off 2.5\pgflinewidth] (W1ww)--(W1ss);
\draw [line width=\edgewidth,line cap=round, dash pattern=on 0pt off 2.5\pgflinewidth] (B2ss)--(B2ee);
\draw [line width=\edgewidth,line cap=round, dash pattern=on 0pt off 2.5\pgflinewidth] (W2ww)--(W2ss);
\draw [line width=\edgewidth,line cap=round, dash pattern=on 0pt off 2.5\pgflinewidth] (B3ss)--(B3ee);
\draw [line width=\edgewidth,line cap=round, dash pattern=on 0pt off 2.5\pgflinewidth] (W3ww)--(W3ss);
\draw [line width=\nodewidth, fill=black] (B1) circle [radius=\noderad] ;
\draw [line width=\nodewidth, fill=black] (B2) circle [radius=\noderad] ;
\draw [line width=\nodewidth, fill=black] (B3) circle [radius=\noderad] ;
\draw [line width=\nodewidth, fill=white] (W1) circle [radius=\noderad] ;
\draw [line width=\nodewidth, fill=white] (W2) circle [radius=\noderad] ;
\draw [line width=\nodewidth, fill=white] (W3) circle [radius=\noderad] ;
\path (B1w) ++(-90:0.2cm) coordinate (B1w+);
\path (B1) ++(-90:0.2cm) coordinate (B1+); \path (W1) ++(-90:0.2cm) coordinate (W1+);
\path (B2) ++(-90:0.2cm) coordinate (B2+); \path (W2) ++(-90:0.2cm) coordinate (W2+);
\path (B3) ++(-90:0.2cm) coordinate (B3+); \path (W3) ++(-90:0.2cm) coordinate (W3+);
\path (W3n) ++(-90:0.2cm) coordinate (W3n+);
\draw [->, rounded corners, line width=0.08cm, red] (B1w+)--(B1+)--(W1+)--(B2+)--(W2+)--(B3+)--(W3+)--(W3n+) ;
\draw[line width=0.05cm, blue] (-1.3,2.1)--(9,2.1); \draw[line width=0.05cm, blue] (-1.3,-0.6)--(9,-0.6);
\node[red] at (9.3,0.35) {\Large$\widetilde{z_i}(\alpha)$} ;
\node[blue] at (9.7,2.1) {\Large $\ell_{i,\alpha}^L$} ; \node[blue] at (9.7,-0.6) {\Large $\ell_{i,\alpha}^R$} ;

\node at (3.8,3) {\Large the left of $\ell_{i,\alpha}^L$} ; \node at (3.8,-1.5) {\Large the right of $\ell_{i,\alpha}^R$} ;
\draw [line width=0.03cm, decorate, decoration={brace, mirror, amplitude=10pt}](-1.6,2.1) -- (-1.6,-0.6) node[black,midway,xshift=-2.5cm,yshift=0.25cm] {\Large the $(i,\alpha)$-th} node[black,midway,xshift=-2cm,yshift=-0.25cm] {\Large deformed part};
\node at (13,0) {};
\end{tikzpicture}
} }
\caption{}\label{def_leftright}
\end{figure}

Also, we call the region obtained as the intersection of the right of $\ell_{i-1,\alpha}^R$ and the left of $\ell_{i,\alpha}^L$ the \emph{$(i,\alpha)$-th irrelevant part}.
Here, we say that the intersection of the right of $\ell_{r,\alpha}^R$ and the left of $\ell_{1,\alpha+1}^L$ is the \emph{$(1,\alpha+1)$-th irrelevant part}.
We remark that sometimes there are no nodes in an irrelevant part. We sometimes omit $\alpha\in\ZZ$ from the notation unless it causes confusion.
We also use these terminologies for~$\Gamma$.
That is, a part of~$\Gamma$ obtained by projecting a deformed (resp.\ irrelevant) part of $\widetilde{\Gamma}$ onto $\Gamma$ is said to be a \emph{deformed} (resp.\ \emph{irrelevant}) \emph{part} of~$\Gamma$.

\begin{figure}[h!]\centering
{\scalebox{0.6}{
\begin{tikzpicture}
\newcommand{\edgewidth}{0.05cm} 
\newcommand{\nodewidth}{0.05cm} 
\newcommand{\noderad}{0.16} 

\filldraw[blue!15, fill=blue!15] (-3.5,-1.8)--(-0.7,-1.8)--(-0.7,9.3)--(-3.5,9.3);
\draw [line width=0.03cm, decorate, decoration={brace, amplitude=10pt}](-3.5,9.5)--(-0.7,9.5) node[black,midway,xshift=-0.5cm,yshift=1.2cm] {the $(i,\alpha)$-th} node[black,midway,xshift=0.2cm,yshift=0.7cm] {irrelevant part};

\draw [line width=0.03cm, decorate, decoration={brace, mirror, amplitude=10pt}](-0.7,-2)--(2.2,-2) node[black,midway,xshift=-0.5cm,yshift=-0.7cm] {the $(i,\alpha)$-th} node[black,midway,xshift=0.2cm,yshift=-1.2cm] {deformed part};

\filldraw[blue!15, fill=blue!15] (2.2,-1.8)--(5,-1.8)--(5,9.3)--(2.2,9.3);
\draw [line width=0.03cm, decorate, decoration={brace, amplitude=10pt}](2.2,9.5)--(5,9.5) node[black,midway,xshift=-0.2cm,yshift=1.2cm] {the $(i+1,\alpha)$-th} node[black,midway,xshift=0.2cm,yshift=0.7cm] {irrelevant part};

\draw [line width=0.03cm, decorate, decoration={brace, mirror, amplitude=10pt}](5,-2)--(7.9,-2) node[black,midway,xshift=-0.2cm,yshift=-0.7cm] {the $(i+1,\alpha)$-th} node[black,midway,xshift=0.2cm,yshift=-1.2cm] {deformed part};

\filldraw[blue!15, fill=blue!15] (7.9,-1.8)--(10.7,-1.8)--(10.7,9.3)--(7.9,9.3);
\draw [line width=0.03cm, decorate, decoration={brace, amplitude=10pt}](7.9,9.5)--(10.7,9.5) node[black,midway,xshift=-0.2cm,yshift=1.2cm] {the $(i+2,\alpha)$-th} node[black,midway,xshift=0.2cm,yshift=0.7cm] {irrelevant part};

\draw [line width=0.03cm, decorate, decoration={brace, mirror, amplitude=10pt}](10.7,-2)--(13.6,-2) node[black,midway,xshift=-0.2cm,yshift=-0.7cm] {the $(i+2,\alpha)$-th} node[black,midway,xshift=0.2cm,yshift=-1.2cm] {deformed part};

\filldraw[blue!15, fill=blue!15] (13.6,-1.8)--(16.4,-1.8)--(16.4,9.3)--(13.6,9.3);
\draw [line width=0.03cm, decorate, decoration={brace, amplitude=10pt}](13.6,9.5)--(16.4,9.5) node[black,midway,xshift=-0.2cm,yshift=1.2cm] {the $(i+3,\alpha)$-th} node[black,midway,xshift=0.2cm,yshift=0.7cm] {irrelevant part};

\draw [line width=\edgewidth,line cap=round, dash pattern=on 0pt off 2.5\pgflinewidth] (-4,3.75)--(-5,3.75);
\draw [line width=\edgewidth,line cap=round, dash pattern=on 0pt off 2.5\pgflinewidth] (16.9,3.75)--(17.9,3.75);

\coordinate (W1) at (0,0); \coordinate (W2) at (0,3); \coordinate (W3) at (0,6);
\coordinate (B1) at (1.5,1.5); \coordinate (B2) at (1.5,4.5); \coordinate (B3) at (1.5,7.5);

\path (W1) ++(135:1.4cm) coordinate (W1w); \path (W1) ++(225:1.4cm) coordinate (W1s); \path (W1) ++(315:1.4cm) coordinate (W1e);
\path (B1) ++(45:1.4cm) coordinate (B1e); \path (B1) ++(315:1.4cm) coordinate (B1s);
\path (W2) ++(135:1.4cm) coordinate (W2w); \path (W2) ++(225:1.4cm) coordinate (W2s);
\path (B2) ++(45:1.4cm) coordinate (B2e); \path (B2) ++(315:1.4cm) coordinate (B2s);
\path (W3) ++(135:1.4cm) coordinate (W3w); \path (W3) ++(225:1.4cm) coordinate (W3s);
\path (B3) ++(45:1.4cm) coordinate (B3e); \path (B3) ++(315:1.4cm) coordinate (B3s); \path (B3) ++(135:1.4cm) coordinate (B3n);

\path (W1) ++(155:0.6cm) coordinate (W1+); \path (W1) ++(205:0.6cm) coordinate (W1-);
\path (W2) ++(155:0.6cm) coordinate (W2+); \path (W2) ++(205:0.6cm) coordinate (W2-);
\path (W3) ++(155:0.6cm) coordinate (W3+); \path (W3) ++(205:0.6cm) coordinate (W3-);
\path (B1) ++(25:0.6cm) coordinate (B1+); \path (B1) ++(335:0.6cm) coordinate (B1-);
\path (B2) ++(25:0.6cm) coordinate (B2+); \path (B2) ++(335:0.6cm) coordinate (B2-);
\path (B3) ++(25:0.6cm) coordinate (B3+); \path (B3) ++(335:0.6cm) coordinate (B3-);
\draw [line width=\edgewidth] (W1)--(B1)--(W2)--(B2)--(W3)--(B3) ;

\draw [line width=\edgewidth] (W1)--(W1w); \draw [line width=\edgewidth] (W1)--(W1s); \draw [line width=\edgewidth] (W1)--(W1e);
\draw [line width=\edgewidth] (B1)--(B1e); \draw [line width=\edgewidth] (B1)--(B1s);
\draw [line width=\edgewidth] (W2)--(W2w); \draw [line width=\edgewidth] (W2)--(W2s);
\draw [line width=\edgewidth] (B2)--(B2e); \draw [line width=\edgewidth] (B2)--(B2s);
\draw [line width=\edgewidth] (W3)--(W3w); \draw [line width=\edgewidth] (W3)--(W3s);
\draw [line width=\edgewidth] (B3)--(B3e); \draw [line width=\edgewidth] (B3)--(B3s); \draw [line width=\edgewidth] (B3)--(B3n);

\draw [line width=\edgewidth,line cap=round, dash pattern=on 0pt off 2.5\pgflinewidth] (W1+)--(W1-);
\draw [line width=\edgewidth,line cap=round, dash pattern=on 0pt off 2.5\pgflinewidth] (W2+)--(W2-);
\draw [line width=\edgewidth,line cap=round, dash pattern=on 0pt off 2.5\pgflinewidth] (W3+)--(W3-);
\draw [line width=\edgewidth,line cap=round, dash pattern=on 0pt off 2.5\pgflinewidth] (B1+)--(B1-);
\draw [line width=\edgewidth,line cap=round, dash pattern=on 0pt off 2.5\pgflinewidth] (B2+)--(B2-);
\draw [line width=\edgewidth,line cap=round, dash pattern=on 0pt off 2.5\pgflinewidth] (B3+)--(B3-);

\path (W1e) ++(-90:0.3cm) coordinate (W1e_dot); \path (W1e) ++(-90:0.9cm) coordinate (W1e_ddot);
\draw [line width=\edgewidth,line cap=round, dash pattern=on 0pt off 2.5\pgflinewidth] (W1e_dot)--(W1e_ddot);
\path (B3n) ++(90:0.3cm) coordinate (B3n_dot); \path (B3n) ++(90:0.9cm) coordinate (B3n_ddot);
\draw [line width=\edgewidth,line cap=round, dash pattern=on 0pt off 2.5\pgflinewidth] (B3n_dot)--(B3n_ddot);

\draw [line width=\nodewidth, fill=black] (B1) circle [radius=\noderad] ;
\draw [line width=\nodewidth, fill=black] (B2) circle [radius=\noderad] ;
\draw [line width=\nodewidth, fill=black] (B3) circle [radius=\noderad] ;
\draw [line width=\nodewidth, fill=white] (W1) circle [radius=\noderad] ;
\draw [line width=\nodewidth, fill=white] (W2) circle [radius=\noderad] ;
\draw [line width=\nodewidth, fill=white] (W3) circle [radius=\noderad] ;
\path (W1e) ++(-90:0.25cm) coordinate (W1ez);
\path (B1) ++(-90:0.25cm) coordinate (B1z); \path (W1) ++(-90:0.25cm) coordinate (W1z);
\path (B2) ++(-90:0.25cm) coordinate (B2z); \path (W2) ++(-90:0.25cm) coordinate (W2z);
\path (B3) ++(-90:0.25cm) coordinate (B3z); \path (W3) ++(-90:0.25cm) coordinate (W3z);
\path (B3) ++(135:1.8cm) coordinate (B3nn); \path (B3nn) ++(-90:0.25cm) coordinate (B3nz);
\draw [->, rounded corners, line width=0.08cm, red] (W1ez)--(W1z)--(B1z)--(W2z)--(B2z)--(W3z)--(B3z)--(B3nz) ;
\path (B3nz) ++(-90:1cm) coordinate (zi);\node[red] at (zi) {$\widetilde{z_i}(\alpha)$} ;

\coordinate (W1) at (5.7,0); \coordinate (W2) at (5.7,3); \coordinate (W3) at (5.7,6);
\coordinate (B1) at (7.2,1.5); \coordinate (B2) at (7.2,4.5); \coordinate (B3) at (7.2,7.5);

\path (W1) ++(135:1.4cm) coordinate (W1w); \path (W1) ++(225:1.4cm) coordinate (W1s); \path (W1) ++(315:1.4cm) coordinate (W1e);
\path (B1) ++(45:1.4cm) coordinate (B1e); \path (B1) ++(315:1.4cm) coordinate (B1s);
\path (W2) ++(135:1.4cm) coordinate (W2w); \path (W2) ++(225:1.4cm) coordinate (W2s);
\path (B2) ++(45:1.4cm) coordinate (B2e); \path (B2) ++(315:1.4cm) coordinate (B2s);
\path (W3) ++(135:1.4cm) coordinate (W3w); \path (W3) ++(225:1.4cm) coordinate (W3s);
\path (B3) ++(45:1.4cm) coordinate (B3e); \path (B3) ++(315:1.4cm) coordinate (B3s); \path (B3) ++(135:1.4cm) coordinate (B3n);

\path (W1) ++(155:0.6cm) coordinate (W1+); \path (W1) ++(205:0.6cm) coordinate (W1-);
\path (W2) ++(155:0.6cm) coordinate (W2+); \path (W2) ++(205:0.6cm) coordinate (W2-);
\path (W3) ++(155:0.6cm) coordinate (W3+); \path (W3) ++(205:0.6cm) coordinate (W3-);
\path (B1) ++(25:0.6cm) coordinate (B1+); \path (B1) ++(335:0.6cm) coordinate (B1-);
\path (B2) ++(25:0.6cm) coordinate (B2+); \path (B2) ++(335:0.6cm) coordinate (B2-);
\path (B3) ++(25:0.6cm) coordinate (B3+); \path (B3) ++(335:0.6cm) coordinate (B3-);
\draw [line width=\edgewidth] (W1)--(B1)--(W2)--(B2)--(W3)--(B3) ;

\draw [line width=\edgewidth] (W1)--(W1w); \draw [line width=\edgewidth] (W1)--(W1s); \draw [line width=\edgewidth] (W1)--(W1e);
\draw [line width=\edgewidth] (B1)--(B1e); \draw [line width=\edgewidth] (B1)--(B1s);
\draw [line width=\edgewidth] (W2)--(W2w); \draw [line width=\edgewidth] (W2)--(W2s);
\draw [line width=\edgewidth] (B2)--(B2e); \draw [line width=\edgewidth] (B2)--(B2s);
\draw [line width=\edgewidth] (W3)--(W3w); \draw [line width=\edgewidth] (W3)--(W3s);
\draw [line width=\edgewidth] (B3)--(B3e); \draw [line width=\edgewidth] (B3)--(B3s); \draw [line width=\edgewidth] (B3)--(B3n);

\draw [line width=\edgewidth,line cap=round, dash pattern=on 0pt off 2.5\pgflinewidth] (W1+)--(W1-);
\draw [line width=\edgewidth,line cap=round, dash pattern=on 0pt off 2.5\pgflinewidth] (W2+)--(W2-);
\draw [line width=\edgewidth,line cap=round, dash pattern=on 0pt off 2.5\pgflinewidth] (W3+)--(W3-);
\draw [line width=\edgewidth,line cap=round, dash pattern=on 0pt off 2.5\pgflinewidth] (B1+)--(B1-);
\draw [line width=\edgewidth,line cap=round, dash pattern=on 0pt off 2.5\pgflinewidth] (B2+)--(B2-);
\draw [line width=\edgewidth,line cap=round, dash pattern=on 0pt off 2.5\pgflinewidth] (B3+)--(B3-);

\path (W1e) ++(-90:0.3cm) coordinate (W1e_dot); \path (W1e) ++(-90:0.9cm) coordinate (W1e_ddot);
\draw [line width=\edgewidth,line cap=round, dash pattern=on 0pt off 2.5\pgflinewidth] (W1e_dot)--(W1e_ddot);
\path (B3n) ++(90:0.3cm) coordinate (B3n_dot); \path (B3n) ++(90:0.9cm) coordinate (B3n_ddot);
\draw [line width=\edgewidth,line cap=round, dash pattern=on 0pt off 2.5\pgflinewidth] (B3n_dot)--(B3n_ddot);

\draw [line width=\nodewidth, fill=black] (B1) circle [radius=\noderad] ;
\draw [line width=\nodewidth, fill=black] (B2) circle [radius=\noderad] ;
\draw [line width=\nodewidth, fill=black] (B3) circle [radius=\noderad] ;
\draw [line width=\nodewidth, fill=white] (W1) circle [radius=\noderad] ;
\draw [line width=\nodewidth, fill=white] (W2) circle [radius=\noderad] ;
\draw [line width=\nodewidth, fill=white] (W3) circle [radius=\noderad] ;
\path (W1e) ++(-90:0.25cm) coordinate (W1ez);
\path (B1) ++(-90:0.25cm) coordinate (B1z); \path (W1) ++(-90:0.25cm) coordinate (W1z);
\path (B2) ++(-90:0.25cm) coordinate (B2z); \path (W2) ++(-90:0.25cm) coordinate (W2z);
\path (B3) ++(-90:0.25cm) coordinate (B3z); \path (W3) ++(-90:0.25cm) coordinate (W3z);
\path (B3) ++(135:1.8cm) coordinate (B3nn); \path (B3nn) ++(-90:0.25cm) coordinate (B3nz);
\draw [->, rounded corners, line width=0.08cm, red] (W1ez)--(W1z)--(B1z)--(W2z)--(B2z)--(W3z)--(B3z)--(B3nz) ;
\path (B3nz) ++(-90:1cm) coordinate (zi);\node[red] at (zi) {$\widetilde{z}_{i+1}(\alpha)$} ;

\coordinate (W1) at (11.4,0); \coordinate (W2) at (11.4,3); \coordinate (W3) at (11.4,6);
\coordinate (B1) at (12.9,1.5); \coordinate (B2) at (12.9,4.5); \coordinate (B3) at (12.9,7.5);

\path (W1) ++(135:1.4cm) coordinate (W1w); \path (W1) ++(225:1.4cm) coordinate (W1s); \path (W1) ++(315:1.4cm) coordinate (W1e);
\path (B1) ++(45:1.4cm) coordinate (B1e); \path (B1) ++(315:1.4cm) coordinate (B1s);
\path (W2) ++(135:1.4cm) coordinate (W2w); \path (W2) ++(225:1.4cm) coordinate (W2s);
\path (B2) ++(45:1.4cm) coordinate (B2e); \path (B2) ++(315:1.4cm) coordinate (B2s);
\path (W3) ++(135:1.4cm) coordinate (W3w); \path (W3) ++(225:1.4cm) coordinate (W3s);
\path (B3) ++(45:1.4cm) coordinate (B3e); \path (B3) ++(315:1.4cm) coordinate (B3s); \path (B3) ++(135:1.4cm) coordinate (B3n);

\path (W1) ++(155:0.6cm) coordinate (W1+); \path (W1) ++(205:0.6cm) coordinate (W1-);
\path (W2) ++(155:0.6cm) coordinate (W2+); \path (W2) ++(205:0.6cm) coordinate (W2-);
\path (W3) ++(155:0.6cm) coordinate (W3+); \path (W3) ++(205:0.6cm) coordinate (W3-);
\path (B1) ++(25:0.6cm) coordinate (B1+); \path (B1) ++(335:0.6cm) coordinate (B1-);
\path (B2) ++(25:0.6cm) coordinate (B2+); \path (B2) ++(335:0.6cm) coordinate (B2-);
\path (B3) ++(25:0.6cm) coordinate (B3+); \path (B3) ++(335:0.6cm) coordinate (B3-);
\draw [line width=\edgewidth] (W1)--(B1)--(W2)--(B2)--(W3)--(B3) ;

\draw [line width=\edgewidth] (W1)--(W1w); \draw [line width=\edgewidth] (W1)--(W1s); \draw [line width=\edgewidth] (W1)--(W1e);
\draw [line width=\edgewidth] (B1)--(B1e); \draw [line width=\edgewidth] (B1)--(B1s);
\draw [line width=\edgewidth] (W2)--(W2w); \draw [line width=\edgewidth] (W2)--(W2s);
\draw [line width=\edgewidth] (B2)--(B2e); \draw [line width=\edgewidth] (B2)--(B2s);
\draw [line width=\edgewidth] (W3)--(W3w); \draw [line width=\edgewidth] (W3)--(W3s);
\draw [line width=\edgewidth] (B3)--(B3e); \draw [line width=\edgewidth] (B3)--(B3s); \draw [line width=\edgewidth] (B3)--(B3n);

\draw [line width=\edgewidth,line cap=round, dash pattern=on 0pt off 2.5\pgflinewidth] (W1+)--(W1-);
\draw [line width=\edgewidth,line cap=round, dash pattern=on 0pt off 2.5\pgflinewidth] (W2+)--(W2-);
\draw [line width=\edgewidth,line cap=round, dash pattern=on 0pt off 2.5\pgflinewidth] (W3+)--(W3-);
\draw [line width=\edgewidth,line cap=round, dash pattern=on 0pt off 2.5\pgflinewidth] (B1+)--(B1-);
\draw [line width=\edgewidth,line cap=round, dash pattern=on 0pt off 2.5\pgflinewidth] (B2+)--(B2-);
\draw [line width=\edgewidth,line cap=round, dash pattern=on 0pt off 2.5\pgflinewidth] (B3+)--(B3-);

\path (W1e) ++(-90:0.3cm) coordinate (W1e_dot); \path (W1e) ++(-90:0.9cm) coordinate (W1e_ddot);
\draw [line width=\edgewidth,line cap=round, dash pattern=on 0pt off 2.5\pgflinewidth] (W1e_dot)--(W1e_ddot);
\path (B3n) ++(90:0.3cm) coordinate (B3n_dot); \path (B3n) ++(90:0.9cm) coordinate (B3n_ddot);
\draw [line width=\edgewidth,line cap=round, dash pattern=on 0pt off 2.5\pgflinewidth] (B3n_dot)--(B3n_ddot);

\draw [line width=\nodewidth, fill=black] (B1) circle [radius=\noderad] ;
\draw [line width=\nodewidth, fill=black] (B2) circle [radius=\noderad] ;
\draw [line width=\nodewidth, fill=black] (B3) circle [radius=\noderad] ;
\draw [line width=\nodewidth, fill=white] (W1) circle [radius=\noderad] ;
\draw [line width=\nodewidth, fill=white] (W2) circle [radius=\noderad] ;
\draw [line width=\nodewidth, fill=white] (W3) circle [radius=\noderad] ;
\path (W1e) ++(-90:0.25cm) coordinate (W1ez);
\path (B1) ++(-90:0.25cm) coordinate (B1z); \path (W1) ++(-90:0.25cm) coordinate (W1z);
\path (B2) ++(-90:0.25cm) coordinate (B2z); \path (W2) ++(-90:0.25cm) coordinate (W2z);
\path (B3) ++(-90:0.25cm) coordinate (B3z); \path (W3) ++(-90:0.25cm) coordinate (W3z);
\path (B3) ++(135:1.8cm) coordinate (B3nn); \path (B3nn) ++(-90:0.25cm) coordinate (B3nz);
\draw [->, rounded corners, line width=0.08cm, red] (W1ez)--(W1z)--(B1z)--(W2z)--(B2z)--(W3z)--(B3z)--(B3nz) ;
\path (B3nz) ++(-90:1cm) coordinate (zi); \node[red] at (zi) {$\widetilde{z}_{i+2}(\alpha)$} ;
\end{tikzpicture}
} }
\caption{}
\label{fig_deformed_irrelevant}
\end{figure}

By the condition of Definition~\ref{def_properly}(3), $z_1,\dots,z_r$ do not have a common node, and therefore the irrelevant parts do not overlap each other.
Since the operations (zig-1)--(zig-4) (or (zag-1)--(zag-4)) are local operations on each deformed part, any irrelevant part will be unchanged even if we apply these operations.
Thus, we continue to use these terminologies ``deformed parts’’ and ``irrelevant parts’’.
We then consider a zigzag path $w$ satisfying the following properties:
\begin{itemize}\itemsep=0pt
\item[(a)] If $[z_i]$ and $[w]$ are linearly independent, then by Lemma~\ref{slope_linearly_independent}, $\widetilde{w}$ intersects with $\widetilde{z_i}(\alpha)$ precisely once, and so does any $\widetilde{z_i}(\alpha)$ with $i=1,\dots,r$ and $\alpha\in\ZZ$.
In particular, all intersections are zigs of $w$ or zags of $w$ by Lemma~\ref{intersect_zigorzag2}.
If $\widetilde{w}$ intersects with $\widetilde{z_i}(\alpha)$ at a zig (resp.\ zag) of $\widetilde{z_i}(\alpha)$, then we easily see that $\widetilde{w}$ crosses the $(i,\alpha)$-th deformed part in the direction from the $(i,\alpha)$-th (resp.~$(i+1,\alpha)$-th) irrelevant part to the $(i+1,\alpha)$-th (resp.~$(i,\alpha)$-th) irrelevant part.

\item[(b)] If $[z_i]$ and $[w]$ are linearly dependent, then by Lemma~\ref{slope_linearly_independent}, $\widetilde{w}$ and $\widetilde{z_i}(\alpha)$ do not intersect for any $i=1,\dots,r$ and $\alpha\in\ZZ$.
This is equivalent to the condition that $\widetilde{w}$ is contained in some irrelevant part.
In this case, $w$ is unchanged even if we apply the extended deformations because (zig-1)--(zig-4) (or (zag-1)--(zag-4)) are operations on the deformed parts, and~(zig-5) (or~(zag-5)) does not affect $w$ by Lemma~\ref{bypass_removable} below.
\end{itemize}
\end{observation}

In the remainder of this subsection, we discuss the behavior of zigzag paths after applying the operations (zig-1)--(zig-4).
(We can apply the same arguments for the case of (zag-1)--(zag-4).)

\begin{observation}\label{obs_deformed_part2}
We consider a zigzag path $y_k$ on a consistent dimer model $\Gamma$ intersecting with $z_i$ at a zig of $z_i$.
Let $\widetilde{z}_i$ and $\widetilde{y}_k$ be zigzag paths on $\widetilde{\Gamma}$ projecting onto $z_i$ and $y_k$, respectively.
By Observation~\ref{obs_deformed_part}(a), $\widetilde{y}_k$ crosses the $i$-th deformed part in the direction from the $i$-th irrelevant part to the $(i+1)$-th irrelevant part, and $\widetilde{y}_k$ intersects with $\widetilde{z_i}$ precisely once.
We suppose that the zig $\widetilde{z_i}[2m-1]$ of $\widetilde{z_i}$ is such an intersection.
In this case, $\widetilde{z_i}[2m-1]$ is also a zag of $\widetilde{y}_k$; thus we may write it as~$\widetilde{y_k}[2m]$.

Now, we apply the operations (zig-1)--(zig-3) to~$\Gamma$, and we denote the resulting dimer model by~$\Gamma^\prime$ and its universal cover by~$\widetilde{\Gamma}^\prime$.
Then, some new nodes are inserted in $\widetilde{z_i}[2m-1]=\widetilde{y_k}[2m]$, and zigzag paths $\widetilde{z}_{i,1},\dots,\widetilde{z}_{i,p_i}$ on $\widetilde{\Gamma}^\prime$, which project onto zigzag paths $z_{i,1},\dots,z_{i,p_i}$ on~$\Gamma^\prime$, respectively, appear in the $i$-th deformed part.

\begin{figure}[h!]\centering
\begin{tikzpicture}
\node at (-0.5,0)
{\scalebox{0.7}{
\begin{tikzpicture}
\newcommand{\edgewidth}{0.05cm} 
\newcommand{\nodewidth}{0.05cm} 
\newcommand{\noderad}{0.16} 

\coordinate (W1) at (0,0); \coordinate (B1) at (2.5,0);
\path (W1) ++(135:1.5cm) coordinate (W1a); \path (W1) ++(150:1.5cm) coordinate (W1a-); \path (W1) ++(225:1.5cm) coordinate (W1b);
\path (B1) ++(45:1.5cm) coordinate (B1a); \path (B1) ++(330:1.5cm) coordinate (B1a-); \path (B1) ++(315:1.5cm) coordinate (B1b);

\draw [line width=\edgewidth] (W1)--(B1);
\draw [line width=\edgewidth] (W1)--(W1a); \draw [line width=\edgewidth] (W1)--(W1a-); \draw [line width=\edgewidth] (W1)--(W1b);
\draw [line width=\edgewidth] (B1)--(B1a); \draw [line width=\edgewidth] (B1)--(B1a-); \draw [line width=\edgewidth] (B1)--(B1b);
\draw [line width=\edgewidth,line cap=round, dash pattern=on 0pt off 2.5\pgflinewidth] (-1,0.25)--(-1,-0.45);
\draw [line width=\edgewidth,line cap=round, dash pattern=on 0pt off 2.5\pgflinewidth] (3.5,0.45)--(3.5,-0.25);

\draw [line width=\nodewidth, fill=white] (W1) circle [radius=\noderad] ;
\draw [line width=\nodewidth, fill=black] (B1) circle [radius=\noderad] ;

\path (W1) ++(270:0.3cm) coordinate (W1-); \path (W1) ++(90:0.3cm) coordinate (W1+);
\path (B1) ++(270:0.3cm) coordinate (B1-); \path (B1) ++(90:0.3cm) coordinate (B1+);
\path (W1-) ++(225:1.7cm) coordinate (W1--); \path (W1+) ++(135:1.7cm) coordinate (W1++);
\path (B1-) ++(315:1.7cm) coordinate (B1--); \path (B1+) ++(45:1.7cm) coordinate (B1++);

\draw [->, rounded corners, line width=0.08cm, red] (B1--)--(B1-)--(W1+)--(W1++) ;
\draw [->, rounded corners, line width=0.08cm, blue] (W1--)--(W1-)--(B1+)--(B1++) ;

\node[blue] at (3,1.3) {\Large$\widetilde{y}_k$};
\node[red] at (-0.5,1.3) {\Large$\widetilde{z}_i$};
\end{tikzpicture}
} };

\draw [->,line width=0.035cm] (2.2,0)--(4.8,0);
\node at (3.5,0.35) {\small(zig-1)--(zig-3)};

\node at (8.5,0)
{\scalebox{0.7}{
\begin{tikzpicture}
\newcommand{\edgewidth}{0.05cm} 
\newcommand{\nodewidth}{0.05cm} 
\newcommand{\noderad}{0.16} 

\coordinate (W1) at (0,0); \coordinate (B1) at (6,0);
\path (W1) ++(135:1.5cm) coordinate (W1a); \path (W1) ++(150:1.5cm) coordinate (W1a-); \path (W1) ++(225:1.5cm) coordinate (W1b);
\path (B1) ++(45:1.5cm) coordinate (B1a); \path (B1) ++(330:1.5cm) coordinate (B1a-); \path (B1) ++(315:1.5cm) coordinate (B1b);

\coordinate (W2) at (2,0); \coordinate (W3) at (5,0);
\coordinate (B2) at (1,0); \coordinate (B3) at (4,0);

\path (W2) ++(120:1.5cm) coordinate (W2+); \path (W3) ++(120:1.5cm) coordinate (W3+);
\path (B2) ++(300:1.5cm) coordinate (B2+); \path (B3) ++(300:1.5cm) coordinate (B3+);

\draw [line width=\edgewidth] (W1)--(2.5,0); \draw [line width=\edgewidth] (3.5,0)--(B1);
\draw [line width=\edgewidth,line cap=round, dash pattern=on 0pt off 2.5\pgflinewidth] (2.6,0)--(3.4,0);
\draw [line width=\edgewidth] (W1)--(W1a-); \draw [line width=\edgewidth] (W1)--(W1b);
\draw [line width=\edgewidth] (B1)--(B1a); \draw [line width=\edgewidth] (B1)--(B1a-);
\draw [line width=\edgewidth,line cap=round, dash pattern=on 0pt off 2.5\pgflinewidth] (-1,0.25)--(-1,-0.45);
\draw [line width=\edgewidth,line cap=round, dash pattern=on 0pt off 2.5\pgflinewidth] (7,0.45)--(7,-0.25);

\draw [line width=\edgewidth] (B2)--(B2+); \draw [line width=\edgewidth] (B3)--(B3+);
\draw [line width=\edgewidth] (W2)--(W2+); \draw [line width=\edgewidth] (W3)--(W3+);

\draw [line width=\nodewidth, fill=white] (W1) circle [radius=\noderad] ;
\draw [line width=\nodewidth, fill=black] (B1) circle [radius=\noderad] ;

\draw [line width=\nodewidth, fill=white] (W2) circle [radius=\noderad] ; \draw [line width=\nodewidth, fill=white] (W3) circle [radius=\noderad] ;
\draw [line width=\nodewidth, fill=black] (B2) circle [radius=\noderad] ; \draw [line width=\nodewidth, fill=black] (B3) circle [radius=\noderad] ;

\path (W1) ++(270:0.3cm) coordinate (W1-); \path (W1) ++(90:0.3cm) coordinate (W1+);
\path (B1) ++(270:0.3cm) coordinate (B1-); \path (B1) ++(90:0.3cm) coordinate (B1+);
\path (W1-) ++(225:1.7cm) coordinate (W1--); \path (W1+) ++(135:1.7cm) coordinate (W1++);
\path (B1-) ++(315:1.7cm) coordinate (B1--); \path (B1+) ++(45:1.7cm) coordinate (B1++);
\draw [->, rounded corners, line width=0.08cm, blue] (W1--)--(W1-)--(2,-0.35)--(4,0.35)--(B1+)--(B1++) ;

\path (B2) ++(30:0.25cm) coordinate (B2z); \path (W2) ++(160:0.3cm) coordinate (W2z);
\path (B2z) ++(300:2cm) coordinate (B2zz); \path (W2z) ++(120:1.5cm) coordinate (W2zz);
\draw [<-, rounded corners, line width=0.08cm, red] (B2zz)--(B2z)--(W2z)--(W2zz);

\path (B3) ++(30:0.25cm) coordinate (B3z); \path (W3) ++(160:0.3cm) coordinate (W3z);
\path (B3z) ++(300:2cm) coordinate (B3zz); \path (W3z) ++(120:1.5cm) coordinate (W3zz);
\draw [<-, rounded corners, line width=0.08cm, red] (B3zz)--(B3z)--(W3z)--(W3zz);
\node[blue] at (6.4,1.3) {\Large$\widetilde{y}^\prime_k$};
\node[red] at (2.8,-1.3) {\large$\widetilde{z}_{i,1}$}; \node[red] at (5.8,-1.3) {\large$\widetilde{z}_{i,p_i}$};
\end{tikzpicture}
} };
\end{tikzpicture}
\caption{}\label{fig_observation_zigzag1}
\end{figure}
We consider the zigzag path $\widetilde{y}^\prime_k$ on $\widetilde{\Gamma}^\prime$ passing through $\widetilde{y_k}[2m-1]$ as a zig.
That is, $\widetilde{y}^\prime_k$ starts from $\widetilde{y_k}[2m-1]$, crosses through zigzag paths $\widetilde{z}_{i,1},\dots,\widetilde{z}_{i,p_i}$ in the $i$-th deformed part, and arrives at $\widetilde{y_k}[2m+1]$ (see the right of Figure~\ref{fig_observation_zigzag1}).
In particular, it crosses the $i$-th deformed part in the direction from the $i$-th irrelevant part to the $(i+1)$-th irrelevant part.
Although $\widetilde{y}^\prime_k$ looks different from $\widetilde{y}_k$ in the deformed parts, it connects the zigs $\widetilde{y_k}[2m-1]$ and $\widetilde{y_k}[2m+1]$ of $\widetilde{y}_k$ in the $i$-th deformed part for all $i$, and thus $\widetilde{y}^\prime_k$ shares the same nodes and edges as $\widetilde{y}_k$ in any irrelevant part.
We conclude that $\widetilde{y}^\prime_k$ coincides with $\widetilde{y}_k$ in all irrelevant parts, and behaves as depicted on the right-hand side of Figure~\ref{fig_observation_zigzag1} in each deformed part.
Also, we see that bypasses inserted in the operation (zig-4) do not affect the behavior of $\widetilde{y}^\prime_k$, because $\widetilde{y}^\prime_k$ never passes through bypasses.
\end{observation}

\begin{observation}
\label{obs_deformed_part3}
We consider a zigzag path $x_j$ on $\Gamma$ intersecting with $z_i$ at a zag of $z_i$.
Let $\widetilde{x}_j$ be a zigzag path on $\widetilde{\Gamma}$ projecting onto $x_j$.
By Observation~\ref{obs_deformed_part}(a), $\widetilde{x}_j$ crosses the $i$-th deformed part in the direction from the $(i+1)$-th irrelevant part to the $i$-th irrelevant part, and $\widetilde{x}_j$ intersects with $\widetilde{z_i}$ precisely once.
We suppose that the zag $\widetilde{z_i}[2m]$ of $\widetilde{z_i}$ is such an intersection.
In this case, $\widetilde{z_i}[2m]$ is also a zig of $\widetilde{x}_j$, and thus we may write it as $\widetilde{x_j}[2m+1]$.
We then apply the operations (zig-1)--(zig-4) to $\Gamma$, and we have the dimer model $\overline{\nu}_\calX^\zig(\Gamma)=\overline{\nu}_\calX^\zig(\Gamma,\{z_1,\dots,z_r\})$.

If $z_i[2m]$, which is the projection of $\widetilde{z_i}[2m]=\widetilde{x_j}[2m+1]$ on $\Gamma$, is not contained in $X_i$, then bypasses are inserted.
We consider the zigzag path $\widetilde{x}^\prime_j$ on the universal cover of $\overline{\nu}_\calX^\zig(\Gamma)$ passing through $\widetilde{x_j}[2m]$ as a zag of $\widetilde{x}^\prime_j$.
That is, $\widetilde{x}^\prime_j$ starts from $\widetilde{x_j}[2m]$, behaves as depicted on the right-hand side of Figure~\ref{fig_observation_zigzag2}, and arrives at $\widetilde{x_j}[2m+2]$.
In particular, $\widetilde{x}^\prime_j$ crosses the $i$-th deformed part in the direction from the $(i+1)$-th irrelevant part to the $i$-th irrelevant part,
and behaves in the same manner as $\widetilde{x}_j$ in each irrelevant part.

\begin{figure}[h!]\centering\vspace*{-31mm}
\begin{tikzpicture}
\newcommand{\edgewidth}{0.05cm} 
\newcommand{\nodewidth}{0.05cm} 
\newcommand{\noderad}{0.16} 
\coordinate (W1) at (0,1.2); \coordinate (W2) at (0,3.6);
\path (W1) ++(0:4cm) coordinate (B1); \path (W2) ++(0:4cm) coordinate (B2);
\path (B1) ++(225:2cm) coordinate (B0); \path (W2) ++(45:2cm) coordinate (W3);

\path (W1) ++(-20:0.8cm) coordinate (B1-1); \path (W1) ++(-20:1.6cm) coordinate (W1-1);
\path (W1) ++(-20:2.4cm) coordinate (B1-2); \path (W1) ++(-20:3.2cm) coordinate (W1-2);
\path (W2) ++(-20:0.8cm) coordinate (B2-1); \path (W2) ++(-20:1.6cm) coordinate (W2-1);
\path (W2) ++(-20:2.4cm) coordinate (B2-2); \path (W2) ++(-20:3.2cm) coordinate (W2-2);

\path (B1) ++(30:1cm) coordinate (B1e); \path (B1) ++(330:1cm) coordinate (B1s);
\path (W1) ++(150:1cm) coordinate (W1w); \path (W1) ++(210:1cm) coordinate (W1s);
\path (B2) ++(30:1cm) coordinate (B2e); \path (B2) ++(330:1cm) coordinate (B2s);
\path (W2) ++(150:1cm) coordinate (W2w); \path (W2) ++(210:1cm) coordinate (W2s);

\path (W1) ++(160:0.7cm) coordinate (W1+); \path (W1) ++(200:0.7cm) coordinate (W1-);
\path (W2) ++(160:0.7cm) coordinate (W2+); \path (W2) ++(200:0.7cm) coordinate (W2-);
\path (B1) ++(20:0.7cm) coordinate (B1+); \path (B1) ++(340:0.7cm) coordinate (B1-);
\path (B2) ++(20:0.7cm) coordinate (B2+); \path (B2) ++(340:0.7cm) coordinate (B2-);

\node (original) at (0,0)
{\scalebox{0.65}{
\begin{tikzpicture}
\draw [line width=\edgewidth] (B1)--(W1); \draw [line width=\edgewidth] (B2)--(W1) ; \draw [line width=\edgewidth] (B2)--(W2) ;
\draw [line width=\edgewidth] (B1)--(B0) ; \draw [line width=\edgewidth] (W2)--(W3) ;
\draw [line width=\edgewidth] (B1)--(B1e); \draw [line width=\edgewidth] (B1)--(B1s);
\draw [line width=\edgewidth] (W1)--(W1w); \draw [line width=\edgewidth] (W1)--(W1s);
\draw [line width=\edgewidth] (B2)--(B2e); \draw [line width=\edgewidth] (B2)--(B2s);
\draw [line width=\edgewidth] (W2)--(W2w); \draw [line width=\edgewidth] (W2)--(W2s);

\draw [line width=\edgewidth,line cap=round, dash pattern=on 0pt off 2.5\pgflinewidth] (W1+)--(W1-);
\draw [line width=\edgewidth,line cap=round, dash pattern=on 0pt off 2.5\pgflinewidth] (W2+)--(W2-);
\draw [line width=\edgewidth,line cap=round, dash pattern=on 0pt off 2.5\pgflinewidth] (B1+)--(B1-);
\draw [line width=\edgewidth,line cap=round, dash pattern=on 0pt off 2.5\pgflinewidth] (B2+)--(B2-);
\draw [line width=\nodewidth, fill=black] (B1) circle [radius=\noderad] ; \draw [line width=\nodewidth, fill=black] (B2) circle [radius=\noderad] ;
\draw [line width=\nodewidth, fill=white] (W1) circle [radius=\noderad] ; \draw [line width=\nodewidth, fill=white] (W2) circle [radius=\noderad] ;

\path (W1) ++(90:0.2cm) coordinate (W1z); \path (W2) ++(90:0.2cm) coordinate (W2z);
\path (B1) ++(90:0.2cm) coordinate (B1z); \path (B2) ++(90:0.2cm) coordinate (B2z);
\path (B1z) ++(225:2cm) coordinate (B0z); \path (W2z) ++(45:2cm) coordinate (W3z);

\path (B2) ++(-90:0.2cm) coordinate (B2zz); \path (B2zz) ++(330:1cm) coordinate (B2szz);
\path (W1) ++(-90:0.2cm) coordinate (W1zz); \path (W1zz) ++(150:1.5cm) coordinate (W1wzz);
\draw [->, rounded corners, line width=0.08cm, blue] (B2szz)--(B2zz)--(W1zz)--(W1wzz);
\draw [->, rounded corners, line width=0.08cm, red] (B0z)--(B1z)--(W1z)--(B2z)--(W2z)--(W3z);

\node at (0.3,4.8) {\color{red}\Large$\widetilde{z_i}$};
\node at (-1.2,2.3) {\color{blue}\Large$\widetilde{x_j}$};
\end{tikzpicture}
}};

\draw [->,line width=0.035cm] (2.6,0)--(5.2,0);
\node at (3.9,0.35) {\small(zig-1)--(zig-4)};

\node (after) at (8.9,0)
{\scalebox{0.65}{
\begin{tikzpicture}
\coordinate (WW1) at (0,1.2); \coordinate (WW2) at (0,3.6);
\path (WW1) ++(0:8cm) coordinate (BB1); \path (WW2) ++(0:8cm) coordinate (BB2);
\path (BB1) ++(200:2cm) coordinate (BB0); \path (WW2) ++(20:2cm) coordinate (WW3);

\path (WW1) ++(0:0.8cm) coordinate (BB1-1); \path (WW1) ++(0:1.6cm) coordinate (WW1-1);
\path (WW1) ++(0:2.4cm) coordinate (BB1-2);
\path (WW2) ++(0:0.8cm) coordinate (BB2-1); \path (WW2) ++(0:1.6cm) coordinate (WW2-1);
\path (WW2) ++(0:2.4cm) coordinate (BB2-2);

\path (WW1) ++(0:5.6cm) coordinate (WW1-2); \path (WW1) ++(0:6.4cm) coordinate (BB1-3);
\path (WW1) ++(0:7.2cm) coordinate (WW1-3);
\path (WW2) ++(0:5.6cm) coordinate (WW2-2); \path (WW2) ++(0:6.4cm) coordinate (BB2-3);
\path (WW2) ++(0:7.2cm) coordinate (WW2-3);

\path (BB1) ++(30:1cm) coordinate (BB1e); \path (BB1) ++(330:1cm) coordinate (BB1s);
\path (WW1) ++(150:1cm) coordinate (WW1w); \path (WW1) ++(210:1cm) coordinate (WW1s);
\path (BB2) ++(30:1cm) coordinate (BB2e); \path (BB2) ++(330:1cm) coordinate (BB2s);
\path (WW2) ++(150:1cm) coordinate (WW2w); \path (WW2) ++(210:1cm) coordinate (WW2s);

\path (WW1) ++(160:0.7cm) coordinate (WW1+); \path (WW1) ++(200:0.7cm) coordinate (WW1-);
\path (WW2) ++(160:0.7cm) coordinate (WW2+); \path (WW2) ++(200:0.7cm) coordinate (WW2-);
\path (BB1) ++(20:0.7cm) coordinate (BB1+); \path (BB1) ++(340:0.7cm) coordinate (BB1-);
\path (BB2) ++(20:0.7cm) coordinate (BB2+); \path (BB2) ++(340:0.7cm) coordinate (BB2-);

\draw [line width=\edgewidth] (WW1)-- ++(0:3.2cm); \draw [line width=\edgewidth] (BB1)-- ++(180:3.2cm);
\draw [line width=\edgewidth] (WW2)-- ++(0:3.2cm); \draw [line width=\edgewidth] (BB2)-- ++(180:3.2cm);

\path (WW1) ++(0:3.5cm) coordinate (dot_start1);
\draw [line width=\edgewidth,line cap=round, dash pattern=on 0pt off 2.5\pgflinewidth] (dot_start1)-- ++(0:1.1cm);
\path (WW2) ++(0:3.5cm) coordinate (dot_start2);
\draw [line width=\edgewidth,line cap=round, dash pattern=on 0pt off 2.5\pgflinewidth] (dot_start2)-- ++(0:1.1cm);

\draw [line width=\edgewidth] (BB1)--(BB1e); \draw [line width=\edgewidth] (BB1)--(BB1s);
\draw [line width=\edgewidth] (WW1)--(WW1w); \draw [line width=\edgewidth] (WW1)--(WW1s);
\draw [line width=\edgewidth] (BB2)--(BB2e); \draw [line width=\edgewidth] (BB2)--(BB2s);
\draw [line width=\edgewidth] (WW2)--(WW2w); \draw [line width=\edgewidth] (WW2)--(WW2s);

\draw [line width=\edgewidth,line cap=round, dash pattern=on 0pt off 2.5\pgflinewidth] (WW1+)--(WW1-);
\draw [line width=\edgewidth,line cap=round, dash pattern=on 0pt off 2.5\pgflinewidth] (WW2+)--(WW2-);
\draw [line width=\edgewidth,line cap=round, dash pattern=on 0pt off 2.5\pgflinewidth] (BB1+)--(BB1-);
\draw [line width=\edgewidth,line cap=round, dash pattern=on 0pt off 2.5\pgflinewidth] (BB2+)--(BB2-);

\path (BB1-1) ++(-70:1.5cm) coordinate (BB1-1-); \draw [line width=\edgewidth] (BB1-1)--(BB1-1-);
\path (BB1-2) ++(-70:1.5cm) coordinate (BB1-2-); \draw [line width=\edgewidth] (BB1-2)--(BB1-2-);
\path (WW2-1) ++(110:1.5cm) coordinate (WW2-1-); \draw [line width=\edgewidth] (WW2-1)--(WW2-1-);
\draw [line width=\edgewidth] (BB2-1)--(WW1-1);
\draw [line width=\edgewidth] (BB2-2)-- ++(-70:1.3cm);

\path (BB1-3) ++(-70:1.5cm) coordinate (BB1-3-); \draw [line width=\edgewidth] (BB1-3)--(BB1-3-);
\path (WW2-2) ++(110:1.5cm) coordinate (WW2-2-); \draw [line width=\edgewidth] (WW2-2)--(WW2-2-);
\path (WW2-3) ++(110:1.5cm) coordinate (WW2-3-); \draw [line width=\edgewidth] (WW2-3)--(WW2-3-);
\draw [line width=\edgewidth] (BB2-3)--(WW1-3);
\draw [line width=\edgewidth] (WW1-2)-- ++(110:1.3cm);

\draw [line width=\edgewidth] (BB2-1)--(WW1); \draw [line width=\edgewidth] (BB2-2)--(WW1-1); 
\draw [line width=\edgewidth] (WW1-3)--(BB2); \draw [line width=\edgewidth] (WW1-2)--(BB2-3);

\draw [line width=\nodewidth, fill=black] (BB1) circle [radius=\noderad] ; \draw [line width=\nodewidth, fill=black] (BB2) circle [radius=\noderad] ;
\draw [line width=\nodewidth, fill=black] (BB1-1) circle [radius=\noderad] ; \draw [line width=\nodewidth, fill=black] (BB1-2) circle [radius=\noderad] ;
\draw [line width=\nodewidth, fill=black] (BB1-3) circle [radius=\noderad] ;
\draw [line width=\nodewidth, fill=black] (BB2-1) circle [radius=\noderad] ; \draw [line width=\nodewidth, fill=black] (BB2-2) circle [radius=\noderad] ;
\draw [line width=\nodewidth, fill=black] (BB2-3) circle [radius=\noderad] ;
\draw [line width=\nodewidth, fill=white] (WW1) circle [radius=\noderad] ; \draw [line width=\nodewidth, fill=white] (WW2) circle [radius=\noderad] ;
\draw [line width=\nodewidth, fill=white] (WW1-1) circle [radius=\noderad] ; \draw [line width=\nodewidth, fill=white] (WW1-2) circle [radius=\noderad] ;
\draw [line width=\nodewidth, fill=white] (WW1-3) circle [radius=\noderad] ;
\draw [line width=\nodewidth, fill=white] (WW2-1) circle [radius=\noderad] ; \draw [line width=\nodewidth, fill=white] (WW2-2) circle [radius=\noderad] ;
\draw [line width=\nodewidth, fill=white] (WW2-3) circle [radius=\noderad] ;

\path (BB2) ++(-90:0.35cm) coordinate (BB2z); \path (BB2z) ++(330:1cm) coordinate (BB2sz);
\path (WW1-3) ++(-90:0.35cm) coordinate (WW1-3z);
\path (BB2-3) ++(-90:0.35cm) coordinate (BB2-3z); \path (WW1-2) ++(-90:0.35cm) coordinate (WW1-2z);
\draw [->, rounded corners, line width=0.08cm, blue] (BB2sz)--(BB2z)--(WW1-3z)--(BB2-3z)--(WW1-2z)-- ++(110:1.8cm);

\path (WW1) ++(-90:0.35cm) coordinate (WW1z);
\path (BB2-1) ++(-90:0.35cm) coordinate (BB2-1z); \path (BB2-2) ++(-90:0.35cm) coordinate (BB2-2z);
\path (WW1-1) ++(-90:0.35cm) coordinate (WW1-1z);
\path (WW1z) ++(150:1.35cm) coordinate (WW1wz);
\draw [<-, rounded corners, line width=0.08cm, blue] (WW1wz)--(WW1z)--(BB2-1z)--(WW1-1z)--(BB2-2z)-- ++(-70:1.8cm);

\path (WW1-1) ++(-30:0.25cm) coordinate (WW1-1z); \path (BB2-1) ++(-30:0.25cm) coordinate (BB2-1z);
\path (WW2-1) ++(-30:0.25cm) coordinate (WW2-1z); \path (BB1-1) ++(-30:0.25cm) coordinate (BB1-1z);
\path (WW2-1z) ++(110:1.8cm) coordinate (WW2-1zz); \path (BB1-1z) ++(290:1.7cm) coordinate (BB1-1zz);
\draw [->, rounded corners, line width=0.08cm, red] (WW2-1zz)--(WW2-1z)--(BB2-1z)--(WW1-1z)--(BB1-1z)--(BB1-1zz);

\path (WW1-3) ++(-30:0.25cm) coordinate (WW1-3z); \path (BB2-3) ++(-30:0.25cm) coordinate (BB2-3z);
\path (WW2-3) ++(-30:0.25cm) coordinate (WW2-3z); \path (BB1-3) ++(-30:0.25cm) coordinate (BB1-3z);
\path (WW2-3z) ++(110:1.8cm) coordinate (WW2-3zz); \path (BB1-3z) ++(290:1.7cm) coordinate (BB1-3zz);
\draw [->, rounded corners, line width=0.08cm, red] (WW2-3zz)--(WW2-3z)--(BB2-3z)--(WW1-3z)--(BB1-3z)--(BB1-3zz);

\node at (-1.3,2) {\color{blue}\Large$\widetilde{x}_j^\prime$};
\node[red] at (2.1,-0.3) {\large$\widetilde{z}_{i,1}$}; \node[red] at (7.8,-0.3) {\large$\widetilde{z}_{i,p_i}$};
\end{tikzpicture}
}};
\end{tikzpicture}
\caption{The case where the projection of $\widetilde{z_i}[2m]=\widetilde{x_j}[2m+1]$ on $\Gamma$ is not contained in $X_i$.}\label{fig_observation_zigzag2}
\end{figure}

\begin{figure}[h!]\centering
\begin{tikzpicture}
\newcommand{\edgewidth}{0.05cm} 
\newcommand{\nodewidth}{0.05cm} 
\newcommand{\noderad}{0.16} 

\node (original) at (0,0)
{\scalebox{0.65}{
\begin{tikzpicture}
\draw [line width=\edgewidth] (B1)--(W1); \draw [line width=\edgewidth] (B2)--(W1) ; \draw [line width=\edgewidth] (B2)--(W2) ;
\draw [line width=\edgewidth] (B1)--(B0) ; \draw [line width=\edgewidth] (W2)--(W3) ;
\draw [line width=\edgewidth] (B1)--(B1e); \draw [line width=\edgewidth] (B1)--(B1s);
\draw [line width=\edgewidth] (W1)--(W1w); \draw [line width=\edgewidth] (W1)--(W1s);
\draw [line width=\edgewidth] (B2)--(B2e); \draw [line width=\edgewidth] (B2)--(B2s);
\draw [line width=\edgewidth] (W2)--(W2w); \draw [line width=\edgewidth] (W2)--(W2s);

\draw [line width=\edgewidth,line cap=round, dash pattern=on 0pt off 2.5\pgflinewidth] (W1+)--(W1-);
\draw [line width=\edgewidth,line cap=round, dash pattern=on 0pt off 2.5\pgflinewidth] (W2+)--(W2-);
\draw [line width=\edgewidth,line cap=round, dash pattern=on 0pt off 2.5\pgflinewidth] (B1+)--(B1-);
\draw [line width=\edgewidth,line cap=round, dash pattern=on 0pt off 2.5\pgflinewidth] (B2+)--(B2-);
\draw [line width=\nodewidth, fill=black] (B1) circle [radius=\noderad] ; \draw [line width=\nodewidth, fill=black] (B2) circle [radius=\noderad] ;
\draw [line width=\nodewidth, fill=white] (W1) circle [radius=\noderad] ; \draw [line width=\nodewidth, fill=white] (W2) circle [radius=\noderad] ;

\path (W1) ++(90:0.2cm) coordinate (W1z); \path (W2) ++(90:0.2cm) coordinate (W2z);
\path (B1) ++(90:0.2cm) coordinate (B1z); \path (B2) ++(90:0.2cm) coordinate (B2z);
\path (B1z) ++(225:2cm) coordinate (B0z); \path (W2z) ++(45:2cm) coordinate (W3z);

\path (B2) ++(-90:0.2cm) coordinate (B2zz); \path (B2zz) ++(330:1cm) coordinate (B2szz);
\path (W1) ++(-90:0.2cm) coordinate (W1zz); \path (W1zz) ++(150:1.5cm) coordinate (W1wzz);
\draw [->, rounded corners, line width=0.08cm, blue] (B2szz)--(B2zz)--(W1zz)--(W1wzz);
\draw [->, rounded corners, line width=0.08cm, red] (B0z)--(B1z)--(W1z)--(B2z)--(W2z)--(W3z);

\node at (0.3,4.8) {\color{red}\Large$\widetilde{z_i}$};
\node at (-1.2,2.3) {\color{blue}\Large$\widetilde{x_j}$};
\end{tikzpicture}
}};

\draw [->,line width=0.035cm] (2.6,0)--(5.2,0);
\node at (3.9,0.35) {\small(zig-1)--(zig-4)};

\node (after) at (8.9,0)
{\scalebox{0.65}{
\begin{tikzpicture}
\draw [line width=\edgewidth] (WW1)-- ++(0:3.2cm); \draw [line width=\edgewidth] (BB1)-- ++(180:3.2cm);
\draw [line width=\edgewidth] (WW2)-- ++(0:3.2cm); \draw [line width=\edgewidth] (BB2)-- ++(180:3.2cm);

\path (WW1) ++(0:3.5cm) coordinate (dot_start1);
\draw [line width=\edgewidth,line cap=round, dash pattern=on 0pt off 2.5\pgflinewidth] (dot_start1)-- ++(0:1.1cm);
\path (WW2) ++(0:3.5cm) coordinate (dot_start2);
\draw [line width=\edgewidth,line cap=round, dash pattern=on 0pt off 2.5\pgflinewidth] (dot_start2)-- ++(0:1.1cm);

\draw [line width=\edgewidth] (BB1)--(BB1e); \draw [line width=\edgewidth] (BB1)--(BB1s);
\draw [line width=\edgewidth] (WW1)--(WW1w); \draw [line width=\edgewidth] (WW1)--(WW1s);
\draw [line width=\edgewidth] (BB2)--(BB2e); \draw [line width=\edgewidth] (BB2)--(BB2s);
\draw [line width=\edgewidth] (WW2)--(WW2w); \draw [line width=\edgewidth] (WW2)--(WW2s);

\draw [line width=\edgewidth,line cap=round, dash pattern=on 0pt off 2.5\pgflinewidth] (WW1+)--(WW1-);
\draw [line width=\edgewidth,line cap=round, dash pattern=on 0pt off 2.5\pgflinewidth] (WW2+)--(WW2-);
\draw [line width=\edgewidth,line cap=round, dash pattern=on 0pt off 2.5\pgflinewidth] (BB1+)--(BB1-);
\draw [line width=\edgewidth,line cap=round, dash pattern=on 0pt off 2.5\pgflinewidth] (BB2+)--(BB2-);

\path (BB1-1) ++(-70:1.5cm) coordinate (BB1-1-); \draw [line width=\edgewidth] (BB1-1)--(BB1-1-);
\path (BB1-2) ++(-70:1.5cm) coordinate (BB1-2-); \draw [line width=\edgewidth] (BB1-2)--(BB1-2-);
\path (WW2-1) ++(110:1.5cm) coordinate (WW2-1-); \draw [line width=\edgewidth] (WW2-1)--(WW2-1-);
\draw [line width=\edgewidth] (BB2-1)--(WW1-1);
\draw [line width=\edgewidth] (BB2-2)-- ++(-70:1.3cm);

\path (BB1-3) ++(-70:1.5cm) coordinate (BB1-3-); \draw [line width=\edgewidth] (BB1-3)--(BB1-3-);
\path (WW2-2) ++(110:1.5cm) coordinate (WW2-2-); \draw [line width=\edgewidth] (WW2-2)--(WW2-2-);
\path (WW2-3) ++(110:1.5cm) coordinate (WW2-3-); \draw [line width=\edgewidth] (WW2-3)--(WW2-3-);
\draw [line width=\edgewidth] (BB2-3)--(WW1-3);
\draw [line width=\edgewidth] (WW1-2)-- ++(110:1.3cm);

\draw [line width=\nodewidth, fill=black] (BB1) circle [radius=\noderad] ; \draw [line width=\nodewidth, fill=black] (BB2) circle [radius=\noderad] ;
\draw [line width=\nodewidth, fill=black] (BB1-1) circle [radius=\noderad] ; \draw [line width=\nodewidth, fill=black] (BB1-2) circle [radius=\noderad] ;
\draw [line width=\nodewidth, fill=black] (BB1-3) circle [radius=\noderad] ;
\draw [line width=\nodewidth, fill=black] (BB2-1) circle [radius=\noderad] ; \draw [line width=\nodewidth, fill=black] (BB2-2) circle [radius=\noderad] ;
\draw [line width=\nodewidth, fill=black] (BB2-3) circle [radius=\noderad] ;
\draw [line width=\nodewidth, fill=white] (WW1) circle [radius=\noderad] ; \draw [line width=\nodewidth, fill=white] (WW2) circle [radius=\noderad] ;
\draw [line width=\nodewidth, fill=white] (WW1-1) circle [radius=\noderad] ; \draw [line width=\nodewidth, fill=white] (WW1-2) circle [radius=\noderad] ;
\draw [line width=\nodewidth, fill=white] (WW1-3) circle [radius=\noderad] ;
\draw [line width=\nodewidth, fill=white] (WW2-1) circle [radius=\noderad] ; \draw [line width=\nodewidth, fill=white] (WW2-2) circle [radius=\noderad] ;
\draw [line width=\nodewidth, fill=white] (WW2-3) circle [radius=\noderad] ;

\path (BB2) ++(-110:0.28cm) coordinate (BB2z); \path (BB2z) ++(330:1.2cm) coordinate (BB2zz);
\path (WW2-3) ++(-110:0.28cm) coordinate (WW2-3z); \path (WW2-3z) ++(110:2cm) coordinate (WW2-3zz);
\draw [->, rounded corners, line width=0.08cm, blue] (BB2zz)--(BB2z)--(WW2-3z)--(WW2-3zz);

\path (WW1-1) ++(-30:0.25cm) coordinate (WW1-1z); \path (BB2-1) ++(-30:0.25cm) coordinate (BB2-1z);
\path (WW2-1) ++(-30:0.25cm) coordinate (WW2-1z); \path (BB1-1) ++(-30:0.25cm) coordinate (BB1-1z);
\path (WW2-1z) ++(110:1.8cm) coordinate (WW2-1zz); \path (BB1-1z) ++(290:1.7cm) coordinate (BB1-1zz);
\draw [->, rounded corners, line width=0.08cm, red] (WW2-1zz)--(WW2-1z)--(BB2-1z)--(WW1-1z)--(BB1-1z)--(BB1-1zz);

\path (WW1-3) ++(-30:0.25cm) coordinate (WW1-3z); \path (BB2-3) ++(-30:0.25cm) coordinate (BB2-3z);
\path (WW2-3) ++(-30:0.25cm) coordinate (WW2-3z); \path (BB1-3) ++(-30:0.25cm) coordinate (BB1-3z);
\path (WW2-3z) ++(110:1.8cm) coordinate (WW2-3zz); \path (BB1-3z) ++(290:1.7cm) coordinate (BB1-3zz);
\draw [->, rounded corners, line width=0.08cm, red] (WW2-3zz)--(WW2-3z)--(BB2-3z)--(WW1-3z)--(BB1-3z)--(BB1-3zz);

\node at (6,4.8) {\color{blue}\Large$\widetilde{x}_j^\prime$};
\node[red] at (2.1,-0.3) {\large$\widetilde{z}_{i,1}$}; \node[red] at (7.8,-0.3) {\large$\widetilde{z}_{i,p_i}$};
\end{tikzpicture}
}};
\end{tikzpicture}
\caption{The case where the projection of $\widetilde{z_i}[2m]=\widetilde{x_j}[2m+1]$ on $\Gamma$ is contained in $X_i$.}
\label{fig_observation_zigzag3}
\end{figure}

If $z_i[2m]$ is contained in $X_i$, then no bypasses are inserted.
We again consider the zigzag path $\widetilde{x}^\prime_j$ on the universal cover of $\overline{\nu}_\calX^\zig(\Gamma)$ passing through $\widetilde{x_j}[2m]$ as a zag of $\widetilde{x}^\prime_j$ (e.g., see Figure~\ref{fig_observation_zigzag3}).
Unlike in the previous case, after passing through $\widetilde{x_j}[2m]$, $\widetilde{x}^\prime_j$ goes to the edge $\big\{\widetilde{b}_i[2m+1],\widetilde{w}_{i,1}[2m+1]\big\}$.
(Here, we denote the edge whose endpoints are a~black node~$b$ and a~white node~$w$ by~$\{b,w\}$.)

If $z_i[2m+2]\in X_i$, in which case bypasses are not inserted, then $\widetilde{x}^\prime_j$ goes to the edge $\big\{\widetilde{w}_{i,1}[2m+1], \widetilde{b}_{i,1}[2m+3]\big\}$.

On the other hand, if $z_i[2m+2]\not\in X_i$, in which case we insert bypasses, then $\widetilde{x}^\prime_j$ goes to the edge $\big\{\widetilde{w}_{i,1}[2m+1], \widetilde{b}_i[2m+3]\big\}$.
In such a way, $\widetilde{x}^\prime_j$ crosses the $i$-th deformed part in the direction from the $(i+1)$-th irrelevant part to the $i$-th irrelevant part.
More precisely, if we assume that~$\widetilde{x}^\prime_j$ goes through the edge $\{\widetilde{b}_{i,s}[2m-1], \widetilde{w}_{i,s+1}[2m-1]\}$ in the $i$-th deformed part,
then $\widetilde{x}^\prime_j$ behaves as follows:
\begin{itemize}\itemsep=0pt
\item[(1)] if $z_i[2m]\in X_i$, in which case bypasses are not inserted, then $\widetilde{x}^\prime_j$ goes through the edge $\big\{\widetilde{w}_{i,s+1}[2m-1],\widetilde{b}_{i,s+1}[2m+1]\big\}$ and then $\big\{\widetilde{b}_{i,s+1}[2m+1], \widetilde{w}_{i,s+2}[2m+1]\big\}$,
\item[(2)] if $z_i[2m]\not\in X_i$, in which case we insert bypasses, then $\widetilde{x}^\prime_j$ goes through the edge $\big\{\widetilde{w}_{i,s+1}[2m-1], \widetilde{b}_{i,s}[2m+1]\big\}$ and then $\big\{\widetilde{b}_{i,s}[2m+1], \widetilde{w}_{i,s+1}[2m+1]\big\}$,
\end{itemize}
where $m=1,\dots,n$ and $s=0,\dots,p_i-1$ with $\widetilde{b}_{i,0}[-]=\widetilde{b}_i[-]$ and $\widetilde{w}_{i,p_i+1}[-]=\widetilde{w}_i[-]$.
When we consider $\widetilde{x}^\prime_j$ in the $i$-th deformed part, we encounter case (1) $|X_i|$ times and case (2) $\ell(z_i)/2-|X_i|$ times.
Since $p_i=|X_i|-1$, we see that $\widetilde{x}^\prime_j$ goes out the $i$-th deformed part from $\widetilde{w}_i[2m-1+2n]=\widetilde{w}_i[2m-1+\ell(z_i)]$,
and then it goes into the $i$-th irrelevant part.
Thus, $\widetilde{x}^\prime_j$ behaves in the same manner as the shift of $\widetilde{x}_j$ in the $i$-th irrelevant part.
For example, if we consider the type I zigzag path $z_i$ with $\ell(z_i)=8$, and the zig-deformation parameter $X_i$ with $|X_i|=3$ and $z_i[2m+4]\not\in X_i$, then the $i$-th deformed part will change as shown in Figure~\ref{fig_observation_zigzag4}.

\begin{figure}[h!]\centering
\scalebox{0.6}{
\begin{tikzpicture}
\newcommand{\edgewidth}{0.05cm} 
\newcommand{\nodewidth}{0.05cm} 
\newcommand{\noderad}{0.17} 

\node at (0,0){
\begin{tikzpicture}
\foreach \blackname/\x/\y in {3/4.5/0, 6/4.5/2, 9/4.5/4, 12/4.5/6, 15/4.5/8}
{\coordinate (B\blackname) at (\x,\y);}
\foreach \whitename/\x/\y in {1/0/0, 4/0/2, 7/0/4, 10/0/6, 13/0/8}
{\coordinate (W\whitename) at (\x,\y);}

\foreach \t/\h in {1/3,4/6,7/9,10/12,13/15}{\draw [line width=\edgewidth] (W\t)--(B\h); } 
\foreach \t/\h in {1/6,4/9,7/12,10/15}{\draw [line width=\edgewidth] (W\t)--(B\h); }
\foreach \t in {3}{\draw [line width=\edgewidth] (B\t)-- ++(210:1.2cm); }
\foreach \t in {13}{\draw [line width=\edgewidth] (W\t)-- ++(30:1.2cm); }

\foreach \t in {1,4,7,10,13}{\draw [line width=\edgewidth] (W\t)-- ++(150:1cm); \draw [line width=\edgewidth] (W\t)-- ++(210:1cm); }
\foreach \t in {1,4,7,10,13}{\path (W\t) ++ (165:0.7cm) coordinate (W\t+); \draw [line width=\edgewidth,line cap=round, dash pattern=on 0pt off 2.5\pgflinewidth] (W\t+)-- ++(-90:0.38cm);}
\foreach \t in {3,6,9,12,15}{\draw [line width=\edgewidth] (B\t)-- ++(30:1cm); \draw [line width=\edgewidth] (B\t)-- ++(-30:1cm); }
\foreach \t in {3,6,9,12,15}{\path (B\t) ++ (15:0.7cm) coordinate (B\t+); \draw [line width=\edgewidth,line cap=round, dash pattern=on 0pt off 2.5\pgflinewidth] (B\t+)-- ++(-90:0.38cm);}

\foreach \n in {3,6,9,12,15} {\draw [line width=\nodewidth, fill=black] (B\n) circle [radius=\noderad];}
\foreach \n in {1,4,7,10,13} {\draw [line width=\nodewidth, fill=white] (W\n) circle [radius=\noderad];}

\path (B6) ++(-90:0.2cm) coordinate (B6_z); \path (W1) ++(-90:0.2cm) coordinate (W1_z);
\path (B6_z) ++(-30:1.2cm) coordinate (B6_zz); \path (W1_z) ++(150:1.2cm) coordinate (W1_zz);
\draw [->, rounded corners, line width=0.09cm, blue] (B6_zz)--(B6_z)--(W1_z)--(W1_zz);

\path (W13) ++(-90:0.2cm) coordinate (W13_z); \path (W13_z) ++(150:1.2cm) coordinate (W13_zz+);
\path (W13_z) ++(30:1.2cm) coordinate (W13_zz-);
\draw [->, rounded corners, line width=0.09cm, blue] (W13_zz-)--(W13_z)--(W13_zz+);

\path (W1_zz) ++(180:0.3cm) node[blue] {\large $\widetilde{x}_j$};
\path (W13_zz+) ++(180:0.6cm) node[blue] {\large $\widetilde{x}_j(1)$};

\node[red] at (2.25,-0.3) {$\widetilde{z}_i[2m-1]$}; \node[red] at (2.25,1.3) {$\widetilde{z}_i[2m]$};
\node[red] at (2.25,2.3) {$\widetilde{z}_i[2m+1]$}; \node[red] at (2.25,8.3) {$\widetilde{z}_i[2m+7]$};
\path (W1) ++(-90:0.86cm) node {$\widetilde{w}_i[2m-1]$}; \path (W13) ++(90:0.86cm) node {$\widetilde{w}_i[2m+7]$};
\path (B3) ++(-90:0.86cm) node {$\widetilde{b}_i[2m-1]$}; \path (B15) ++(90:0.86cm) node {$\widetilde{b}_i[2m+7]$};
\end{tikzpicture}};

\draw [->,line width=0.05cm] (4.5,0)--(9.7,0);
\node at (7.1,0.5) {\LARGE(zig-1)--(zig-4)};

\node at (15,0){
\begin{tikzpicture}
\foreach \blackname/\x/\y in {1/1.5/0, 2/4.5/0, 3/7.5/0, 4/1.5/2, 5/4.5/2, 6/7.5/2, 7/1.5/4, 8/4.5/4, 9/7.5/4, 10/1.5/6, 11/4.5/6, 12/7.5/6, 13/1.5/8, 14/4.5/8, 15/7.5/8}
{\coordinate (B\blackname) at (\x,\y);}
\foreach \whitename/\x/\y in {1/0/0, 2/3/0, 3/6/0, 4/0/2, 5/3/2, 6/6/2, 7/0/4, 8/3/4, 9/6/4, 10/0/6, 11/3/6, 12/6/6, 13/0/8, 14/3/8, 15/6/8}
{\coordinate (W\whitename) at (\x,\y);}

\foreach \t/\h in {1/3,4/6,7/9,10/12,13/15}{\draw [line width=\edgewidth] (W\t)--(B\h); } 
\foreach \t/\h in {2/4,3/5,5/7,6/8,8/10,9/11,11/13,12/14}{\draw [line width=\edgewidth] (W\t)--(B\h); }
\foreach \t/\h in {7/10,8/11,9/12}{\draw [line width=\edgewidth] (W\t)--(B\h); } 

\foreach \t in {1,2}{\draw [line width=\edgewidth] (B\t)-- ++(-50:1.2cm); }
\foreach \t in {14,15}{\draw [line width=\edgewidth] (W\t)-- ++(130:1.2cm); }

\foreach \t in {1,4,7,10,13}{\draw [line width=\edgewidth] (W\t)-- ++(150:1cm); \draw [line width=\edgewidth] (W\t)-- ++(210:1cm); }
\foreach \t in {1,4,7,10,13}{\path (W\t) ++ (165:0.7cm) coordinate (W\t+); \draw [line width=\edgewidth,line cap=round, dash pattern=on 0pt off 2.5\pgflinewidth] (W\t+)-- ++(-90:0.38cm);}
\foreach \t in {3,6,9,12,15}{\draw [line width=\edgewidth] (B\t)-- ++(30:1cm); \draw [line width=\edgewidth] (B\t)-- ++(-30:1cm); }
\foreach \t in {3,6,9,12,15}{\path (B\t) ++ (15:0.7cm) coordinate (B\t+); \draw [line width=\edgewidth,line cap=round, dash pattern=on 0pt off 2.5\pgflinewidth] (B\t+)-- ++(-90:0.38cm);}

\foreach \n in {1,2,3,4,5,6,7,8,9,10,11,12,13,14,15} {\draw [line width=\nodewidth, fill=black] (B\n) circle [radius=\noderad];}
\foreach \n in {1,2,3,4,5,6,7,8,9,10,11,12,13,14,15} {\draw [line width=\nodewidth, fill=white] (W\n) circle [radius=\noderad];}

\foreach \t in {6,8,11,13}{\path (B\t) ++(-90:0.2cm) coordinate (B\t_z); }
\foreach \t in {6,8,11,13}{\path (W\t) ++(-90:0.2cm) coordinate (W\t_z); }
\path (B6_z) ++(-30:1.2cm) coordinate (B6_zz); \path (W13_z) ++(150:1.2cm) coordinate (W13_zz);
\draw [->, rounded corners, line width=0.09cm, blue] (B6_zz)--(B6_z)--(W6_z)--(B8_z)--(W8_z)--(B11_z)--(W11_z)--(B13_z)--(W13_z)--(W13_zz);

\path (W13_zz) ++(180:0.3cm) node[blue] {\large $\widetilde{x}_j^\prime$};

\path (W1) ++(-90:0.86cm) node {$\widetilde{w}_i[2m-1]$}; \path (W13) ++(90:0.86cm) node {$\widetilde{w}_i[2m+7]$};
\path (B3) ++(-90:0.86cm) node {$\widetilde{b}_i[2m-1]$}; \path (B15) ++(90:0.86cm) node {$\widetilde{b}_i[2m+7]$};
\end{tikzpicture}};

\end{tikzpicture}
}
\caption{An example of the behavior of $\widetilde{x}^\prime_j$ in the $i$-th deformed part.}
\label{fig_observation_zigzag4}
\end{figure}
\end{observation}

\subsection{Properties of zigzag paths on deformed dimer models}\label{subsec_property_zigzag}

In this subsection, we study the slopes of zigzag paths on deformed dimer models.
In particular, we can describe them in terms of zigzag paths of the original dimer model, and this description plays a crucial role when discussing the relationship with the combinatorial mutations of the associated polygons.

\begin{Proposition}\label{deform_vector}
The path $z_{i,j}$ given in {\rm(zig-3)} of Definition~{\rm \ref{def_deformation_zig}} $($resp.\ {\rm(zag-3)} of Definition~{\rm \ref{def_deformation_zag})} is a type I zigzag path $z_{i,j}$ of $\nu^\zig_\calX(\Gamma)$ $($resp.~$\nu^\zag_\calY(\Gamma)$$)$ with $\ell(z_{i,j})=\ell(z_i)$.
Moreover, $z_{i,j}$ does not have any self-intersections on the universal cover, and satisfies $[z_{i,j}]=-[z_i]=-v$, and hence it is not homologically trivial.
\end{Proposition}

\begin{proof}
We consider the case of $\nu^\zig_\calX(\Gamma)$. The case of $\nu^\zag_\calY(\Gamma)$ is similar.

First, $z_{i,j}$ is a zigzag path of $\overline{\nu}^\zig_\calX(\Gamma)$ by definition.
We see that zigzag paths of $\overline{\nu}^\zig_\calX(\Gamma)$ intersecting with $\widetilde{z}_{i,j}$ take either the form of $\widetilde{y}_k^\prime$ given in Observation~\ref{obs_deformed_part2} or $\widetilde{x}_j^\prime$ given in Observation~\ref{obs_deformed_part3}.
In particular, these intersect with $\widetilde{z}_{i,j}$ precisely once in each deformed part, and
$\widetilde{y}_k^\prime$ crosses the $i$-th deformed part in the direction from the $i$-th irrelevant part to the $(i+1)$-th irrelevant part, and $\widetilde{x}_j^\prime$ crosses the $i$-th deformed part in the direction from the $(i+1)$-th irrelevant part to the $i$-th irrelevant part for all $i$.
Since other zigzag paths do not intersect with $\widetilde{z}_{i,j}$, we see that $z_{i,j}$ is a type I zigzag path of $\overline{\nu}^\zig_\calX(\Gamma)$.
Also, $\widetilde{z}_{i,j}$ does not have any self-intersections by definition.
By these properties, the edges constituting $z_{i,j}$ are not removed by the operation (zig-5).
In addition, since $z_{i,j}$ contains no $2$-valent nodes, the operation (join) does not affect $z_{i,j}$.
Therefore, we see that $z_{i,j}$ is a type I zigzag path of $\nu^\zig_\calX(\Gamma)$, and the remaining assertions follow from the definition of $z_{i,j}$.
\end{proof}

\begin{Proposition}\label{zigzag_afterdeform1}
The paths $y_1,\dots, y_t$ $($resp.\ the paths $x_1,\dots, x_s)$ are zigzag paths of $\Gamma$ intersecting with a chosen type~I zigzag path~$z$
at some zigs $($resp.\ zags$)$ of $z$. We have the following.
\begin{itemize}\itemsep=0pt
\item[\rm (1)] For any zigzag path $y_k$ of $\Gamma$, there exists a unique zigzag path $y_k^\prime$ of $\nu^\zig_\calX(\Gamma)$ without self-intersections
on the universal cover and satisfying $[y_k]=[y_k^\prime]$ where $k=1,\dots,t$.
\item[\rm (2)] For any zigzag path $x_j$ of $\Gamma$, there exists a unique zigzag path $x_j^\prime$ of $\nu^\zag_\calY(\Gamma)$ without self-intersections
on the universal cover and satisfying $[x_j]=[x_j^\prime]$ where $j=1,\dots,s$.
\end{itemize}
\end{Proposition}

\begin{proof}We consider the case of $\nu^\zig_\calX(\Gamma)$. The case of $\nu^\zag_\calY(\Gamma)$ is similar.

We use the same notation used in Observation~\ref{obs_deformed_part2}.
In particular, we consider the zigzag path~$\widetilde{y}^\prime_k$ of~$\widetilde{\Gamma}^\prime$ which coincides with $\widetilde{y}_k$ in all irrelevant parts, and behaves as depicted on the right-hand side of Figure~\ref{fig_observation_zigzag1} in each deformed part,
and therefore does not have any self-intersections.
Since bypasses inserted in the operation~(zig-4) do not affect the behavior of $\widetilde{y}^\prime_k$, we can extend $\widetilde{y}^\prime_k$ to a zigzag path of the universal cover of $\overline{\nu}^\zig_\calX(\Gamma)$.
By projecting $\widetilde{y}^\prime_k$ onto~$\overline{\nu}^\zig_\calX(\Gamma)$, we have the zigzag path $y_k^\prime$ of $\overline{\nu}^\zig_\calX(\Gamma)$.
By the construction given in Observation~\ref{obs_deformed_part2}, we see that $[y_k]=[y_k^\prime]$.
Since~$\widetilde{y}^\prime_k$ coincides with $\widetilde{y}_k$ in all irrelevant parts and $\Gamma$ is consistent, it does not behave pathologically in irrelevant parts as it infringes on the consistency condition.
Furthermore, $\widetilde{y}^\prime_k$ intersects with zigzag paths of the forms $\widetilde{z}_{i,j}$ and $\widetilde{x}^\prime_j$ (see Observation~\ref{obs_deformed_part3}) in some deformed parts, but they do not intersect with each other in the same direction more than once.
Therefore, the edges constituting~$y_k^\prime$ are not removed by the operation~(zig-5), and hence~$y_k^\prime$ is not changed by~(zig-5).
In addition, the operation (join) does not change the slopes.
Thus, we naturally extend this zigzag path $y_k^\prime$ as the path of $\nu^\zig_\calX(\Gamma)$, which is determined uniquely and satisfies $[y_k]=[y_k^\prime]$ by construction.
\end{proof}

By the proof of Propositions~\ref{deform_vector} and \ref{zigzag_afterdeform1}, the zigzag paths $\widetilde{z}_{i,j}$
and $\widetilde{y}^\prime_k$ do not have any self-intersections,
and do not intersect with other zigzag paths in the same direction more than once.
Furthermore, the intersections between $\widetilde{x}^\prime_j$ and $\widetilde{z}_{i,j}$ or $\widetilde{y}^\prime_k$ are not bypasses
(see Observations~\ref{obs_deformed_part2} and \ref{obs_deformed_part3}). This proves the following lemma.

\begin{Lemma}\label{bypass_removable}
We have the following.
\begin{itemize}\itemsep=0pt
\item[\rm (1)] The edges removed by the operation {\rm(zig-5)} are a part of the bypasses added in {\rm(zig-4)} or edges appearing in some irrelevant parts that are intersections between pairs of zigzag paths $x_1,\dots,x_s$.
\item[\rm (2)] The edges removed by the operation {\rm(zag-5)} are a part of the bypasses added in {\rm(zag-4)} or edges appearing in some irrelevant parts that are intersections between pairs of zigzag paths $y_1,\dots,y_t$.
\end{itemize}
\end{Lemma}

Before showing the next proposition, we introduce some notation.

\begin{setting}\label{setting_PMs_z}
Let $z$ be a zigzag path on a consistent dimer model $\Gamma$.
We recall that corner perfect matchings are ordered in the anti-clockwise direction along the vertices of~$\Delta_\Gamma$ (see Section~\ref{subsection_PM}).
Let~$\sfP$,~$\sfP^\prime$ be adjacent corner perfect matchings on~$\Gamma$ such that the difference of~$\sfP$ and~$\sfP^\prime$ contains~$z$ (see Proposition~\ref{zigzag_sidepolygon}).
We assume that~$\sfP$,~$\sfP^\prime$ are ordered with this order, in which case $\sfP\cap z=\Zig(z)$ and $\sfP^\prime\cap z=\Zag(z)$.
Then, we set $\sfP_z\coloneqq \sfP$ and $\sfP^\prime_z\coloneqq (\sfP{\setminus}\Zig(z))\cup\Zag(z)$.
By Proposition~\ref{char_bound}, $\sfP_z$, $\sfP^\prime_z$ are boundary perfect matchings corresponding to certain lattice points on the edge of $\Delta_\Gamma$ whose outer normal vector is~$[z]$,
and we see that the difference of~$\sfP_z$ and~$\sfP^\prime_z$, namely $\sfP_z\cup\sfP^\prime_z{\setminus}\sfP_z\cap\sfP^\prime_z$, forms~$z$.
Thus, by this construction, $h(\sfP^\prime_z,\sfP_z)\in\ZZ^2$ is a primitive lattice element with $\langle[z],h(\sfP^\prime_z,\sfP_z)\rangle=0$.
\end{setting}

\begin{Proposition}\label{zigzag_afterdeform2}
Let $h(\sfP^\prime_z,\sfP_z)$ be a primitive lattice element as above.
We have the following.
\begin{itemize}\itemsep=0pt
\item[\rm (1)] For any zigzag path $x_j$ of $\Gamma$,
there exists a unique zigzag path $x_j^\prime$ of $\nu^\zig_\calX(\Gamma)$ such that for every $j=1,\dots,s$
\begin{displaymath}
[x^{\prime}_j]=[x_j]+\langle [x_j],h(\sfP_z^\prime,\sfP_z)\rangle[z].
\end{displaymath}
\item[\rm (2)] For any zigzag path $y_k$ of $\Gamma$, there exists a unique zigzag path $y_k^\prime$ of $\nu^\zag_\calY(\Gamma)$ such that for every $k=1,\dots,t$
\begin{displaymath}
[y^{\prime}_k]=[y_k]+\langle [y_k],h(\sfP_z^\prime,\sfP_z)\rangle[z].
\end{displaymath}
\end{itemize}
\end{Proposition}

\begin{proof}We consider the case of $\nu^\zig_\calX(\Gamma)$. The case of $\nu^\zag_\calY(\Gamma)$ is similar.

We recall that $|x_j\cap z|=|x_j\cap z_i|$ for all $i=1,\dots,r$, and that this number is denoted by $m_j$ (see Definition~\ref{def_deformation_data}).
We first show that
\begin{equation}\label{desired_eq1}
m_j=\langle [x_j],h(\sfP_z^\prime,\sfP_z)\rangle
\end{equation}
for $j=1,\dots,s$. Let $p_{x_j}$ be the path of the quiver $Q_\Gamma$ going along the left side of~$x_j$ (see Observation~\ref{height_zigzag}).
In particular, considering $p_{x_j}$ as an element in $\rmH_1(\TT)$, we have $[p_{x_j}]=[x_j]$.
By our assumption, the intersections~$x_j\cap z$ are contained in $\Zag(z)=\sfP^\prime\cap z$.
Thus, $p_{x_j}$ crosses~$z$ at a~zig of~$z$.
Since $\sfP\cap z=\Zig(z)$, every time $p_{x_j}$ crosses $z$, the height function $\sfh_{\sfP^\prime,\sfP}$ increases by~$1$.
Since $m_j=|x_j\cap z|$, we have the equation~(\ref{desired_eq1}).

Then, we show that for each $x_j$ where $j=1,\dots,s$, there exists a zigzag path $x^{\prime}_j$ on $\overline{\nu}_\calX^\zig(\Gamma)=\overline{\nu}_\calX^\zig(\Gamma,\{z_1,\dots,z_r\})$ such that
\begin{equation}\label{desired_eq2}
[x^{\prime}_j]=[x_j]+m_j[z].
\end{equation}
We divide $x_j$ into sub-zigzag paths $x_j^{(1)},\dots,x_j^{(m_j)}$. By definition, $x_j^{(1)}$ intersects with $z_r$ at a~zag of~$z_r$.
We denote this zag by $z_r[2m]\coloneqq x_j^{(1)}\cap z_r$, in which case the white (resp.\ black) node that is the end point of $z_r[2m]$ is denoted by $w_r[2m-1]$ (resp.~$b_r[2m+1]$).
By considering the universal cover $\widetilde{\Gamma}$, we naturally define $\widetilde{x}_j$, $\widetilde{x}_j^{(1)}$, $\widetilde{z}_r$, $\widetilde{z}_r[2m]$, $\widetilde{w}_r[2m-1]$, $\widetilde{b}_r[2m+1]$, etc.
Also, we assume that~$\widetilde{z}_r$ is contained in the~$(r,0)$-th deformed part.
Then, $\widetilde{x}_j$ crosses the $(r,0)$-th deformed part in the direction from the $(r+1,0)$-th irrelevant part to the $(r,0)$-th irrelevant part,
in which case the entrance of the $(r,0)$-th deformed part is $\widetilde{b}_r[2m+1]$ and the exit is $\widetilde{w}_r[2m-1]$.

In what follows, we use the same notation used in Observation~\ref{obs_deformed_part3}.
In particular, we pay attention to the zigzag path $\widetilde{x}_j^\prime$ of the universal cover $\overline{\nu}_\calX^\zig(\Gamma)^\sim$ of $\overline{\nu}_\calX^\zig(\Gamma)$, which behaves as follows:
\begin{itemize}\itemsep=0pt
\item[$(A_r)$] If $z_r[2m]\not\in X_r$, then $\widetilde{x}_j^\prime$ goes into the $(r,0)$-th deformed part of $\overline{\nu}_\calX^\zig(\Gamma)^\sim$ from $\widetilde{b}_r[2m+1]$ and goes out from $\widetilde{w}_r[2m-1]$ (see also Figure~\ref{fig_observation_zigzag2}).
After crossing the $(r,0)$-th deformed part, it goes into the $(r,0)$-th irrelevant part, and it behaves in the same manner as $\widetilde{x}_j$ in that part.
\item[$(B_r)$] If $z_r[2m]\in X_r$, then $\widetilde{x}_j^\prime$ goes into the $(r,0)$-th deformed part of $\overline{\nu}_\calX^\zig(\Gamma)^\sim$ from $\widetilde{b}_r[2m+1]$ and goes out from $\widetilde{w}_r[2m-1+2n]=\widetilde{w}_r[2m-1+\ell(z)]$ (see also Figures~\ref{fig_observation_zigzag3} and \ref{fig_observation_zigzag4}).
After crossing the $(r,0)$-th deformed part, it goes into the $(r,0)$-th irrelevant part, and it behaves in the same manner as the shift of $\widetilde{x}_j$, which we denote as $\widetilde{x}_j(1)$, in that part.
\end{itemize}
Then, $\widetilde{x}_j^\prime$ goes into the $(r-1,0)$-th deformed part of $\overline{\nu}_\calX^\zig(\Gamma)^\sim$.
We let $\widetilde{z}_{r-1}[2m^\prime]\coloneqq \widetilde{x}_j\cap \widetilde{z}_{r-1}$ and $\widetilde{z}_{r-1}[2m^{\prime\prime}]\coloneqq \widetilde{x}_j(1)\cap \widetilde{z}_{r-1}$.
We note that on the dimer model $\Gamma$ we have $z_{r-1}[2m^\prime]=z_{r-1}[2m^{\prime\prime}]=x_j^{(1)}\cap z_{r-1}$ by definition.
\begin{itemize}\itemsep=0pt
\item If $z_r[2m]\in X_r$ in the above argument, then $z_{r-1}[2m^\prime]\not\in X_{r-1}$ by the definition of~$X_r$ and~$X_{r-1}$.
(Furthermore, $x_j^{(1)}\cap z_i\not\in X_i$ for any $i\neq r$.)
In this case, $\widetilde{x}_j^\prime$ crosses the $(r-1,0)$-th deformed part of $\overline{\nu}_\calX^\zig(\Gamma)^\sim$ in the same manner as $(A_r)$ above.
Then, it behaves in the same manner as $\widetilde{x}_j(1)$ in the $(r-1,0)$-th irrelevant part.
\item Let $z_r[2m]\not\in X_r$ in the above argument.
\begin{itemize}\itemsep=0pt
\item If $z_{r-1}[2m^\prime]\not\in X_{r-1}$, then $\widetilde{x}_j^\prime$ crosses the $(r-1,0)$-th deformed part of $\overline{\nu}_\calX^\zig(\Gamma)^\sim$ in the same way as~$(A_r)$.
Then, it behaves in the same manner as $\widetilde{x}_j$ in the $(r-1,0)$-th irrelevant part.
\item If $z_{r-1}[2m^\prime]\in X_{r-1}$, then $\widetilde{x}_j^\prime$ crosses the $(r-1,0)$-th deformed part of $\overline{\nu}_\calX^\zig(\Gamma)^\sim$ in the same way as~$(B_r)$.
Then, it behaves in the same manner as $\widetilde{x}_j(1)$ in the~$(r-1,0)$-th irrelevant part.
\end{itemize}
\end{itemize}
Repeating these inductive arguments, we see that $\widetilde{x}^\prime_j$ crosses the $(i,0)$-th deformed part of $\overline{\nu}_\calX^\zig(\Gamma)^\sim$ for $i=r,r-1,\dots,1$ in this order, and goes into the $(1,0)$-th irrelevant part.
In any case, $\widetilde{x}^\prime_j$ behaves in the same manner as $\widetilde{x}_j(1)$ in this irrelevant part.

Then, $\widetilde{x}^\prime_j$ goes into the $(r,-1)$-th deformed part, in which case we consider the sub-zigzag path $x_j^{(2)}$ of~$\Gamma$
and the intersection between $\widetilde{z}_r(-1)$ and $\widetilde{x}_j^{(2)}$ on~$\widetilde{\Gamma}$.
By the same arguments as above, we see that $\widetilde{x}^\prime_j$ crosses the $(i,-1)$-th deformed part of $\overline{\nu}_\calX^\zig(\Gamma)^\sim$ for $i=r,r-1,\dots,1$ in this order.
After that, it goes into $(1,-1)$-th irrelevant part and behaves in the same manner as~$\widetilde{x}_j(2)$ in this irrelevant part.

Repeating these arguments, we finally see that $\widetilde{x}^\prime_j$ crosses the $(i,-m_j+1)$-th deformed part of $\overline{\nu}_\calX^\zig(\Gamma)^\sim$ for $i=r,r-1,\dots,1$ in this order,
and behaves in the same manner as $\widetilde{x}_j(m_j)$ in the $(1,-m_j+1)$-th irrelevant part.
Then, $\widetilde{x}^\prime_j$ goes into the $(r,-m_j)$-th deformed part of $\overline{\nu}_\calX^\zig(\Gamma)^\sim$, in which case
we denote the black node that is the entrance of this deformed part by~$B$.
Since $m_j=|x_j\cap z_i|$, the projection of $\widetilde{x}_j\cap \widetilde{z}_i(-m_j)$ on~$\Gamma$ coincides with $z_i[2m]$, which is the starting edge of our arguments.
Thus, $B$ coincides with $\widetilde{b}_r[2m+1]$ if they are projected onto $\overline{\nu}_\calX^\zig(\Gamma)$,
which means we can follow all edges of the zigzag path $x^\prime_j$ of $\overline{\nu}_\calX^\zig(\Gamma)$.
By these arguments, we see that the slope of~$\widetilde{x}^\prime_j$ changes $[z]$ in each deformed part, and thus we have $[x^{\prime}_j]=[x_j]+m_j[z]$.

Finally, we apply the operations (zig-5) and (join) to $\overline{\nu}_\calX^\zig(\Gamma)$.
Then, we obtain the deformed dimer model $\nu_\calX^\zig(\Gamma)$ and the zigzag path on it having the same slope as $\widetilde{x}^\prime_j$.
This zigzag path is determined uniquely by the construction, and we use the same notation for this zigzag path by abuse of the notation.
Since (zig-5) and (join) do not change the slopes of zigzag paths, we have~(\ref{desired_eq2}).
\end{proof}

Since the slopes of zigzag paths $\nu^\zig_\calX(\Gamma)$ and $\nu^\zag_\calY(\Gamma$) do not depend on
the choice of $X_1,\dots,X_r$ (resp.\ $Y_1,\dots,Y_r$) by Propositions~\ref{deform_vector}, \ref{zigzag_afterdeform1} and \ref{zigzag_afterdeform2},
we obtain the proposition below.
However, we remark that the deformed dimer model depends on the choice of zig-deformation parameters; thus $\nu^\zig_\calX(\Gamma,\{z_1,\dots,z_r\})\not\cong\nu^\zig_{\calX^\prime}(\Gamma,\{z_1,\dots,z_r\})$ in general (the zag version is similar).

\begin{Proposition}\label{other_weight_PMpolygon}
For zig-deformation parameter $\calX^\prime\coloneqq\{X_1^\prime,\dots,X_r^\prime\}$ and
zag-deformation parameter $\calY^\prime\coloneqq\{Y_1^\prime,\dots,Y_r^\prime\}$ different from $\calX$ and $\calY$, we have
\begin{displaymath}
\Delta_{\nu^\zig_\calX(\Gamma,\{z_1,\dots,z_r\})}=\Delta_{\nu^\zig_{\calX^\prime}(\Gamma,\{z_1,\dots,z_r\})}\qquad\text{and}\qquad \Delta_{\nu^\zag_\calY(\Gamma,\{z_1,\dots,z_r\})}=\Delta_{\nu^\zag_{\calY^\prime}(\Gamma,\{z_1,\dots,z_r\})}.
\end{displaymath}
\end{Proposition}

\subsection[Remarks on the extended deformations of hexagonal and rectangular dimer models]{Remarks on the extended deformations of hexagonal\\ and rectangular dimer models}\label{sec_remark_deform}

As we mentioned in Remark~\ref{rem_def_deform}, we can skip the operations (zig-4) and (zig-5) (resp.~(zag-4) and~(zag-5))
in the case of hexagonal and rectangular dimer models (see Definition~\ref{def_hexagonal_square}), as we will see below.
We note that hexagonal and rectangular dimer models are isoradial.
These dimer models have been studied in several papers, with the following results being particularly noteworthy.

\begin{Proposition}[{e.g., \cite{IN,Nak_semisteady,UY}}]\label{char_hexagonal_square}
Let $\Gamma$ be a consistent dimer model. Then, we have the following.
\begin{itemize}\itemsep=0pt
\item[\rm (1)] $\Gamma$ is a hexagonal dimer model if and only if the PM polygon $\Delta_\Gamma$ is a triangle.
\item[\rm (2)] If $\Gamma$ is a rectangular dimer model, then the PM polygon $\Delta_\Gamma$ is a parallelogram.
\end{itemize}
\end{Proposition}

For these nice classes of dimer models, we may skip the operations (zig-4) and (zig-5) (or (zag-4) and (zag-5)) when we apply the extended deformation.

\begin{Proposition}\label{skip_hexagonal_square}
Let $\Gamma$ be a hexagonal or rectangular dimer model.
Then, the extended deformation $\nu_\calX^\zig(\Gamma,\{z_1,\dots,z_r\})$ is defined by the operations {\rm(zig-1)--(zig-3)} and {\rm(join)}.
Similarly, the extended deformation $\nu_\calY^\zag(\Gamma,\{z_1,\dots,z_r\})$ is defined by the operations {\rm(zag-1)--(zag-3)} and {\rm(join)}.

In particular, the extended deformations coincide with the usual deformations:
\begin{gather*}
\nu_\calX^\zig(\Gamma,\{z_1,\dots,z_r\})=\nu_\bfp^\zig(\Gamma,\{z_1,\dots,z_r\})\qquad \text{and} \\
\nu_\calY^\zag(\Gamma,\{z_1,\dots,z_r\})=\nu_\bfq^\zag(\Gamma,\{z_1,\dots,z_r\}).
\end{gather*}
\end{Proposition}

\begin{proof}We will prove the case of the extended zig-deformation. The case of the extended zag-deformation is similar.

The zigzag paths $z_1,\dots,z_r$ have the same slope, and this slope is the outer normal vector of an edge of the PM polygon $\Delta_\Gamma$ by Proposition~\ref{zigzag_sidepolygon}.
\begin{itemize}\itemsep=0pt
\item[\rm (1)] Let $\Gamma$ be a hexagonal dimer model.
Then, $\Delta_\Gamma$ is a triangle by Proposition~\ref{char_hexagonal_square}(1).
Let $e_1$, $e_2$, $e_3$ be the edges of $\Delta_\Gamma$ ordered cyclically in the anti-clockwise direction.
We may assume that the slopes of $z_1,\dots,z_r$ are the outer normal vector of~$e_1$.
We then consider the zigzag paths $x_1,\dots,x_s$ (resp.\ $y_1,\dots,y_t$) intersecting with $z_i$ at zags (resp.\ zigs) of~$z_i$.
Since $\Gamma$ is isoradial, it is properly ordered.
Thus, by Proposition~\ref{zigzag_sidepolygon} we see that $[x_1]=\cdots=[x_s]$ (resp.\ $[y_1]=\cdots=[y_t]$), and $[x_j]$ (resp.~$[y_k]$) is the outer normal vector of $e_2$ (resp.\ of~$e_3$).

\item[\rm (2)] Let $\Gamma$ be a rectangular dimer model.
Then, $\Delta_\Gamma$ is a parallelogram by Proposition~\ref{char_hexagonal_square}(2).
Let $e_1$, $e_2$, $e_3$, $e_4$ be the edges of $\Delta_\Gamma$ ordered cyclically in the anti-clockwise direction.
In particular, $\{e_1,e_3\}$ and $\{e_2,e_4\}$ are pairs of edges that are parallel.
We may assume that the slopes of $z_1,\dots,z_r$ are the outer normal vector of $e_1$, in which case
the zigzag paths having the slope $-[z_i]$ correspond to~$e_3$.
Then, in a similar way as above, we have the zigzag paths $x_1,\dots,x_s$ (resp.\ $y_1,\dots,y_t$) such that $[x_j]$ (resp.~$[y_k]$) is the outer normal vector of~$e_2$ (resp.\ $e_4$).
\end{itemize}
In both cases, we see that any pair of zigzag paths in $x_1,\dots,x_s$ (resp.\ $y_1,\dots,y_t$) does not have intersections on the universal cover by Lemma~\ref{slope_linearly_independent} because~$\Gamma$ is isoradial.

Next, we consider $\overline{\nu}_\calX^\zig(\Gamma,\{z_1,\dots,z_r\})$ for the case of $r\neq1$ (see Remark~\ref{r=1case} for the case of $r=1$).
By Lemma~\ref{bypass_removable} and the fact that there is no intersection between $x_1,\dots,x_s$, we see that the edges removed by (zig-5) are bypasses added in (zig-4).
Furthermore, by Observations~\ref{obs_deformed_part2} and \ref{obs_deformed_part3} any zigzag path passing through a bypass takes the form $\widetilde{x}_j^\prime$.
Since the slopes of $x_1,\dots,x_s$ are all the same in our situation, those of $x_1^\prime, \dots, x_s^\prime$ are all the same (see Proposition~\ref{zigzag_afterdeform2}).
Thus, any bypass on the universal cover of $\overline{\nu}_\calX^\zig(\Gamma,\{z_1,\dots,z_r\})$ is either
\begin{itemize}\itemsep=0pt
\item[\rm (i)] a self-intersection of a zigzag path $\widetilde{x}_j^\prime$, or
\item[\rm (ii)] the intersection of a pair of zigzag paths $\widetilde{x}_j^\prime$, $\widetilde{x}_{j^\prime}^\prime$ with $[x_j^\prime]=[x_{j^\prime}^\prime]$.
\end{itemize}
We also see that an edge which is either~(i) or~(ii) is certainly a bypass, because of Observation~\ref{obs_deformed_part3} and the fact that such an intersection can not appear in the irrelevant part.
Moreover, if a bypass is the intersection of zigzag paths $\widetilde{x}_j^\prime$ and $\widetilde{x}_{j^\prime}^\prime$, then they have another intersection because $[x_j]=[x_{j^\prime}]$, and such an intersection is also a bypass.
Thus, if there exists a bypass that can not be removed by (zig-5), the consistency condition is prevented.
Therefore, we can remove all bypasses added in (zig-4) by using (zig-5), and hence we may skip these operations.
\end{proof}

\section[Combinatorial mutations of the PM polygon are realized\\ by extended deformations]{Combinatorial mutations of the PM polygon are realized\\ by extended deformations}
\label{mutation_vs_deformation}

Throughout this section, we still keep the notation of Sections~\ref{sec_def_deform} and~\ref{sec_def_exdeform} unless otherwise stated.
In this section, we show that the combinatorial mutation of the PM polygon of a consistent dimer model coincides with the PM polygon of the deformed dimer model (see Theorem~\ref{mutation=deformation}).
First, we observe the relationship between the deformation data (see Definition~\ref{def_deformation_data}) and the mutation data (see Definition~\ref{def_mutation_data}).

\begin{setting}
\label{mutation_deformation_setting}
Let $\Gamma$ be a reduced consistent dimer model, and $\Delta_\Gamma$ be the PM polygon of $\Gamma$.
We take a type I zigzag path $z$ of $\Gamma$ with $v\coloneqq [z]\in\ZZ^2$. Then, by Proposition~\ref{zigzag_sidepolygon}
there is an edge $E$ of $\Delta_\Gamma$ whose outer normal vector is $v$.
Since $\Delta_\Gamma$ is determined up to translation, there is ambiguity concerning the position of the origin.
Thus, we fix the origin ${\bf 0}$ for $\Delta_\Gamma$ so that ${\bf 0}\in\Delta_\Gamma$.
Let $w\coloneqq -v$, and consider
\begin{gather*}
h_{\max}= h_{\max}(\Delta_\Gamma,w)\coloneqq{\max}\{\langle w,u\rangle \,|\, u\in \Delta_\Gamma\},
\\
h_{\min}= h_{\min}(\Delta_\Gamma,w)\coloneqq{\min}\{\langle w,u\rangle \,|\, u\in \Delta_\Gamma\}.
\end{gather*}
Now, we let $r\coloneqq-h_{\min}$ and assume that $r\le\big|\calZ_v^\rmI(\Gamma)\big|$.
Since the length of the line segments of~$E$ is~$|E\cap N|-1$ and this is equal to $|\calZ_v(\Gamma)|$ by Proposition~\ref{zigzag_sidepolygon}, we have
\begin{displaymath}
-h_{\min}=r\le|\calZ_v^\rmI(\Gamma)|\le|\calZ_v(\Gamma)|=|E\cap N|-1.
\end{displaymath}
Thus $\Delta_\Gamma$ admits the combinatorial mutation with respect to~$w$ (see Remark~\ref{rem_def_mutation}).
Let $\ell(z)\coloneqq2n$. Then, by Lemma~\ref{zigzag_lem1} we have
\begin{align*}
n=\ell(z)/2&=|\sfP\cap z|+\langle h(\sfP,\sfP_i),w\rangle =|\sfP\cap z|+\langle h(\sfP,\sfP_0)-h(\sfP_i,\sfP_0),w\rangle\\
&=|\sfP\cap z|+\langle h(\sfP,\sfP_0),w\rangle-\langle h(\sfP_i,\sfP_0),w\rangle,
\end{align*}
where $\sfP$ is a perfect matching on $\Gamma$, $\sfP_0$ is the reference perfect matching, and $\sfP_i\in\PM_{\max}(z)$.
Since $h(\sfP_i,\sfP_0)$ is a lattice point on $E$ by Lemma~\ref{number_zigzag_corner}, we have $\langle h(\sfP_i,\sfP_0),w\rangle=h_{\min}$.
If $\sfP\in\PM_{\min}(z)$, then $|\sfP\cap z|=0$ by Lemma~\ref{lem_existence_pm}, and this means that $\langle h(\sfP,\sfP_0),w\rangle=h_{\max}$.
Thus, we have $n=h_{\max}-h_{\min}=\width(\Delta_\Gamma,w)$.

We show how the correspondence between mutation data and deformation data in Table~\ref{mutation_deformation_data}.

\begin{table}[h!]\centering
\begin{tabular}{|c|c|}
\hline
Mutation data&Deformation data \\ \hline
$w$&$-v$ \\ \hline
$h_{\min}$&$-r$\\ \hline
$h_{\max}$&$h$ \\ \hline
$\width(\Delta_\Gamma,w)$&$n$ \\ \hline
\end{tabular}

\caption{Relationships between the mutation data and the deformation data.}\label{mutation_deformation_data}
\end{table}

Using the integers $r$, $h$, we take type I zigzag paths $z_1,\dots,z_r$ and the zig-deformation (resp.\ zag-deformation) parameter~$\calX$ (resp.~$\calY$) with respect to $z_1,\dots,z_r$ as in Setting~\ref{def_deformation_exdata}.
We then have the deformed consistent dimer models $\nu^\zig_\calX(\Gamma)=\nu^\zig_\calX(\Gamma, \{z_1,\dots,z_r\})$
and $\nu^\zag_\calY(\Gamma)=\nu^\zag_\calY(\Gamma, \{z_1,\dots,z_r\})$.

We determine the origin of the PM polygons $\Delta_{\nu^\zig_\calX(\Gamma)}$ and $\Delta_{\nu^\zag_\calY(\Gamma)}$ as follows.
First, there are zigzag paths $y_1^\prime,\dots,y_t^\prime$ on $\nu^\zig_\calX(\Gamma)$ whose slope respectively corresponds to the one of zigzag paths $y_1,\dots, y_t$ on $\Gamma$ by Proposition~\ref{zigzag_afterdeform1}.
Then, we put $\Delta_{\nu^\zig_\calX(\Gamma)}$ on $\Delta_\Gamma$ so that the edges corresponding to $y_1^\prime,\dots,y_t^\prime$ respectively coincide with the edges of $\Delta_\Gamma$ corresponding to $y_1,\dots,y_t$.
We determine the origin for $\Delta_{\nu^\zig_\calX(\Gamma)}$ so that it is in the same position as the origin for $\Delta_\Gamma$.
Considering the zigzag paths on $\nu^\zag_\calY(\Gamma)$ obtained from the zigzag paths $x_1,\dots,x_s$ on $\Gamma$,
we can also determine the origin for $\Delta_{\nu^\zag_\calY(\Gamma)}$.
\end{setting}

\begin{Remark}
\label{rem_typeI_isoradial}
In Setting~\ref{mutation_deformation_setting}, we assumed that $r\le|\calZ_v^\rmI(\Gamma)|$ for defining the deformation data.
As we mentioned in Remark~\ref{rem_typeI_to_typeII}, even if the number of type I zigzag paths is insufficient,
we can sometimes change a type II zigzag path into a type I zigzag path without changing the PM polygon by using mutations of dimer models (see Appendix~\ref{app_mutationdimer}).
Moreover, it is known that for a given lattice polygon $P$ there exists an isoradial dimer model giving $P$ as the PM polygon by~\cite{Gul}, in which case all zigzag paths are type I (see Definition~\ref{def_isoradial}), and hence $\big|\calZ_v^\rmI(\Gamma)\big|=|\calZ_v(\Gamma)|$.
Thus, if $-h_{\min}\le|E\cap N|-1$ we can find a certain isoradial dimer model $\Gamma$ satisfying $-h_{\min}=r\le|\calZ_v^{\rmI}(\Gamma)|=|E\cap N|-1$.
\end{Remark}

For any edge $E$ of $\Delta_\Gamma$ as in Setting~\ref{mutation_deformation_setting}, we take a primitive lattice element $u_E\in N$ such that $\langle w,u_E\rangle=0$.
Here, there are two choices of~$u_E$ and we fix~$u_E$ as follows.
We recall the primitive lattice element $h(\sfP^\prime_z,\sfP_z)$ given in Settings~\ref{setting_PMs_z}, which satisfies $\langle[z],h(\sfP^\prime_z,\sfP_z)\rangle=0$.
We let $u_E\coloneqq h(\sfP^\prime_z,\sfP_z)$, and hence $\langle w,u_E\rangle=\langle-[z],u_E\rangle=0$.
We set the line segment $F$ as $F\coloneqq \operatorname{conv}\{{\bf 0},u_E\}$.
Using this with Setting~\ref{mutation_deformation_setting} (and also Table~\ref{mutation_deformation_data}),
our main theorem can be stated as follows.

\begin{Theorem}\label{mutation=deformation}
Let $\Gamma$ be a reduced consistent dimer model with ${\bf 0}\in\Delta_\Gamma$.
Then, we have
\begin{gather*}
\mutation_{w}(\Delta_\Gamma,F)=\Delta_{\nu^\zig_\calX(\Gamma,\{z_1,\dots,z_r\})},\\
\mutation_{w}(\Delta_\Gamma,-F)=\Delta_{\nu^\zag_\calY(\Gamma,\{z_1,\dots,z_r\})}.
\end{gather*}
\end{Theorem}

\begin{proof}We will prove the first equation. The other one follows from a similar argument.

First, we show that
\begin{displaymath}
\varphi(\Delta_\Gamma^*)=\Delta_{\nu^\zig_\calX(\Gamma,\{z_1,\dots,z_r\})}^*,
\end{displaymath}
where $\varphi=\varphi_{w,F}$ is the map given in (\ref{mutation_M}).

Let $E_1\coloneqq E, E_2,\dots,E_m$ be the edges of $\Delta_\Gamma$ ordered cyclically in the anti-clockwise direction.
As we mentioned in Setting~\ref{mutation_deformation_setting}, we suppose that ${\bf 0}\in\Delta_\Gamma$.
Let $w_1\coloneqq w,w_2, \dots, w_m$ be inner normal vectors corresponding to $E_1,\dots,E_m$, respectively (see Figure~\ref{PMpolygon_normalvec}).
Also, we let $v_i=-w_i$ for $i=1,\dots,r$, which are the outer normal vectors corresponding to $E_i$.
We then consider $u\in\Delta_\Gamma$ such that $\langle w_1,u\rangle=h_{\max}(\Delta_\Gamma,w_1)=h$; that is, we consider $w_{h_{\max}}(\Delta_\Gamma)$, which is either a~vertex or an edge of $\Delta_\Gamma$.
If $w_{h_{\max}}(\Delta_\Gamma)$ is an edge, we easily see that it is parallel to $E$, in which case we may write $E_a\coloneqq w_{h_{\max}}(\Delta_\Gamma)$ for some $1<a<m$.
If $w_{h_{\max}}(\Delta_\Gamma)$ is a vertex, we set the edges intersecting at $w_{h_{\max}}(\Delta_\Gamma)$ as $E_{a-1}$, $E_{a+1}$ and set $E_a=\varnothing$ where $1<a<m$.
Recall that by Proposition~\ref{zigzag_sidepolygon}, for each $E_i\neq\varnothing$ there exist zigzag paths on $\Gamma$ such that the slopes coincide with $v_i$,
and the set of such zigzag paths is denoted by $\calZ_{v_i}=\calZ_{v_i}(\Gamma)$.

\begin{figure}[h!]\centering
\scalebox{0.9}{
\begin{tikzpicture}[sarrow/.style={-latex, very thick}]
\foreach \name/\r in {0/2,1/2,2/2,3/2,4/2,5/2,6/2,7/2}{\coordinate (v\name) at (22.5+45*\name:\r cm);}
\draw [line width=0.05cm] (v0)--(v1)--(v2)--(v3); \draw [line width=0.05cm] (v4)--(v5)--(v6)--(v7);

\path (v3) ++(-90:0.3cm) coordinate (v3-); \path (v4) ++(90:0.3cm) coordinate (v4+);
\draw [line width=0.05cm, dotted] (v3-)--(v4+);
\path (v0) ++(-90:0.3cm) coordinate (v0-); \path (v7) ++(90:0.3cm) coordinate (v7+);
\draw [line width=0.05cm, dotted] (v0-)--(v7+);

\foreach \name in {1,2,3,5,6,7}{\coordinate (vec\name) at (45*\name:1.8 cm);}
\foreach \name in {1,2,3,5,6,7}{\draw[red, sarrow, line width=0.05cm] (vec\name)-- ++(180+45*\name:0.7cm);}

\path (vec1) ++(-90:0.4cm) node[red] {\footnotesize$w_m$};
\path (vec2) ++(-90:0.85cm) node[red] {\footnotesize$w_1=w$};
\path (vec3) ++(-90:0.4cm) node[red] {\footnotesize$w_2$};
\path (vec5) ++(90:0.5cm) node[red] {\footnotesize$w_{a-1}$};
\path (vec6) ++(90:0.85cm) node[red] {\footnotesize$w_a$};
\path (vec7) ++(90:0.5cm) node[red] {\footnotesize$w_{a+1}$};
\filldraw [blue] (0,0) circle [radius=0.05cm];
\draw[blue, sarrow, line width=0.05cm] (0,0)-- ++(180:0.7cm) node[inner sep=0.5pt, circle, midway, xshift=0cm, yshift=0.2cm, blue] {\footnotesize$u_E$};
\node at (0,-0.25) {\footnotesize${\bf 0}$};
\end{tikzpicture}
}
\caption{The PM polygon $\Delta_\Gamma$ and its inner normal vectors (the case where the origin is contained in the strict interior of $\Delta_\Gamma$).}
\label{PMpolygon_normalvec}
\end{figure}

First, we consider the edge $E_1$ and zigzag paths in $\calZ_{v_1}=\calZ_{(-w)}$.
By definition, we have $\langle -v_1,u_E\rangle=0$ and $\{z_1,\dots,z_r\}\subseteq\calZ_{(-w)}^\rmI\subseteq \calZ_{(-w)}$.
If $|\calZ_{(-w)}|>r$, then there exists a zigzag path in $\calZ_{(-w)}$ that is not in $\{z_1,\dots,z_r\}$.
If $E_a\neq\varnothing$, we have zigzag paths in $\calZ_{v_a}$.
Since $E_1$ and $E_a$ are parallel, $v_1$ and $v_a$ are linearly dependent, and hence $\langle v_a,u_E\rangle=0$.
Then, we see that zigzag paths in $\calZ_{v_a}$ do not intersect with any type I zigzag path $z$ satisfying $[z]=v_1$ in the universal cover (see Lemma~\ref{slope_linearly_independent}).

Next, we consider the edges $E_2,\dots,E_{a-1}$ of $\Delta_\Gamma$ and zigzag paths in $\calZ_1\coloneqq\calZ_{v_2}\cup\cdots\cup\calZ_{v_{a-1}}$.
We see that $v_i$ with $i=2,\dots,a-1$ satisfies $\langle -v_i,u_E\rangle<0$ by a choice of the edges $E_2,\dots,E_{a-1}$.
By the same argument as in the proof of Proposition~\ref{prop_nondegenerate}, we see that
the zigzag paths in $\calZ_1$ intersect with a type I zigzag path $z$ satisfying $[z]=v_1$ precisely once in the universal cover (see Lemma~\ref{slope_linearly_independent}); specifically, they intersect with $z$ in $\Zag(z)$.
Thus, we have $\calZ_1=\{x_1,\dots,x_s\}$, and for each $i=2,\dots ,a-1$ the vector $v_i$ satisfies $v_i=[x_j]$ for some $j=1,\dots,s$.

Next, we consider the edges $E_{a+1},\dots,E_m$ of $\Delta_\Gamma$ and zigzag paths in $\calZ_2\coloneqq\calZ_{v_{a+1}}\cup\cdots\allowbreak\cup\calZ_{v_m}$.
We see that~$v_i$ with $i=a+1,\dots,m$ satisfies $\langle -v_i,u_E\rangle>0$ by the choice of the edges $E_{a+1},\dots,E_m$.
By a similar argument as above, we see that the zigzag paths in $\calZ_2$ intersect with a type I zigzag path $z$ satisfying $[z]=v_1$ precisely once in the universal cover; specifically, they intersect with $z$ in $\Zig(z)$.
Thus, we have $\calZ_2=\{y_1,\dots,y_t\}$, and for each $i=a+1,\dots,m$ the vector $v_i$ satisfies $v_i=[y_k]$ for some $k=1,\dots,t$.

Collectively, we see that any zigzag path of $\Gamma$ takes one of the following forms:
\begin{itemize}\itemsep=0pt
\item $z_1,\dots,z_r$,
\item $z_1^\prime,\dots,z_{r^\prime}^\prime$ contained in $\calZ_{(-w)}{\setminus}\{z_1,\dots,z_r\}$ for $w=-v_1$ with $\langle w,u_E\rangle=0$ if $|\calZ_{(-w)}|>r$,
\item $z_1^{\prime\prime},\dots,z_{r^{\prime\prime}}^{\prime\prime}$ contained in $\calZ_w$ for $w=-v_1$ with $\langle w,u_E\rangle=0$ if $E_a\neq\varnothing$,
\item $x_j$ where $j=1,\dots,s$, in which case it satisfies $\langle-[x_j],u_E\rangle<0$,
\item $y_k$ where $k=1,\dots,t$, in which case it satisfies $\langle-[y_k],u_E\rangle>0$.
\end{itemize}
The slopes of these zigzag paths give the supporting hyperplanes of $\Delta_\Gamma$ by Proposition~\ref{zigzag_sidepolygon}. More precisely, if $\Delta_\Gamma$ contains the origin ${\bf 0}$ as an interior lattice point, then
\begin{displaymath}
H_{-[\zeta],\ge -k_\zeta}=\{u\in N_\RR\,|\,\langle-[\zeta],u\rangle\ge -k_\zeta\}
\end{displaymath}
is the supporting hyperplane of $\Delta_\Gamma$ for any zigzag path $\zeta$ of $\Gamma$ and a certain positive integer $k_\zeta$.
If the origin ${\bf 0}$ lies on the boundary of $\Delta_\Gamma$, $k_\zeta$ is replaced by $0$ for the zigzag paths corresponding to the edges that contain ${\bf 0}$.
By Proposition~\ref{prop_dual_polytope} and its proof, $\Delta_\Gamma^*$ can be written as $\Delta_\Gamma^*=Q+C$, where $Q$ is a polygon and $C$ is a polyhedral cone.
Since the set of the slopes of zigzag paths of $\Gamma$ coincides with $\{v_1,\dots,v_m\}$ if we identify the same slopes,
we see that the set $\{u_1,\dots,u_p,u_1^\prime,\dots,u_q^\prime\}$, which generates $Q$ and $C$ in the proof of Proposition~\ref{prop_dual_polytope},
is given by $\{w_1,\dots,w_m\}$ in our situation.
In what follows, we assume that ${\bf 0}$ is contained in the strict interior of $\Delta_\Gamma$, in which case $\Delta_\Gamma^*=Q$ and
$Q=\operatorname{conv}\big(\big\{\frac{1}{k_1}w_1,\dots,\frac{1}{k_m}w_m\big\}\big)$ for some positive integers $k_i$,
giving the supporting hyperplanes $H_{w_i,\ge-k_i}$ of $\Delta_\Gamma$ for $i=1,\dots,m$.
We remark that $\frac{1}{k_a}w_a$ appears in the above generating set if $E_a\neq\varnothing$.
Since $\langle w_i,u_E\rangle\ge0$ for $i=1,a, a+1,\dots, m$ and $\langle w_i,u_E\rangle<0$ for $i=2,\dots, a-1$,
we see that
\begin{gather}
\varphi(\Delta_\Gamma^*)=\operatorname{conv}\bigg(\bigg\{\frac{1}{k_1}w_1,\frac{1}{k_2}w_2^\prime,\dots,\frac{1}{k_{a-1}}w_{a-1}^\prime,
\frac{1}{k_a}w_a \ \text{(if $E_a\neq\varnothing$)}, \nonumber\\
\hphantom{\varphi(\Delta_\Gamma^*)=\operatorname{conv}\bigg(\bigg\{}{}
\frac{1}{k_{a+1}}w_{a+1}, \dots,\frac{1}{k_m}w_m\bigg\}\bigg),\label{generating_set_dual}
\end{gather}
where $w_i^\prime\coloneqq w_i-\langle w_i,u_E\rangle w$ for $i=2,\dots,a-1$.
We also note that when $E_a=\varnothing$, we can take the positive integer $k_a$ so that the line $\{u\in N_\RR\,|\,\langle v_1,u\rangle= -k_a\}$, which is parallel to $E_1$, passes through the vertex of $\Delta_\Gamma$ that is the intersection of $E_{a-1}$ and $E_{a+1}$.
By the choice of $k_a$, we have $\langle \frac{1}{k_a}v_1,u\rangle\ge -1$ for any $u\in \Delta_\Gamma$; thus $\frac{1}{k_a}v_1=\frac{1}{k_a}(-w_1)\in\Delta_\Gamma^*$ and hence $\frac{1}{k_a}v_1=\frac{1}{k_a}(-w_1)\in\varphi(\Delta_\Gamma^*)$.

We then consider the deformed dimer model $\nu^\zig_\calX(\Gamma)=\nu^\zig_\calX(\Gamma,\{z_1,\dots,z_r\})$.
By Observation~\ref{obs_deformed_part}, the lift of any zigzag path of the form $z_i^\prime$ or $z_i^{\prime\prime}$ on the universal cover is contained in some irrelevant part, and hence it does not change even if we apply the extended deformation.
Also, by Proposition~\ref{zigzag_afterdeform2}, we have the zigzag paths $x_1^\prime,\dots,x_s^\prime$ on $\nu^\zig_\calX(\Gamma)$ satisfying
\begin{displaymath}
-[x^{\prime}_j]=-[x_j]-\langle [x_j],h(\sfP_z^\prime,\sfP_z)\rangle[z]=w_i-\langle w_i,u_E\rangle w
\end{displaymath}
for $j=1,\dots,s$ and some $i=2,\dots,a-1$.
Furthermore, by Proposition~\ref{zigzag_afterdeform1}, we have the zigzag paths $y_1^\prime,\dots,y_t^\prime$ on $\nu^\zig_\calX(\Gamma)$ satisfying
\begin{displaymath}
-[y_k^\prime]=-[y_k]=w_i
\end{displaymath}
for $k=1,\dots,t$ and some $i=a+1,\dots,m$.
Thus, we see that the zigzag paths $x_j, y_k$ vary as they satisfy the condition (\ref{mutation_M}) when we apply the extended deformation $\nu^\zig_\calX$ to $\Gamma$.
In addition, we have the zigzag path of the form $z_{i,j}$ defined in (zig-3).
Thus, the zigzag paths on the consistent dimer model $\nu^\zig_\calX(\Gamma)$ are
\begin{gather*}
\{z_i^\prime\}_{1\le i\le r^\prime}\quad \text{(if $|\calZ_{(-w)}|>r$)},\qquad \{z_i^{\prime\prime}\}_{1\le i\le r^{\prime\prime}}\quad \text{(if $E_a\neq\varnothing$)},\\
\{x_j^\prime\}_{1\le j\le s},\qquad \{y_k^\prime\}_{1\le k\le t},\qquad \text{and} \qquad \{z_{i,j}\}_{\substack{1\le i\le r \\ 1\le j\le p_i}}.
\end{gather*}
By the description of their slopes and Proposition~\ref{zigzag_sidepolygon}, we see that the inner normal vectors of~$\Delta_{\nu^\zig_\calX(\Gamma)}$ are
\begin{displaymath}
\{w_1,w_2^\prime,\dots,w_{a-1}^\prime,w_a=-w_1, w_{a+1}, \dots,w_m\},
\end{displaymath}
and these vectors give the supporting hyperplanes of $\Delta_{\nu^\zig_\calX(\Gamma)}$ just like for $\Delta_\Gamma$.
Here, $w_1$ appears in the above set if $|\calZ_{(-w_1)}|=|\calZ_{v_1}|>r$,
but we always have $\frac{1}{k_1}w_1\in \Delta_{\nu^\zig_\calX(\Gamma)}^*$ by the same argument as we used for showing $\frac{1}{k_a}(-w_1)\in\Delta_\Gamma^*$ above, whereas $w_a$ certainly appears since $[z_{i,j}]=-[z_i]=-w_1=w_a$ (see Lemma~\ref{deform_vector}). Thus, we have
\begin{displaymath}
\Delta_{\nu^\zig_\calX(\Gamma)}^*=\operatorname{conv}\bigg(\bigg\{\frac{1}{k_1}w_1,\frac{1}{k_2}w_2^\prime,\dots,\frac{1}{k_{a-1}}w_{a-1}^\prime,
\frac{1}{k_a}w_a, \frac{1}{k_{a+1}}w_{a+1}, \dots,\frac{1}{k_m}w_m\bigg\}\bigg).
\end{displaymath}
By the description (\ref{generating_set_dual}) and the fact that $\frac{1}{k_1}w_1$ and $\frac{1}{k_a}w_a=\frac{1}{k_a}(-w_1)$ are contained in both $\varphi(\Delta_\Gamma^*)$ and $\Delta_{\nu^\zig_\calX(\Gamma)}^*$, we see that $\varphi(\Delta_\Gamma^*)=\Delta_{\nu^\zig_\calX(\Gamma)}^*$.

The case where the origin ${\bf 0}$ lies on the boundary of $\Delta_\Gamma$ can be proved by a similar argument if we consider the hyperplane $\{u\in N_\RR\,|\,\langle-[\zeta],u\rangle\ge 0\}$ instead of $\{u\in N_\RR\,|\,\langle-[\zeta],u\rangle\ge -k_\zeta\}$ for the zigzag paths corresponding to the edges that contain ${\bf 0}$, in which case $-[\zeta]$ will be a~generator of a polyhedral cone~$C$.

By Propositions~\ref{prop_dual_polytope}(ii) and \ref{prop_mutation_dual_polytope}, we conclude that
$\mutation_{w}(\Delta_\Gamma,F)=\Delta_{\nu^\zig_\calX(\Gamma)}$.
\end{proof}

Since $\mutation_w(\Delta_\Gamma,F)\cong\mutation_w(\Delta_\Gamma,-F)$ (see Remark~\ref{rem_def_mutation}), we immediately have the following.

\begin{Corollary}
\label{cor_unimodular_deformedDM}
We have
\begin{displaymath}
\Delta_{\nu^\zig_\calX(\Gamma, \{z_1,\dots,z_r\})}\cong\Delta_{\nu^\zag_\calY(\Gamma, \{z_1,\dots,z_r\})}.
\end{displaymath}
That is, they are $\GL(2,\ZZ)$-equivalent.
\end{Corollary}

We next show that the extended zig-deformation and zag-deformation are mutually inverse operations on the level of the associated PM polygon
as in Corollary~\ref{cor_mut_inverse_on_polygon} below.
However, it is not true on the level of dimer models as we saw in Example~\ref{zigzag_counterEX}.

\begin{setting}
\label{setting_inverse}
Let $\nu^\zig_\calX(\Gamma,\{z_1,\dots,z_r\})$ be the reduced consistent dimer model.
We consider the following deformation data for $\nu^\zig_\calX(\Gamma,\{z_1,\dots,z_r\})$.
Let $z_{i,j}$ be a type I zigzag path which satisfies $[z_{i,j}]=-v\eqqcolon w$ (see Proposition~\ref{deform_vector}).
Since $\{z_{i,j}\}_{\substack{1\le i\le r \\ 1\le j\le p_i}}$ is the subset of type I zigzag paths, we have $|\calZ_{w}^\rmI(\nu^\zig_\calX(\Gamma))|\ge\sum_{i=1}^rp_i=h$.
We take a subset $\{z_1^\prime,\dots,z_h^\prime\}$ of type~I zigzag paths of~$\nu^\zig_\calX(\Gamma)$, and perform the same procedure as in Setting~\ref{def_deformation_exdata}.
Then we have the zag-deformation parameter $\calY^\prime=\{Y_1^\prime,\dots,Y_h^\prime\}$ of the weight $\bfq^\prime=(q_1^\prime,\dots,q_h^\prime)$ with $\sum_{i=1}^hq_i^\prime=r$.
We can use the same arguments for the case of $\nu^\zag_\calY(\Gamma,\{z_1,\dots,z_r\})$.
Specifically, we take a subset $\{z_1^{\prime\prime},\dots,z_h^{\prime\prime}\}$ of type I zigzag paths on $\nu^\zag_\calY(\Gamma)$ and have the zig-deformation parameter $\calX^\prime=\{X_1^\prime,\dots,X_h^\prime\}$ of the weight $\bfp^\prime=(p_1^\prime,\dots,p_h^\prime)$ with $\sum_{i=1}^hp_i^\prime=r$.
\end{setting}

\begin{Corollary}\label{cor_mut_inverse_on_polygon}
Let the notation be the same as in Setting~{\rm \ref{setting_inverse}}. Then, we have
\begin{displaymath}
\Delta_{\nu^\zag_{\calY^\prime}\left(\nu^\zig_\calX(\Gamma, \{z_i\}_{i=1}^r),\{z_\ell^\prime\}_{\ell=1}^h\right)}=\Delta_\Gamma \qquad \text{and} \qquad
\Delta_{\nu^\zig_{\calX^\prime}\left(\nu^\zag_\calY(\Gamma, \{z_i\}_{i=1}^r),\{z_\ell^{\prime\prime}\}_{\ell=1}^h\right)}=\Delta_\Gamma.
\end{displaymath}
\end{Corollary}

\begin{proof}
This follows from Proposition~\ref{properties_mutation_polygon}(1) and Theorem~\ref{mutation=deformation}.
\end{proof}

As we mentioned in Section~\ref{sec_intro}, the combinatorial mutation of Fano polygons is important from the viewpoint of mirror symmetry and the classification of Fano manifolds.
To study the combinatorial mutation of Fano polygons using extended deformations of dimer models, we assume that the polygons $\Delta_\Gamma$, $\Delta_{\nu^\zig_\calX(\Gamma, \{z_1,\dots,z_r\})}$ and $\Delta_{\nu^\zag_\calY(\Gamma, \{z_1,\dots,z_r\})}$ contain the origin in their strict interiors.
Then, we have the following corollary.

\begin{Corollary}\label{cor_Fano_deformedDM}
Let the notation be the same as above. Then, we see that~$\Delta_\Gamma$ is Fano if and only if $\Delta_{\nu^\zig_\calX(\Gamma, \{z_1,\dots,z_r\})}$ $($resp.\ $\Delta_{\nu^\zag_\calY(\Gamma, \{z_1,\dots,z_r\})}$$)$ is Fano.
\end{Corollary}

\begin{proof}This follows from Proposition~\ref{properties_mutation_polygon}(2) and Theorem~\ref{mutation=deformation}.
\end{proof}

Note that since any lattice polygon can be obtained as the PM polygon of a reduced consistent dimer model by Theorem~\ref{existence_dimer},
we can discuss the combinatorial mutation of a polygon in terms of the extended zig-deformations and zag-deformation by the results shown in this section.

\appendix
\section{Mutations of dimer models}\label{app_mutationdimer}

In this section, we introduce another operation, called the \emph{mutation of dimer models}.
From the viewpoint of physics, dimer models and their mutations correspond to quiver gauge theories and Seiberg duality.
The mutation of dimer models can be defined for each quadrangle face of a~dimer model, and the operation called \emph{spider move} (e.g., \cite{Boc_abc, GK}), which is the inverse operation shown in Figure~\ref{fig_spider}, is the main component used in defining this mutation.

\begin{figure}[h!]\centering
{\scalebox{0.9}{
\begin{tikzpicture}

\node at (3.3,0.3) {spider move} ;
\draw[<->, line width=0.03cm] (2.1,0)--(4.5,0);

\node (mutate_a) at (0,0)
{\scalebox{0.7}{
\begin{tikzpicture}
\coordinate (B1) at (-1,0); \coordinate (B2) at (1,0);
\coordinate (W1) at (-2,0); \coordinate (W2) at (0,-2); \coordinate (W3) at (2,0);
\coordinate (W4) at (0,2);

\draw[line width=0.05cm] (W1)--(B1); \draw[line width=0.05cm] (W2)--(B1);
\draw[line width=0.05cm] (W4)--(B1);
\draw[line width=0.05cm] (W2)--(B2); \draw[line width=0.05cm] (W3)--(B2);
\draw[line width=0.05cm] (W4)--(B2);

\filldraw [line width=0.05cm, fill=black] (B1) circle [radius=0.16] ; \filldraw [line width=0.05cm, fill=black] (B2) circle [radius=0.16] ;

\draw [line width=0.05cm, fill=white] (W1) circle [radius=0.16] ; \draw [line width=0.05cm, fill=white] (W2) circle [radius=0.16] ;
\draw [line width=0.05cm, fill=white] (W3) circle [radius=0.16] ; \draw [line width=0.05cm, fill=white] (W4) circle [radius=0.16] ;
\end{tikzpicture} }} ;

\node (mutate_b) at (6.5,0)
{\scalebox{0.7}{
\begin{tikzpicture}

\coordinate (B1) at (0,1); \coordinate (B2) at (0,-1);
\coordinate (W1) at (-2,0); \coordinate (W2) at (0,-2); \coordinate (W3) at (2,0);
\coordinate (W4) at (0,2);

\draw[line width=0.05cm] (W1)--(B1); \draw[line width=0.05cm] (W3)--(B1);
\draw[line width=0.05cm] (W4)--(B1);
\draw[line width=0.05cm] (W1)--(B2); \draw[line width=0.05cm] (W2)--(B2);
\draw[line width=0.05cm] (W3)--(B2);

\filldraw [line width=0.05cm, fill=black] (B1) circle [radius=0.16] ; \filldraw [line width=0.05cm, fill=black] (B2) circle [radius=0.16] ;

\draw [line width=0.05cm, fill=white] (W1) circle [radius=0.16] ; \draw [line width=0.05cm, fill=white] (W2) circle [radius=0.16] ;
\draw [line width=0.05cm, fill=white] (W3) circle [radius=0.16] ; \draw [line width=0.05cm, fill=white] (W4) circle [radius=0.16] ;
\end{tikzpicture} }} ;

\end{tikzpicture}
}}
\caption{}\label{fig_spider}
\end{figure}

We note that there are two types of the spider move (and hence the mutation) depending on the color of the two interior nodes.

\begin{Definition}[mutation of dimer models]\label{def_mutation}
Let $\Gamma$ be a dimer model. We pick a quadrangle face $f\in\Gamma_2$.
Then, the \emph{mutation} of $\Gamma$ at $f$, denoted by $\mu_f(\Gamma)$, is the operation consisting of the following procedures:
\begin{itemize}\itemsep=0pt
\item[(I)] If there exist black nodes on the boundary of $f$ that are not $3$-valent,
we apply split moves to those nodes and make them $3$-valent as shown in Figure~\ref{fig_mutation1}.
\item[(II)] We apply the spider move to $f$ (see Figure~\ref{fig_spider}).
\item[(III)] If the resulting dimer model contains $2$-valent nodes, we remove them by applying the join moves.
\end{itemize}
\end{Definition}

\begin{figure}[h!]\centering
{\scalebox{0.9}{
\begin{tikzpicture}

\node at (3.1,0.3) {split move} ;
\draw[->, line width=0.03cm] (2.1,0)--(4.3,0);

\node (mutate_a) at (0,0)
{\scalebox{0.7}{
\begin{tikzpicture}
\node at (0,0) {{\Large$f$}} ;
\coordinate (B1) at (-1,0); \coordinate (B2) at (1,0);
\coordinate (W1) at (0,-2);
\coordinate (W2) at (0,2);
\draw[line width=0.05cm] (-2,0)--(B1);
\draw[line width=0.05cm] (W1)--(B1);
\draw[line width=0.05cm] (W2)--(B1);
\draw[line width=0.05cm] (W1)--(B2);
\draw[line width=0.05cm] (B2)--(2,0.5); \draw[line width=0.05cm] (B2)--(2,-0.5);
\draw[line width=0.05cm] (W2)--(B2);
\filldraw [line width=0.05cm, fill=black] (B1) circle [radius=0.16] ; \filldraw [line width=0.05cm, fill=black] (B2) circle [radius=0.16] ;
\draw [line width=0.05cm, fill=white] (W1) circle [radius=0.16] ;
\draw [line width=0.05cm, fill=white] (W2) circle [radius=0.16] ;
\end{tikzpicture} }} ;

\node (mutate_b) at (7,0)
{\scalebox{0.7}{
\begin{tikzpicture}
\node at (0,0) {{\Large$f$}} ;
\coordinate (B1) at (-1,0); \coordinate (B2) at (1,0);
\coordinate (W1) at (0,-2);
\coordinate (W2) at (0,2);

\coordinate (W3) at (2,0);
\coordinate (B3) at (3,0);
\draw[line width=0.05cm] (-2,0)--(B1);
\draw[line width=0.05cm] (W1)--(B1); \draw[line width=0.05cm] (W2)--(B1);
\draw[line width=0.05cm] (W1)--(B2); \draw[line width=0.05cm] (W2)--(B2);

\draw[line width=0.05cm] (B3)--(4,0.5); \draw[line width=0.05cm] (B3)--(4,-0.5);
\draw[line width=0.05cm] (B2)--(W3); \draw[line width=0.05cm] (B3)--(W3);

\filldraw [line width=0.05cm, fill=black] (B1) circle [radius=0.16] ; \filldraw [line width=0.05cm, fill=black] (B2) circle [radius=0.16] ;
\filldraw [line width=0.05cm, fill=black] (B3) circle [radius=0.16] ;
\draw [line width=0.05cm, fill=white] (W1) circle [radius=0.16] ;\draw [line width=0.05cm, fill=white] (W2) circle [radius=0.16] ;
\draw [line width=0.05cm, fill=white] (W3) circle [radius=0.16] ;
\end{tikzpicture} }} ;

\end{tikzpicture}
}}
\caption{}\label{fig_mutation1}
\end{figure}

Applying a mutation at a quadrangle face, we obtain the new dimer model from a given one, although the mutation sometimes induces an isomorphic one.
We also remark that the mutation is an involutive operation; that is, $\mu_f(\mu_f(\Gamma))=\Gamma$ holds.
We say that dimer models $\Gamma$ and $\Gamma^\prime$ are \emph{mutation-equivalent} if they are transformed into
each other by repeating the mutation of dimer models.
Moreover, we also see that the join, split and spider moves do not change the slopes of zigzag paths and preserve conditions in Definition~\ref{def_properly}. Thus we have the following.

\begin{Proposition}\label{mutation_preserve_toric}
A mutation of dimer models turns consistent dimer models into consistent dimer models associated with the same lattice polygon.
\end{Proposition}

Note that it has been conjectured that all consistent dimer models associated with the same lattice polygon are mutation-equivalent.
This conjecture is still open in general (see \cite[pp.~396--397]{Boc_abc}).
We note that partial answers were given in several papers (e.g., \cite{Boc_toricNCCR,GK,HS,Nak}).

\begin{Remark}A mutation of a dimer model is also defined as the dual of the \emph{mutation of a quiver with potential} (= \emph{QP}) in the sense of~\cite{DWZ} (see also \cite[Section~7.2]{Boc_toricNCCR}, \cite[Section~4]{Nak}).
Although we can consider the mutation of a QP for any vertex of the quiver having no loops and 2-cycles, the resulting QP is not necessarily the dual of a dimer model.
To make the resulting quiver the dual of a dimer model, we need to assume that the mutated vertex has two incoming (equivalently, two outgoing) arrows, which is equivalent to assuming that the face of a dimer model corresponding to such a vertex is a quadrangle.
\end{Remark}

In the theory of (extended) deformations of dimer models, type~I zigzag paths are important.
We can use mutations to change a type~II zigzag path into type~I as in the example below.

\begin{Example}\label{mutation_typeII_to_I}
We consider the following dimer model $\Gamma$ (the image on the left). Since the face~$0$ is a quadrangle, we can apply the mutation and obtain the dimer model $\mu_0(\Gamma)$ as follows.

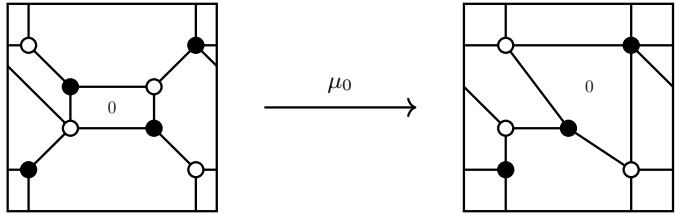
\begin{center}
\newcommand{\edgewidth}{0.055cm}
\newcommand{\nodewidth}{0.055cm}
\newcommand{\noderad}{0.18} 
\newcommand{\arrowwidth}{0.07cm}
\begin{tikzpicture}
\node (DM1) at (0,0)
{\scalebox{0.55}{
\begin{tikzpicture}
\coordinate (W1) at (4,1); \coordinate (W2) at (1,2); \coordinate (W3) at (3,3); \coordinate (W4) at (0,4);
\coordinate (B1) at (0,1); \coordinate (B2) at (3,2); \coordinate (B3) at (1,3); \coordinate (B4) at (4,4);
\draw[line width=\edgewidth] (-0.5,0) rectangle (4.5,5);
\draw[line width=\edgewidth] (B1)--(W2); \draw[line width=\edgewidth] (B2)--(W1); \draw[line width=\edgewidth] (B2)--(W2);
\draw[line width=\edgewidth] (B2)--(W3); \draw[line width=\edgewidth] (B3)--(W2); \draw[line width=\edgewidth] (B3)--(W3);
\draw[line width=\edgewidth] (B3)--(W4); \draw[line width=\edgewidth] (B4)--(W3);
\draw[line width=\edgewidth] (W2)--(-0.5,3.5); \draw[line width=\edgewidth] (B4)--(4.5,3.5);
\draw[line width=\edgewidth] (B1)--(0,0); \draw[line width=\edgewidth] (B1)--(-0.5,1);
\draw[line width=\edgewidth] (W1)--(4,0); \draw[line width=\edgewidth] (W1)--(4.5,1);
\draw[line width=\edgewidth] (B4)--(4.5,4); \draw[line width=\edgewidth] (B4)--(4,5);
\draw[line width=\edgewidth] (W4)--(-0.5,4); \draw[line width=\edgewidth] (W4)--(0,5);
\draw [line width=\nodewidth, fill=black] (B1) circle [radius=\noderad] ;
\draw [line width=\nodewidth, fill=black] (B2) circle [radius=\noderad] ;
\draw [line width=\nodewidth, fill=black] (B3) circle [radius=\noderad] ;
\draw [line width=\nodewidth, fill=black] (B4) circle [radius=\noderad] ;
\draw [line width=\nodewidth, fill=white] (W1) circle [radius=\noderad] ;
\draw [line width=\nodewidth, fill=white] (W2) circle [radius=\noderad] ;
\draw [line width=\nodewidth, fill=white] (W3) circle [radius=\noderad] ;
\draw [line width=\nodewidth, fill=white] (W4) circle [radius=\noderad] ;

\node at (2,2.5) {\Large$0$};
\end{tikzpicture}
} };

\node (DM2) at (6,0)
{\scalebox{0.55}{
\begin{tikzpicture}
\coordinate (W1) at (4,1); \coordinate (W2) at (1,2); \coordinate (W3) at (1,4);
\coordinate (B1) at (1,1); \coordinate (B2) at (2.5,2); \coordinate (B3) at (4,4);
\draw[line width=\edgewidth] (0,0) rectangle (5,5);
\draw[line width=\edgewidth] (B1)--(W2); \draw[line width=\edgewidth] (B2)--(W1); \draw[line width=\edgewidth] (B2)--(W2);
\draw[line width=\edgewidth] (B3)--(W1); \draw[line width=\edgewidth] (B2)--(W3); \draw[line width=\edgewidth] (B3)--(W3);
\draw[line width=\edgewidth] (W2)--(0,3); \draw[line width=\edgewidth] (B3)--(5,3);
\draw[line width=\edgewidth] (B1)--(1,0); \draw[line width=\edgewidth] (B1)--(0,1);
\draw[line width=\edgewidth] (W1)--(4,0); \draw[line width=\edgewidth] (W1)--(5,1);
\draw[line width=\edgewidth] (B3)--(5,4); \draw[line width=\edgewidth] (B3)--(4,5);
\draw[line width=\edgewidth] (W3)--(0,4); \draw[line width=\edgewidth] (W3)--(1,5);
\draw [line width=\nodewidth, fill=black] (B1) circle [radius=\noderad] ;
\draw [line width=\nodewidth, fill=black] (B2) circle [radius=\noderad] ;
\draw [line width=\nodewidth, fill=black] (B3) circle [radius=\noderad] ;
\draw [line width=\nodewidth, fill=white] (W1) circle [radius=\noderad] ;
\draw [line width=\nodewidth, fill=white] (W2) circle [radius=\noderad] ;
\draw [line width=\nodewidth, fill=white] (W3) circle [radius=\noderad] ;

\node at (3,3) {\Large$0$};
\end{tikzpicture}
} };

\path (DM1) ++(0:2cm) coordinate (DM1+); \path (DM2) ++(180:2cm) coordinate (DM2+);
\draw[->, line width=0.03cm] (DM1+)--(DM2+) node[midway,xshift=0cm,yshift=0.3cm] {\small $\mu_0$};
\end{tikzpicture}
\end{center}

We can see that the zigzag path on $\Gamma$ whose slope is $(-1,1)$ or $(1,-1)$ is type II.
On the other hand, we see that $\mu_0(\Gamma)$ is isoradial, and hence all zigzag paths are type I.
\end{Example}

\section{Large examples}\label{app_large_example}

As we mentioned in Remark~\ref{r=1case} and Proposition~\ref{skip_hexagonal_square},
we sometimes skip the operations (zig-4) and (zig-5) (resp.~(zag-4) and~(zag-5)) when we define the extended deformation $\nu^\zig_\calX(\Gamma,\{z_1,\allowbreak\dots,\allowbreak z_r\})$ (resp.~$\nu^\zag_\calY(\Gamma,\{z_1,\dots, z_r\})$) of a consistent dimer model $\Gamma$.
However, as the following example shows, (zig-4) and (zig-5) (resp.~(zag-4) and~(zag-5)) are indispensable to define extended deformations
compatible with combinatorial mutations of polygons, as shown in Theorem~\ref{mutation=deformation}.
For example, we often encounter such a situation when we consider a consistent dimer model whose PM polygon is relatively large.

\begin{Example}
We consider the lattice polygon $P$ shown on the left-hand side of Figure~\ref{ex_large_polygon}.
We assume that the double circle stands for the origin ${\bf 0}$.
We consider the edge $E$ whose primitive inner normal vector is $w=(0,-1)$ with $h_{\min}=-3$ and $h_{\max}=1$.
We take $u_E=(-1,0)$, which satisfies $\langle w,u_E\rangle=0$, and consider the line segment $F=\operatorname{conv}\{{\bf 0},u_E\}$.
Then, the combinatorial mutation $\mut_w(P,F)$ of the polygon $P$ is as shown on the right-hand side of Figure~\ref{ex_large_polygon}.

\begin{figure}[h!]\centering
\begin{tikzpicture}

\node at (0,0)
{\scalebox{0.45}{
\begin{tikzpicture}
\coordinate (v0) at (0,0); \coordinate (v1) at (2,-1); \coordinate (v2) at (2,3); \coordinate (v3) at (-2,3); \coordinate (v4) at (-3,2);
\coordinate (v5) at (-3,1); \coordinate (v6) at (-2,-1);

\draw [step=1, gray] (-4.3,-2.3) grid (3.3,4.3);
\foreach \h/\t in {1/2,2/3,3/4,4/5,5/6,6/1}
{\draw [line width=0.08cm] (v\h)--(v\t) ;}

\draw [line width=0.03cm] (v0) circle [radius=0.25] ;
\foreach \vertex in {0,1,2,3,4,5,6}
{\draw [line width=0.05cm, fill=black] (v\vertex) circle [radius=0.1] ; }
\end{tikzpicture}
}};

\draw[->,line width=0.03cm] (2.5,0)--(4.5,0);
\node at (3.5,0.5) {$\mut_w(-,F)$};

\node at (7,0)
{\scalebox{0.45}{
\begin{tikzpicture}
\coordinate (v0) at (0,0); \coordinate (v1) at (2,-1); \coordinate (v2) at (2,3); \coordinate (v3) at (1,3); \coordinate (v4) at (-1,2);
\coordinate (v5) at (-2,1); \coordinate (v6) at (-3,-1);

\draw [step=1, gray] (-4.3,-2.3) grid (3.3,4.3);
\foreach \h/\t in {1/2,2/3,3/4,4/5,5/6,6/1}
{\draw [line width=0.08cm] (v\h)--(v\t) ;}

\draw [line width=0.03cm] (v0) circle [radius=0.25] ;
\foreach \vertex in {0,1,2,3,4,5,6}
{\draw [line width=0.05cm, fill=black] (v\vertex) circle [radius=0.1] ; }
\end{tikzpicture}
}};
\end{tikzpicture}
\caption{The lattice polygons $P$ and $\mut_w(P,F)$ for $w=(0,-1)$ and $F=\mathrm{conv}\{{\bf 0},(-1,0)\}$.}
\label{ex_large_polygon}
\end{figure}
Next, we consider the dimer model $\Gamma$ shown in Figure~\ref{ex_large_dimer1}.
The zigzag paths on this dimer model $\Gamma$ are depicted in Figure~\ref{ex_large_zigzag}.
One can check that $\Gamma$ is consistent, and hence the PM polygon $\Delta_\Gamma$ coincides with $P$ by Proposition~\ref{zigzag_sidepolygon}.

\newcommand{\basicdimerEx}{
\foreach \blackname/\x/\y in
{1/1/0,2/1/2,3/1/4,4/1/6,5/1/8,6/3/0,7/3/2,8/3/4,9/3/6,10/3/8,11/5/0,12/5/2,13/5/4,14/5/6,15/5/8,
16/7/0,17/7/2,18/7/4,19/7/6,20/7/8,21/9/0,22/9/2,23/9/4,24/9/6,25/9/8}
{\coordinate (B\blackname) at (\x,\y);}
\foreach \whitename/\x/\y in
{1/0/1,2/0/3,3/0/5,4/0/7,5/2/1,6/2/3,7/2/5,8/2/7,9/4/1,10/4/3,11/4/5,12/4/7,13/6/1,14/6/3,15/6/5,16/6/7,
17/8/1,18/8/3,19/8/5,20/8/7,21/10/1,22/10/3,23/10/5,24/10/7}
{\coordinate (W\whitename) at (\x,\y);}

\draw[line width=\edgewidth] (0,0) rectangle (10,8);

\foreach \w/\b in {1/1,1/2,2/2,2/3,3/3,3/4,4/4,4/5,5/1,5/2,5/6,5/7,6/2,6/3,6/7,6/8,7/3,7/4,7/8,7/9,8/4,8/5,8/9,8/10}{\draw [line width=\edgewidth] (W\w)--(B\b);}
\foreach \w/\b in {9/6,9/11,9/12,10/7,10/12,10/13,11/8,11/9,11/13,11/14,12/9,12/10,12/14,12/15}{\draw [line width=\edgewidth] (W\w)--(B\b);}
\foreach \w/\b in {13/11,13/12,13/16,13/17,14/12,14/13,14/17,14/18,15/13,15/14,15/18,15/19,16/14,16/15,16/19,16/20}{\draw [line width=\edgewidth] (W\w)--(B\b);}
\foreach \w/\b in {17/16,17/17,17/21,17/22,18/17,18/18,18/22,18/23,19/18,19/19,19/23,19/24,20/19,20/20,20/24,21/21,21/22,22/22,22/23,23/23,23/24,24/24,24/25}{\draw [line width=\edgewidth] (W\w)--(B\b);}

\foreach \n in {1,2,3,4,5,6,7,8,9,10,11,12,13,14,15,16,17,18,19,20,21,22,23,24,25} {\draw [line width=\nodewidth, fill=black] (B\n) circle [radius=\noderad];}
\foreach \n in {1,2,3,4,5,6,7,8,9,10,11,12,13,14,15,16,17,18,19,20,21,22,23,24} {\draw [line width=\nodewidth, fill=white] (W\n) circle [radius=\noderad];}
}

\begin{figure}[h!]\centering
\scalebox{0.5}{
\begin{tikzpicture}
\newcommand{\edgewidth}{0.05cm} 
\newcommand{\nodewidth}{0.05cm} 
\newcommand{\noderad}{0.17} 
\basicdimerEx
\end{tikzpicture}
}
\caption{A consistent dimer model $\Gamma$ whose PM polygon coincides with $P$.}
\label{ex_large_dimer1}
\end{figure}

\begin{figure}[h!]\centering
\begin{tikzpicture}
\newcommand{\edgewidth}{0.05cm} 
\newcommand{\nodewidth}{0.05cm} 
\newcommand{\noderad}{0.17} 

\node at (0,0) {
\scalebox{0.4}{
\begin{tikzpicture}
\basicdimerEx
\newcommand{\zzwidth}{0.2cm} 
\newcommand{\zzcolor}{red} 
\draw[->, line width=\zzwidth, rounded corners, color=\zzcolor] (B1)--(W1)--(B2)--(W2)--(B3)--(W3)--(B4)--(W4)--(B5);
\draw[->, line width=\zzwidth, rounded corners, color=\zzcolor] (B6)--(W5)--(B7)--(W6)--(B8)--(W7)--(B9)--(W8)--(B10);
\draw[->, line width=\zzwidth, rounded corners, color=\zzcolor] (B11)--(W9)--(B12)--(W10)--(B13)--(W11)--(B14)--(W12)--(B15);
\draw[->, line width=\zzwidth, rounded corners, color=\zzcolor] (B16)--(W13)--(B17)--(W14)--(B18)--(W15)--(B19)--(W16)--(B20);
\node[red] at (0.8,5) {\Huge $z_1$}; \node[red] at (2.8,5) {\Huge $z_2$};
\node[red] at (4.8,5) {\Huge $z_3$}; \node[red] at (6.8,5) {\Huge $z_4$};
\end{tikzpicture}
}};

\node at (5,0) {
\scalebox{0.4}{
\begin{tikzpicture}
\basicdimerEx
\newcommand{\zzwidth}{0.2cm} 
\newcommand{\zzcolor}{red} 
\draw[->, line width=\zzwidth, rounded corners, color=\zzcolor, dotted] (B5)--(W8)--(B4)--(W7)--(B3)--(W6)--(B2)--(W5)--(B1);
\draw[->, line width=\zzwidth, rounded corners, color=\zzcolor, dotted] (B15)--(W16)--(B14)--(W15)--(B13)--(W14)--(B12)--(W13)--(B11);
\draw[->, line width=\zzwidth, rounded corners, color=\zzcolor, dotted] (B20)--(W20)--(B19)--(W19)--(B18)--(W18)--(B17)--(W17)--(B16);
\draw[->, line width=\zzwidth, rounded corners, color=\zzcolor, dotted] (B25)--(W24)--(B24)--(W23)--(B23)--(W22)--(B22)--(W21)--(B21);
\node[red] at (1,5) {\Huge $z_1^\prime$};
\node[red] at (5,5) {\Huge $z_2^\prime$}; \node[red] at (7,5) {\Huge $z_3^\prime$}; \node[red] at (9,5) {\Huge $z_4^\prime$};
\end{tikzpicture}
}};

\node at (10,0) {
\scalebox{0.4}{
\begin{tikzpicture}
\basicdimerEx
\newcommand{\zzwidth}{0.2cm} 
\newcommand{\zzcolor}{blue!50} 
\draw[->, line width=\zzwidth, rounded corners, color=\zzcolor] (W1)--(B1)--(W5)--(B6)--(W9)--(B11)--(W13)--(B16)--(W17)--(B21)--(W21);
\draw[->, line width=\zzwidth, rounded corners, color=\zzcolor] (W2)--(B2)--(W6)--(B7)--(W10)--(B12)--(W14)--(B17)--(W18)--(B22)--(W22);
\draw[->, line width=\zzwidth, rounded corners, color=\zzcolor] (W3)--(B3)--(W7)--(B8)--(W11)--(B13)--(W15)--(B18)--(W19)--(B23)--(W23);
\draw[->, line width=\zzwidth, rounded corners, color=\zzcolor] (W4)--(B4)--(W8)--(B9)--(W12)--(B14)--(W16)--(B19)--(W20)--(B24)--(W24);
\node[blue!50] at (5,6.8) {\Huge $y_1$}; \node[blue!50] at (5,4.8) {\Huge $y_2$};
\node[blue!50] at (5,2.8) {\Huge $y_3$}; \node[blue!50] at (5,0.8) {\Huge $y_4$};
\end{tikzpicture}
}};

\node at (0,-4) {
\scalebox{0.4}{
\begin{tikzpicture}
\basicdimerEx
\newcommand{\zzwidth}{0.2cm} 
\newcommand{\zzcolor}{blue} 
\draw[->, line width=\zzwidth, rounded corners, color=\zzcolor] (B21)--(W17)--(B22)--(W18)--(B23)--(W19)--(B24)--(W20)--(B20)--(W16)--(B15)--(W12)--(B10)--(W8)--(B5)--(W4);
\draw[->, line width=\zzwidth, rounded corners, color=\zzcolor] (W24)--(B25);
\node[blue] at (5,7.2) {\Huge $x_1$};
\end{tikzpicture}
}};

\node at (5,-4) {
\scalebox{0.4}{
\begin{tikzpicture}
\basicdimerEx
\newcommand{\zzwidth}{0.2cm} 
\newcommand{\zzcolor}{orange} 
\draw[->, line width=\zzwidth, rounded corners, color=\zzcolor] (W23)--(B24)--(W19)--(B19)--(W15)--(B14)--(W11)--(B9)--(W7)--(B4)--(W3);
\node[orange] at (5,5.2) {\Huge $x_2$};
\end{tikzpicture}
}};

\node at (10,-4) {
\scalebox{0.4}{
\begin{tikzpicture}
\basicdimerEx
\newcommand{\zzwidth}{0.2cm} 
\newcommand{\zzcolor}{green} 
\draw[->, line width=\zzwidth, rounded corners, color=\zzcolor] (W21)--(B22)--(W17)--(B17)--(W13)--(B12)--(W9)--(B6);
\draw[->, line width=\zzwidth, rounded corners, color=\zzcolor] (B10)--(W12)--(B9)--(W11)--(B8)--(W6)--(B3)--(W2);
\draw[->, line width=\zzwidth, rounded corners, color=\zzcolor] (W22)--(B23)--(W18)--(B18)--(W14)--(B13)--(W10)--(B7)--(W5)--(B2)--(W1);
\node[teal] at (5,1.2) {\Huge $x_3$}; \node at (4.5,3.9) {\Huge $\checkmark$};
\end{tikzpicture}
}};

\end{tikzpicture}
\caption{The zigzag paths of $\Gamma$.}\label{ex_large_zigzag}
\end{figure}

We now consider the extended zig-deformation of $\Gamma$ that realizes $\mut_w(P,F)$ as the PM polygon (see Theorem~\ref{mutation=deformation}).
To do this, we first fix the deformation data (see Definition~\ref{def_deformation_data}) as follows.
First, the zigzag path $z_i$ on $\Gamma$ is type I with $\ell(z_i)=8$, and its slope is $[z_i]=(0,1)=-w$ for $i=1,\dots,4$.
Let $r\coloneqq -h_{\min}=3$ and $h\coloneqq h_{\max}=1$ (see Table~\ref{mutation_deformation_data}).
These satisfy $r=3<\big|\calZ^\rmI_{(-w)}(\Gamma)\big|=4$ and $r+h=\ell(z_i)/2=4$.
We take the set of type I zigzag paths $\{z_1,z_2,z_3\}$, and consider the extended zig-deformation $\nu^\zig_\calX(\Gamma,\{z_1,z_2,z_3\})$ of $\Gamma$ at $\{z_1,z_2,z_3\}$ with respect to $\calX$,
where $\calX$ is the zig-deformation parameter (see Setting~\ref{def_deformation_exdata}) defined as follows.
To define $\calX$, we focus on $x_1$, $x_2$, $x_3$ shown in Figure~\ref{ex_large_zigzag}, which are the zigzag paths intersecting with~$z_i$ at some zags of~$z_i$.
They satisfy $m_1\coloneqq|x_1\cap z_i|=1$, $m_2\coloneqq|x_2\cap z_i|=1$, and $m_3\coloneqq|x_3\cap z_i|=2$ for any $i$;
thus we consider the set of sub-zigzag paths $\big\{x_1=x_1^{(1)}, x_2=x_2^{(1)}, x_3^{(1)}, x_3^{(2)}\big\}$.
Here, we fix the intersection of $z_3$ and $x_3$ marked by $\checkmark$ in Figure~\ref{ex_large_zigzag} as the starting edge of $x_3^{(1)}$.
We then set the zig-deformation parameter $\calX\coloneqq\{X_1,X_2,X_3\}$ with respect to $\{z_1,z_2,z_3\}$,
where $X_1=\big\{x_1\cap z_1=x_1^{(1)}\cap z_1\big\}$, $X_2=\big\{x_2\cap z_2=x_2^{(1)}\cap z_2\big\}$ and $X_3=\big\{x_3^{(1)}\cap z_3, x_3^{(2)}\cap z_3\big\}$,
and hence the weight of $\calX$ is $\bfp=(0,0,1)$.

Using these deformation data, we apply the operations (zig-1)--(zig-3) to $\Gamma$.
This gives us the dimer model shown in the left of Figure~\ref{ex_large_dimer2}.
Here, we can easily see that the slopes of zigzag paths on this dimer model do not correspond bijectively to the primitive side segments of
$\mut_w(P,F)$, and thus we can not obtain Theorem~\ref{mutation=deformation} without the operations~(zig-4) and~\mbox{(zig-5)}. We therefore apply~(zig-4)~-- that is, we insert some bypasses.
Then we have the dimer model $\overline{\nu}^\zig_\calX(\Gamma)=\overline{\nu}^\zig_\calX(\Gamma, \{z_1,z_2,z_3\})$, as shown in the right of Figure~\ref{ex_large_dimer2}.
Some zigzag paths on~$\overline{\nu}^\zig_\calX(\Gamma)$ are given in Figure~\ref{ex_large_dimer3_zigzag}.
For example, the zigzag path $x_3$ on $\Gamma$ behaves as Figure~\ref{fig_observation_zigzag2} in
$\big\{x_3^{(1)}\cap z_i,\allowbreak x_3^{(2)}\cap z_i \,|\, i=1,2\big\}$ and behaves as Figure~\ref{fig_observation_zigzag3} in $X_3=\big\{x_3^{(1)}\cap z_3, x_3^{(2)}\cap z_3\big\}$
by our choice of $\calX$, and hence we obtain the zigzag path $x_3^\prime$ on $\overline{\nu}^\zig_\calX(\Gamma)$ as in Figure~\ref{ex_large_dimer3_zigzag}.
By Proposition~\ref{prop_nondegenerate}, this dimer model is non-degenerate, but it is not consistent.
Indeed, we can see that some nodes appearing on the zigzag path $x^\prime_3$ shown in Figure~\ref{ex_large_dimer3_zigzag} are not properly ordered (see Definition~\ref{def_properly}(4)).

\begin{figure}[h!]\centering
\begin{tikzpicture}
\node at (0,0) {
\scalebox{0.5}{
\begin{tikzpicture}
\newcommand{\edgewidth}{0.05cm} 
\newcommand{\nodewidth}{0.05cm} 
\newcommand{\noderad}{0.17} 
\foreach \blackname/\x/\y in
{1/1/0,2/1/2,3/1/4,4/1/6,5/1/8,6/3/0,7/3/2,8/3/4,9/3/6,10/3/8,
11/7/0,12/7/2,13/7/4,14/7/6,15/7/8,
16/9/0,17/9/2,18/9/4,19/9/6,20/9/8,21/11/0,22/11/2,23/11/4,24/11/6,25/11/8}
{\coordinate (B\blackname) at (\x,\y);}
\foreach \whitename/\x/\y in
{1/0/1,2/0/3,3/0/5,4/0/7,5/2/1,6/2/3,7/2/5,8/2/7,9/4/1,10/4/3,11/4/5,12/4/7,
13/8/1,14/8/3,15/8/5,16/8/7, 17/10/1,18/10/3,19/10/5,20/10/7,21/12/1,22/12/3,23/12/5,24/12/7}
{\coordinate (W\whitename) at (\x,\y);}

\draw[line width=\edgewidth] (0,0) rectangle (12,8);

\foreach \w/\b in {1/1,2/2,3/3,4/4,5/1,5/2,5/6,6/2,6/3,6/7,7/3,7/4,7/8,8/4,8/5,8/9}{\draw [line width=\edgewidth] (W\w)--(B\b);}
\foreach \w/\b in {9/6,9/11,10/7,10/12,11/8,11/9,11/13,12/9,12/10,12/14}{\draw [line width=\edgewidth] (W\w)--(B\b);}
\foreach \w/\b in {13/11,13/12,13/16,13/17,14/12,14/13,14/17,14/18,15/13,15/14,15/18,15/19,16/14,16/15,16/19,16/20}{\draw [line width=\edgewidth] (W\w)--(B\b);}
\foreach \w/\b in {17/16,17/17,17/21,17/22,18/17,18/18,18/22,18/23,19/18,19/19,19/23,19/24,20/19,20/20,20/24,21/21,21/22,22/22,22/23,23/23,23/24,24/24,24/25}{\draw [line width=\edgewidth] (W\w)--(B\b);}

\foreach \n in {1,2,3,4,5,6,7,8,9,10,11,12,13,14,15,16,17,18,19,20,21,22,23,24,25} {\draw [line width=\nodewidth, fill=black] (B\n) circle [radius=\noderad];}
\foreach \n in {1,2,3,4,5,6,7,8,9,10,11,12,13,14,15,16,17,18,19,20,21,22,23,24} {\draw [line width=\nodewidth, fill=white] (W\n) circle [radius=\noderad];}

\path (W9) ++(-18:1cm) coordinate (B30); \path (W9) ++(-18:2cm) coordinate (W30);
\path (W10) ++(-18:1cm) coordinate (B31); \path (W10) ++(-18:2cm) coordinate (W31);
\path (W11) ++(-18:1cm) coordinate (B32); \path (W11) ++(-18:2cm) coordinate (W32);
\path (W12) ++(-18:1cm) coordinate (B33); \path (W12) ++(-18:2cm) coordinate (W33);

\foreach \w/\b in {30/31,31/32,32/33}{\draw [line width=\edgewidth] (W\w)--(B\b);}
\draw [line width=\edgewidth] (B30)--(5.3,0); \draw [line width=\edgewidth] (W33)--(5.3,8);

\foreach \n in {30,31,32,33}{\draw [line width=\nodewidth, fill=black] (B\n) circle [radius=\noderad];}
\foreach \n in {30,31,32,33}{\draw [line width=\nodewidth, fill=white] (W\n) circle [radius=\noderad];}
\end{tikzpicture}
}};

\node at (8,0) {
\scalebox{0.5}{
\begin{tikzpicture}
\newcommand{\edgewidth}{0.05cm} 
\newcommand{\nodewidth}{0.05cm} 
\newcommand{\noderad}{0.17} 

\foreach \blackname/\x/\y in
{1/1/0,2/1/2,3/1/4,4/1/6,5/1/8,6/3/0,7/3/2,8/3/4,9/3/6,10/3/8,
11/7/0,12/7/2,13/7/4,14/7/6,15/7/8,
16/9/0,17/9/2,18/9/4,19/9/6,20/9/8,21/11/0,22/11/2,23/11/4,24/11/6,25/11/8}
{\coordinate (B\blackname) at (\x,\y);}
\foreach \whitename/\x/\y in
{1/0/1,2/0/3,3/0/5,4/0/7,5/2/1,6/2/3,7/2/5,8/2/7,9/4/1,10/4/3,11/4/5,12/4/7,
13/8/1,14/8/3,15/8/5,16/8/7, 17/10/1,18/10/3,19/10/5,20/10/7,21/12/1,22/12/3,23/12/5,24/12/7}
{\coordinate (W\whitename) at (\x,\y);}

\draw[line width=\edgewidth] (0,0) rectangle (12,8);

\foreach \w/\b in {1/1,2/2,3/3,4/4,5/1,5/2,5/6,6/2,6/3,6/7,7/3,7/4,7/8,8/4,8/5,8/9}{\draw [line width=\edgewidth] (W\w)--(B\b);}
\foreach \w/\b in {9/6,9/11,10/7,10/12,11/8,11/9,11/13,12/9,12/10,12/14}{\draw [line width=\edgewidth] (W\w)--(B\b);}
\foreach \w/\b in {13/11,13/12,13/16,13/17,14/12,14/13,14/17,14/18,15/13,15/14,15/18,15/19,16/14,16/15,16/19,16/20}{\draw [line width=\edgewidth] (W\w)--(B\b);}
\foreach \w/\b in {17/16,17/17,17/21,17/22,18/17,18/18,18/22,18/23,19/18,19/19,19/23,19/24,20/19,20/20,20/24,21/21,21/22,22/22,22/23,23/23,23/24,24/24,24/25}{\draw [line width=\edgewidth] (W\w)--(B\b);}

\path (W9) ++(-18:1cm) coordinate (B30); \path (W9) ++(-18:2cm) coordinate (W30);
\path (W10) ++(-18:1cm) coordinate (B31); \path (W10) ++(-18:2cm) coordinate (W31);
\path (W11) ++(-18:1cm) coordinate (B32); \path (W11) ++(-18:2cm) coordinate (W32);
\path (W12) ++(-18:1cm) coordinate (B33); \path (W12) ++(-18:2cm) coordinate (W33);

\draw [line width=\edgewidth] (W1)--(B2); \draw [line width=\edgewidth] (W2)--(B3); \draw [line width=\edgewidth] (W3)--(B4);
\draw [line width=\edgewidth] (W5)--(B7); \draw [line width=\edgewidth] (W6)--(B8); \draw [line width=\edgewidth] (W8)--(B10);
\draw [line width=\edgewidth] (W11)--(B33); \draw [line width=\edgewidth] (W32)--(B14); \draw [line width=\edgewidth] (W33)--(B15);
\draw [line width=\edgewidth] (W12)--(4.6,8); \draw [line width=\edgewidth] (4.6,0)--(B30);

\foreach \n in {1,2,3,4,5,6,7,8,9,10,11,12,13,14,15,16,17,18,19,20,21,22,23,24,25} {\draw [line width=\nodewidth, fill=black] (B\n) circle [radius=\noderad];}
\foreach \n in {1,2,3,4,5,6,7,8,9,10,11,12,13,14,15,16,17,18,19,20,21,22,23,24} {\draw [line width=\nodewidth, fill=white] (W\n) circle [radius=\noderad];}

\foreach \w/\b in {30/31,31/32,32/33}{\draw [line width=\edgewidth] (W\w)--(B\b);}
\draw [line width=\edgewidth] (B30)--(5.3,0); \draw [line width=\edgewidth] (W33)--(5.3,8);

\foreach \n in {30,31,32,33}{\draw [line width=\nodewidth, fill=black] (B\n) circle [radius=\noderad];}
\foreach \n in {30,31,32,33}{\draw [line width=\nodewidth, fill=white] (W\n) circle [radius=\noderad];}
\end{tikzpicture}
}};

\end{tikzpicture}

\caption{The dimer model obtained by applying (zig-1)--(zig-3) to $\Gamma$ (left), and the dimer model $\overline{\nu}^\zig_\calX(\Gamma, \{z_1,z_2,z_3\})$ obtained by applying (zig-1)--(zig-4) to $\Gamma$ (right).}\label{ex_large_dimer2}
\end{figure}

\newcommand{\deformeddimerEx}{
\foreach \blackname/\x/\y in
{1/1/0,2/1/2,3/1/4,4/1/6,5/1/8,6/3/0,7/3/2,8/3/4,9/3/6,10/3/8,
11/7/0,12/7/2,13/7/4,14/7/6,15/7/8,
16/9/0,17/9/2,18/9/4,19/9/6,20/9/8,21/11/0,22/11/2,23/11/4,24/11/6,25/11/8}
{\coordinate (B\blackname) at (\x,\y);}
\foreach \whitename/\x/\y in
{1/0/1,2/0/3,3/0/5,4/0/7,5/2/1,6/2/3,7/2/5,8/2/7,9/4/1,10/4/3,11/4/5,12/4/7,
13/8/1,14/8/3,15/8/5,16/8/7, 17/10/1,18/10/3,19/10/5,20/10/7,21/12/1,22/12/3,23/12/5,24/12/7}
{\coordinate (W\whitename) at (\x,\y);}

\draw[line width=\edgewidth] (0,0) rectangle (12,8);

\foreach \w/\b in {1/1,2/2,3/3,4/4,5/1,5/2,5/6,6/2,6/3,6/7,7/3,7/4,7/8,8/4,8/5,8/9}{\draw [line width=\edgewidth] (W\w)--(B\b);}
\foreach \w/\b in {9/6,9/11,10/7,10/12,11/8,11/9,11/13,12/9,12/10,12/14}{\draw [line width=\edgewidth] (W\w)--(B\b);}
\foreach \w/\b in {13/11,13/12,13/16,13/17,14/12,14/13,14/17,14/18,15/13,15/14,15/18,15/19,16/14,16/15,16/19,16/20}{\draw [line width=\edgewidth] (W\w)--(B\b);}
\foreach \w/\b in {17/16,17/17,17/21,17/22,18/17,18/18,18/22,18/23,19/18,19/19,19/23,19/24,20/19,20/20,20/24,21/21,21/22,22/22,22/23,23/23,23/24,24/24,24/25}{\draw [line width=\edgewidth] (W\w)--(B\b);}

\path (W9) ++(-18:1cm) coordinate (B30); \path (W9) ++(-18:2cm) coordinate (W30);
\path (W10) ++(-18:1cm) coordinate (B31); \path (W10) ++(-18:2cm) coordinate (W31);
\path (W11) ++(-18:1cm) coordinate (B32); \path (W11) ++(-18:2cm) coordinate (W32);
\path (W12) ++(-18:1cm) coordinate (B33); \path (W12) ++(-18:2cm) coordinate (W33);

\draw [line width=\edgewidth] (W1)--(B2); \draw [line width=\edgewidth] (W2)--(B3); \draw [line width=\edgewidth] (W3)--(B4);
\draw [line width=\edgewidth] (W5)--(B7); \draw [line width=\edgewidth] (W6)--(B8); \draw [line width=\edgewidth] (W8)--(B10);
\draw [line width=\edgewidth] (W11)--(B33); \draw [line width=\edgewidth] (W32)--(B14); \draw [line width=\edgewidth] (W33)--(B15);
\draw [line width=\edgewidth] (W12)--(4.6,8); \draw [line width=\edgewidth] (4.6,0)--(B30);

\foreach \n in {1,2,3,4,5,6,7,8,9,10,11,12,13,14,15,16,17,18,19,20,21,22,23,24,25} {\draw [line width=\nodewidth, fill=black] (B\n) circle [radius=\noderad];}
\foreach \n in {1,2,3,4,5,6,7,8,9,10,11,12,13,14,15,16,17,18,19,20,21,22,23,24} {\draw [line width=\nodewidth, fill=white] (W\n) circle [radius=\noderad];}

\foreach \w/\b in {30/31,31/32,32/33}{\draw [line width=\edgewidth] (W\w)--(B\b);}
\draw [line width=\edgewidth] (B30)--(5.3,0); \draw [line width=\edgewidth] (W33)--(5.3,8);

\foreach \n in {30,31,32,33}{\draw [line width=\nodewidth, fill=black] (B\n) circle [radius=\noderad];}
\foreach \n in {30,31,32,33}{\draw [line width=\nodewidth, fill=white] (W\n) circle [radius=\noderad];}
}

\begin{figure}[h!]\centering
\begin{tikzpicture}
\newcommand{\edgewidth}{0.05cm} 
\newcommand{\nodewidth}{0.05cm} 
\newcommand{\noderad}{0.17} 

\node at (0,0) {
\scalebox{0.38}{
\begin{tikzpicture}
\deformeddimerEx

\newcommand{\zzwidth}{0.2cm} 
\newcommand{\zzcolor}{blue} 
\draw[->, line width=\zzwidth, rounded corners, color=\zzcolor] (W24)--(B25);
\draw[->, line width=\zzwidth, rounded corners, color=\zzcolor] (B21)--(W17)--(B22)--(W18)--(B23)--(W19)--(B24)--(W20)--(B20)--(W16)--(B15)--(W33)--(5.3,8);
\draw[->, line width=\zzwidth, rounded corners, color=\zzcolor] (5.3,0)--(B30)--(4.6,0);
\draw[->, line width=\zzwidth, rounded corners, color=\zzcolor] (4.6,8)--(W12)--(B10)--(W8)--(B5);
\draw[->, line width=\zzwidth, rounded corners, color=\zzcolor] (B1)--(W1)--(B2)--(W2)--(B3)--(W3)--(B4)--(W4);
\node[blue] at (7,7) {\Huge $x_1^\prime$};
\end{tikzpicture}
}};

\node at (5.5,0) {
\scalebox{0.38}{
\begin{tikzpicture}
\deformeddimerEx
\newcommand{\zzwidth}{0.2cm} 
\newcommand{\zzcolor}{orange} 
\draw[->, line width=\zzwidth, rounded corners, color=\zzcolor] (W23)--(B24)--(W19)--(B19)--(W15)--(B14)--(W32)--(B33)--(W11)--(B9)--(W8)--(B10);
\draw[->, line width=\zzwidth, rounded corners, color=\zzcolor] (B6)--(W5)--(B7)--(W6)--(B8)--(W7)--(B4)--(W3);
\node[orange] at (7,5) {\Huge $x_2^\prime$};
\end{tikzpicture}
}};

\node at (11,0) {
\scalebox{0.38}{
\begin{tikzpicture}
\deformeddimerEx
\newcommand{\zzwidth}{0.2cm} 
\newcommand{\zzcolor}{green} 
\draw[->, line width=\zzwidth, rounded corners, color=\zzcolor] (W21)--(B22)--(W17)--(B17)--(W13)--(B12)--(W31)--(B32)--(W11)--(B33)--(W12)--(4.6,8);
\draw[->, line width=\zzwidth, rounded corners, color=\zzcolor] (B10)--(W12)--(B9)--(W11)--(B8)--(W6)--(B3)--(W2);
\draw[->, line width=\zzwidth, rounded corners, color=\zzcolor] (W22)--(B23)--(W18)--(B18)--(W14)--(B13)--(W32)--(B14)--(W33)--(B15);
\draw[->, line width=\zzwidth, rounded corners, color=\zzcolor] (B11)--(W30)--(B31)--(W10)--(B7)--(W5)--(B2)--(W1);
\draw[->, line width=\zzwidth, rounded corners, color=\zzcolor] (4.6,0)--(B30)--(W9)--(B6);
\node[teal] at (9,3) {\Huge $x_3^\prime$};
\end{tikzpicture}
}};

\end{tikzpicture}
\caption{The zigzag paths $x_1^\prime$, $x_2^\prime$, $x_3^\prime$ of $\overline{\nu}^\zig_\calX(\Gamma, \{z_1,z_2,z_3\})$.}
\label{ex_large_dimer3_zigzag}
\end{figure}

By Lemma~\ref{bypass_removable}, some edges constituting the zigzag paths $x^\prime_1$, $x^\prime_2$, $x^\prime_3$ shown in Figure~\ref{ex_large_dimer3_zigzag} might be removed by the operation (zig-5).
Thus, paying attention to these zigzag paths, we apply~(zig-5) to $\overline{\nu}^\zig_\calX(\Gamma)$ and make it consistent.
Namely, we remove pairs of edges (1)--(5) shown in the type~A (left) of Figure~\ref{ex_large_dimer3}, which are the intersections of pairs of zigzag paths on the universal cover that intersect with each other in the same direction more than once, from $\overline{\nu}^\zig_\calX(\Gamma)$ with this order. Then we have the dimer model shown in the left of Figure~\ref{ex_large_dimer4}. Applying (join), we finally have the dimer model $\nu^\zig_\calX(\Gamma, \{z_1,z_2,z_3\})$ shown in the right of Fi\-gu\-re~\ref{ex_large_dimer4}.\looseness=1

There are also other ways to remove edges. For example, if we remove pairs of edges (1)--(5) shown in the type B (right) of Fi\-gu\-re~\ref{ex_large_dimer3},
then we have the dimer model shown in the left of Figure~\ref{ex_large_dimer5},
and applying (join) we have the dimer model $\nu^\zig_\calX(\Gamma, \{z_1,z_2,z_3\})$ shown in the right of Figure~\ref{ex_large_dimer5}.

\begin{figure}[h!]\centering
\begin{tikzpicture}
\newcommand{\edgewidth}{0.05cm} 
\newcommand{\nodewidth}{0.05cm} 
\newcommand{\noderad}{0.17} 

\node at (0,0)
{\scalebox{0.5}{
\begin{tikzpicture}
\deformeddimerEx
\draw[line width=\edgewidth] (W5)--(B7) node[fill=white,inner sep=0.5pt, circle, midway,xshift=0cm,yshift=0cm] {$(1)$};
\draw [line width=\nodewidth, fill=white] (W5) circle [radius=\noderad];
\draw[line width=\edgewidth] (W11)--(B33) node[fill=white,inner sep=0.5pt, circle, midway,xshift=0cm,yshift=0cm] {$(1)$};
\draw [line width=\nodewidth, fill=white] (W11) circle [radius=\noderad];
\draw[line width=\edgewidth] (W6)--(B8) node[fill=white,inner sep=0.5pt, circle, midway,xshift=0cm,yshift=0cm] {$(2)$};
\draw [line width=\nodewidth, fill=white] (W6) circle [radius=\noderad];
\draw[line width=\edgewidth] (W32)--(B14) node[fill=white,inner sep=0.5pt, circle, midway,xshift=0cm,yshift=0cm] {$(2)$};
\draw [line width=\nodewidth, fill=white] (W32) circle [radius=\noderad];
\draw[line width=\edgewidth] (W12)--(B10) node[fill=white,inner sep=0.5pt, circle, midway,xshift=0cm,yshift=0cm] {$(3)$};
\draw [line width=\nodewidth, fill=white] (W12) circle [radius=\noderad];
\draw[line width=\edgewidth] (W12)--(4.6,8) node[fill=white,inner sep=0.5pt, circle, midway,xshift=0cm,yshift=0cm] {$(3)$};
\draw [line width=\nodewidth, fill=white] (W12) circle [radius=\noderad];
\draw[line width=\edgewidth] (W33)--(B15) node[fill=white,inner sep=0.5pt, circle, midway,xshift=0cm,yshift=0cm] {$(4)$};
\draw [line width=\nodewidth, fill=white] (W33) circle [radius=\noderad];
\draw[line width=\edgewidth] (W19)--(B24) node[fill=white,inner sep=0.5pt, circle, midway,xshift=0cm,yshift=0cm] {$(4)$};
\draw [line width=\nodewidth, fill=white] (W19) circle [radius=\noderad];
\draw[line width=\edgewidth] (W1)--(B2) node[fill=white,inner sep=0.5pt, circle, midway,xshift=0cm,yshift=0cm] {$(5)$};
\draw [line width=\nodewidth, fill=white] (W1) circle [radius=\noderad];
\draw[line width=\edgewidth] (W8)--(B10) node[fill=white,inner sep=0.5pt, circle, midway,xshift=0cm,yshift=0cm] {$(5)$};
\draw [line width=\nodewidth, fill=white] (W8) circle [radius=\noderad];
\end{tikzpicture}
}};

\node at (8,0)
{\scalebox{0.5}{
\begin{tikzpicture}
\deformeddimerEx
\draw[line width=\edgewidth] (W5)--(B7) node[fill=white,inner sep=0.5pt, circle, midway,xshift=0cm,yshift=0cm] {$(1)$};
\draw [line width=\nodewidth, fill=white] (W5) circle [radius=\noderad];
\draw[line width=\edgewidth] (W11)--(B33) node[fill=white,inner sep=0.5pt, circle, midway,xshift=0cm,yshift=0cm] {$(1)$};
\draw [line width=\nodewidth, fill=white] (W11) circle [radius=\noderad];
\draw[line width=\edgewidth] (W6)--(B8) node[fill=white,inner sep=0.5pt, circle, midway,xshift=0cm,yshift=0cm] {$(2)$};
\draw [line width=\nodewidth, fill=white] (W6) circle [radius=\noderad];
\draw[line width=\edgewidth] (W32)--(B14) node[fill=white,inner sep=0.5pt, circle, midway,xshift=0cm,yshift=0cm] {$(2)$};
\draw [line width=\nodewidth, fill=white] (W32) circle [radius=\noderad];
\draw[line width=\edgewidth] (W12)--(4.6,8) node[fill=white,inner sep=0.5pt, circle, midway,xshift=0cm,yshift=0cm] {$(3)$};
\draw [line width=\nodewidth, fill=white] (W12) circle [radius=\noderad];
\draw[line width=\edgewidth] (W18)--(B23) node[fill=white,inner sep=0.5pt, circle, midway,xshift=0cm,yshift=0cm] {$(3)$};
\draw [line width=\nodewidth, fill=white] (W18) circle [radius=\noderad];
\draw[line width=\edgewidth] (W33)--(B15) node[fill=white,inner sep=0.5pt, circle, midway,xshift=0cm,yshift=0cm] {$(4)$};
\draw [line width=\nodewidth, fill=white] (W33) circle [radius=\noderad];
\draw[line width=\edgewidth] (W19)--(B24) node[fill=white,inner sep=0.5pt, circle, midway,xshift=0cm,yshift=0cm] {$(4)$};
\draw [line width=\nodewidth, fill=white] (W19) circle [radius=\noderad];
\draw[line width=\edgewidth] (W8)--(B10) node[fill=white,inner sep=0.5pt, circle, midway,xshift=0cm,yshift=0cm] {$(5)$};
\draw [line width=\nodewidth, fill=white] (W8) circle [radius=\noderad];
\draw[line width=\edgewidth] (W17)--(B22) node[fill=white,inner sep=0.5pt, circle, midway,xshift=0cm,yshift=0cm] {$(5)$};
\draw [line width=\nodewidth, fill=white] (W17) circle [radius=\noderad];
\end{tikzpicture}
}};

\node at (0,-2.5) {Type A}; \node at (8,-2.5) {Type B};

\end{tikzpicture}
\caption{The two ways to remove edges from $\overline{\nu}^\zig_\calX(\Gamma, \{z_1,z_2,z_3\})$ by (zig-5).}\label{ex_large_dimer3}
\end{figure}

\begin{figure}[h!]\centering
\begin{tikzpicture}
\newcommand{\edgewidth}{0.05cm} 
\newcommand{\nodewidth}{0.05cm} 
\newcommand{\noderad}{0.17} 

\node at (0,0)
{\scalebox{0.5}{
\begin{tikzpicture}
\foreach \blackname/\x/\y in
{1/1/0,2/1/2,3/1/4,4/1/6,5/1/8,6/3/0,7/3/2,8/3/4,9/3/6,10/3/8,
11/7/0,12/7/2,13/7/4,14/7/6,15/7/8,
16/9/0,17/9/2,18/9/4,19/9/6,20/9/8,21/11/0,22/11/2,23/11/4,24/11/6,25/11/8}
{\coordinate (B\blackname) at (\x,\y);}
\foreach \whitename/\x/\y in
{1/0/1,2/0/3,3/0/5,4/0/7,5/2/1,6/2/3,7/2/5,8/2/7,9/4/1,10/4/3,11/4/5,12/4/7,
13/8/1,14/8/3,15/8/5,16/8/7, 17/10/1,18/10/3,19/10/5,20/10/7,21/12/1,22/12/3,23/12/5,24/12/7}
{\coordinate (W\whitename) at (\x,\y);}

\draw[line width=\edgewidth] (0,0) rectangle (12,8);
\foreach \w/\b in {1/1,2/2,3/3,4/4,5/1,5/2,5/6,6/2,6/3,6/7,7/3,7/4,7/8,8/4,8/5,8/9}{\draw [line width=\edgewidth] (W\w)--(B\b);}
\foreach \w/\b in {9/6,9/11,10/7,10/12,11/8,11/9,11/13,12/9,12/14}{\draw [line width=\edgewidth] (W\w)--(B\b);}
\foreach \w/\b in {13/11,13/12,13/16,13/17,14/12,14/13,14/17,14/18,15/13,15/14,15/18,15/19,16/14,16/15,16/19,16/20}{\draw [line width=\edgewidth] (W\w)--(B\b);}
\foreach \w/\b in {17/16,17/17,17/21,17/22,18/17,18/18,18/22,18/23,19/18,19/19,19/23,20/19,20/20,20/24,21/21,21/22,22/22,22/23,23/23,23/24,24/24,24/25}{\draw [line width=\edgewidth] (W\w)--(B\b);}
\path (W9) ++(-18:1cm) coordinate (B30); \path (W9) ++(-18:2cm) coordinate (W30);
\path (W10) ++(-18:1cm) coordinate (B31); \path (W10) ++(-18:2cm) coordinate (W31);
\path (W11) ++(-18:1cm) coordinate (B32); \path (W11) ++(-18:2cm) coordinate (W32);
\path (W12) ++(-18:1cm) coordinate (B33); \path (W12) ++(-18:2cm) coordinate (W33);

\draw [line width=\edgewidth] (W2)--(B3); \draw [line width=\edgewidth] (W3)--(B4);

\foreach \n in {1,2,3,4,5,6,7,8,9,10,11,12,13,14,15,16,17,18,19,20,21,22,23,24,25} {\draw [line width=\nodewidth, fill=black] (B\n) circle [radius=\noderad];}
\foreach \n in {1,2,3,4,5,6,7,8,9,10,11,12,13,14,15,16,17,18,19,20,21,22,23,24} {\draw [line width=\nodewidth, fill=white] (W\n) circle [radius=\noderad];}

\foreach \w/\b in {30/31,31/32,32/33}{\draw [line width=\edgewidth] (W\w)--(B\b);}
\draw [line width=\edgewidth] (B30)--(5.3,0); \draw [line width=\edgewidth] (W33)--(5.3,8);

\foreach \n in {30,31,32,33}{\draw [line width=\nodewidth, fill=black] (B\n) circle [radius=\noderad];}
\foreach \n in {30,31,32,33}{\draw [line width=\nodewidth, fill=white] (W\n) circle [radius=\noderad];}
\end{tikzpicture}
}};

\node at (9,0)
{\scalebox{0.5}{
\begin{tikzpicture}
\foreach \blackname/\x/\y in
{1/1/0,2/1/2,3/1/4,4/1/6,5/1/8,6/3/0,7/3/2,8/3/4,9/3/6,10/3/8,
11/5/0,12/5/2,13/5/4,14/5/6,15/5/8,
16/7/0,17/7/2,18/7/4,19/7/6,20/7/8,21/9/0,22/9/2,23/9/4,24/9/6,25/9/8}
{\coordinate (B\blackname) at (\x,\y);}
\foreach \whitename/\x/\y in
{1/0/1,2/0/3,3/0/5,4/0/7,5/2/1,6/2/3,7/2/5,8/2/7,9/4/1,10/4/3,11/4/5,12/4/7,
13/6/1,14/6/3,15/6/5,16/6/7, 17/8/1,18/8/3,19/8/5,20/8/7,21/10/1,22/10/3,23/10/5,24/10/7}
{\coordinate (W\whitename) at (\x,\y);}

\draw[line width=\edgewidth] (0,0) rectangle (10,8);
\foreach \w/\b in {1/1,2/2,2/3,3/3,3/4,4/4,5/1,5/2,6/2,6/3,7/3,7/4,8/4,8/5}{\draw [line width=\edgewidth] (W\w)--(B\b);}
\foreach \w/\b in {5/6,6/7,7/8,7/9,8/9,9/6,9/7,10/7,10/8,11/8,11/9,12/9,12/10}{\draw [line width=\edgewidth] (W\w)--(B\b);}
\foreach \w/\b in {9/11,10/12,11/13,12/14}{\draw [line width=\edgewidth] (W\w)--(B\b);}
\foreach \w/\b in {13/11,13/12,13/16,13/17,14/12,14/13,14/17,14/18,15/13,15/14,15/18,15/19,16/14,16/15,16/19,16/20}{\draw [line width=\edgewidth] (W\w)--(B\b);}
\foreach \w/\b in {17/16,17/17,17/21,17/22,18/17,18/18,18/22,18/23,19/18,19/19,19/23,20/19,20/20,20/24,21/21,21/22,22/22,22/23,23/23,23/24,24/24,24/25}{\draw [line width=\edgewidth] (W\w)--(B\b);}

\foreach \n in {1,2,3,4,5,6,7,8,9,10,11,12,13,14,15,16,17,18,19,20,21,22,23,24,25} {\draw [line width=\nodewidth, fill=black] (B\n) circle [radius=\noderad];}
\foreach \n in {1,2,3,4,5,6,7,8,9,10,11,12,13,14,15,16,17,18,19,20,21,22,23,24} {\draw [line width=\nodewidth, fill=white] (W\n) circle [radius=\noderad];}
\end{tikzpicture}
}};

\draw[->,line width=0.03cm] (3.5,0)--(6,0);
\node at (4.75,0.35) {(join)};

\end{tikzpicture}
\caption{The dimer model $\nu^\zig_\calX(\Gamma, \{z_1,z_2,z_3\})$ obtained from the type A of Figure~\ref{ex_large_dimer3}.}
\label{ex_large_dimer4}
\end{figure}

\begin{figure}[h!]\centering
\begin{tikzpicture}
\newcommand{\edgewidth}{0.05cm} 
\newcommand{\nodewidth}{0.05cm} 
\newcommand{\noderad}{0.17} 

\node at (0,0)
{\scalebox{0.5}{
\begin{tikzpicture}
\foreach \blackname/\x/\y in
{1/1/0,2/1/2,3/1/4,4/1/6,5/1/8,6/3/0,7/3/2,8/3/4,9/3/6,10/3/8,
11/7/0,12/7/2,13/7/4,14/7/6,15/7/8,
16/9/0,17/9/2,18/9/4,19/9/6,20/9/8,21/11/0,22/11/2,23/11/4,24/11/6,25/11/8}
{\coordinate (B\blackname) at (\x,\y);}
\foreach \whitename/\x/\y in
{1/0/1,2/0/3,3/0/5,4/0/7,5/2/1,6/2/3,7/2/5,8/2/7,9/4/1,10/4/3,11/4/5,12/4/7,
13/8/1,14/8/3,15/8/5,16/8/7, 17/10/1,18/10/3,19/10/5,20/10/7,21/12/1,22/12/3,23/12/5,24/12/7}
{\coordinate (W\whitename) at (\x,\y);}
\draw[line width=\edgewidth] (0,0) rectangle (12,8);
\foreach \w/\b in {1/1,2/2,3/3,4/4,5/1,5/2,5/6,6/2,6/3,6/7,7/3,7/4,7/8,8/4,8/5,8/9}{\draw [line width=\edgewidth] (W\w)--(B\b);}
\foreach \w/\b in {9/6,9/11,10/7,10/12,11/8,11/9,11/13,12/9,12/14}{\draw [line width=\edgewidth] (W\w)--(B\b);}
\foreach \w/\b in {13/11,13/12,13/16,13/17,14/12,14/13,14/17,14/18,15/13,15/14,15/18,15/19,16/14,16/15,16/19,16/20}{\draw [line width=\edgewidth] (W\w)--(B\b);}
\foreach \w/\b in {17/16,17/17,17/21,18/17,18/18,18/22,19/18,19/19,19/23,20/19,20/20,20/24,21/21,21/22,22/22,22/23,23/23,23/24,24/24,24/25}{\draw [line width=\edgewidth] (W\w)--(B\b);}
\path (W9) ++(-18:1cm) coordinate (B30); \path (W9) ++(-18:2cm) coordinate (W30);
\path (W10) ++(-18:1cm) coordinate (B31); \path (W10) ++(-18:2cm) coordinate (W31);
\path (W11) ++(-18:1cm) coordinate (B32); \path (W11) ++(-18:2cm) coordinate (W32);
\path (W12) ++(-18:1cm) coordinate (B33); \path (W12) ++(-18:2cm) coordinate (W33);

\draw [line width=\edgewidth] (W1)--(B2); \draw [line width=\edgewidth] (W2)--(B3); \draw [line width=\edgewidth] (W3)--(B4);
\draw [line width=\edgewidth] (W12)--(B10);

\foreach \n in {1,2,3,4,5,6,7,8,9,10,11,12,13,14,15,16,17,18,19,20,21,22,23,24,25} {\draw [line width=\nodewidth, fill=black] (B\n) circle [radius=\noderad];}
\foreach \n in {1,2,3,4,5,6,7,8,9,10,11,12,13,14,15,16,17,18,19,20,21,22,23,24} {\draw [line width=\nodewidth, fill=white] (W\n) circle [radius=\noderad];}

\foreach \w/\b in {30/31,31/32,32/33}{\draw [line width=\edgewidth] (W\w)--(B\b);}
\draw [line width=\edgewidth] (B30)--(5.3,0); \draw [line width=\edgewidth] (W33)--(5.3,8);

\foreach \n in {30,31,32,33}{\draw [line width=\nodewidth, fill=black] (B\n) circle [radius=\noderad];}
\foreach \n in {30,31,32,33}{\draw [line width=\nodewidth, fill=white] (W\n) circle [radius=\noderad];}
\end{tikzpicture}
}};

\node at (9,0)
{\scalebox{0.5}{
\begin{tikzpicture}
\foreach \blackname/\x/\y in
{1/1/0,2/1/2,3/1/4,4/1/6,5/1/8,6/3/6, 7/5/0,8/5/2,9/5/4,10/5/6,11/5/8,
12/7/0,13/7/2,14/7/4,15/7/6,16/7/8,
17/9/0,18/9/2,19/9/4,20/9/6,21/9/8,22/11/0,23/11/2,24/11/4,25/11/6,26/11/8}
{\coordinate (B\blackname) at (\x,\y);}
\foreach \whitename/\x/\y in
{1/0/1,2/0/3,3/0/5,4/0/7,5/2/1,6/2/3,7/2/5,8/2/7,9/4/7,10/6/1,11/6/3,12/6/5,13/6/7,
14/8/1,15/8/3,16/8/5,17/8/7, 18/10/1,19/10/3,20/10/5,21/10/7,22/12/1,23/12/3,24/12/5,25/12/7}
{\coordinate (W\whitename) at (\x,\y);}
\draw[line width=\edgewidth] (0,0) rectangle (12,8);
\foreach \w/\b in {1/1,1/2,2/2,2/3,3/3,3/4,4/4,5/1,5/2,5/7,6/2,6/3,6/8,7/3,7/4,7/6,7/9,8/4,8/5,8/6,
9/6,9/10,9/11,10/7,10/8,10/12,11/8,11/9,11/13,12/9,12/10,12/14,13/10,13/11,13/15,
14/12,14/13,14/17,14/18,15/13,15/14,15/18,15/19,16/14,16/15,16/19,16/20,17/15,17/16,17/20,17/21,
18/17,18/18,18/22,19/18,19/19,19/23,20/19,20/20,20/24,21/20,21/21,21/25,
22/22,22/23,23/23,23/24,24/24,24/25,25/25,25/26}{\draw [line width=\edgewidth] (W\w)--(B\b);}

\foreach \n in {1,2,3,4,5,6,7,8,9,10,11,12,13,14,15,16,17,18,19,20,21,22,23,24,25,26} {\draw [line width=\nodewidth, fill=black] (B\n) circle [radius=\noderad];}
\foreach \n in {1,2,3,4,5,6,7,8,9,10,11,12,13,14,15,16,17,18,19,20,21,22,23,24,25} {\draw [line width=\nodewidth, fill=white] (W\n) circle [radius=\noderad];}

\end{tikzpicture}
}};

\draw[->,line width=0.03cm] (3.5,0)--(5.5,0);
\node at (4.5,0.35) {(join)};

\end{tikzpicture}
\caption{The dimer model $\nu^\zig_\calX(\Gamma, \{z_1,z_2,z_3\})$ obtained from the type B of Figure~\ref{ex_large_dimer3}.}
\label{ex_large_dimer5}
\end{figure}

Let $\Gamma_A$ (resp.~$\Gamma_B$) be the deformed dimer model shown in the right of Figure~\ref{ex_large_dimer4} (resp.\ Fi\-gu\-re~\ref{ex_large_dimer5}).
We can check that $\Gamma_A$ and $\Gamma_B$ are not isomorphic, but they are mutation-equivalent.
Indeed, by applying the mutations of $\Gamma_B$ at the faces $1,\dots,10$ in this order (see Figure~\ref{ex_mutation_equiv_AB}), we recover $\Gamma_A$.
Here, we recall that a mutation of a dimer model can be defined at any quadrangle face.
Although some faces indexed by $\{1,\dots,10\}$ are not quadrangle, such faces will become quadrangles during the process of this series of mutations.

The PM polygons of the dimer models $\Gamma_A$ and $\Gamma_B$ are the same (see Proposition~\ref{mutation_preserve_toric}),
and it coincides with the lattice polygon shown on the right-hand side of Figure~\ref{ex_large_polygon} by Theorem~\ref{mutation=deformation}.\looseness=-1

\begin{figure}[h!]\centering
\begin{tikzpicture}
\newcommand{\edgewidth}{0.05cm} 
\newcommand{\nodewidth}{0.05cm} 
\newcommand{\noderad}{0.17} 

\node at (0,0)
{\scalebox{0.45}{
\begin{tikzpicture}
\foreach \blackname/\x/\y in
{1/1/0,2/1/2,3/1/4,4/1/6,5/1/8,6/3/6, 7/5/0,8/5/2,9/5/4,10/5/6,11/5/8,
12/7/0,13/7/2,14/7/4,15/7/6,16/7/8,
17/9/0,18/9/2,19/9/4,20/9/6,21/9/8,22/11/0,23/11/2,24/11/4,25/11/6,26/11/8}
{\coordinate (B\blackname) at (\x,\y);}
\foreach \whitename/\x/\y in
{1/0/1,2/0/3,3/0/5,4/0/7,5/2/1,6/2/3,7/2/5,8/2/7,9/4/7,10/6/1,11/6/3,12/6/5,13/6/7,
14/8/1,15/8/3,16/8/5,17/8/7, 18/10/1,19/10/3,20/10/5,21/10/7,22/12/1,23/12/3,24/12/5,25/12/7}
{\coordinate (W\whitename) at (\x,\y);}
\draw[line width=\edgewidth] (0,0) rectangle (12,8);
\foreach \w/\b in {1/1,1/2,2/2,2/3,3/3,3/4,4/4,5/1,5/2,5/7,6/2,6/3,6/8,7/3,7/4,7/6,7/9,8/4,8/5,8/6,
9/6,9/10,9/11,10/7,10/8,10/12,11/8,11/9,11/13,12/9,12/10,12/14,13/10,13/11,13/15,
14/12,14/13,14/17,14/18,15/13,15/14,15/18,15/19,16/14,16/15,16/19,16/20,17/15,17/16,17/20,17/21,
18/17,18/18,18/22,19/18,19/19,19/23,20/19,20/20,20/24,21/20,21/21,21/25,
22/22,22/23,23/23,23/24,24/24,24/25,25/25,25/26}{\draw [line width=\edgewidth] (W\w)--(B\b);}

\foreach \n in {1,2,3,4,5,6,7,8,9,10,11,12,13,14,15,16,17,18,19,20,21,22,23,24,25,26} {\draw [line width=\nodewidth, fill=black] (B\n) circle [radius=\noderad];}
\foreach \n in {1,2,3,4,5,6,7,8,9,10,11,12,13,14,15,16,17,18,19,20,21,22,23,24,25} {\draw [line width=\nodewidth, fill=white] (W\n) circle [radius=\noderad];}

\node at (1,1) {\LARGE $1$}; \node at (3,7.5) {\LARGE $2$}; \node at (4,5.5) {\LARGE $3$}; \node at (6.5,3.5) {\LARGE $4$};
\node at (8,2) {\LARGE $5$}; \node at (9,3) {\LARGE $6$}; \node at (5,7) {\LARGE $7$}; \node at (6.5,5.5) {\LARGE $8$};
\node at (8,4) {\LARGE $9$}; \node at (9,5) {\LARGE $10$};
\end{tikzpicture}
}};

\node at (9,0)
{\scalebox{0.45}{
\begin{tikzpicture}
\foreach \blackname/\x/\y in
{1/1/0,2/1/2,3/1/4,4/1/6,5/1/8,6/3/0,7/3/2,8/3/4,9/3/6,10/3/8,
11/5/0,12/5/2,13/5/4,14/5/6,15/5/8,
16/7/0,17/7/2,18/7/4,19/7/6,20/7/8,21/9/0,22/9/2,23/9/4,24/9/6,25/9/8}
{\coordinate (B\blackname) at (\x,\y);}
\foreach \whitename/\x/\y in
{1/0/1,2/0/3,3/0/5,4/0/7,5/2/1,6/2/3,7/2/5,8/2/7,9/4/1,10/4/3,11/4/5,12/4/7,
13/6/1,14/6/3,15/6/5,16/6/7, 17/8/1,18/8/3,19/8/5,20/8/7,21/10/1,22/10/3,23/10/5,24/10/7}
{\coordinate (W\whitename) at (\x,\y);}

\draw[line width=\edgewidth] (0,0) rectangle (10,8);
\foreach \w/\b in {1/1,2/2,2/3,3/3,3/4,4/4,5/1,5/2,6/2,6/3,7/3,7/4,8/4,8/5}{\draw [line width=\edgewidth] (W\w)--(B\b);}
\foreach \w/\b in {5/6,6/7,7/8,7/9,8/9,9/6,9/7,10/7,10/8,11/8,11/9,12/9,12/10}{\draw [line width=\edgewidth] (W\w)--(B\b);}
\foreach \w/\b in {9/11,10/12,11/13,12/14}{\draw [line width=\edgewidth] (W\w)--(B\b);}
\foreach \w/\b in {13/11,13/12,13/16,13/17,14/12,14/13,14/17,14/18,15/13,15/14,15/18,15/19,16/14,16/15,16/19,16/20}{\draw [line width=\edgewidth] (W\w)--(B\b);}
\foreach \w/\b in {17/16,17/17,17/21,17/22,18/17,18/18,18/22,18/23,19/18,19/19,19/23,20/19,20/20,20/24,21/21,21/22,22/22,22/23,23/23,23/24,24/24,24/25}{\draw [line width=\edgewidth] (W\w)--(B\b);}
\foreach \n in {1,2,3,4,5,6,7,8,9,10,11,12,13,14,15,16,17,18,19,20,21,22,23,24,25} {\draw [line width=\nodewidth, fill=black] (B\n) circle [radius=\noderad];}
\foreach \n in {1,2,3,4,5,6,7,8,9,10,11,12,13,14,15,16,17,18,19,20,21,22,23,24} {\draw [line width=\nodewidth, fill=white] (W\n) circle [radius=\noderad];}
\end{tikzpicture}
}};

\draw[->,line width=0.03cm] (3.2,0)--(6.3,0);
\node at (4.75,0.35) {\footnotesize The mutation $\mu_f$};
\node at (4.75,-0.35) {\footnotesize where $f=1,\dots,10$};

\end{tikzpicture}

\caption{The mutations of $\Gamma_B$ (left) at the faces $1,\dots,10$ induce $\Gamma_A$ (right).}\label{ex_mutation_equiv_AB}
\end{figure}
\end{Example}

\subsection*{Acknowledgements}
The authors would like to thank Alexander Kasprzyk for valuable lectures and discussions on the combinatorial mutation of Fano polygons.
The authors would also like to thank the anonymous referees for their numerous valuable comments and suggestions.
The first author is supported by JSPS Grant-in-Aid for Scientific Research (C) 20K03513.
The second author was supported by World Premier International Research Center Initiative (WPI initiative), MEXT, Japan, and is supported by JSPS Grant-in-Aid for Early-Career Scientists 20K14279.

\LastPageEnding

\end{document}